\documentclass[a4paper, bibliography=totoc, twoside]{scrreprt}

\usepackage[utf8]{inputenc}
\usepackage{lmodern}
\usepackage[english]{babel} 

\usepackage{pdfpages}
\usepackage{microtype}

\usepackage{amsmath, amssymb, amsthm, amsfonts, mathtools, nicefrac}
\usepackage{amsbsy} 
\usepackage{braket} 
\usepackage{floatflt} 
\usepackage{enumitem}  
\usepackage{array} 
\usepackage{varwidth}
\usepackage[autostyle]{csquotes} 
\setlist[enumerate]{label=\alph*)} 
\usepackage{hhline}
\usepackage{url}

\let\OLDthebibliography\thebibliography
\renewcommand\thebibliography[1]{
	\OLDthebibliography{#1}
	\setlength{\parskip}{3pt}
	\setlength{\itemsep}{3pt plus 0.3ex}
}

\usepackage[type={CC}, modifier={by-sa}, version={4.0}]{doclicense}

\usepackage{graphicx}
\usepackage{xcolor}
\definecolor{Red}{RGB}{228,26,28}
\definecolor{Green}{RGB}{30,188,77}
\definecolor{Blue}{RGB}{49, 147, 226}
\definecolor{Orange}{RGB}{255,127,0}
\definecolor{lightOrange}{RGB}{255, 183, 112}
\definecolor{Purple}{RGB}{152,78,163}
\definecolor{Lightgray}{RGB}{200,200,200}
\definecolor{Gray}{RGB}{175,175,175}
\definecolor{Superlightgray}{RGB}{225,225,225}

\usepackage{tikz}
\usetikzlibrary{matrix,arrows}   
  \usetikzlibrary{intersections}
\usetikzlibrary{decorations.pathmorphing,decorations.pathreplacing,decorations.markings}

\usetikzlibrary{calc}    
\usetikzlibrary{shapes.misc}
\usetikzlibrary{cd} 
\usetikzlibrary{patterns}
\usetikzlibrary{shapes.geometric} 
\usetikzlibrary{positioning}
\usetikzlibrary{patterns}
\usetikzlibrary{backgrounds} 
\pgfdeclarelayer{foreground}
\pgfdeclarelayer{background}
\pgfsetlayers{background,main,foreground}

\usepackage[format=plain,margin=30pt,labelfont=bf, labelsep=space]{caption} 
\usepackage{todo}
\DeclareGraphicsExtensions{.pdf, .jpg, .png}
\usepackage{hyperref}

\usepackage{tikz-cd}
\usetikzlibrary{cd}

\usepackage[noabbrev,capitalize]{cleveref} 

\theoremstyle{plain}
\newtheorem{thm}{Theorem}[chapter]

\newtheorem*{thm*}{Theorem}
\newtheorem{prop}[thm]{Proposition}
\newtheorem{lem}[thm]{Lemma}
\newtheorem{cor}[thm]{Corollary}
\newtheorem{oq}[thm]{Open Question}
\newtheorem{conj}[thm]{Conjecture}
\newtheorem{intro}{Theorem}

\newtheorem{shifting}[thm]{Shifting Property}
\newtheorem{subword}[thm]{Subword Property}
\newtheorem{prefix}[thm]{Prefix Property}
\newtheorem{carter}[thm]{Carter's Lemma}
\newtheorem{nohole}[thm]{Link Property}

\newtheorem{gromov}[thm]{Gromov's Link Criterion}
\newtheorem{bowditch}[thm]{Bowditch's Criterion}
\newtheorem{unique}[thm]{Unique Geodesics Criterion}
\newtheorem{curvcon}[thm]{Curvature Conjecture}

\theoremstyle{definition}
\newtheorem{defi}[thm]{Definition}
\newtheorem{exa}[thm]{Example}

\newtheorem{con}[thm]{Convention}

\theoremstyle{remark}
\newtheorem{rem}[thm]{Remark}
\newtheorem{nota}[thm]{Notation}
\newtheorem{obs}[thm]{Observation}

\numberwithin{equation}{thm}

\newcommand{\A}{\mathcal{A}}

\newcommand{\C}{\mathcal{C}}
\newcommand{\E}{\mathcal{E}}
\newcommand{\F}{\mathbb{F}}
\newcommand{\Z}{\mathbb{Z}}
\newcommand{\R}{\mathbb{R}}
\newcommand{\N}{\mathbb{N}}
\newcommand{\B}{\mathcal{B}}
\newcommand{\D}{\mathcal{D}}
\renewcommand{\S}{\mathcal{S}}
\newcommand{\DD}{\hat{\mathcal{D}}}
\newcommand{\De}{\D^\ast}
\newcommand{\gl}{\mathrm{GL}}
\newcommand{\di}{\mathrm{d}}
\newcommand{\din}{\di_{\nc}}
\newcommand{\dip}{\di_{\pn}}
\newcommand{\dig}{\di_{\si}}
\newcommand{\eps}{\varepsilon}
\newcommand{\inv}{^{-1}}
\newcommand{\com}{^{\perp}}
\newcommand{\st}{^\ast}
\newcommand{\tr}{^\intercal}
\newcommand{\ith}{i^{\text{th}}}
\newcommand{\pow}{\leq_S}

\newcommand{\mi}{\hat{0}}
\newcommand{\ma}{\hat{1}}
\newcommand{\dka}{(\!(} 
\newcommand{\dkz}{)\!)}

\newcommand{\catz} {\mathrm{CAT}(0)}
\newcommand{\cato} {\mathrm{CAT}(1)}
\newcommand{\cox}{c}
\renewcommand{\l}{l}
\renewcommand{\r}{r}
\newcommand{\auto}{\varphi}
\newcommand{\ncauto}{\varphi}
\newcommand{\vecauto}{\psi}
\newcommand{\lamauto}{\Psi}
\newcommand{\antiauto}{\psi}
\newcommand{\aab}{\antiauto_\bil}
\newcommand{\ncaa}{\phi}
\newcommand{\autn}{\auto_{n}}
\newcommand{\autl}{\auto_{\l}}
\newcommand{\autr}{\auto_{\r}}
\newcommand{\autd}{\ncaa_{\l}}
\newcommand{\emb}{f}
\newcommand{\bil}{b}
\newcommand{\bila}{\bil_{A}}
\newcommand{\bilb}{\bil_{B}}
\newcommand{\bild}{\bil_{D}}
\newcommand{\bo}{\bil_1}
\newcommand{\bt}{\bil_2}
\newcommand{\bou}{\bo^U}
\newcommand{\btu}{\bt^U}
\newcommand{\subo}{\mathrm{sub}}

\newcommand{\ncb}{\mathrm{NCB}}
\newcommand{\ncd}{\mathrm{NCD}}
\newcommand{\exo}{\zeta}
\newcommand{\si}{\Delta}
\newcommand{\ord}{\mathrm{ord}}
\newcommand{\ls}{\ell_S}

\newcommand{\pal}{\alpha^{(p)}}
\newcommand{\Vp}{V^{(p)}}
\newcommand{\pmov}{\mov^{(p)}}

\newcommand{\VZ}{V_{\Z}}
\newcommand{\prho}{\rho^{(p)}}
\newcommand{\pemb}{\emb^{(p)}}
\newcommand{\hp}{^{(p)}}
\newcommand{\vs}{\F_2^{n-1}}
\newcommand{\tphi}{\tilde{\Phi}}
\newcommand{\tal}{\tilde{\alpha}}
\newcommand{\tpal}{\tilde{\alpha}^{(p)}}
\newcommand{\tmov}{\widetilde{\mov}}
\newcommand{\tpmov}{\tmov^{(p)}}
\newcommand{\ab}{\mathrm{cov}_\gtrdot} 
\newcommand{\bel}{\mathrm{cov}_\lessdot} 
\newcommand{\invn}{\mathrm{inv}}
\newcommand{\typ}{t}
\newcommand{\Typ}{T}
\newcommand{\conv}{\mathrm{conv}}
\newcommand{\convn}{\conv_{\nc}}

\newcommand{\convg}{\conv_{\si}}
\newcommand{\Cn}{\C(\ncpn)}
\newcommand{\An}{\A(\ncpn)}
\DeclareMathOperator{\ecc}{ecc} 
\DeclareMathOperator{\rad}{rad} 
\DeclareMathOperator{\diam}{diam} 
\DeclareMathOperator{\sd}{sd} 
\DeclareMathOperator{\id}{id} 
\DeclareMathOperator{\im}{im} 
\DeclareMathOperator{\rk}{rk} 
\DeclareMathOperator{\fix}{fix} 
\DeclareMathOperator{\mov}{mov} 

\DeclareMathOperator{\lk}{lk} 
\DeclareMathOperator{\sta}{st} 
\DeclareMathOperator{\osta}{st^\circ} 

\DeclareMathOperator{\aut}{Aut}
\newcommand{\nc}{\mathrm{NC}}

\newcommand{\ocncw}{|\mathrm{NC}(W)|}
\newcommand{\ocnc}{|\mathrm{NC}|}
\newcommand{\ocncpn}{|\mathrm{NCP}_n|}
\newcommand{\on}{\ocncpn}
\newcommand{\pn}{\mathrm{P}_n}
\newcommand{\p}{\mathrm{P}}
\newcommand{\op}{|\pn|}
\newcommand{\ncp}{\mathrm{NCP}}
\newcommand{\ncpn}{\mathrm{NCP}_n}
\newcommand{\str}{:\,}
\newcommand{\lam}{\Lambda}

\tikzset{kpunkt/.style={circle, fill, inner sep=0, minimum size=3pt}}

\tikzset{skpunkt/.style={circle, fill, inner sep=0, minimum size=2pt}}

\tikzset{lkpunkt/.style={circle, fill = white, draw= black, inner sep=0, minimum size=3pt}}

\tikzset{mpunkt/.style={circle, fill, inner sep=0, minimum size=5pt}}

\tikzset{gpunkt/.style={circle, fill, inner sep=0, minimum size=7pt}}

\tikzset{frage/.style = {rectangle, rounded corners, draw=black, fill=white,  text centered, align = center}}

\tikzset{info/.style = {rectangle, rounded corners, fill=black!10,   text centered, align = center}}

\tikzset{verm/.style = {rectangle, rounded corners, fill=Orange!30,   text centered, align = center}}

\tikzset{anwei/.style = {rectangle, rounded corners, fill=Blue!30,  text centered, align = center}}

\tikzset{janein/.style = {ellipse, fill=white,draw=black,  text centered, align = center}}

\tikzset{blub/.style = {circle, minimum size = 26pt, fill=white,draw=black,  text centered, align = center}}

\tikzset{fertig/.style = {ellipse, fill=Green!70, text centered, align = center}}

\tikzset{elli/.style = {rounded rectangle, rounded rectangle arc length=180, fill=black!5, inner sep = 3mm,  text centered, align = center}}

\tikzset{krei/.style = {circle, fill=KITgreen!50, inner sep = 1.5mm,  text centered, align = center}}

\tikzset{ellig/.style = {rounded rectangle, rounded rectangle arc length=180, fill=KITgreen!50, inner sep = 3mm,  text centered, align = center}}

\tikzset{elligg/.style = {rounded rectangle, rounded rectangle arc length=180, fill=KITgreen!30, inner sep = 3mm,  text centered, align = center}}

\tikzset{verband/.style = {rectangle, rounded corners, draw=KITgreen!80, very thick, inner sep = 3.8mm, text centered, align = center}}

\tikzset{verbandB/.style = {rectangle, rounded corners, draw=KITgreen!80, very thick, inner sep = 4.3mm, text centered, align = center}}

\tikzset{verbandS/.style = {rectangle, rounded corners, draw=KITgreen!80, very thick, inner sep = 1.2mm, text centered, align = center}}

\tikzset{verbandL/.style = {rectangle, rounded corners, draw=KITgreen!80, very thick, inner sep = 0.83mm, text centered, align = center}}
\newcommand{\viereck}{
	\foreach \w in {1,2,3,4} 
	\node (p\w) at (-\w * 360/4 + 135  : 4mm) [kpunkt] {};
}

\newcommand{\mpviereck}{
	\node[kpunkt] (p0) at (0,0){};
	\foreach \w in {1,2,3,4} 
	\node (p\w) at (-\w * 360/4 + 135  : 4mm) [kpunkt] {};
}

\newcommand{\sechseck}{
	\foreach \w in {1,...,6} 
	\node (p\w) at (-\w * 360/6 +60  : 4mm) [kpunkt] {};
}

\newcommand{\sechseckmp}{
	\sechseck
	\node (p0) at (0,0) [kpunkt] {};
}

\newcommand{\gsechsecklab}{
	\foreach \w in {1,...,6} 
	\node (p\w) at (-\w * 360/6 +60  : 6mm) [kpunkt] {};
	\foreach \w in {1,...,6} 
	\node (q\w) at (-\w * 360/6 +60  : 9mm)  {$\w$};
}

\newcommand{\gsechseck}{
	\foreach \w in {1,...,6} 
	\node (p\w) at (-\w * 360/6 +60  : 6mm) [kpunkt] {};
}

\newcommand{\fnfeck}{
	\foreach \w in {1,...,5} 
	\node (p\w) at (-\w * 360/5 +90  : 9mm) [kpunkt] {};
}

\newcommand{\skfnfeck}{
	\foreach \w in {1,...,5} 
	\node (p\w) at (-\w * 360/5 +90  : 4mm) [kpunkt] {};
}

\newcommand{\kfnfeck}{
	\foreach \w in {1,...,5} 
	\node (p\w) at (-\w * 360/5 +90  : 6mm) [kpunkt] {};
}

\newcommand{\achtecklab}{
	\foreach \w in {1,...,8} 
	\node (p\w) at (-\w  * 360/8 +67.5   : 8mm) [kpunkt] {};
	\foreach \w in {1,...,8} 
	\node (q\w) at (-\w  * 360/8 +67.5  : 11mm) {$\w$};
}

\newcommand{\achteckmp}{
	\foreach \w in {1,...,8} 
	\node (p\w) at (-\w  * 360/8 +67.5   : 6mm) [kpunkt] {};
	\node[kpunkt](p0) at (0,0){};
}

\newcommand{\sgsechseck}{
	\foreach \w in {1,...,6}
	\node (p\w) at (-\w * 360/6 +60  : 8mm) [kpunkt] {};
}

\newcommand{\sgsechseckmp}{
	\sgsechseck
	\node (p0) at (0,0) [kpunkt] {};
}

\newcommand{\gachteck}{
	\foreach \w in {1,...,8}
	\node (p\w) at (-\w  * 360/8 +67.5   : 8mm) [kpunkt] {};
}

\newcommand{\grauKreis}{
	\def\r{0.8} 
	
	\draw[Lightgray] (0,0) circle (\r);
	\node[mpunkt] (1) at (0:\r) {};
	\node[mpunkt] (-1) at (180:\r) {};
	\node[mpunkt] (-k) at (60:\r) {};
	\node[mpunkt] (k) at (60:-\r) {};
	
	\draw (-1) -- (-k) (k) -- (1);
	
}

\newcommand{\ez}{\viereck\draw (p1) -- (p2);}
\newcommand{\zd}{\viereck\draw (p3) -- (p2);}
\newcommand{\dv}{\viereck\draw (p3) -- (p4);}
\newcommand{\ev}{\viereck\draw (p1) -- (p4);}
\newcommand{\ed}{\viereck\draw (p1) -- (p3);}
\newcommand{\zv}{\viereck\draw (p2) -- (p4);}

\newcommand{\ezd}{\viereck\draw (p1) -- (p2) -- (p3) -- (p1);}
\newcommand{\zdv}{\viereck\draw (p4) -- (p2) -- (p3) -- (p4);}
\newcommand{\edv}{\viereck\draw (p1) -- (p3) -- (p4) -- (p1);}
\newcommand{\ezv}{\viereck\draw (p1) -- (p2) -- (p4) -- (p1);}
\newcommand{\ezudv}{\viereck\draw (p1) -- (p2);\draw(p4) -- (p3);}
\newcommand{\evuzd}{\viereck\draw (p1) -- (p4);\draw(p2) -- (p3);}

\newcommand{\pez}{\begin{tikzpicture}\ez\end{tikzpicture}}
\newcommand{\pzd}{\begin{tikzpicture}\zd\end{tikzpicture}}
\newcommand{\pdv}{\begin{tikzpicture}\dv\end{tikzpicture}}
\newcommand{\pev}{\begin{tikzpicture}\ev\end{tikzpicture}}
\newcommand{\ped}{\begin{tikzpicture}\ed\end{tikzpicture}}
\newcommand{\pzv}{\begin{tikzpicture}\zv\end{tikzpicture}}

\newcommand{\pezd}{\begin{tikzpicture}\ezd\end{tikzpicture}}
\newcommand{\pzdv}{\begin{tikzpicture}\zdv\end{tikzpicture}}
\newcommand{\pedv}{\begin{tikzpicture}\edv\end{tikzpicture}}
\newcommand{\pezv}{\begin{tikzpicture}\ezv\end{tikzpicture}}
\newcommand{\pezudv}{\begin{tikzpicture}\ezudv\end{tikzpicture}}
\newcommand{\pevuzd}{\begin{tikzpicture}\evuzd\end{tikzpicture}}
\newcommand{\peduzv}{\begin{tikzpicture}\foreach \w in {1,2,3,4} 
	\node (p\w) at (-\w * 360/4 + 135  : 4mm) [kpunkt, Lightgray] {};\draw[, Lightgray] (p1) -- (p3)(p2) -- (p4);\end{tikzpicture}}

\newcommand{\pfull}{\begin{tikzpicture}\viereck\draw (p1) -- (p2)(p2) -- (p3)(p3) -- (p4)(p4) -- (p1);\end{tikzpicture}}

\newcommand{\Bez}{\sechseck\draw (p1) -- (p2);\draw (p5) -- (p4);}
\newcommand{\Bzd}{\begin{scope}[rotate = -60] \Bez\end{scope}}
\newcommand{\Bemd}{\begin{scope}[rotate = -120] \Bez\end{scope}}
\newcommand{\Be}{\sechseck\draw (p1) -- (p4);}
\newcommand{\Bz}{\begin{scope}[rotate = -60] \Be\end{scope}}
\newcommand{\Bd}{\begin{scope}[rotate = -120] \Be\end{scope}}
\newcommand{\Bed}{\sechseck\draw (p1) -- (p3);\draw (p4) -- (p6);}
\newcommand{\Bemz}{\begin{scope}[rotate = -60] \Bed\end{scope}}
\newcommand{\Bzmd}{\begin{scope}[rotate = -120] \Bed\end{scope}}

\newcommand{\Bezd}{\sechseck\draw (p1)--(p2)--(p3)--(p1);\draw (p4)--(p5)--(p6)--(p4);}
\newcommand{\Bemzmd}{\begin{scope}[rotate = -60] \Bezd\end{scope}}
\newcommand{\Bezmd}{\begin{scope}[rotate = -120] \Bezd\end{scope}}
\newcommand{\BZez}{\sechseck\draw (p1)--(p2)--(p4)--(p5)--(p1);}
\newcommand{\BZzd}{\begin{scope}[rotate = -60] \BZez\end{scope}}
\newcommand{\BZed}{\begin{scope}[rotate = -120] \BZez\end{scope}}
\newcommand{\Beuzd}{\sechseck\draw (p1)--(p4);\draw (p2)--(p3); \draw (p6)--(p5);}
\newcommand{\Bemduz}{\begin{scope}[rotate = -60] \Beuzd\end{scope}}
\newcommand{\Bezud}{\begin{scope}[rotate = -120] \Beuzd\end{scope}}

\newcommand{\pBfull}{\begin{tikzpicture}\sechseck\draw (p1)--(p2)(p3)--(p2)(p3)--(p4)(p4)--(p5)(p5)--(p6)(p1)--(p6);\end{tikzpicture}}
\newcommand{\pBez}{\begin{tikzpicture}\Bez\end{tikzpicture}}
\newcommand{\pBzd}{\begin{tikzpicture}\Bzd\end{tikzpicture}}
\newcommand{\pBemd}{\begin{tikzpicture}\Bemd\end{tikzpicture}}
\newcommand{\pBe}{\begin{tikzpicture}\Be\end{tikzpicture}}
\newcommand{\pBz}{\begin{tikzpicture}\Bz\end{tikzpicture}}
\newcommand{\pBd}{\begin{tikzpicture}\Bd\end{tikzpicture}}
\newcommand{\pBed}{\begin{tikzpicture}\Bed\end{tikzpicture}}
\newcommand{\pBemz}{\begin{tikzpicture}\Bemz\end{tikzpicture}}
\newcommand{\pBzmd}{\begin{tikzpicture}\Bzmd\end{tikzpicture}}

\newcommand{\pBezd}{\begin{tikzpicture}\Bezd\end{tikzpicture}}
\newcommand{\pBemzmd}{\begin{tikzpicture}\Bemzmd\end{tikzpicture}}
\newcommand{\pBezmd}{\begin{tikzpicture}\Bezmd\end{tikzpicture}}
\newcommand{\pBZez}{\begin{tikzpicture}\BZez\end{tikzpicture}}
\newcommand{\pBZzd}{\begin{tikzpicture}\BZzd\end{tikzpicture}}
\newcommand{\pBZed}{\begin{tikzpicture}\BZed\end{tikzpicture}}
\newcommand{\pBeuzd}{\begin{tikzpicture}\Beuzd\end{tikzpicture}}
\newcommand{\pBemduz}{\begin{tikzpicture}\Bemduz\end{tikzpicture}}
\newcommand{\pBezud}{\begin{tikzpicture}\Bezud\end{tikzpicture}}

\newcommand{\Dev}{\sechseckmp \draw (p0)--(p1);}
\newcommand{\Dzv}{\begin{scope}[rotate = -60] \Dev\end{scope}}
\newcommand{\Ddv}{\begin{scope}[rotate = -60*2] \Dev\end{scope}}
\newcommand{\Dmev}{\begin{scope}[rotate = -60*3] \Dev\end{scope}}
\newcommand{\Dmzv}{\begin{scope}[rotate = -60*4] \Dev\end{scope}}
\newcommand{\Dmdv}{\begin{scope}[rotate = -60*5] \Dev\end{scope}}
\newcommand{\Dez}{\sechseckmp\draw (p1) -- (p2);\draw (p5) -- (p4);}
\newcommand{\Dzd}{\begin{scope}[rotate = -60] \Dez\end{scope}}
\newcommand{\Demd}{\begin{scope}[rotate = -60*2] \Dez\end{scope}}
\newcommand{\Ded}{\sechseckmp \draw (p1)--(p3);\draw (p4)--(p6);}
\newcommand{\Demz}{\begin{scope}[rotate = -60] \Ded\end{scope}}
\newcommand{\Dzmd}{\begin{scope}[rotate = -60*2] \Ded\end{scope}}

\newcommand{\Dezv}{\sechseckmp \draw (p0)--(p1)--(p2)--(p0);}
\newcommand{\Dzdv}{\begin{scope}[rotate = -60] \Dezv\end{scope}}
\newcommand{\Ddmev}{\begin{scope}[rotate = -60*2] \Dezv\end{scope}}
\newcommand{\Dmemzv}{\begin{scope}[rotate = -60*3] \Dezv\end{scope}}
\newcommand{\Dmzmdv}{\begin{scope}[rotate = -60*4] \Dezv\end{scope}}
\newcommand{\Dmdev}{\begin{scope}[rotate = -60*5] \Dezv\end{scope}}
\newcommand{\Devuzd}{\sechseckmp \draw (p0)--(p1);\draw (p2)--(p3);\draw (p5)--(p6);}
\newcommand{\Dzvuemd}{\begin{scope}[rotate = -60] \Devuzd\end{scope}}
\newcommand{\Ddvuez}{\begin{scope}[rotate = -60*2] \Devuzd\end{scope}}
\newcommand{\Dmevuzd}{\begin{scope}[rotate = -60*3] \Devuzd\end{scope}}
\newcommand{\Dmzvuemd}{\begin{scope}[rotate = -60*4] \Devuzd\end{scope}}
\newcommand{\Dmdvuez}{\begin{scope}[rotate = -60*5] \Devuzd\end{scope}}
\newcommand{\Deme}{\sechseckmp \draw (p1)--(p0)--(p4);}
\newcommand{\Dzmz}{\begin{scope}[rotate = -60*1] \Deme\end{scope}}
\newcommand{\Ddmd}{\begin{scope}[rotate = -60*2] \Deme\end{scope}}
\newcommand{\Dezmd}{\sechseckmp \draw (p1)--(p2)--(p6)--(p1);\draw (p3)--(p4)--(p5)--(p3);}
\newcommand{\Dezd}{\begin{scope}[rotate = -60] \Dezmd\end{scope}}
\newcommand{\Demzmd}{\begin{scope}[rotate = -60*2] \Dezmd\end{scope}}
\newcommand{\Dmdzv}{\sechseckmp \draw (p2)--(p0)--(p6)--(p2);}
\newcommand{\Dedv}{\begin{scope}[rotate = -60] \Dmdzv\end{scope}}
\newcommand{\Dzmev}{\begin{scope}[rotate = -60*2] \Dmdzv\end{scope}}
\newcommand{\Ddmzv}{\begin{scope}[rotate = -60*3] \Dmdzv\end{scope}}
\newcommand{\Dmemdv}{\begin{scope}[rotate = -60*4] \Dmdzv\end{scope}}
\newcommand{\Dmzev}{\begin{scope}[rotate = -60*5] \Dmdzv\end{scope}}

\newcommand{\Dezdv}{\sechseckmp \draw (p1)--(p2)--(p3)--(p0)--(p1);}
\newcommand{\Dzdmev}{\begin{scope}[rotate = -60] \Dezdv\end{scope}}
\newcommand{\Ddmemzv}{\begin{scope}[rotate = -60*2] \Dezdv\end{scope}}
\newcommand{\Dmemzmdv}{\begin{scope}[rotate = -60*3] \Dezdv\end{scope}}
\newcommand{\Dmzmdev}{\begin{scope}[rotate = -60*4] \Dezdv\end{scope}}
\newcommand{\Dmdezv}{\begin{scope}[rotate = -60*5] \Dezdv\end{scope}}
\newcommand{\DZez}{\sechseckmp \draw (p1)--(p2)--(p4)--(p5)--(p1);}
\newcommand{\DZzd}{\begin{scope}[rotate = -60] \DZez\end{scope}}
\newcommand{\DZed}{\begin{scope}[rotate = -60*2] \DZez\end{scope}}
\newcommand{\Demeuzd}{\sechseckmp \draw (p1)--(p0)--(p4);\draw (p2)--(p3);\draw (p5)--(p6);}
\newcommand{\Dzmzuemd}{\begin{scope}[rotate = -60] \Demeuzd\end{scope}}
\newcommand{\Ddmduez}{\begin{scope}[rotate = -60*2] \Demeuzd\end{scope}}

\newcommand{\pDev}{\begin{tikzpicture}\Dev\end{tikzpicture}}
\newcommand{\pDzv}{\begin{tikzpicture}\Dzv\end{tikzpicture}}
\newcommand{\pDdv}{\begin{tikzpicture}\Ddv\end{tikzpicture}}
\newcommand{\pDmev}{\begin{tikzpicture}\Dmev\end{tikzpicture}}
\newcommand{\pDmzv}{\begin{tikzpicture}\Dmzv\end{tikzpicture}}
\newcommand{\pDmdv}{\begin{tikzpicture}\Dmdv\end{tikzpicture}}
\newcommand{\pDez}{\begin{tikzpicture}\Dez\end{tikzpicture}}
\newcommand{\pDzd}{\begin{tikzpicture}\Dzd\end{tikzpicture}}
\newcommand{\pDemd}{\begin{tikzpicture}\Demd\end{tikzpicture}}
\newcommand{\pDed}{\begin{tikzpicture}\Ded\end{tikzpicture}}
\newcommand{\pDemz}{\begin{tikzpicture}\Demz\end{tikzpicture}}
\newcommand{\pDzmd}{\begin{tikzpicture}\Dzmd\end{tikzpicture}}

\newcommand{\pDezv}{\begin{tikzpicture}\Dezv\end{tikzpicture}}
\newcommand{\pDzdv}{\begin{tikzpicture}\Dzdv\end{tikzpicture}}
\newcommand{\pDdmev}{\begin{tikzpicture}\Ddmev\end{tikzpicture}}
\newcommand{\pDmemzv}{\begin{tikzpicture}\Dmemzv\end{tikzpicture}}
\newcommand{\pDmzmdv}{\begin{tikzpicture}\Dmzmdv\end{tikzpicture}}
\newcommand{\pDmdev}{\begin{tikzpicture}\Dmdev\end{tikzpicture}}
\newcommand{\pDevuzd}{\begin{tikzpicture}\Devuzd\end{tikzpicture}}
\newcommand{\pDzvuemd}{\begin{tikzpicture}\Dzvuemd\end{tikzpicture}}
\newcommand{\pDdvuez}{\begin{tikzpicture}\Ddvuez\end{tikzpicture}}
\newcommand{\pDmevuzd}{\begin{tikzpicture}\Dmevuzd\end{tikzpicture}}
\newcommand{\pDmzvuemd}{\begin{tikzpicture}\Dmzvuemd\end{tikzpicture}}
\newcommand{\pDmdvuez}{\begin{tikzpicture}\Dmdvuez\end{tikzpicture}}
\newcommand{\pDeme}{\begin{tikzpicture}\Deme\end{tikzpicture}}
\newcommand{\pDzmz}{\begin{tikzpicture}\Dzmz\end{tikzpicture}}
\newcommand{\pDdmd}{\begin{tikzpicture}\Ddmd\end{tikzpicture}}
\newcommand{\pDezmd}{\begin{tikzpicture}\Dezmd\end{tikzpicture}}
\newcommand{\pDezd}{\begin{tikzpicture}\Dezd\end{tikzpicture}}
\newcommand{\pDemzmd}{\begin{tikzpicture}\Demzmd\end{tikzpicture}}
\newcommand{\pDmdzv}{\begin{tikzpicture}\Dmdzv\end{tikzpicture}}
\newcommand{\pDedv}{\begin{tikzpicture}\Dedv\end{tikzpicture}}
\newcommand{\pDzmev}{\begin{tikzpicture}\Dzmev\end{tikzpicture}}
\newcommand{\pDdmzv}{\begin{tikzpicture}\Ddmzv\end{tikzpicture}}
\newcommand{\pDmemdv}{\begin{tikzpicture}\Dmemdv\end{tikzpicture}}
\newcommand{\pDmzev}{\begin{tikzpicture}\Dmzev\end{tikzpicture}}

\newcommand{\pDezdv}{\begin{tikzpicture}\Dezdv\end{tikzpicture}}
\newcommand{\pDzdmev}{\begin{tikzpicture}\Dzdmev\end{tikzpicture}}
\newcommand{\pDdmemzv}{\begin{tikzpicture}\Ddmemzv\end{tikzpicture}}
\newcommand{\pDmemzmdv}{\begin{tikzpicture}\Dmemzmdv\end{tikzpicture}}
\newcommand{\pDmzmdev}{\begin{tikzpicture}\Dmzmdev\end{tikzpicture}}
\newcommand{\pDmdezv}{\begin{tikzpicture}\Dmdezv\end{tikzpicture}}
\newcommand{\pDZez}{\begin{tikzpicture}\DZez\end{tikzpicture}}
\newcommand{\pDZzd}{\begin{tikzpicture}\DZzd\end{tikzpicture}}
\newcommand{\pDZed}{\begin{tikzpicture}\DZed\end{tikzpicture}}
\newcommand{\pDemeuzd}{\begin{tikzpicture}\Demeuzd\end{tikzpicture}}
\newcommand{\pDzmzuemd}{\begin{tikzpicture}\Dzmzuemd\end{tikzpicture}}
\newcommand{\pDdmduez}{\begin{tikzpicture}\Ddmduez\end{tikzpicture}}

\newcommand{\Tetraeder}{
	\begin{scope}[every path/.style={thick}]
		\node[Blue] (e1) at (-4.95,2.16) [gpunkt] {};
		\node[Blue] (e2) at (0,0.45) [gpunkt] {};
		\node[Blue] (e3) at (-0.45,8) [gpunkt] {};
		\node[Blue] (e4) at (3.96,2.16) [gpunkt] {};
		
		\draw[black!20,dashed] (e1)--(e4);
		\foreach \i/\j in {1/3,1/2,2/3,3/4,2/4}
		\node[Green] (f\i\j) at ($(e\i)!0.5!(e\j)$) [gpunkt] {};
		\draw[Orange] (e1) -- (f12) -- (e2) -- (f23) -- (e3) -- (f34) --
		(e4) -- (f24) -- (e2) (e1) -- (f13) -- (e3);
		
		\path[name path=A] (e1)--(f23)--(e4);
		\path[name path=B] (e3)--(f12);
		\node (AB) [Orange, name intersections={of=A and B, by=AB}] at
		(AB) [gpunkt] {};
		
		\path[name path=C] (e3)--(f24);
		\node (AC) [Orange, name intersections={of= A and C, by=AC}] at
		(AC) [gpunkt] {};
		
		\draw[Blue] (f12) -- (AB) -- (f23) -- (AC) -- (f34) (f13) --
		(AB) (f24) -- (AC);
		\draw[Green] (e1) -- (AB) (AB) -- (e2) -- (AC) (AB) -- (e3) --
		(AC) -- (e4);
		
	\end{scope}
}

\newcommand{\CoxA}{
	\begin{tikzpicture}
	\foreach \w in {0,1,...,5}
	\node[kpunkt] (p\w) at (\w,0){};
	\draw (p0)--(p1)--(p2) (p3)--(p4)--(p5);
	\node at(2.5,0){$\ldots$};
	\end{tikzpicture}
}

\newcommand{\CoxB}{
	\begin{tikzpicture}
	\foreach \w in {0,1,...,5}
	\node[kpunkt] (p\w) at (\w,0){};
	\draw (p0)--(p1)--(p2) (p3)--(p4)--(p5);
	\node at(2.5,0){$\ldots$};
	\node[above] at(4.5,0){$4$};
	\end{tikzpicture}
}

\newcommand{\CoxD}{
	\begin{tikzpicture}
	\foreach \w in {0,1,...,5}
	\node[kpunkt] (p\w) at (\w,0){};
	\node[kpunkt] (p6) at (4,1){};
	\draw (p0)--(p1)--(p2) (p3)--(p4)--(p5) (p4)--(p6); 
	\node at(2.5,0){$\ldots$};
	\end{tikzpicture}
}

\newcommand{\CoxEsechs}{
	\begin{tikzpicture}
	\foreach \w in {0,1,...,4}
	\node[kpunkt] (p\w) at (\w,0){};
	\node[kpunkt] (p6) at (2,1){};
	\draw (p0)--(p1)--(p2)--(p3)--(p4) (p2)--(p6); 
	\end{tikzpicture}
}

\newcommand{\CoxEsieben}{
	\begin{tikzpicture}
	\foreach \w in {0,1,...,5}
	\node[kpunkt] (p\w) at (\w,0){};
	\node[kpunkt] (p7) at (2,1){};
	\draw (p0)--(p1)--(p2)--(p3)--(p4)--(p5) (p2)--(p7); 
	\end{tikzpicture}
}

\newcommand{\CoxEacht}{
	\begin{tikzpicture}
	\foreach \w in {0,1,...,6}
	\node[kpunkt] (p\w) at (\w,0){};
	\node[kpunkt] (p8) at (2,1){};
	\draw (p0)--(p1)--(p2)--(p3)--(p4)--(p5)--(p6) (p2)--(p8); 
	\end{tikzpicture}
}

\newcommand{\CoxF}{
	\begin{tikzpicture}
	\foreach \w in {0,1,...,3}
	\node[kpunkt] (p\w) at (\w,0){};
	\draw (p0)--(p1)--(p2)--(p3); 
	\node[above] at(1.5,0){$4$};
	\end{tikzpicture}
}

\newcommand{\CoxHdrei}{
	\begin{tikzpicture}
	\foreach \w in {0,1,2}
	\node[kpunkt] (p\w) at (\w,0){};
	\draw (p0)--(p1)--(p2); 
	\node[above] at(0.5,0){$5$};
	\end{tikzpicture}
}

\newcommand{\CoxHvier}{
	\begin{tikzpicture}
	\foreach \w in {0,1,...,3}
	\node[kpunkt] (p\w) at (\w,0){};
	\draw (p0)--(p1)--(p2)--(p3); 
	\node[above] at(0.5,0){$5$};
	\end{tikzpicture}
}

\newcommand{\CoxIm}{
	\begin{tikzpicture}
	\foreach \w in {0,1}
	\node[kpunkt] (p\w) at (\w,0){};
	\draw (p0)-- (p1) node[midway,above] {$m$}; 
	\end{tikzpicture}
}

\begin{document}

\includepdf{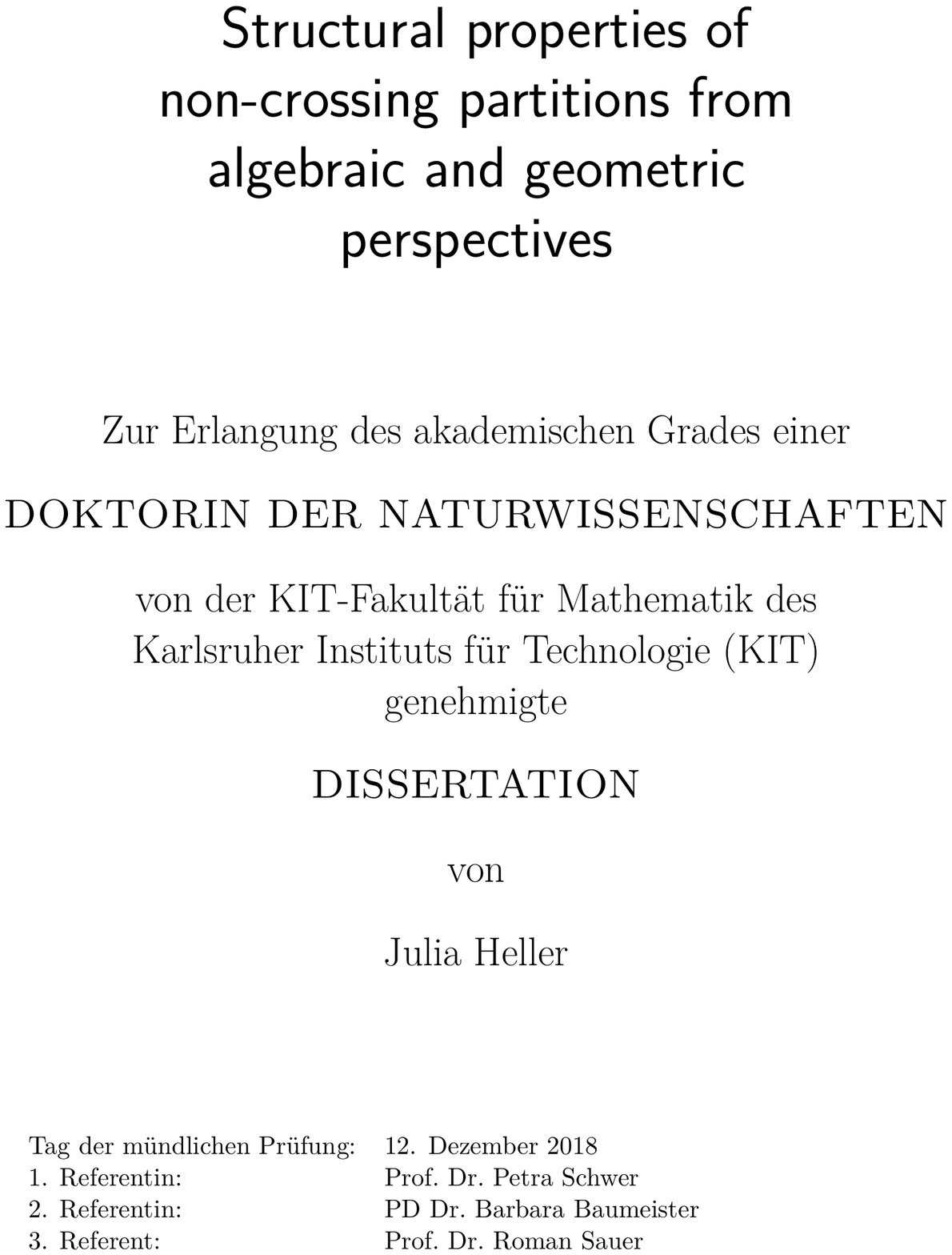}

\thispagestyle{empty}
\vspace*{5cm}
\vfill
\doclicenseThis

\chapter*{Introduction}
\addcontentsline{toc}{chapter}{Introduction}

Consider visualizations of partitions of, for instance, a six-element set. The elements of this set are represented by circularly ordered vertices in the plane with clockwise numbering. Each part of the partition is visualized by a polygon on the corresponding vertices. For instance, the two pictures
\begin{center}
	\begin{tikzpicture}
	
	\begin{scope}
	\gsechsecklab
	\draw (p1)--(p3)--(p4)--(p1)(p5)--(p6);
	\end{scope}
	
	\begin{scope}[xshift=4cm]
	\gsechsecklab
	\draw (p1)--(p3)--(p4)--(p1)(p5)--(p2);
	\end{scope}
	
	\end{tikzpicture}
\end{center}
correspond to the partitions $\Set{\Set{1,3,4}, \Set{2}, \Set{5,6}}$ and $\Set{\Set{1,3,4}, \Set{2,5}, \Set{6}}$ of the set $\Set{1, 2, 3, 4, 5, 6}$, respectively. It is immediately apparent that the left partition classifies as \emph{non-crossing}, whereas the right is \emph{crossing}.\medskip

In the 1970s Kreweras started the systematic study of non-crossing partitions of circularly arranged $n$-element sets, which are now called the \emph{classical non-crossing partitions} \cite{kre}. He showed that with the refinement order, they form a graded lattice and described first descriptive and enumerative properties. From that moment on, the non-crossing partitions drew attention not only in enumerative combinatorics and poset theory \cite{edel}, but also in a wide range of fields such as geometric combinatorics \cite{simion_ass}, free probability theory \cite{speicher}, and even molecular biology \cite{pen_water}. 
The article \cite{simion_ncp} surveys these works in the early years of non-crossing partitions. 
See also \cite{mcc_surp} for non-crossing partitions in surprising locations. Aside from that, non-crossing partitions appear in representation theory of quivers \cite{thom}, the theory of cluster algebras \cite{fom_zel}, and they are a central object of Coxeter-Catalan combinatorics \cite{armstr}.\medskip

A new direction of research started with the works of Biane, Brady, Bessis, Watt, and Reiner, namely the connections of non-crossing partitions to Coxeter groups \cite{biane}, \cite{bra_kpi}, \cite{bes}, \cite{bra_watt_par_ord} \cite{bra_watt_kpi}, \cite{rei}.

Motivated from probability theory, Biane showed that the lattice of non-crossing partitions as defined by Kreweras is isomorphic to a partially ordered subset of the symmetric group. This allowed him to compute the group of lattice automorphisms and lattice skew-automorphisms of the classical non-crossing partitions. 

The same partial order on the symmetric group was studied by Brady, whose motivation came from a geometric origin. 
This partial order, now known as absolute order, allowed him to find a new generating set for the braid groups. These generators are exactly the non-crossing partitions and they give rise to new classifying spaces for the braid groups.

The key in the studies of both Biane and Brady was to consider the symmetric group with the generating set of all transpositions. Bessis, and independently Brady and Watt, systematically studied finite Coxeter groups with a set of generators given by the set of all reflections. This initiated the so-called dual study of finite Coxeter groups and thus the study of non-crossing partitions from an algebraic perspective. 

Non-crossing partitions can be seen as a subset of finite Coxeter groups. This allows to transfer the types of Coxeter groups to the corresponding non-crossing partitions. For example, the symmetric group $S_n$ is the Coxeter group of type $A_{n-1}$, hence the classical non-crossing partitions of an $n$-element set correspond to the non-crossing partitions of type $A_{n-1}$. The subscript refers to the rank of the Coxeter group. Since the Coxeter groups of type $A$, $B$ and $D$ are called the Coxeter groups of classical types, the non-crossing partitions corresponding to these groups are called of \emph{classical type} as well.

Generalizations of non-crossing partitions to well-generated complex reflection groups exist as well \cite{bes_cor}, but in the present thesis we only consider non-crossing partitions inside finite Coxeter groups.

In the 1990s Reiner took yet a different approach and realized that the classical non-crossing partition lattice can be interpreted as a poset of intersection subspaces arising from a hyperplane arrangement of type $A$ \cite{rei}. Using this interpretation, he defined a pictorial analog for the non-crossing partitions for type $B$ and $D$. Later it turned out that Reiner's non-crossing partitions of type $B$ are pictorial representations of the group-theoretic ones \cite{bra_watt_kpi}. Pictorial representations fitting to the group-theoretic non-crossing partitions of type $D$ were later constructed by Athanasiadis and Reiner \cite{ath_rei}.

By the work of Brady and McCammond, the non-crossing partitions became popular in geometric group theory as well \cite{bra-mcc}. Based on the construction of the classifying space from Brady, they observed that the curvature of the order complex of the classical non-crossing partition lattice is closely related to the curvature of the universal cover of the classifying space for the braid group.

\medskip

The present thesis studies structural properties of non-crossing partitions from different perspectives and makes use of the resulting interactions. 
For the elements of the non-crossing partition lattices of classical types, that is for the types $A$, $B$ and $D$, three different interpretations exist. Either, we see them as elements of a Coxeter group, as a set-theoretic partition, or as an embedded graph, which is what the pictorial representations are.
Also the whole set of non-crossing partitions can be considered from different perspectives.
On the one hand, we regard the non-crossing partitions inside an arbitrary finite Coxeter group as a partially ordered set. On the other hand, we see them as a simplicial complex whose vertices are labeled with non-crossing partitions. This complex is called the \emph{complex of non-crossing partitions} subsequently.

One aspect is to compute the automorphism group of non-crossing partition lattices. A lattice automorphism preserves the partial order, where a lattice skew-automorphism may reverse the partial order.
Biane used the relation of group theory and set theory to 
show that the group of lattice automorphisms of the non-crossing partitions of type $A_{n-1}$ is isomorphic to the dihedral group of order $2n$. Moreover, he showed that the group of lattice skew-automorphisms is isomorphic to the dihedral group of order $4n$ \cite{biane}.
The interaction of the pictorial and group-theoretic non-crossing partitions of type $B$ allows us to prove an analogous statement in \cref{thm:typeB_full_aut}. 

\begin{intro}\label{i1}
	The group of lattice automorphisms of the non-crossing partitions of type $B_n$ is isomorphic to the dihedral group of order $2n$.
\end{intro}

The non-crossing partitions of type $D$ are more involved and the existing pictorial representations from \cite{ath_rei} do not fit our purpose of computing the automorphism group for type $D$. Therefore, we introduce a new pictorial representation, which is based on the aforementioned one from Athanasiadis and Reiner. Based on this, we show the following in \cref{thm:typeD_full_aut}.

\begin{intro}\label{i2}
	The group of lattice automorphisms of the non-crossing partitions of type $D_n$ is isomorphic to the dihedral group of order $4(n-1)$ for $n\neq 4$. If $n=4$, then the automorphism group contains a proper subgroup that is isomorphic to the dihedral group of order $12$.
\end{intro}

Further, we give a group-theoretic description of the automorphism groups  and show that every automorphism can be realized as a symmetry of the pictorial representations, provided that $n\neq4$ for type $D$.

Let us now turn to the different interpretations of the set of non-crossing partitions. The study of non-crossing partitions as simplicial complexes is not very established yet, and various properties of the simplicial complex, such as shellability, are proven in a poset-theoretic setting \cite{abw}. Based on poset-theoretic results \cite{bra_watt_kpi,bra_watt_par_ord}, we show together with Schwer that the complexes of non-crossing partitions have a nice simplicial structure \cite{hs}, which is \cref{thm:simplicial_embedding} in the present thesis. For type $A$, this was already remarked in \cite{bra-mcc}. 

\begin{intro}\label{i3}
	For every finite Coxeter group, the non-crossing partition complex is isomorphic to a chamber subcomplex of a spherical building. Moreover, if the Coxeter group is crystallographic, then the building can be chosen to be finite.
\end{intro}

This theorem has direct consequences on the structure of non-crossing partition complexes. For instance, they are unions of Coxeter complexes of type $A$. This gives the non-crossing partition complex a building-like structure, which we investigate further. 
As an application of the above theorem, we give a bound on the radius of the Hurwitz graph, partially answering a question of \cite{ar}. Together with Schwer we show \cref{thm:hurwitz}, which is the following.

\begin{intro}\label{i4}
	The radius of the Hurwitz graph of a finite Coxeter group of rank $n$ is bounded from below by $\binom{n}{2}$.
\end{intro}

Note that lattice automorphisms of the non-crossing partition lattice give rise to simplicial automorphisms of the corresponding simplicial complex. The automorphisms of the non-crossing partition complexes interact with simplicial automorphisms of buildings as follows, which is proven in its poset-theoretic version in \cref{thm:autos_commute}.

\begin{intro}\label{i5}
	For every finite Coxeter group there exists a dihedral group of automorphisms of the corresponding non-crossing partition complex whose elements extend uniquely to type-preserving simplicial automorphisms of the ambient spherical building. For the classical types, every automorphism extends uniquely, where for type $D$ we further have to assume that $n\neq 4$.
\end{intro}

For the non-crossing partitions of classical type, there exist type-reversing automorphisms that extend as well. The following is the simplicial version of \cref{thm:anti-autos_extend}.

\begin{intro}\label{i6}
	For the Coxeter groups of classical types, there exists a dihedral group of skew-automorphisms of the non-crossing partition complex whose elements extend to simplicial automorphisms of the ambient spherical building.
\end{intro}

Let us now focus on the classical non-crossing partitions, their relation to the braid groups and the resulting impact on geometric group theory, initiated by Brady and McCammond \cite{bra_kpi}, \cite{bra-mcc}. Their starting point was the new classifying space for the braid group constructed by Brady, 
and the question whether braid groups are $\catz$ groups. An affirmative answer to this question is guaranteed if the universal cover, endowed with a suitable metric, is a $\catz$ space. Brady and McCammond reduced the question whether this infinite simplicial complex is $\catz$ to the question whether the finite complex of non-crossing partitions with a spherical metric is $\cato$. They could prove that this is the case for the braid groups up to five strands. The case for six strands was solved in \cite{hks}. 
The proof techniques of both works heavily depend on the number of strands of the braid groups and with this on the dimension of the non-crossing partition complex.\medskip

In this thesis we do not prove that the braid groups are $\catz$ for any new dimensions. However, we provide a uniform strategy that might prove the \emph{\cref{conj:ncpn_cat1}}, which states that the complex of classical non-crossing partitions is $\cato$. Further, we show some intermediate results concerning the simplicial structure. We often use the fact that the non-crossing partition complex is a subcomplex of a finite spherical building, which is $\cato$ \cite{davis}.
Moreover, we investigate the partition complex, which we regard as a subcomplex of the spherical building as well.

In our strategy, we make use of the criterion that a spherical simplicial complex is $\cato$ if and only if between any two points of distance less than $\pi$ there exists a unique geodesic \cite{bh}. 
The following two relations of distances of maximal simplices, so-called chambers, measured in the different complexes are shown in \cref{thm:dist_pn_si} and \cref{thm:dist_pn_ncpn}, respectively.

\begin{intro}\label{i7}
	The distance of chambers measured in the partition complex coincides with the one measured in the spherical building.
\end{intro}

\begin{intro}\label{i8}
	If two chambers of the non-crossing partition complex are contained in a common apartment of the partition complex, then their distance measured in the non-crossing partition complex coincides with the one measured in the partition complex.
\end{intro}

The key of our approach is the study of convex hulls of maximal simplices in the non-crossing partition complex, for which the knowledge of chamber distances is most helpful. More concretely, we use that if each convex hull of chambers in the non-crossing partition complex is $\cato$, then the non-crossing partition complex is $\cato$.

Hence the problem of finding out whether the non-crossing partition complex is $\cato$ reduces to the question whether the convex hulls of its chambers are $\cato$.
If two chambers are contained in a common apartment, then their convex hull equals the convex hull in the building and therefore is $\cato$. We show for another class of chamber pairs that their convex hull is $\cato$ in  \cref{cor:conv_intersec}.

\begin{intro}\label{i9}
	If two chambers of the non-crossing partition complex have non-empty intersection, then their convex hull is $\cato$.
\end{intro}

Proving that a space is $\cato$ means, roughly speaking, that the space does not have any \enquote{small holes}. 
We prove the \cref{nohole} for the non-crossing partition complex, which is as follows. An analog holds for the partition complex as well.

\begin{intro}\label{i10}
	Consider an apartment of the building. If the link complex of a vertex in this apartment is contained in the non-crossing partition complex, then the vertex is a non-crossing partition as well, provided that $n \geq 5$.
\end{intro}

We use the Link Property to show the next statement, which is \cref{thm:large_holes}.

\begin{intro}\label{i11}
	Every vertex of the spherical building that is not a non-crossing partition lies on a building geodesic of length $\pi$ that does not intersect the non-crossing partition complex, provided that $n \geq 5$.
\end{intro}

If we consider the non-crossing partition complex as being cut out from the building, then the above statement means that whenever one vertex is removed, a large portion of the building is removed as well.

The collection of statements above is not a proof for the non-crossing partition complex being $\cato$, but it promises to pave the way towards proving it.

\medskip

This thesis is divided into three parts and structured as follows.

\cref{part1} reviews the background material of lattices and simplicial complexes, Coxeter groups, and buildings in the Chapters \ref{chap:lattices}, \ref{chap:cox}, and \ref{cha:buildings}, respectively. 

\cref{part2} is devoted to the non-crossing partitions associated to Coxeter groups. We define non-crossing partitions and consider their basic properties in \cref{chap:elem_props_nc}. 
\cref{chap:classical_types} reviews the non-crossing partitions of  classical types. 
In particular, the interaction of the algebraic and pictorial characterizations of the non-crossing partitions is described in detail. The new pictorial representations of non-crossing partitions of type $D$ is introduced here. 
\cref{chap:embedding} deals with the embedding of non-crossing partitions from poset-theoretic and simplicial viewpoints with emphasis on the embeddings into finite spherical buildings.
In this chapter we also investigate the Hurwitz graph.
\cref{chap:autos} is the main chapter of this part of the thesis and is concerned with the study of automorphisms and anti-automorphisms of the non-crossing partition lattices.

\cref{part3} studies the geometry of the classical non-crossing partitions particularly with regard to its curvature properties. 
\cref{chap:motivation} serves multiple purposes. First, we explain the motivation from geometric group theory to study the geometry of the classical non-crossing partition complex in more detail. Then, we provide the necessary background concerning geodesic metric spaces and in particular $\catz$ and $\cato$ spaces. Finally, we outline a strategy that could prove the Curvature Conjecture. 
Before we fill this strategy with details, we study the simplicial structure of the classical non-crossing partition complex as well as the one of the partition complex, seen as subcomplexes of a finite spherical building, in \cref{sec:simpl_structure}. 
We translate the algebraic embedding of the classical non-crossing partitions into purely pictorial terms. This embedding is compatible with a pictorial embedding of the partitions into the same building, which is constructed here as well. The rest of this chapter is a collection of structural properties and descriptions, known and new ones, of the chambers and apartments of both the partition and the non-crossing partition complexes. 
The last chapter of this thesis, \cref{chap:dist}, provides the details corresponding to the different cases of the strategy outlined before.

The appendix contains two technical proofs, which we relocated there for the sake of readability of the main part.

\subsubsection*{Acknowledgments}

At this point I want to express my gratitude to all those people who accompanied me during my PhD studies. Foremost, I would like to thank my supervisor Petra Schwer and my referees Barbara Baumeister and Roman Sauer. Moreover, I thank my friends and my former and present colleagues from the Institute of Algebra and Geometry at the KIT and from the RTG, as well as all those people I met at numerous conferences and seminars I had the chance to visit. Finally, I would like to thank my parents, sisters and Benni for always being there for me.

\cleardoublepage
\tableofcontents

\cleardoublepage
\part{Preliminaries}\label{part1}

\cleardoublepage
\chapter{Posets and simplicial complexes}\label{chap:lattices}

This chapter is a summary of the main concepts in use, with emphasis on the interaction of posets and simplicial complexes. The main references are \cite{staece} and \cite{graetzer} for posets, and \cite{ab} and \cite{bb} for simplicial complexes. The definitions and statements we are reviewing here are taken from these books.

\section{Posets and lattices}

\begin{defi}
	A \emph{partially ordered set}, or \emph{poset} for short, $(P,\leq)$ consists of a set $P$ and a binary relation $\leq$, called \emph{partial order}, such that for all $x,y,z\in P$ it holds that
	\begin{enumerate}
		\item $x\leq x$,
		\item if $x \leq y$ and $y \leq x$, then $x = y$, and
		\item if $x \leq y$ and $y\leq z$, then $x \leq z$.
	\end{enumerate}
\end{defi}

If it is clear from the context which partial order is used, we also refer to $P$ alone as poset. For $x,y\in P$ the symbol $y \geq x$ means $x \leq y$. By $x < y$ we mean $x \leq y$ and $x \neq y$, and by $y > x$ we mean $x < y$. We say that $y$ \emph{covers} $x$, and denote it by $x \lessdot y$, if $x \leq y$ and no element $z\in P$ satisfies $x < z <y$. We also say $x$ is \emph{covered} by $y$. By $y \gtrdot x$ we mean $x \lessdot y$.    

\begin{defi}\label{defi:cov_sets}
	Let $P$ be a poset and $x\in P$ an element.
	The set 
	\[
	\ab(x)\coloneqq \set{ y \in P \str  y \gtrdot x}
	\] 
	consists of all elements $y\in P$ that cover $x$ and is called the \emph{cover set} of $x$. Analogously, the set 
	\[
	\bel(x)\coloneqq \set{ y \in P \str y \lessdot x}
	\]
	consists of all elements $y\in P$ that are covered by $x$ and is called the \emph{covered set} of $x$.
\end{defi}

A finite poset is completely determined by its cover relations and it is convenient to encode the information in a graph.

\begin{defi}
	The \emph{Hasse diagram} of a finite poset $P$ is an embedded graph. Its vertices are the elements of $P$ and there is an edge between $x$ and $y$ if $x \lessdot y$. The vertices are arranged in the plane in such a way that $x$ is drawn \enquote{below} $y$ whenever $x \lessdot y$.
\end{defi}

\begin{exa}\label{exa:bool3}
	Let $P$ be the set $\B_3$ of subsets of the set $\set{1,2,3}$ with partial order given by inclusion. This is indeed a poset and its Hasse diagram is depicted in \cref{fig:hasse_bool3}. 
\end{exa}

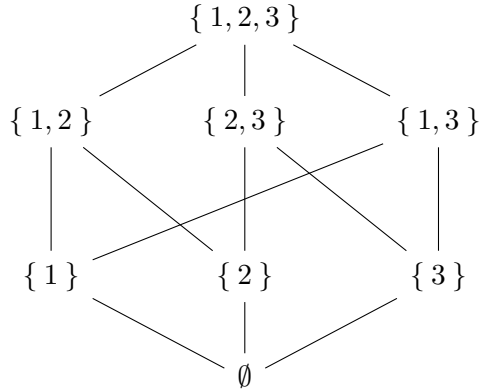
\begin{figure}
	\begin{center}
		\begin{tikzcd}
		&\Set{1,2,3} \ar[d,dash] \ar[dr, dash] \ar[dl,dash]&\\
		\Set{1,2}\ar[dd,dash]\ar[ddr,dash]&\Set{2,3}\ar[dd,dash]\ar[ddr,dash]&\Set{1,3}\ar[dd,dash]\ar[ddll,dash]\\
		&&\\
		\Set{1}\ar[dr,dash]&\Set{2}\ar[d,dash]&\Set{3}\ar[dl,dash]\\
		&\emptyset&
		\end{tikzcd}
		\caption{The Hasse diagram of the poset $\B_3$.}
		\label{fig:hasse_bool3}
	\end{center}
\end{figure}

Let $P$ be a poset. A subset $Q \subseteq P$ is an \emph{induced subposet}, or \emph{subposet} for short, of $P$ if for $x,y\in Q$ it holds that $x \leq y$ in $Q$ if and only if $x \leq y$ in $P$. A subposet $C=\set{x_0, x_1, \ldots,x_k} \subseteq P$ is called \emph{chain} if $x_0 < x_1 < \ldots < x_k$. We also denote the chain as ordered tupel $C=(x_0, \ldots, x_k)$. The integer $k$ is then called \emph{length} of the chain $C$, and $C$ is said to be a \emph{$k$-chain}. A chain is called \emph{maximal} if it is not contained in any longer chain of $P$. The supremum of the lengths of all maximal chains in $P$ is the \emph{rank} of $P$. If all maximal chains in $P$ are of the same length $k<\infty$, then $P$ is called a \emph{graded} poset. In this case, there exists a unique \emph{rank function} $\rk\colon P \to \set{0, 1, \ldots, k}$ such that $\rk(m)=0$ for all minimal elements $m\in P$ and $\rk(y)=\rk(x)+1$ if $x\lessdot y$. The integer $\rk(x)$ is called the \emph{rank} of the element $x\in P$. Note that the rank of $x\in P$ equals the rank of the subposet $\set{z \in P\str z\leq x} \subseteq P$. For $x,y\in P$ we call the subposet
\[
[x,y]\coloneqq \Set{z\in P\str x \leq z \leq y} \subseteq P
\]
the \emph{interval} between $x$ and $y$.

\begin{exa}
	We continue \cref{exa:bool3} from above. The poset $\B_3$ is a graded poset of rank $3$ and the rank of $X \subseteq \set{1,2,3}$ is given by the number of its elements, that is $\rk(X)=\#X$. Note that the maximal chains in $\B_3$ are given by the total orderings of the set $\set{1,2,3}$ and vice versa. For instance, the order $2-1-3$ gives rise to the maximal chain 
	\[
	\emptyset < \set{2} < \set{1,2} < \set{1,2,3}.
	\]
\end{exa}

Let $P$ and $Q$ be posets. The \emph{direct product} of $P$ and $Q$ is the poset 
\[
P\times Q = \Set{(p,q) \str p \in P, q\in Q}
\]
with partial order given by $(p',q') \leq (p,q)$ if and only if $p'\leq p$ in $P$ and $q' \leq q$ in $Q$. If $P$ and $Q$ are graded, then $P\times Q$ is graded as well.  

The rest of this section is devoted to lattices.
Let $P$ be a poset and $x,y\in P$. An \emph{upper bound} of $x$ and $y$ is an element $z\in P$ satisfying $x,y \leq z$. An upper bound $z\in P$ of $x$ and $y$ is called \emph{join} if for all upper bounds $u$ of $x$ and $y$ it holds that $z \leq u$, that is $z$ is a least upper bound. If a join of $x$ and $y$ exists, it is unique and denoted by $x \vee y$. \emph{Lower bounds} of $x$ and $y$ are defined dually. If a greatest lower bound of $x$ and $y$ exists, it is unique and called \emph{meet} of $x$ and $y$. It is denoted by $x \wedge y$.

\begin{defi}
	A \emph{lattice} is a poset $L$ in which for every pair of elements there exists both a meet and a join.
\end{defi}

Let $L$ be a lattice. A \emph{minimum} is an element $\mi\in L$ satisfying $\mi \leq x$ for all $x\in L$. Dually, a \emph{maximum} is an element $\ma\in L$ satisfying $x \leq \ma$ for all $x\in L$. If a minimum or a maximum exists, it is unique. Clearly every finite lattice has a minimum and a maximum. 
An \emph{atom} of $L$ is an element covering the minimum $\mi$. The lattice is called \emph{atomic} if every element $x\in L$ is a join of atoms. If the lattice is graded, then the atoms are exactly the elements of rank $1$.  A subposet $K \subseteq L$ is called \emph{sublattice} if $K$ is closed under the meet and join operations of $L$, that is if $x,y\in K \subseteq L$, then also $x\vee y\in K$ and $x\wedge y\in K$.
Let $M \subseteq L$ be a finite subset of a lattice $L$ with $\mi \in L$. The \emph{span} of $M$ in $L$ is defined to be the smallest sublattice $\braket{M} \subseteq L$ that contains $M$, that is if $K$ is a sublattice of $L$ containing the set $M$, then $\braket{M} \subseteq K$.   
We say that a set $M \subseteq L$ \emph{spans} a sublattice $K$ of $L$ if $\braket{M}=K$.

\begin{exa}
	The poset $\B_3$ from the examples above is a lattice with minimum $\mi = \emptyset$ and maximum $\ma=\set{1,2,3}$. The join is given by union of sets and the meet is given by intersection. Clearly the lattice $\B_3$ is atomic, since every set $X \subseteq \set{1,2,3}$ can be written as the union of its singletons $\bigcup_{x\in X} \set{x} = X$.  
\end{exa}

\subsection{Automorphisms and anti-automorphisms}

\begin{defi}
	Let $P$ and $Q$ be posets and let $\auto\colon P \to Q$ be a map. Then $\auto$ is called \emph{order-preserving} if $x \leq y$ in $P$ implies $\auto(x)\leq \auto(y)$ in $Q$ for all $x,y\in P$.
	The map $\auto$ is called \emph{order-reversing} if $x \leq y$ in $P$ implies $\auto(y)\leq \auto(x)$ in $Q$ for all $x,y\in P$. The map $\auto$ is called \emph{poset map} if it is order-preserving or order-reversing.
	An injective and order-preserving poset map is called a \emph{poset embedding} and a \emph{poset isomorphism} is a surjective poset embedding with order-preserving inverse. If $P$ and $Q$ are lattices, then a poset isomorphism $P \to Q$ is called \emph{lattice isomorphism}. Two posets are called \emph{isomorphic} if there exists a poset isomorphism between them.
\end{defi} 

We can generalize the example of the lattice $\B_3$ from above to $\B_n$ for all $n\in \N$.

\begin{exa}\label{exa:boolean_lattice}	
	Let $\B_n$ be the set of subsets of the set $\set{1,\ldots,n}$ with partial order given by inclusion. Then $\B_n$ is a graded poset of rank $n$. The rank of $X\subseteq \set{1,\ldots,n}$ is $\rk(X)=\#X$. With the meet operation given by intersection of sets and the join operation given by union of sets, $\B_n$ is an atomic lattice.  
\end{exa}

\begin{defi}
	A lattice that is isomorphic to the lattice $\B_n$ for some $n\in \N$ is called \emph{Boolean lattice} of rank $n$.
\end{defi}

Poset embeddings are central for \cref{chap:embedding}. The next definition introduces the main objects of \cref{chap:autos}.

\begin{defi}
	Let $L$ be a lattice and $\auto\colon L \to L$ a map. Then $\auto$ is called \emph{lattice automorphism} if it is a poset isomorphism.  
	The map $\auto$ is called \emph{lattice anti-automorphism} if it is a order-reversing bijection with order-reversing inverse. A \emph{lattice skew-automorphism} is a map $L\to L$ that is either a lattice automorphism or a lattice anti-automorphism.
\end{defi}

Note that for a lattice automorphism $\auto$ of a lattice $L$ it holds for $x,y\in L$ that $x \leq y$ if and only if $\auto(x) \leq \auto(y)$. Dually, for a lattice anti-automorphism we have $x\leq y$ if and only if $\auto(y)\leq \auto(x)$. Moreover, the join and meet operations of a lattice are compatible with automorphisms and anti-automorphisms.

\begin{lem}\label{lem:lattice_auto_join_interchange}
	If $\auto \colon L \to L$ is a lattice automorphism, then
	\[
	\auto(x\vee y)=\auto(x)\vee\auto(y)\quad \text{ and } \quad\auto(x\wedge y)=\auto(x)\wedge\auto(y)
	\]
	hold for all $x,y \in L$.
\end{lem}

\begin{proof}
	We only show the first equality, the second follows analogously.
	Let $x,y\in L$ be arbitrary. Then $x,y \leq x\vee y$ by definition and since $\auto$ is an automorphism, we have $\auto(x), \auto(y)\leq \auto(x\vee y)$. The definition of the join then yields $\auto(x) \vee \auto(y)\leq \auto(x\vee y)$. To prove the other inequality, note that there is a $z\in L$ such that $\auto(x)\vee\auto(y)=\auto(z)$, since $\auto$ is a bijection. This means that $x,y \leq z$, as $\auto\inv$ is order-preserving, and hence $x \vee y \leq z$. Applying the automorphism $\auto$ once again yields $\auto(x\vee y) \leq \auto(z)=\auto(x)\vee \auto(y)$.
\end{proof}

\begin{lem}\label{lem:lattice_antiauto_join_interchange}
	If $\antiauto \colon L \to L$ is a lattice anti-automorphism, then
	\[
	\antiauto(x\vee y)=\antiauto(x)\wedge\antiauto(y) \quad\text{ and } \quad	\antiauto(x\wedge y)=\antiauto(x)\vee\antiauto(y)
	\]
	hold for all $x,y \in L$.
\end{lem}

\begin{proof}
	The proof is the dual version of the proof of \cref{lem:lattice_auto_join_interchange}. We only show the first equality, the second is the analog.
	Let $x,y\in L$ be arbitrary. Since $x,y\leq x \vee y$ holds in $L$ and $\antiauto$ is an anti-automorphism, it follows that $\antiauto(x\vee y)\leq \antiauto(x),\antiauto(y)$ and hence $\antiauto(x\vee y)\leq \antiauto(x)\wedge\antiauto(y)$. For the other inequality, let $z\in L$ be such that $\antiauto(x)\wedge\antiauto(y)=\antiauto(z)$. Such an element exists, because $\antiauto$ is a self-bijection. By definition of the meet we have $\antiauto(z) = \antiauto(x)\wedge\antiauto(y) \leq \antiauto(x), \antiauto(y)$, which in turn implies $x,y\leq z$, since $\antiauto$ is order-reversing. Hence we have that $x \vee y \leq z$ and applying the anti-automorphism $\antiauto$ gives $\antiauto(x)\wedge\antiauto(y) = \antiauto(z) \leq \antiauto(x \vee y)$.
\end{proof}

The cover sets and covered sets behave nicely under lattice skew-automorphisms as well. We state the version for automorphisms here, because this is the one we need in \cref{chap:autos}.

\begin{lem}\label{lem:cov_set_autos}
	Let $L$ be a lattice and $\ncauto$ be a lattice automorphism of $L$. Then it holds for all $x\in L$ that $\ncauto(\bel(x))=\bel(\ncauto(x))$ and $\ncauto(\ab(x))=\ab(\ncauto(x))$. In particular, the cardinalities of cover sets and covered sets are preserved under automorphisms.
\end{lem}

\begin{proof}
	We show that $y \in \ncauto(\bel(x))$ if and only if $y \in \bel(\ncauto(x))$. The proof for $\ncauto(\ab(x))=\ab(\ncauto(x))$ works analogously. 
	
	If $y \in \ncauto(\bel(x))$, then there is a $z \lessdot x$ such that $y=\ncauto(z)$. Suppose that there is a $v\in L$ such that $\ncauto(z) < v < \ncauto(x)$. Since $\ncauto\inv$ is order-preserving, we get that $z < \ncauto\inv(v) < x$, which contradicts $z \lessdot x$. Hence $y = \ncauto(z) \lessdot \ncauto(x)$, which means that $y \in \bel(\ncauto(x))$.
	
	For the other direction let $z\in L$ be such that $y=\ncauto(z) \lessdot \ncauto(x)$, that is $y \in \bel(\auto(x))$.
	If there were a $v \in L$ such that $z < v < x$, then $\ncauto(z) < \ncauto(v) < \ncauto(x)$ because $\ncauto$ is order-preserving. This is a contradiction to $y=\ncauto(z) \lessdot \ncauto(x)$. Hence $y=\ncauto(z)$ for $z\lessdot x$, which means that $y \in \ncauto(\bel(x))$.
\end{proof}

\subsection{The lattice associated to a vector space}\label{sec:lattice_vec_space}

To every finite dimensional vector space we associate a lattice as follows. Let $V$ be a vector space of dimension $n$ defined over the field $\F$. Let
\[
\lam(V) \coloneqq \Set{U \subseteq V \str U \text{ linear subspace of } V}
\]
be the set of all linear subspaces of $V$. We abbreviate $\lam$ for $\lam(V)$ if the vector space $V$ is understood. A partial order $\leq$ on $\lam$ is defined by
\[
U' \leq U \text{ if and only if } U' \subseteq U,
\]
that is by the inclusion of subspaces. With this partial order, the poset $\lam$ is a graded lattice of rank $n$. The rank of an element $U\in\lam$ is given by its dimension as a vector space, that is $\rk(U)=\dim(U)$. The join operation $\vee$ is given by the sum of subspaces and the meet operation $\wedge$ is the intersection of subspaces.
\cref{fig:lattice_F23} shows the lattice $\lam$ associated to the finite $3$-dimensional vector space $\F_2^3$. By $e_i$ we denote the $\ith$ standard basis vector of $\F_2^3$ and the brackets $\braket{\cdot }$ denote the linear span.

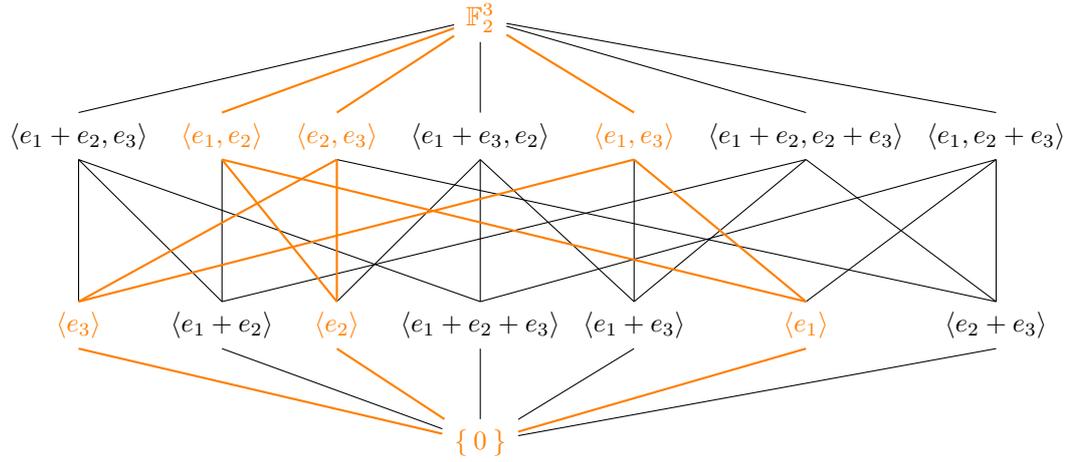
\begin{figure}
	\begin{center}
		\small{		
			\begin{tikzcd}[column sep=0mm,row sep = large]
			&&&\textcolor{Orange}{\F_2^3}
			\ar[d,dash,end anchor = north]\ar[dl,dash,Orange,thick,end anchor = north]\ar[dll,dash,Orange,thick,end anchor = north]\ar[dlll,dash,end anchor = north]\ar[dr,dash,Orange,thick,end anchor = north]\ar[drr,dash,end anchor = north]\ar[drrr,dash,end anchor = north]&&&\\
			
			\Braket{e_1+e_2,e_3} & \textcolor{Orange}{\Braket{e_1,e_2}} &\textcolor{Orange}{\Braket{e_2,e_3}}&\Braket{e_1+e_3,e_2}&
			\textcolor{Orange}{\Braket{e_1,e_3}}&\Braket{e_1+e_2,e_2+e_3} &\Braket{e_1,e_2+e_3}\\
			&&&&&&\\
			\textcolor{Orange}{\Braket{e_3}}
			\ar[uu,dash, start anchor=north, end anchor = south]
			&
			\Braket{e_1+e_2}
			\ar[uul,dash, start anchor=north, end anchor = south]
			\ar[uu,dash, start anchor=north, end anchor = south]
			\ar[uurrrr,dash, start anchor=north, end anchor = south]&
			\textcolor{Orange}{\Braket{e_2}}
			\ar[uur,dash, start anchor=north, end anchor = south]&
			\Braket{e_1+e_2+e_3}
			\ar[uu,dash, start anchor=north, end anchor = south]
			\ar[uurrr,dash, start anchor=north, end anchor = south]
			\ar[uulll,dash, start anchor=north, end anchor = south]&
			\Braket{e_1+e_3}
			\ar[uul,dash, start anchor=north, end anchor = south]
			\ar[uu,dash, start anchor=north, end anchor = south]
			\ar[uur,dash, start anchor=north, end anchor = south]&
			\textcolor{Orange}{\Braket{e_1}}
			\ar[uur,dash, start anchor=north, end anchor = south]&
			\Braket{e_2+e_3}
			\ar[uu,dash, start anchor=north, end anchor = south]
			\ar[uul,dash, start anchor=north, end anchor = south]
			\ar[uullll,dash, start anchor=north, end anchor = south]
			\ar[from=llllll,uull,Orange,thick,dash,start anchor=north, end anchor = south]
			\ar[from=llllll,uullll,Orange,thick,dash,start anchor=north, end anchor = south]
			\ar[from=llll,uulllll,Orange,thick,dash,start anchor=north, end anchor = south]
			\ar[from=llll,uullll,Orange,thick,dash,start anchor=north, end anchor = south]
			\ar[from=l,uulllll,Orange,thick,dash,start anchor=north, end anchor = south]
			\ar[from=l,uull,Orange,thick,dash,start anchor=north, end anchor = south]\\
			&&&\textcolor{Orange}{\Set{0}}
			\ar[u,dash,end anchor = south]\ar[ul,dash,Orange,thick,end anchor = south]\ar[ull,dash,end anchor = south]\ar[ulll,dash,Orange,thick,end anchor = south]\ar[ur,dash,end anchor = south]\ar[urr,dash,Orange,thick,end anchor = south]\ar[urrr,dash,end anchor = south]&&&
			\end{tikzcd}
		}
		\caption{The Hasse diagram of the lattice $\lam(\F_2^3)$. The  subset highlighted in orange is a Boolean sublattice.}
		\label{fig:lattice_F23}
	\end{center}
\end{figure}

If we again take a look at \cref{fig:hasse_bool3} and compare it to the Hasse diagram of the lattice $\lam=\lam(\F_2^3)$, we see that $\lam$ contains subgraphs that are isomorphic to the Hasse diagram of $\B_3$. Hence $\lam$ has sublattices that are Boolean lattices. Note that the atoms of the Boolean sublattice highlighted in \cref{fig:lattice_F23} are $\braket{e_1}$, $\braket{e_2}$ and $\braket{e_3}$ and that $\set{e_1,e_2,e_3}$ is a basis of $\F_2^3$.  
We can describe the Boolean sublattices of $\lam$ by bases, or more generally, by frames of $V$. A \emph{frame} of an $n$-dimensional vector space is a set of $n$ $1$-dimensional subspaces $L_1, \ldots, L_n$ such that $V$ is the direct sum $L_1\oplus \ldots \oplus L_n =V$. Note that every basis of $V$ induces a frame of $V$. Moreover, two bases give rise to the same frame if and only if their elements differ only by scalar multiples. This correspondence allows us to describe lattice automorphisms of $\lam$ using linear algebra.

\begin{lem}
	Every vector space automorphism $\vecauto\colon V \to V$ induces a unique lattice automorphism $\lamauto\colon \lam(V)\to\lam(V)$.
\end{lem}

\begin{lem}\label{lem:lamauto_induces_vecautos}
	Let $V$ be a vector space defined over $\F$ and $\Set{ \alpha_1, \ldots, \alpha_n }$ a basis of $V$. 
	Let $\lamauto\colon \lam(V) \to \lam(V)$ be a lattice automorphism such that $\lamauto(\braket{ \alpha_i})=\braket{ \alpha_i}$ for all $i\in \set{1, \ldots, n}$.
	Then there exist $\lambda_1, \ldots, \lambda_n \in \F\setminus \set{0}$ such that the vector space automorphism $\vecauto\colon V \to V$ defined via $\alpha_i \mapsto \lambda_i\alpha_i$ induces the lattice automorphism $\lamauto$. 
	Moreover, the vector $(\lambda_1, \ldots, \lambda_n)\in \F^n$ is unique up to non-zero scalar multiples.
\end{lem}

\section{Simplicial complexes}\label{sec:simpl_compl}

There are different ways to define a simplicial complex in the literature. We use the definition of \cite{staece}.

\begin{defi}
	A \emph{simplicial complex} with vertex set $V$ is a collection $\si$ of subsets of $V$ such that
	\begin{enumerate}
		\item $\set{v} \in \si$ for all $v\in V$, and
		\item if $A \in \si$, then $A' \in \si$ for all $A' \subseteq A$.
	\end{enumerate}	
	The elements of $\si$ are called \emph{simplices}. If $A' \subseteq A \in \si$, then $A'$ is said to be a \emph{face} of $A$. 
	The \emph{rank} of a simplex $A$ is defined to be its cardinality, that is $\rk(A) \coloneqq \#A$, and the \emph{dimension} of $A$ is $\dim(A)\coloneqq \rk(A)-1$. A simplex of dimension $k$ is called \emph{$k$-simplex}. The dimension of the simplicial complex is the supremum of the dimensions of its simplices, that is $\dim(\si)\coloneqq \sup_{A\in \si}\dim(A)$. The rank $1$ elements of $\si$ are called \emph{vertices}. A subset $\si'\subseteq \si$ is called \emph{subcomplex} of $\si$ if for all $A\in \si'$ and  $A' \subseteq A$ it holds that $A' \in \si'$. 
	The \emph{$k$-skeleton} of $\si$ is the $k$-dimensional subcomplex $\si^{(k)}$ of $\si$ that consists of all simplices of $\si$ that have dimension at most $k$.
	A subcomplex $\si'$ with vertex set $V'\subseteq V$ is said to be \emph{spanned} by $V'$ if for every $A\in \si$ with $A \subseteq V'$ also $A\in \si'$. We then say that $\si'$ is an \emph{induced} subcomplex.   
\end{defi}

If $\si_1$ and $\si_2$ are simplicial complexes with respective vertex sets $V_1$ and $V_2$, then the \emph{(simplicial) join} of $\si_1$ and $\si_2$ is the simplicial complex $\si_1 \ast \si_2$ with vertex set $V_1\cup V_2$ and simplices $A_1 \cup A_2$ for every $A_1\in \si_1$ and $A_2 \in \si_2$. We call $\si_1$ and $\si_2$ the \emph{join factors} of $\si_1\ast\si_2$. Note that the join factors can be naturally interpreted as subcomplex of the join by considering simplices of the form $A \cup \emptyset$ for $A\in \si_i$. The join of a finite collection of simplicial complexes is defined analogously. 

The \emph{affine realization} of a $d$-dimensional simplex is the convex hull of its vertices in general position in $\R^d$. The \emph{affine realization} of a simplicial complex is the union of the affine realizations of all its simplices, hence it is a topological space. It is convenient to abuse notation and use the same symbol for both abstract simplices and simplicial complexes, and their affine realizations. This allows us to speak, for instance, of the homotopy type of the simplicial complex, which is the homotopy type of its affine realization \cite[Chap. I.7]{bh}. In \cref{defi:metric} we define the analog in spherical geometry, the spherical realization of a simplicial complex.

\begin{defi}
	Let $\si$ be a $d$-dimensional simplicial complex. If every simplex of $\si$ is a face of a $d$-simplex, then $\si$ is called \emph{pure}. Now suppose that $\si$ is pure. The simplices of dimension $d-1$ are called \emph{panels}. Two $d$-dimensional simplices $A,B \in \si$ are called \emph{adjacent} if their intersection is a panel. 
	
	A \emph{gallery} of length $k$ in $\si$ is a sequence $\gamma = (C_0, \ldots, C_k)$ of maximal simplices such that $C_{i-1}$ and $C_i$ are adjacent for all $1\leq i <k$. We say that the gallery $\gamma$ \emph{connects} $C_0$ and $C_k$. If any two maximal simplices of $\si$ can be connected by a gallery, then $\si$ is called a \emph{chamber complex} and its maximal simplices are called \emph{chambers}. The set of chambers of a chamber complex $\si$ is denoted by $\C(\si)$. The \emph{(gallery) distance} of two chambers $C,C' \in \si$ is the minimal length of a gallery connecting $C$ and $C'$ and is denoted by $\dig(C,C')$. A gallery connecting $C$ and $C'$ is called \emph{minimal} if it has length $\dig(C,C')$. 
	
	A subcomplex $\si'$ of a chamber complex $\si$ is called \emph{chamber subcomplex} of $\si$ if $\si'$ is a chamber complex of dimension $\dim(\si)$. The \emph{convex hull} of two chambers $C, D\in \si$ is the chamber subcomplex $\convg(C,D)\subseteq \si$ whose chambers are the chambers of minimal galleries connecting $C$ and $D$. 
\end{defi} 

The following concepts allow the local study of simplicial complexes. 
Let $\si$ be a simplicial complex and $A\in \si$ a simplex. The \emph{open star} of $A$ in $\si$ can be seen as the analog of an open neighborhood in a metric space and is defined as
\[
\osta(A,\si) \coloneqq \Set{A'\in \si \str A \subseteq A'} \subseteq \si,
\]
that is the collection of all simplices that contain $A$.
Note that the open star is \emph{not} a subcomplex of $\si$ in general, since it is not closed under containment. However, if we remove an open star from a simplicial complex, the remaining part is still a simplicial complex. The simplicial analog of the boundary of a neighborhood is the \emph{link} of $A$ in $\si$, which is defined as  
\[
\lk(A,\si) \coloneqq \Set{A'\in\si \str A \cap A' = \emptyset \text{ and } A \cup A' \in \si} \subseteq \si.
\]
The link is always a subcomplex of $\si$, as is its union with the open star 
\[
\sta(A,\si) \coloneqq \lk(A,\si) \cup \osta(A,\si),
\]
which is called the \emph{star}, or \emph{closed star} if we want emphasize the closure property, of $A$ in $\si$. \cref{fig:link_star} shows an example of a link and an open star. If $\si$ is a chamber complex, then the link and the star are chamber complexes themselves, provided they are connected. Note that the link is \emph{not} a chamber subcomplex of $\si$ in general, since $\dim(\lk(A,\si))<\dim(\si)$ for all non-empty simplices $A\in\si$. However, the star of each simplex is a chamber subcomplex of $\si$.
\begin{figure}
	\begin{center}
		\begin{tikzpicture}
		\def\r{2.1} 
		\def\grauOut{Gray} 
		\def\grauIn{Superlightgray} 
		\def\ColorOut{Blue} 
		\def\SurfColorIn{lightOrange} 
		\def\EdgeColorIn{Orange} 
		\foreach \w in {1,2,...,6}
		\coordinate (e\w) at (\w * 60 : \r );
		\coordinate (f1) at ($(e5) + (\r,0)$);
		\coordinate (f2) at ($(e2) - (\r,0)$);
		\filldraw[fill=\grauIn, draw=\grauOut] (e2)--(e3)--(f2)--(e2)
		(e5)--(e6)--(f1)--(e5);
		\filldraw[fill=\SurfColorIn,  draw=\ColorOut,very thick] (e1) -- (e2) --
		(e3) -- (e4) -- (e5) -- (e6) -- (e1);
		\foreach \n in {1,2,...,6}
		\draw[\EdgeColorIn,very thick] (0,0) -- (e\n);
		\foreach \w in {1,2,...,6}
		\node[\ColorOut] at (e\w) [mpunkt] {};
		\node[Orange] at (0,0) [mpunkt] {};
		\foreach \n in {1,2}
		\node[\grauOut] at (f\n) [mpunkt] {};
		\end{tikzpicture}
		\caption{The link of the central vertex is the subcomplex highlighted in blue and the simplices of the open star are colored in orange. The star is the colored subcomplex spanned by the blue and orange vertices.}%
		\label{fig:link_star}%
	\end{center}
\end{figure}
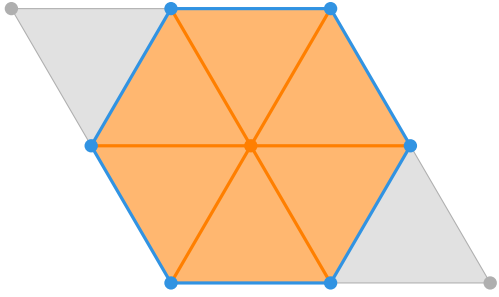

\begin{rem}\label{rem:star_contractible}
	The closed star of a simplex $A\in\si$ is the simplicial join of the link with the simplex $A$ itself, that is
	\[
	\sta(A,\si)=\lk(A,\si) \ast A.
	\]
	In metric terms, the closed star is the cone of the link with apex $A$, and as such it is contractible.
\end{rem}

\begin{rem}\label{rem:link_induced_subcplx}
	If $\si' \subseteq \si$ is a subcomplex and $A\in \si'$ a simplex, then it holds that $\lk(A,\si')\subseteq \lk(A,\si)\cap \si'$. In general, the equality does not hold, as one can see for instance if $\si$ is a $2$-simplex and $\si'$ its boundary. However, if we further assume that $\si'$ is an \emph{induced} subcomplex of $\si$, we get that $\lk(A,\si')= \lk(A,\si)\cap \si'$. In this case it also holds that $\sta(A,\si')= \sta(A,\si)\cap \si'$.
\end{rem}

\begin{defi}
	Let $\si$ and $\si'$ be two simplicial complexes with vertex sets $V$ and $V'$, respectively. Let $\auto\colon V \to V'$ be a map. The induced map $\auto\colon \si \to \si'$ is called \emph{simplicial map} if simplices are mapped to simplices. It is called \emph{non-degenerate} if $\auto$ preserves the dimension of all simplices. A non-degenerate simplicial map between chamber complexes is called \emph{chamber map}. An \emph{isomorphism} between two simplicial complexes is a bijective simplicial map $\si \to \si'$. In this case, $\si$ and $\si'$ are said to be \emph{isomorphic}. 
\end{defi}

\subsection{Order complexes}\label{sec:order_complex}

For every poset we can construct a simplicial complex, which basically \enquote{encodes} the same information as the poset itself.

\begin{defi}
	For a poset $P$, we define $\|P\|$ to be the simplicial complex  with
	\begin{enumerate}
		\item vertex set $P$, and
		\item $\set{p_0, p_1, \ldots, p_k}$ is a $k$-simplex of $\|P\|$ if and only if $(p_0,p_1,\ldots, p_k)$ is a chain in $P$.
	\end{enumerate}
\end{defi}

If $P$ is a graded poset of rank $n$, then the complex $\|P\|$ is a pure simplicial complex of dimension $n$. Moreover, we assign the rank $k$ of the element $p\in P$ to the vertex $p\in \|P\|$ and call $p\in \|P\|$ a \emph{rank $k$ vertex}.
A poset map $P \to Q$ induces a simplicial map $\|P\| \to \|Q\|$. If the poset map is injective, then the simplicial map is an embedding and in particular non-degenerate.

\begin{rem}\label{rem:oc_ind_subcplx}
	If $Q \subseteq P$ is a subposet, then $\|Q\| \subseteq \|P\|$ is an induced subcomplex, since the simplices of $\|P\|$ and $\|Q\|$ are uniquely determined by their vertex sets.
\end{rem}

If $L$ is a graded lattice with minimum $\mi$ and maximum $\ma$, then every maximal simplex in $\|L\|$ contains the two vertices $\mi$ and $\ma$.

\begin{defi}
	Let $L$ be a lattice with minimum $\mi$ and maximum $\ma$. Then the \emph{order complex} of the lattice $L$ is the simplicial complex $|L|\coloneqq \|L \setminus \set{\mi,\ma}\|$. 
\end{defi}

For a lattice $L$ with minimum $\mi$ and maximum $\ma$, the simplicial complex $\|L\|$ is the simplicial join of the order complex $|L|$ with the edge $\Set{\mi, \ma}$. Hence $\|L\|$ is contractible whereas the complex $|L|$ may have interesting homotopy type. This is for instance the case for a Boolean lattice $\B_3$, whose order complex is depicted in \cref{fig:order_complex_bool3}. Note that it is isomorphic to the barycentric subdivision of the boundary of a triangle and therefore has the homotopy type of a circle.

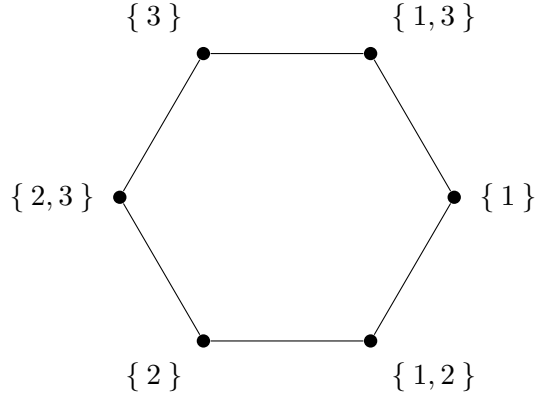
\begin{figure}
	\begin{center}
		\begin{tikzpicture}
		\foreach \w in {1,...,6} 
		\node[mpunkt] (p\w) at (-\w * 360/6 +60  : 2.2cm)  {};
		\foreach \w/\lab/\pos in {1/\Set{1}/right,2/\Set{1,2}/below right,3/\Set{2}/below left,4/\Set{2,3}/left,5/\Set{3}/above left,6/\Set{1,3}/above right} 
		\node[\pos=2mm]  at (p\w)  {$\lab$};
		
		\draw (p1)--(p2)--(p3)--(p4)--(p5)--(p6)--(p1);
		\end{tikzpicture}
		
		\caption{The order complex of the Boolean lattice $\B_3$.}
		\label{fig:order_complex_bool3}
	\end{center}
\end{figure}

Recall that the \emph{barycentric subdivision} of a simplicial complex $\si$ is a simplicial complex $\sd(\si)$, which has
\begin{enumerate}
	\item vertex set $\si \setminus \emptyset$, and
	\item $\set{A_0, A_1, \ldots, A_k}$ is a $k$-simplex of $\sd(\si)$ if and only if $A_0 \subsetneq A_1 \subsetneq \ldots \subsetneq A_k$ holds in $\si$.
\end{enumerate}

Hence the non-empty simplices of $\si$ are in bijection with the vertices of $\sd(\si)$.

\begin{lem}\label{rem:boolean_lattice_sd_bdry}
	The order complex of the Boolean lattice $\B_n$ is isomorphic to the barycentric subdivision of the boundary of the $(n-1)$-simplex.
\end{lem}

\begin{proof}
	Let $V=\Set{1, \ldots, n}$ and let $\si$ be the $(n-1)$-simplex with vertex set $V$. Its boundary is the $(n-2)$-skeleton, which we denote by $\partial\si$. The vertices of $\sd(\partial\si)$ are  the proper subsets of $V$. The simplices of $\sd(\partial \si)$ are given by nested sequences $A_0 \subsetneq A_1 \subsetneq \ldots \subsetneq A_k$ of vertices $A_0, \ldots, A_k$. If $A,A' \in \sd(\partial\si)$ are two simplices, then $A' \subseteq A$ if and only if the sequence of vertices corresponding to $A'$ is a subsequence of the one corresponding to $A$. 
	
	On the other hand, the vertices of the order complex $|\B_n|$ of the Boolean lattice $\B_n$ are the proper subsets of $V$. The simplices in the order complex are given by chains in the lattice $\B_n$. Let $C,C' \subseteq \B_n$ be chains and $A,A'\in |\B_n|$ the corresponding simplices. Then $C' \subseteq C$ if and only if $A' \subseteq A$ by definition of the order complex. 
	
	As $|\B_n|$ and $\sd(\partial\si)$ have the same vertices, simplices and inclusion relations, they are isomorphic simplicial complexes.
\end{proof}

We saw that the essential information of a Boolean lattice $\B_n$ is equivalently stored in its order complex, the barycentric subdivision of the boundary of the $(n-1)$-simplex. Sometimes it is useful to have the theory of posets and lattices available, so we use the poset-theoretic interpretations. This perspective is mainly used in \cref{part2}. But at other points, mainly in \cref{part3} of this thesis, it is more convenient to study simplicial complexes and use the theory established there.

Having available two different but equivalent perspectives allows us to transfer concepts from one to the other. In \cref{cha:buildings} we see that the order complex of the lattice $\lam(V)$ is a spherical building. Buildings are a well-studied class of simplicial complexes surrounded by a rich theory.

\cleardoublepage
\chapter{Coxeter groups}\label{chap:cox}

Coxeter groups play an important role in various fields of mathematics. They are central to both buildings and non-crossing partitions. The former are built-up from Coxeter complexes, and the latter sit inside a Coxeter group. The main references for Coxeter groups are the books \cite{ab}, \cite{bro}, \cite{davis} and \cite{hum}.

\section{Coxeter systems}

\begin{defi}
	Let $S=\set{s_1, \ldots, s_n}$ be a finite set and $m=(m(s_i,s_j))_{1\leq i,j \leq n}$ be a matrix with entries in $\N \cup \infty$. Then $m$ is called a \emph{Coxeter matrix} if it satisfies
	\begin{enumerate}
		\item $m(s_i,s_j)=m(s_j,s_i)$ for all $1\leq i,j \leq n$, and
		\item $m(s_i,s_j)=1$ if and only if $i=j$.
	\end{enumerate}
	A Coxeter matrix $m$ gives rise to a group $W$ via the presentation
	\[
	\Braket{ S | (s_is_j)^{m(s_i,s_j)} \text{ if } m(s_i,s_j)\neq \infty}
	\]
	and the pair $(W,S)$ is called \emph{Coxeter system}. We refer to $W$ as \emph{Coxeter group} and to $S$ as the set of \emph{Coxeter generators} or \emph{Coxeter generating set}. The \emph{rank} of the Coxeter system $(W,S)$ is the cardinality of $S$. A Coxeter system or a Coxeter group is called \emph{finite} or \emph{spherical} if $W$ is a finite group.
\end{defi}

Note that a Coxeter system is always equipped with a group presentation and a Coxeter matrix.

Although it is not a priori clear, the Coxeter generators are different elements in the group and the Coxeter generating set $S$ is a minimal generating set for $W$. Moreover, the order in $W$ of the product $s_is_j$ is indeed $m(s_i,s_j)$ \cite[Prop. 1.1.1, Cor. 1.4.8]{bb}. To simplify notation we set $m_{ij}\coloneqq m(s_i,s_j)$. 

It is convenient to encode the information of a Coxeter matrix $m$, which completely determines the Coxeter system $(W,S)$ up to isomorphism, in a graphical way.

\begin{defi}
	Let $(W,S)$ be a Coxeter system with Coxeter matrix $m$. The \emph{Coxeter diagram} corresponding to $(W,S)$ is an undirected, labeled graph with vertex set $S$. There is an edge between $s_i$ and $s_j$ if and only $m_{ij} \geq 3$. An edge $\set{s_i,s_j}$ is labeled with $m_{ij}$ whenever $m_{ij}>3$. 
\end{defi} 

A Coxeter system $(W,S)$ is said to be \emph{irreducible} if its Coxeter diagram is connected, and \emph{reducible} otherwise. The finite Coxeter systems are completely classified up to isomorphism via the classification of possible Coxeter diagrams \cite[Chap. 2, Thm. 6.4]{hum}. The complete list of Coxeter diagrams for finite irreducible Coxeter systems is shown in \cref{fig:cox_diagrams}. Every connected Coxeter diagram comes equipped with a type $X_n$, which is shown in \cref{fig:cox_diagrams}, where the index $n$ equals the number of vertices of the diagram. The \emph{type} of a irreducible Coxeter system, and also of the Coxeter group, is defined to be the type of the corresponding Coxeter diagram. Note that a Coxeter system of type $X_n$ has rank $n$. We denote the Coxeter group of type $X_n$ by $W(X_n)$. Whenever we refer to a Coxeter group $W$ only, we assume that the corresponding Coxeter generating set $S$ and hence the Coxeter system $(W,S)$ is understood.

In this thesis, we are only interested in finite Coxeter groups, but there is a rich theory of infinite Coxeter groups as well. A class of infinite Coxeter systems, the so called \emph{affine} Coxeter systems, are also classified by its Coxeter diagrams \cite{hum}. 

\renewcommand{\arraystretch}{1,5}
\begin{table}
	\begin{center}
		\begin{tabular}{|m{2.4cm}||m{6.2cm}|}
			\hhline{|-||-|}
			type of $W$ & Coxeter diagram\\
			\hhline{:=::=:}
			$A_n$, $n \geq 1$ & \CoxA\\
			$B_n$, $n \geq 2$ & \CoxB\\[3mm]
			$D_n$,  $n \geq 3$ & \CoxD\\
			$E_6$  & \CoxEsechs\\
			$E_7$  & \CoxEsieben\\
			$E_8$  & \CoxEacht\\
			$F_4$  & \CoxF\\
			$H_3$  & \CoxHdrei\\
			$H_4$  & \CoxHvier\\
			$I_2(m)$, $m\geq 3$  & \CoxIm\\
			\hhline{|-||-|}
		\end{tabular}
		\caption{The complete list of connected Coxeter diagrams that give rise to all finite irreducible Coxeter groups.}
		\label{fig:cox_diagrams}
	\end{center}
\end{table}
\renewcommand{\arraystretch}{1}

Note that all Coxeter diagrams corresponding to finite Coxeter systems are trees. There are three infinite families of finite Coxeter groups with increasing rank, those of type $A_n$ for $n\geq 1$, type $B_n$ for $n \geq 2$, and type $D_n$ for $n\geq 3$. These three families are called the \emph{classical types}. A thorough study of the Coxeter groups of classical types, and in particular their corresponding non-crossing partitions, is given in \cref{chap:classical_types}. Note that the Coxeter group of type $A_{n-1}$ is the symmetric group $S_n$ on $n$ letters. 

There is another infinite family of finite irreducible Coxeter systems of rank $2$, the family of type $I_2(m)$ for an integer $m \geq 3$. The Coxeter group of type $I_2(m)$ is the dihedral group $D_m$ of order $2m$, which is the symmetry group of the regular $m$-gon. Dihedral groups play an important role in \cref{chap:autos}, where we study automorphisms and anti-automorphisms of non-crossing partitions.

\begin{figure}
	\begin{center}
		\begin{tikzpicture}
		\def\l{3} 
		\def\d{\l/6} 
		\def\rd{\l/1.3} 
		\def\FarbemarkObj{Orange} 
		\def\FarbeObj{black} 
		
		f\clip (-\l -1, -\l +1) rectangle (\l + 1, \l + 1);
		\draw[Lightgray] (330:\rd) -- (90:\rd) -- (210:\rd) -- cycle;
		
		
		\foreach \winkel/\farbe in {5/\FarbeObj, 11/\FarbeObj}
		\draw[\farbe] (0,0) -- (\winkel * 30: \l);
		
		\foreach \winkel/\farbe in {1/\FarbemarkObj, 3/\FarbemarkObj, 7/\FarbemarkObj, 9/\FarbemarkObj}
		\draw[\farbe, thick] (0,0) -- (\winkel * 30: \l);

		\begin{scope}[Blue]
		\node[above right=0mm and 0.9mm, fill=white, fill opacity = 0.8, text opacity = 1, inner sep=0] at (0,0) {$\frac{2\pi}{3}$};	
		\draw (0,0)+(330:\rd/3) arc (330:450:\rd/3);
		\draw[->, thick] (0,0)+(335:\rd) arc (335:445:\rd) node[pos=0.75,above right] {$\sigma$};
		\end{scope}
		
		\begin{scope}[\FarbemarkObj]
		\node[below left=1mm and -0.5mm] at (0,0) {$\frac{\pi}{3}$};
		\draw[thick] (0,0)+(210:\rd/3) arc (210:270:\rd/3);
		\end{scope}
		
		\draw[thick, \FarbemarkObj,<->] (30 : \l + \d/2) + (0,-\d/1.3) arc (-10:80:\d) node[midway,above right] {$\sigma_2$}; 
		\draw[thick, \FarbemarkObj,<->] (90 : \l + \d/2) + (\d/1.5, -\d/4) arc (45:135:\d) node[midway, above] {$\sigma_1$}; 
		\end{tikzpicture}	
		\caption{The symmetries of the triangle are generated by reflections.}
		\label{fig:symmetry_triangle}
	\end{center}
\end{figure}
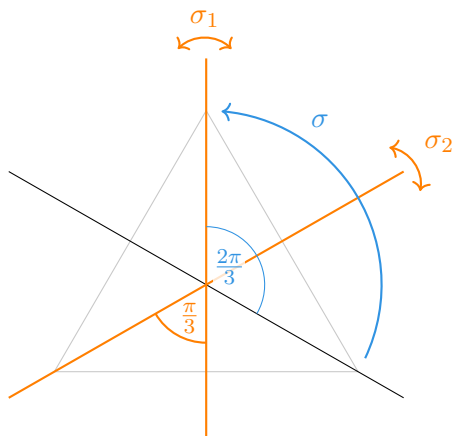

\begin{exa}\label{exa:dihedral_group}
	Let $(W,S)$ be the Coxeter system of type $I_2(m)$ for an integer $m \geq 3$. Then the Coxeter group $W(I_2(m))$ has the presentation 
	\[
	\Braket{ s_1, s_2 | s_1^2,\,\; s_2^2,\,\; (s_1s_2)^m }
	\]
	and can be interpreted as the symmetry group of the regular $m$-gon, that is the dihedral group $D_m$ of order $2m$, as follows. Consider two lines $H_1$ and $H_2$ in the Euclidean plane that intersect in the origin at an angle of $\frac{\pi}{m}$. \cref{fig:symmetry_triangle} depicts the situation for $m=3$. Let $\sigma_1$ and $\sigma_2$ be the Euclidean reflections on $H_1$ and $H_2$, respectively, which have order $2$. The composition $\sigma_1\circ\sigma_2$ is rotation through an angle of $\frac{2\pi}{m}$ and hence of order $m$. The group $D_m$ of isometries of the Euclidean plane generated by $\sigma_1$ and $\sigma_2$ leaves an inscribed $m$-gon invariant if the lines $H_1$ and $H_2$ go through diagonals of the $m$-gon. If $m$ is even, a \emph{diagonal} of the $m$-gon is either the segment connecting two opposite vertices, or the segment connecting the midpoints of two opposite sides. If $m$ is odd, then a \emph{diagonal} of an $m$-gon is a segment connecting a vertex with the midpoint of its opposite side. Note that since the angle between $H_1$ and $H_2$ is $\frac{\pi}{m}$, this guarantees that $H_1$ and $H_2$ pass through consecutive diagonals of the $m$-gon.
	
	The map $W(I_2(m)) \to D_m$ given by $s_1 \mapsto \sigma_1$ and $s_2\mapsto \sigma_2$ is a group isomorphism. In particular, the product $s_1s_2$ gets mapped to a rotation through an angle of $\frac{2\pi}{m}$.
	Another canonical generating set for the dihedral group $D_m$ is given by a reflection and a rotation through an angle of $\frac{2\pi}{m}$. Hence the Coxeter group $W(I_2(m))$ has a group presentation
	\[
	\Braket{ s,  \rho | s^2,\,\; \rho^m,\,\; (s\rho)^2 },
	\]
	where we get from the second presentation to the first by mapping $s \mapsto s_1$ and $\rho \mapsto s_1s_2$. Note that although this presentation is a presentation for the group $W(I_2(m))$, it does \emph{not} give rise to a Coxeter system, since the generator $\rho$ has order $m$, which is different from $2$.
	Since the Coxeter diagrams of type $I_2(3)$ and $A_2$ are the same, the corresponding Coxeter groups are the same. Hence this example also shows how the symmetric group $S_3$ acts on $\R^2$.
\end{exa}

We realized the Coxeter group of type $I_2(m)$ as finite reflection group. This is not a phenomenon of the dihedral group. The finite Coxeter groups are exactly the finite reflection groups, that is finite groups generated by Euclidean reflections \cite[Thm. 6.5]{hum}.

\section{The geometric representation and roots}\label{sec:geom_rep}

In this section we are sketching how a finite Coxeter group can be realized as Euclidean reflection group via the \emph{geometric representation}. The geometric representation is defined for all Coxeter groups in a more general setting, but we restrict ourselves to the finite case, since this is the only one we need. We follow \cite[Chap. 5f.]{hum}.

Let $(W,S)$ be a finite Coxeter system of rank $n$ with $S=\set{s_1, \ldots, s_n}$ and let $V\cong \R^n$ be a Euclidean vector space with the inner product $(\cdot,\cdot)$. For every element $s\in S$ choose a vector $\alpha_s$ such that $\set{\alpha_s \str s\in S }$ forms a basis of $V$ and that
\[
(\alpha_{s_i}, \alpha_{s_j}) = - \cos \frac{\pi}{m_{ij}}
\]
holds for all $1 \leq i,j \leq n$. The last condition ensures that the hyperplanes $\alpha_{s_i}\com$ and $\alpha_{s_j}\com$ meet at an angle of $\frac{\pi}{m_{ij}}$. 
Moreover, note that $(\alpha_{s_i}, \alpha_{s_i}) = - \cos \pi =1$ for all $1\leq i \leq n$. Hence all vectors $\alpha_{s_i}$ have the same length.

For every $s\in S$ define the \emph{reflection} $\sigma_s \colon V \to V$ by
\[
\sigma_s(v)=v-2(\alpha_s,v)\alpha_s,
\]
which fulfills $\sigma(\alpha_s)=-\alpha_s$ and fixes $\alpha_s\com$ pointwise. Then the unique homomorphism
\begin{align*}
\sigma \colon W \to \gl(V)
\end{align*}
that maps $s$ to $\sigma_s$ is a faithful representation of the Coxeter group $W$, called the \emph{geometric representation}. 
Note that $\sigma(W)$ preserves the inner product on $V$. To simplify notation, we write $w(v)$ instead of $\sigma(w)(v)$ for $w\in W$ and $v\in V$.

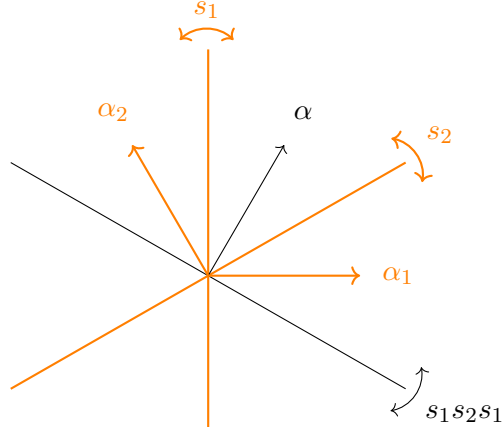
\begin{figure}
	\begin{center}
		\begin{tikzpicture}
		\def\l{3} 
		\def\d{\l/6} 
		\def\FarbemarkObj{Orange} 
		\def\FarbeObj{black} 

		\clip (-\l -1, -\l +1) rectangle (\l + 1, \l + 1);
		\foreach \winkel/\farbe in {5/\FarbeObj, 11/\FarbeObj}
		\draw[\farbe] (0,0) -- (\winkel * 30: \l);	
		
		\foreach \winkel/\farbe in {1/\FarbemarkObj, 3/\FarbemarkObj, 7/\FarbemarkObj, 9/\FarbemarkObj}
		\draw[\farbe, thick] (0,0) -- (\winkel * 30: \l);  
		
		\draw[thick, \FarbemarkObj,<->] (30 : \l + \d/2) + (0,-\d/1.3) arc (-10:80:\d)  node[midway,above right] {$s_2$}; 
		\draw[thick, \FarbemarkObj,<->] (90 : \l + \d/2) + (\d/1.5, -\d/4) arc (45:135:\d)  node[midway,above] {$s_1$}; 
		\draw[\FarbeObj, <->] (330 : \l + \d/2) + (-\d/1.2,-\d/3) arc (280 : 370 : \d)  node[midway,below right] {$s_1s_2s_1$}; 
		
		\foreach \winkel/\farbe in {2/\FarbeObj}
		\draw[\farbe,->] (0,0) -- (\winkel * 30 : \l/1.5);
		
		\foreach \winkel/\farbe in {0/\FarbemarkObj,4/\FarbemarkObj}
		\draw[thick, \farbe,->] (0,0) -- (\winkel * 30 : \l/1.5);
		
		\foreach \winkel/\farbe/\desc in {0/\FarbemarkObj/\alpha_1, 60/\FarbeObj/\alpha, 120/\FarbemarkObj/\alpha_2}
		\node[\farbe] at (\winkel : \l/1.5 + \d) {$\desc$};
		\end{tikzpicture}
		\caption{The roots and reflections corresponding to the symmetric group $S_3$. The simple roots and simple reflections are highlighted in orange.}
		\label{fig:roots_A2}
	\end{center}
\end{figure}

The vectors of the form $\alpha_s$ for $s\in S$ are called \emph{simple roots}. \cref{fig:roots_A2} shows the simple roots $\alpha_1=\alpha_{s_1}$ and $\alpha_2=\alpha_{s_2}$ corresponding to the Coxeter group $W(I_2(3))=S_3$, which we already discussed in \cref{exa:dihedral_group}. Note that the reflection $s_1$ maps the hyperplane $H_2 = \alpha_2\com$ to a hyperplane $H$, on which $s_1s_2s_1$ acts as reflection. For $\alpha = s_1(\alpha_2)$ this reflection is given by $v \mapsto v-2(\alpha,v)\alpha$.

We define the \emph{root system} of $(W,S)$ to be the set
\[
\Phi \coloneqq \Set{ w(\alpha_s) \str w \in W, s \in S }
\]
and call its elements \emph{roots}. 
Since $\set{\alpha_s \str s\in S }$ is a basis of $V$, every root $\alpha \in \Phi$ can be uniquely written as $\alpha = \sum_{i=1}^n a_i \alpha_{s_i}$ for $a_i \in \R$. A root is called \emph{positive} if all $a_i \geq 0$ and \emph{negative} if all $a_i \leq 0$. The set of positive roots is denoted by $\Phi^+$.

As we have already seen in the example of the Coxeter group $S_3$ above, we now can associate a reflection $s_\alpha\in W$ to every root $\alpha\in \Phi$ by setting
\[
s_\alpha(v)= v - 2(\alpha,v)\alpha.
\]
If $\alpha = w(\alpha_s)$ for $w\in W$ and $s\in S$, then $s_\alpha=wsw\inv$. The set
\[
T \coloneqq \set{wsw\inv \str w\in W, s \in S }
\]
is called the \emph{set of reflections} of the Coxeter system $(W,S)$ and its elements \emph{reflections}. The map $\Phi \to T$ given by $w(\alpha_s) \to wsw\inv$ is well-defined and two-to-one, where the restriction to $\Phi^+$ is one-to-one. Moreover, it maps the subset of simple roots to the set $S$ of Coxeter generators. The elements of $S$ therefore are called \emph{simple reflections} as well. For a root $\alpha \in \Phi$ we denote the corresponding reflection in $T$ by $s_\alpha$. The positive root corresponding to $t\in T$ is denoted by $\alpha_t$ and the negative root by $-\alpha_t$.

To every element in $W$ we associate two subspaces of $V$.

\begin{defi}
	The \emph{fixed space} of $w\in W$ is $\fix(w) \coloneqq \ker(w-\id)$ and the \emph{moved space} of $w$ is defined as $\mov(w) \coloneqq \im(w-\id)$.
\end{defi}

The fixed space consists, as the name suggests, of all elements of $V$ that are fixed under the linear transformation $w$. The orthogonal complement of $\fix(w)$ in $V$ is precisely the moved space $\mov(w)=\fix(w)\com$  \cite[Prop. 1]{bra_watt_par_ord}. If $w=t\in T$, then the fixed space is $\braket{ \alpha_t }\com$, also denoted by $\alpha_t\com$, and is called \emph{hyperplane}. 

\begin{rem}\label{rem:simplicial_structure_sphere}
	The hyperplanes $\alpha\com$ for $\alpha \in \Phi^+$ induce a simplicial structure on the unit sphere. In the case of the symmetric group $S_3$, \cref{fig:symmetry_triangle} indicates the simplicial structure on the circle, which is isomorphic to the barycentric subdivision of the boundary of a triangle. 
\end{rem}

Note that the set of vectors $\Phi$, which arises from the geometric representation, indeed satisfies the axioms of a \emph{root system}, which are
\begin{enumerate}
	\item $\Phi \cap \R\alpha = \set{ -\alpha, \alpha }$ for all $a\in \Phi$, and
	\item $s_\alpha \Phi = \Phi$ for all $\alpha \in \Phi$.
\end{enumerate}
A root system is called \emph{crystallographic} if it additionally satisfies
\begin{enumerate}\setcounter{enumi}{2}
	\item $\frac{2(\alpha,\beta)}{(\beta,\beta)} \in \Z$ for all  roots $\alpha$ and $\beta$.
\end{enumerate}

A finite Coxeter group $W$ is called \emph{crystallographic} if it admits a crystallographic root system $\Psi$, which means that $W$ is isomorphic to the group generated by the reflections $s_\alpha$ for $\alpha \in \Psi$. As a consequence we have that $s_\alpha \beta = \beta + n \alpha$ for an \emph{integer} $n\in \Z$. This implies that for every root $\beta \in \Psi$ there exist integers $n_s$ for $s\in S$ such that
\[
\beta = \sum_{s\in S} n_s \alpha_s,
\]
that is every root is an \emph{integer} linear combination of simple roots.
Note that the root systems arising from the geometric representation of Coxeter groups are \emph{not} crystallographic in general.
The existence of crystallographic root systems for finite Coxeter groups depends only on the isomorphism type of the Coxeter system  \cite[Chap. 2.9]{hum}. A finite Coxeter system $(W,S)$ is crystallographic if and only if $m(s,t) \in \set{2,3,4,6}$ for all different $s,t \in S$. Hence \enquote{most} finite Coxeter groups are crystallographic. In particular, the three infinite families of Coxeter groups of classical type are crystallographic and their crystallographic root systems differ from the ones arising from the geometric representation only by the lengths of roots.

\section{Coxeter complexes}\label{sec:cox_cplx}

To every Coxeter system $(W,S)$ of rank $n$ with $S=\set{ s_1, \ldots, s_n }$ there is an associated labeled chamber complex $\Sigma=\Sigma(W,S)$ of dimension $n-1$, called the \emph{Coxeter complex} \cite[Thm. III.1]{bro}. Its simplices are all \emph{special cosets} of $W$, which are subsets of the form $w\braket{ S' } \subseteq W$ for a $w\in W$ and $S' \subsetneq S$. The vertex set is 
\[
V\coloneqq \Set{w\braket{ S  \setminus \set{s_i} }\str w\in W,\ 1\leq i \leq n }
\]
and $\set{v_1, \ldots, v_k}$ with $v_i=w_{j_i}\braket{ S \setminus \set{s_{j_i} }} \in V$ is a simplex in $\Sigma$ if and only if
\[
w_{j_1}\braket{ S \setminus \set{s_{j_1} }} \cap \ldots \cap w_{j_k}\braket{ S \setminus \set{s_{j_k} }} =w\braket{ S' }
\]
for some $w\in W$ and $S'\subseteq S$. The chambers of $\Sigma$ are special cosets of the form $w\braket{ \emptyset } = \set{w}$, hence they are in bijection with the Coxeter group $W$. This allows us to identify the set of chambers $\C=\C(\Sigma)$ of $\Sigma$ with the Coxeter group $W$. 

The \emph{type} of a simplex $w\braket{ S' } \in \si$ is defined to be $S \setminus S'$. Hence the vertices have types with values in $S$ and the types of the $n$ vertices of a chamber are exactly $s_1, \ldots, s_n$. We say that the assignment of a type to every simplex is a \emph{labeling}.
The Coxeter group $W$ acts naturally on its Coxeter complex $\Sigma$ by left-multiplication. This action is simply transitive on chambers and type-preserving. 

\begin{rem}\label{rem:cox_cplx_homeo_sphere}
	The Coxeter complex of a finite Coxeter group $W$ is isomorphic to the simplicial complex arising from the geometric representation we presented in \cref{rem:simplicial_structure_sphere}. Consequently, the Coxeter complex is homeomorphic to a sphere. The action of $W$ on the triangulated sphere induced from the geometric representation coincides with the action of $W$ on its Coxeter complex. \cref{fig:cox_cplx_A3} shows the Coxeter complex of the symmetric group $S_4$ with labeled vertices, depicted as a $2$-dimensional sphere.
\end{rem}

\begin{figure}
	\begin{center}
		\includegraphics[width=6cm]{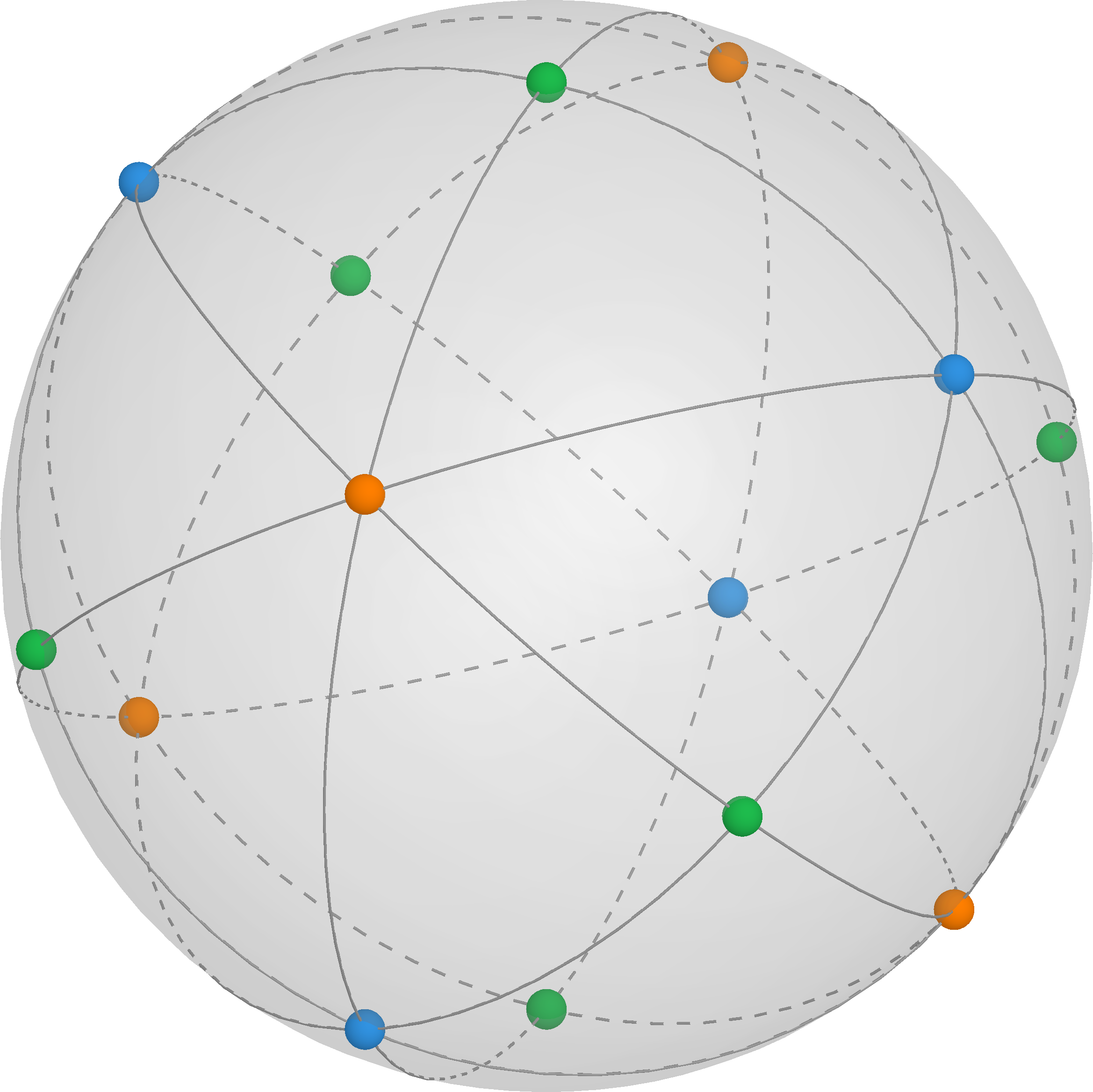}
		\caption{The Coxeter complex of $S_4$ with labeled vertices, where $s_1$ is depicted in blue, $s_2$ in green, and $s_3$ in orange.}%
		\label{fig:cox_cplx_A3}%
	\end{center}
\end{figure}

It is convenient to \emph{color} the panels with colors in $S$, that is the color of a panel $w\braket{ s } \coloneqq w\braket{ \set{s} }$ equals $s \in S$. 
\cref{fig:cox_cplx_A2} shows the Coxeter complex $\Sigma=\Sigma(S_3)$ of the symmetric group $S_3$ with Coxeter generators $s_1$ and $s_2$. Note that the vertices of $\Sigma$ are also the panels of $\Sigma$. Hence the \emph{panel} $\braket{ s_1}$ has \emph{color} $s_1$, while the \emph{vertex} $\braket{ s_1 }$ has \emph{type} $s_2$. Note that if $C$ is a chamber containing a panel $p$, then the color of $p$ equals the type of the unique vertex in $C$ that is not contained in $p$.

\begin{figure}
	\begin{center}
		\begin{tikzpicture}[thick]
		\foreach \w in {1,3,5} 
		\node[gpunkt, Blue] (p\w) at (-\w * 360/6 +60  : 2.5cm)  {};
		\foreach \w in {2,4,6} 
		\node[gpunkt, Orange] (p\w) at (-\w * 360/6 +60  : 2.5cm)  {};
		\foreach \w/\lab/\pos in 
		{1/\textcolor{Blue}{\Braket{s_2}}=\Set{1,s_2}/right,
			2/s_2\textcolor{Orange}{\Braket{s_1}}=\Set{s_2, s_2s_1}/below right,
			3/s_2s_1\textcolor{Blue}{\Braket{s_2}}=\Set{s_2s_1,s_2s_1,s_2}/below left,
			4/s_1s_2\textcolor{Orange}{\Braket{s_1}}=\Set{s_1s_2,s_1s_2s_1}/left,
			5/s_1\textcolor{Blue}{\Braket{s_2}}=\Set{s_1,s_1s_2}/above left,
			6/\textcolor{Orange}{\Braket{s_1}}=\Set{1,s_1}/above right} 
		\node[\pos=2mm]  at (p\w)  {$\lab$};
		
		\draw 	(p1) to node[midway,right=2mm]{$s_2$}
		(p2)to node[midway,below=2mm]{$s_2s_1$}
		(p3)to node[midway,left=2mm]{$s_1s_2s_1=s_2s_1s_2$}
		(p4)to node[midway,left=2mm]{$s_1s_2$}
		(p5)to node[midway,above=2mm]{$s_1$}
		(p6)to node[midway,right=2mm]{$1$}(p1);
		\end{tikzpicture}
		\caption{The Coxeter complex of $S_3$ with colored panels is depicted here. The chambers are labeled with elements of $S_3$.}
		\label{fig:cox_cplx_A2}
	\end{center}
\end{figure}
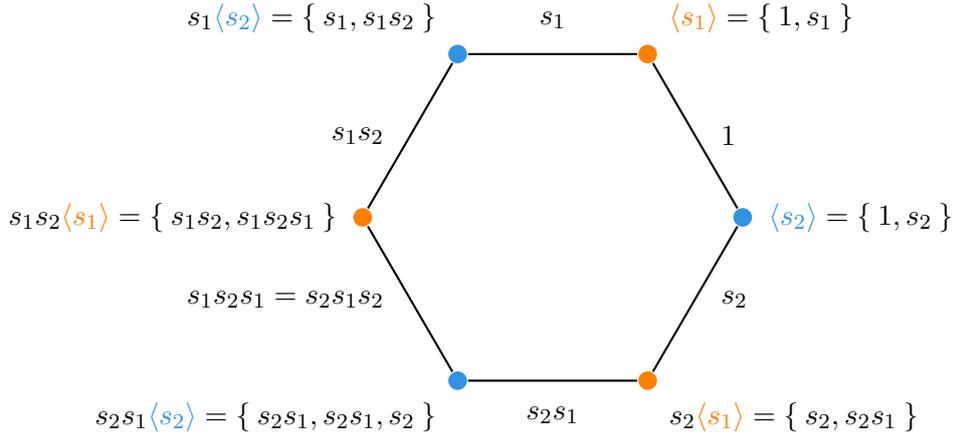

The colors of panels allow us to describe galleries in the Coxeter complex. For this we define two adjacent chambers $C,C' \in \C(\Sigma)$ to be \emph{$s$-adjacent} if their common panel has color $s \in S$. A gallery $(C_0, \ldots, C_k)$ has \emph{type} $(s_{j_1}, \ldots, s_{j_k})$ for $s_{j_i}\in S$  if $C_{i-1}$ and $C_{i}$ are $s_{j_i}$ adjacent for all $1\leq i \leq k$. In this case it holds that $C_k=wC_0$ for $w=s_{j_1}\ldots s_{j_k}$. For a fixed chamber $C$, a tupel $(s_{j_1}, \ldots, s_{j_k})$ uniquely determines a gallery connecting $C$ to $s_{j_1}\ldots s_{j_k}C$. The other way around, there are multiple galleries connecting two chambers $C$ and $wC$, since there are multiple ways to express $w\in W$ as a product of Coxeter generators.

Recall that for every generating set $X$ of a group $G$ the \emph{length} of $g\in G$ with respect to $X$ is the minimal $k$ such that there exist $x_1, \ldots, x_k \in X$ such that $g=x_1\ldots x_k$. We denote the length with respect to $X$ by $\ell_X$.

Since $S$ generates $W$, galleries in $\Sigma$ connecting $\id$ to $w$ are in bijection with decomposition of $w$ by sending the type $(s_{j_1}, \ldots, s_{j_k})$ to the decomposition $s_{j_1} \ldots s_{j_k}$. In particular, minimal galleries correspond to reduced decompositions with respect to $S$ and the gallery distance from $C$ to $wC$ coincides with the length $\ls(w)$ \cite[Cor. 1.75]{ab}. 

Suppose that $W$ is finite. Then for every chamber $C\in \C(\Sigma)$ there is a unique chamber $C' \in \C(\Sigma)$ such that they have maximal distance, that is $\di(C, C') = \max_{D,E\in\C}\di(D,E)$ \cite[Prop. 1.57]{ab}. In this case, the chambers $C$ and $C'$ are called \emph{opposite} chambers. In particular, there is a unique chamber opposite to the chamber corresponding to $\id \in W$, which corresponds to the \emph{longest element} $w_0 \in W$ in the Coxeter group \cite[Prop. 1.77]{ab}. The chamber corresponding to $\id \in W$ is called \emph{fundamental chamber} of $\Sigma$. In \cref{cha:buildings} and \cref{sec:more_Cox} we investigate the Coxeter complex of type $A$ in more detail.  

The geometry of the Coxeter complex $\Sigma(W,S)$ gives rise to a partial order on its chambers $\C$ by calling a chamber $C'$ smaller than a chamber $C$ if there is a minimal gallery connecting the fundamental chamber with $C$ that goes through $C'$. Since chambers of $\Sigma$ are identified with the elements of $W$, this also induces a partial order on $W$, which is defined in group-theoretic terms as follows.
Let $v,w \in W$. We say that $v \pow w$ if $\ls(w)=\ls(v)+\ls(v\inv w)$. This indeed defines a partial order on $W$, called the \emph{weak order}. Actually $(W, \pow)$ is a lattice, ranked by $\ls$, with minimal element $\id$ and maximal element $w_0$  for all finite Coxeter systems \cite[Chap. 3.2]{bb}.

\section{The absolute order and reduced words}\label{sec:abs_ord}

In the last section we saw how the geometry of the Coxeter complex encodes the group-theoretic weak order. 
But what happens if we substitute the generating set $S$ by a larger one, for instance the set of all reflections $T$, and define a partial order analogously? This question initiated the \emph{dual study} of Coxeter groups, which started systematically in \cite{bes}. Important fundamentals for the dual study of Coxeter groups are already established in \cite{car}.
In this section, we introduce the absolute order 
and some of its basic properties.
If not stated otherwise, the definitions and statements of this section can be found in \cite{armstr} or \cite{hum}.\\

Let $(W,S)$ be a finite Coxeter system of rank $n$ with $S=\set{s_1, \ldots, s_n}$, and let $V\cong \R^n$ be the vector space from the geometric representation $\sigma\colon W \to \gl(V)$.
Recall that the set of reflection of $W$ is given by 
\[
T = \Set{ wsw\inv \str w\in W, s\in S }
\]
and that $T$ is a generating set of $W$, since it contains the set of Coxeter generators $S$. We denote the \emph{length} of an element $w\in W$ with respect to $T$ by $\ell(w)\coloneqq \ell_T(w)$ and call it the \emph{absolute length}. In the literature, this length is also known as reflection length. Note that the absolute length is invariant under conjugation, since $T$ is. 
By a \emph{decomposition} of an element $w\in W$ we always mean a word $t_1\ldots t_k$ with $t_1, \ldots, t_k \in T$ such that the product $t_1\cdot \ldots \cdot t_k$ equals the element $w\in W$. 
A decomposition of $w\in W$ is called \emph{reduced} if it consists of $\ell(w)$ letters.

\begin{defi}
	For $v,w\in W$ we set $v \leq w$ if $\ell(w)= \ell(v)+\ell(v\inv w)$ and call the partial order $\leq$ on $W$ the \emph{absolute order}.
\end{defi}

The absolute order can be characterized in terms of subwords. Recall that if $t_1\ldots t_k$ is a decomposition, a \emph{subword} is a decomposition that arises from the former one by deleting entries, that is it is of the form $t_{i_1}\ldots t_{i_m}$ for $1\leq i_1 < \ldots < i_m \leq k$, or the empty decomposition.

\begin{subword}\label{subword_prop}
	For $v,w \in W$ we have $v \leq w$ if and only if there exists a reduced decomposition of $w$ that contains a reduced decomposition of $v$ as a subword. 
\end{subword}

The Coxeter group $W$ together with the absolute order is a graded poset $(W,\leq)$ with rank function $\ell$. We simply write $W$ instead of $(W,\leq)$. Since the absolute length $\ell$ is invariant under conjugation, every element $w\in W$ induces a poset automorphism by conjugation.
The identity $\id\in W$ is the unique minimal element, but in general, there are multiple maximal elements. Hence the absolute order does not induce a lattice structure on $W$. The rank of the poset $W$ equals the rank of the Coxeter group, which follows from the following important facts.

Recall that the set of reflections $T$ is in bijection with the set of positive roots $\Phi^+$ 
via $t \mapsto \alpha_t$. The following is \cite[Le. 3]{car}.

\begin{carter}\label{lem:carter}
	Let $w\in W$ with decomposition $t_1\ldots t_k$. 
	Then this decomposition is reduced if and only if the roots $\alpha_{t_1}, \ldots, \alpha_{t_k}$ are linearly independent in $V$.
\end{carter}

In particular, $\ell(w) \leq n$ for all $w\in W$ and every reflection appears at most once in a reduced decomposition. We can say even more about the set of roots associated to a reduced decomposition. 

\begin{lem}\label{lem:basis_of_moved_space}
	Let $w\in W$ and let $t_1\ldots t_k$ be a reduced decomposition of it. Then the set $\set{\alpha_{t_1}, \ldots, \alpha_{t_k}}$ of roots in $V$ is a basis for the moved space $\mov(w)$.
\end{lem}

Hence the absolute length of an element $w\in W$ coincides with the dimension of its moved space $\mov(w)\subseteq V$, that is $\ell(w)=\dim (\mov(w))$.
The following basic relations in the absolute order are frequently used in the course of this thesis \cite[Le. 3.9, 3.10]{bra_kpi}.

\begin{lem}\label{obs:abs_ord}
	Let $u,v,w \in W$ be such that $u\leq v \leq w$. Then it holds that
	\begin{enumerate}
		\item $v \inv w \leq w$,
		\item $wv\inv \leq w$,
		\item $u\inv v \leq u\inv w$, and
		\item $v\inv w \leq u \inv w $.
	\end{enumerate}
\end{lem}

Furthermore, the manipulation of decompositions, and in particular reduced decomposition, is a useful tool.

\begin{shifting}\label{shifting_prop}
	Let $w\in W$ and let  $t_1\ldots t_k$ be a decomposition of it. Then 
	\begin{enumerate}
		\item $t_1\; t_2\;\ldots\; t_{i-1}\;(t_i \;t_{i+1}\; t_i)\; t_i \; t_{i+2} \;\ldots\; t_k$ and
		\item $t_1 \; t_2 \; \ldots \; t_{i-1} \; t_{i+1} \;(t_{i+1}\; t_i\; t_{i+1}) \; t_{i+2} \;\ldots\; t_k$
	\end{enumerate}
	are again decompositions of $w$ of the same length $k$ for all $1\leq i < k$.
\end{shifting}

The operation in \emph{a)} is called \emph{$i^{\text{th}}$ right-shift} and the one in \emph{b)} is called \emph{$i^{\text{th}}$ left-shift}. Note that the $\ith$ right-shift shifts $t_i$ to the right, and the $i^{\text{th}}$ left-shift shifts $t_{i+1}$ to the left. This makes sense, since the first left-shift one can perform is shifting $t_2$ to the left.

The combination of the shifting property and the subword property allows us to characterize the absolute order of $W$ with prefixes. A \emph{prefix} of a reduced decomposition $t_1\ldots t_k$ is a subword of the form $t_1\ldots t_i$ for some $i\leq k$ or the empty decomposition. 

\begin{prefix}\label{prefix}
	For $v,w\in W$ it holds that $v \leq w$ if and only if there exists a reduced decomposition of $w$ such that a prefix of it is a reduced decomposition of $v$.
\end{prefix}

In particular this means that if $v\leq w$, then every reduced decomposition $t_1\ldots t_i$ of $v$ can be extended to a reduced decomposition $t_1\ldots t_i\ldots t_k$ of $w$.

\cleardoublepage

\chapter{Buildings}\label{cha:buildings}

This section is devoted to define buildings and provide some basic facts about them. The focus is on spherical buildings of type $A$, which are the buildings that appear in this thesis. The construction of the type $A$ spherical building associated to a vector space is essential for this thesis. The books \cite{bro} and \cite{ab} serve as main reference for buildings. 


\begin{defi}
	A simplicial complex $\si$ is called \emph{building} if it can be expressed as a union of subcomplexes, which are called \emph{apartments}, such that the following axioms are satisfied. 	
	\begin{itemize}
		\item[(B0)]
		Each apartment is isomorphic to a Coxeter complex.
		\item[(B1)]\label{ax:b2}
		For any two simplices in $\si$ there is an apartment containing both of them.
		\item[(B2)]
		If two simplices $A, B \in \si$ are both contained in the apartments $\Sigma$ and $\Sigma'$, then there exists an isomorphism $\Sigma \to \Sigma'$ fixing $A$ and $B$ pointwise.
	\end{itemize}
	If all apartments are finite Coxeter complexes, then $\si$ is called a \emph{spherical building}. 
\end{defi}

In particular, every Coxeter complex is a building in its own right.
Buildings are chamber complexes, since any two maximal simplices of $\si$ are contained in a common apartment, where they can be connected by a gallery. Therefore, the maximal simplices of a building $\si$ are called \emph{chambers}. The set of chambers of $\si$ is denoted by $\C(\si)$ or $\C$. 

Any two apartments of $\si$ are isomorphic by axiom (B2), hence they all come from the same Coxeter system $(W,S)$. This allows us to define the \emph{type} and the \emph{rank} of the building as the type and the rank of the Coxeter system corresponding to the apartments, respectively. \cref{fig:building_typeA2} shows a spherical building of type $A_2$. 

\begin{figure}%
	\begin{center}
		\begin{tikzpicture}
		\foreach \w in {1,2,...,26}
		\node (e\w) at (-\w * 360/26 : 3) [mpunkt] {};
		
		\foreach \i [evaluate=\i as \j using int(\i+1)] in {1,...,25}
		\draw[thick] (e\i) -- (e\j);
		\draw[thick] (e1) -- (e26);
		
		\foreach \i [evaluate=\i as \j using int(\i+5)] in {1,3,...,19,21}
		\draw[thick] (e\i) -- (e\j) ;   
		\draw[thick] (e23) -- (e2); 
		\draw[thick] (e25) -- (e4);
		
		\foreach \i [evaluate=\i as \j using int(\i+9)] in {2,4,...,16}
		\draw[thick] (e\i) -- (e\j);
		\foreach \i [evaluate=\i as \j using int(\i+17)] in {1,3,...,9}
		\draw[thick] (e\i) -- (e\j);
		
		\end{tikzpicture}
		\caption{A spherical building of type $A_2$.}%
		\label{fig:building_typeA2}%
	\end{center}
\end{figure}
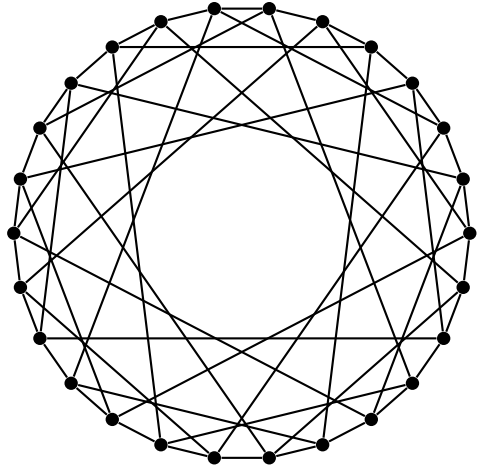

We associate a \emph{type} $S'\subseteq S$ to every simplex in $\si$ such that the labeling of apartments fits to the labeling of the Coxeter complex described in \cref{sec:cox_cplx}.
Since the type of a vertex is an element in $S$, we can color the panels with colors in $S$ as before. This allows us speak of \emph{types of galleries}, which are defined analogously to the Coxeter complex case. 
The following is \cite[Prop. 4.40]{ab}. 

\begin{prop}\label{prop:convex_aptms}
	Let $C,D\in \si$ be two chambers and $A\in \A(\si)$ an apartment containing them. Then every minimal gallery connecting $C$ and $D$ is contained in $A$.
\end{prop}

\section{Spherical buildings}

From now on we are restricting ourselves to spherical buildings. In a spherical building $\si$, there exists a unique collection of apartments satisfying the axioms of a building. This set of apartments of $\si$ is denoted by $\A(\si)$. 
The apartments are the convex hulls of opposite chambers in $\si$, where we call two chambers of $\si$ \emph{opposite} if their gallery distance is maximal. Hence, a pair of opposite chambers determines a unique apartment $A\in \A(\si)$.

\begin{rem}\label{rem:homotopy_type_building}
	Let $C$ be a chamber of a spherical building $\si$ of rank $n$. For all chambers $D\in \C$ there exists an apartment containing the chambers $C$ and $D$. Hence $\si$ is the union of all apartments containing the fixed chamber $C$, that is
	\[
	\si = \bigcup \Set{\Sigma \in \A(\si) \str C \in \C(\Sigma) }.
	\]
	Since each apartment is homeomorphic to a $(n-1)$-sphere, spherical buildings are homotopy equivalent to a wedge of $(n-1)$-spheres. 
\end{rem}

\section{The spherical building associated to a vector space}\label{sec:sph_build}

In \cref{sec:lattice_vec_space} we associated to every vector space $V$ its lattice of linear subspaces. Its order complex is a spherical building of type $A$ \cite[Chap. 4.3]{ab}. 
The aim of this section is to emphasize the relations between the vector space $V$, the lattice $\lam(V)$ and its order complex $|\lam|$.

\begin{prop}\label{prop:LV_building}
	Let $V$ be a vector space of dimension $n\geq 2$. Then the order complex  $|\lam(V)|$ of the lattice $\lam(V)$ of linear subspaces of $V$ is a spherical building of type $A_{n-1}$ and dimension $n-2$.
\end{prop}	

The spherical building of type $A_2$ associated to the vector space $\F_2^3$ is shown in \cref{fig:A2-F_2^3}. It is labeled with the proper linear subspaces of $\F_2^3$, where $e_i$ denotes the $\ith$ standard basis vector of $\F_2^3$. \cref{fig:building_typeA2} depicts the spherical building associated to $\F_3^3$.

\begin{figure}
	\begin{center}
		
		\begin{tikzpicture}
		
		\foreach \w in {1,2,...,14}
		\node (e\w) at (-\w * 360/14 : 3) [mpunkt] {};
		
		\foreach \w/\s/\p in {1/{e_2,e_1+e_3}/right,
			2/{e_2}/below right,
			3/{e_2,e_3}/below,
			4/{e_2+e_3}/below,
			5/{e_1,e_2+e_3}/below left,
			6/{e_1+e_2+e_3}/left,
			7/{e_1+e_2,e_3}/left,
			8/{e_3}/left,
			9/{e_1,e_3}/above left,
			10/{e_1}/above,
			11/{e_1,e_2}/above,
			12/{e_1+e_2}/above right,
			13/{e_1+e_2,e_1+e_3}/right,
			14/{e_1+e_3}/right}
		\node[\p=2mm] (f\w) at (-\w * 360/14  : 3)  {$\Braket{\s}$};

		\foreach \i [evaluate=\i as \j using int(\i+1)] in {1,...,13}
		\draw[thick] (e\i) -- (e\j);
		\draw[thick] (e1) -- (e14);
		
		\foreach \i [evaluate=\i as \j using int(\i+5)] in {1,3,5,7,9}
		\draw[thick] (e\i) -- (e\j) ;   
		\foreach \i [evaluate=\i as \j using int(\i+9)] in {2,4}
		\draw[thick] (e\i) -- (e\j);
		
		\end{tikzpicture}
		\caption{The spherical building associated to the vector space $\F_2^3$, whose standard basis vectors are denoted by $e_1$, $e_2$ and $e_3$.}
		\label{fig:A2-F_2^3}
	\end{center}
\end{figure}
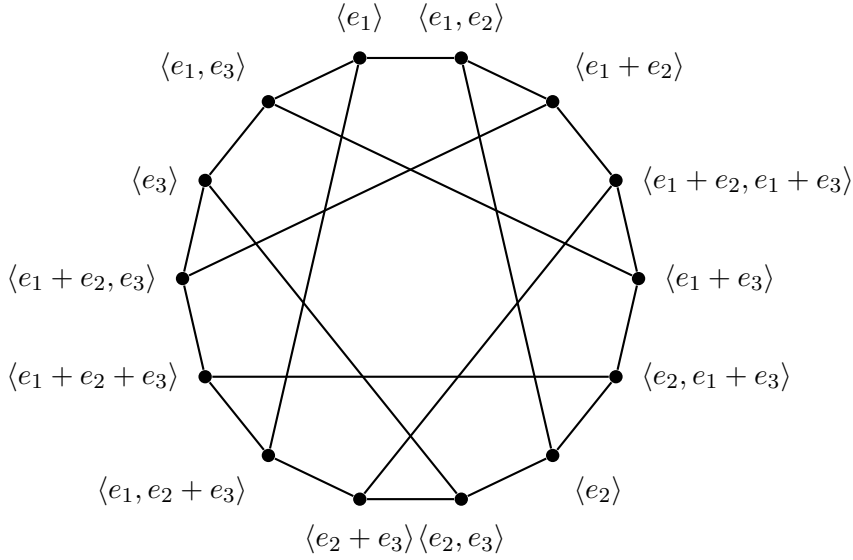

Let $\si=|\lam(V)|$ be the order complex of the lattice $\lam=\lam(V)$.
We start with the study of Coxeter complexes of type $A_{n-1}$, which are the apartments of $\si$. This study of the simplicial structure of the apartments is crucial for \cref{part3} of this thesis. 

Let $\Sigma$ be the Coxeter complex of type $A_{n-1}$ and $\Xi$ the barycentric subdivision of the boundary of the $(n-1)$-simplex $\si_{n-1}$.   
The simplicial complexes $\Sigma$ and $\Xi$ are isomorphic \cite[Exc. 1.112]{ab}. In the following we demonstrate the correspondence and describe how the respective labels of chambers and vertices fit together. For $n=4$, \cref{fig:bary_sd_3simpl} depicts the complex $\Xi$, whose chambers are labeled with total orderings of the set $\Set{1,2,3,4}$, and \cref{fig:cox_cplx_3simpl} shows the labeled chamber complex $\Sigma$, whose chambers are labeled with elements of $S_4$. 

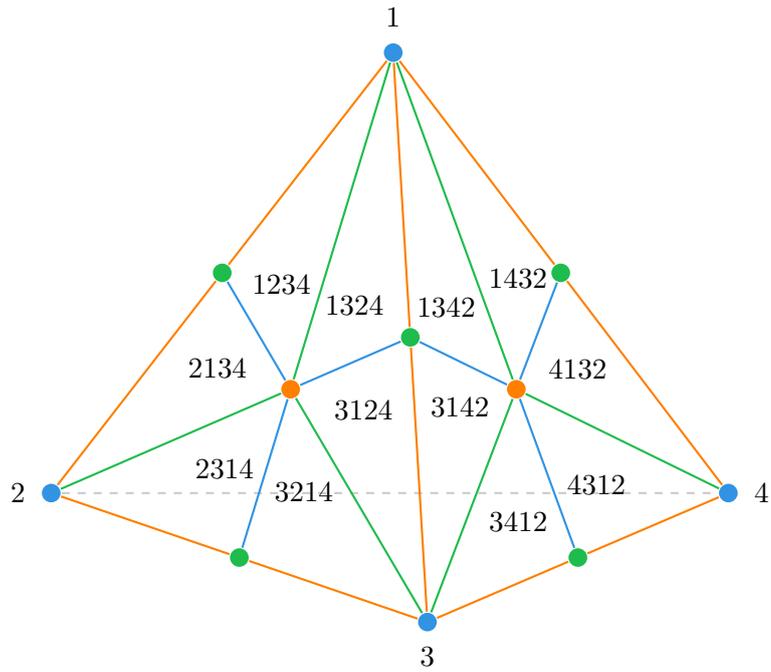
\begin{figure}
	\begin{center}
		\begin{tikzpicture}[every path/.style={thick}]
		\def\r{6}
		\Tetraeder
		
		\node[above=2mm] at (e3) {1};
		\node[left=2mm] at (e1) {2};
		\node[below=2mm] at (e2) {3};
		\node[right=2mm] at (e4) {4};
		
		\begin{scope}[rotate=-16]

		\node at ($(AB) + (0 * 60:  1cm and 1.5cm)$) {$3124$};
		\node at ($(AB) + (1 * 60:  1cm and 1.5cm)$) {$1324$};
		\node at ($(AB) + (2 * 60:  1cm and 1.5cm)$) {$1234$};
		\node at ($(AB) + (3 * 60:  1cm and 1.5cm)$) {$2134$};
		\node at ($(AB) + (4 * 60-3:  1cm and 1.5cm)$) {$2314$};
		\node at ($(AB) + (5 * 60+3:  1cm and 1.5cm)$) {$3214$};
		\end{scope}
		
		\begin{scope}[rotate=18]
		\node at ($(AC) + (0 * 60  : 0.85cm and 2cm)$) {$4132$};
		\node at ($(AC) + (1 * 60  : 0.95cm and 1.6cm)$) {$1432$};
		\node at ($(AC) + (2 * 60 +5 : 0.95cm and 1.6cm)$) {$1342$};
		\node at ($(AC) + (3 * 60  : 0.78cm and 2cm)$) {$3142$};
		\node at ($(AC) + (4 * 60 -3 : 0.95cm and 2cm)$) {$3412$};
		\node at ($(AC) + (5 * 60 + 10 : 0.95cm and 2cm)$) {$4312$};
		\end{scope}
		
		\end{tikzpicture}
	\end{center}
	\caption{A part of the chamber complex $\Xi$, which is the barycentric subdivision of the boundary of the $3$-simplex. The chambers are labeled with total orderings of $\set{1,2,3,4}$ and the different colors of panels and vertices correspond to their respective types, where type $1$ is blue, type $2$ is green and type $3$ is orange. The labels $1,2,3,4$ of the type $1$ vertices are the labels of the $3$-simplex from which $\Xi$ arose.}
	\label{fig:bary_sd_3simpl}
\end{figure}

\begin{figure}
	\begin{center}
		\begin{tikzpicture}[every path/.style={thick}]
		\def\r{6}
		\Tetraeder
		
		\node[Blue, above=2mm] at (e3) {$s_1$};
		\node[Blue, left=2mm] at (e1) {$s_1$};
		\node[Blue, below=2mm] at (e2) {$s_1$};
		\node[Blue, right=2mm] at (e4) {$s_1$};

		\foreach \place/\direction in {f12/below, f13/above left, f23/left, f24/below, f34/above right}
		\node[\direction=2mm, Green] at (\place) {$s_2$};
		
		\node[Orange] at ($(AB) + (100:  0.6cm)$) {$s_3$};
		\node[Orange] at ($(AC) + (90:  0.6cm)$) {$s_3$};
		
		\begin{scope}[rotate=-16]
		
		\node at ($(AB) + (0 * 60:  1cm and 1.5cm)$) {$(1\,3\,2)$};
		\node at ($(AB) + (1 * 60:  1cm and 1.5cm)$) {$(2\,3)$};
		\node at ($(AB) + (2 * 60:  1cm and 1.5cm)$) {$\id$};
		\node at ($(AB) + (3 * 60:  1cm and 1.5cm)$) {$(1\,2)$};
		\node at ($(AB) + (4 * 60-8:  1cm and 1.5cm)$) {$(1\,2\,3)$};
		\node at ($(AB) + (5 * 60+3:  1cm and 1.5cm)$) {$(1\,3)$};
		\end{scope}

		\begin{scope}[rotate=18]
		
		\node at ($(AC) + (0 * 60  : 0.85cm and 2cm)$) {$(1\,4\,2)$};
		\node at ($(AC) + (1 * 60  : 0.95cm and 1.6cm)$) {$(2\,4)$};
		\node at ($(AC) + (2 * 60 +8 : 0.9cm and 1.5cm)$) {$(2\,3\,4)$};
		\node at ($(AC) + (3 * 60 -2 : 0.74cm and 1.9cm)$) {$(1\,3\,4\,2)$};
		\node at ($(AC) + (4 * 60 -11 : 1cm and 2.6cm)$) {$(1\,3)(2\,4)$};
		\node at ($(AC) + (5 * 60 + 10 : 1.3cm and 2.3cm)$) {$(1\,4\,2\,3)$};
		\end{scope}	
		
		\end{tikzpicture}
		\caption{A part of the Coxeter complex $\Sigma$ of type $A_3$. The chambers are labeled with elements of the symmetric group $S_4$, and the vertices are labeled with the simple reflections $s_1$ in blue, $s_2$ in green and $s_3$ in orange.}
		\label{fig:cox_cplx_3simpl}
	\end{center}
\end{figure}
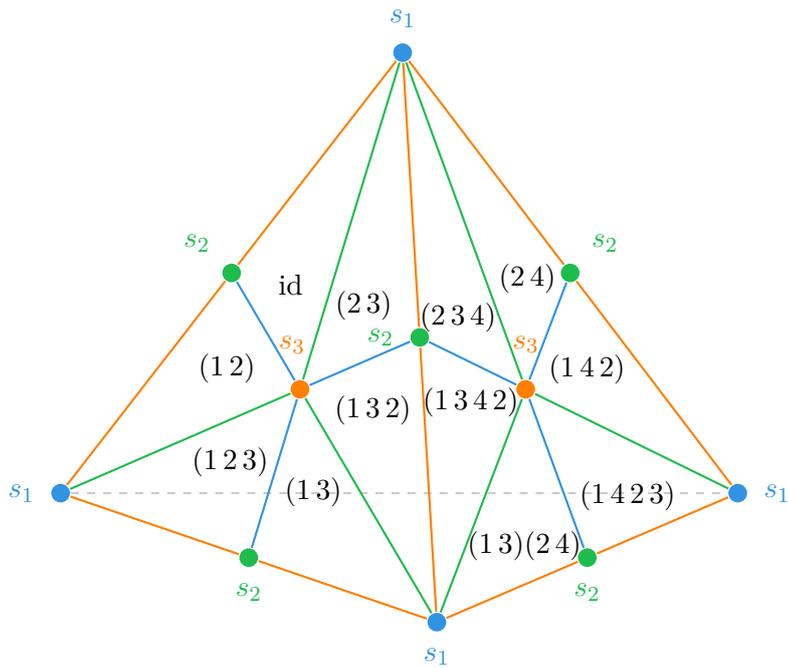

We start with describing the labeling of the complex $\Xi$. Let $P=\set{1,\ldots, n}$ be the vertex set of $\si_{n-1}$. Recall that the vertices of $\Xi$ correspond to the proper subsets of $P$. The \emph{type} of a vertex $X\in \Xi$ is defined to be the cardinality $\#X$ of the subset $X\subseteq P$. Maximal simplices in $\Xi$ have exactly $n-1$ vertices, which all have different types, and a panel has $n-2$ vertices of different types. The \emph{color} of a panel is defined to be $k$ if none of its vertices has type $k$. A maximal simplex $A_1 \subsetneq A_2 \subsetneq \ldots \subsetneq A_{n-1}$ defines a total order 
\[
A_1,\ \ A_2\setminus A_1, \ \ldots \ , \ A_i\setminus A_{i-1}, \ \ldots\ , \ A_{n-1}\setminus A_{n-2}, \ P \setminus A_{n-1}
\]
of the set $\set{1,\ldots, n}$. Hence every maximal simplex can be labeled uniquely by this sequence of type $1$ vertices and different chambers have different labels. In \cref{fig:bary_sd_3simpl}, the chamber labels are the total orders of blue rank $1$ vertices.
If two maximal simplices intersect in a blue panel of color $k<n$, then the corresponding sequences agree in all positions except for $k$ and $k+1$, whose entries are swapped. For instance in \cref{fig:bary_sd_3simpl}, the two maximal simplices with labels $1324$ and $3124$ intersect in a panel of color $1$. We can pass from one sequence to the other by swapping the first and the second entry. Note that two maximal simplices are adjacent if and only if their corresponding sequences differ in exactly two entries.

Every total order $i_1\ldots i_n$ of $P=\set{1, \ldots, n}$ gives rise to a self-bijection of $P$ via $j \mapsto i_j$ and every bijection $P\to P$ arises in this way. A sequence of the form $i_1\ldots i_n$ is called \emph{one line notation} of the bijection it describes. 
Two sequences differ in position $k$ and $k+1$ if and only if the corresponding bijections differ by the pre-composition with the transposition that swaps $k$ and $k+1$, which corresponds to the simple reflection $s_k$. The more familiar cycle notation of the bijection with one line notation $1324$ is $(2\,\;3)$ and $3124$ corresponds to $(1\,\;3\,\;2)$. It holds that $(1\,\;3\,\;2)=(2\,\;3)(1\,\;2)$, which means that the two chambers $(2\,\;3)$ and $(1\,\;3\,\;2)$ are $s_1$-adjacent in $\Sigma$.

Hence the sets of maximal simplices of both $\Xi$ and $\Sigma$ are in bijection with $S_n$. The adjacency relations in $\Xi$ and $\Sigma$ coincide after the identification of maximal simplices with group elements of $S_n$ as well. Finally note that the vertex in $\Xi$ of type $k$ of a chamber corresponding to $w\in W$ corresponds to a vertex of type $s_k$ of the chamber in $\Sigma$ with label $w$.

Recall that Boolean sublattices of $\lam(V)$ are given by frames of $V$, which we discussed in \cref{sec:lattice_vec_space}. Therefore apartments in the building $\si$, which are order complexes of Boolean sublattices, correspond to frames in $V$, and vice versa. 

When we say \enquote{apartment} of $\si=|\lam|$, we mean it as part of the building, that is as the order complex of a Boolean sublattice of $\lam$, which we denoted by $\Xi$ in this discussion. Nevertheless, we stick to the standard notation of the building theory and denote apartments with the symbol $\Sigma$.   

\cref{tab:relations_V_lam_buil} summarizes the correspondences of the different concepts of $V$, $\lam(V)$ and $|\lam(V)|$.

\renewcommand{\arraystretch}{1,5}
\begin{table}
	\begin{center}
		\begin{tabular}{|c||c||c|}
			\hhline{|-||-||-|}
			
			vector space $V$ & lattice $\lam = \lam(V)$ & spherical building $\si=|\lam|$ \\
			\hhline{:=::=::=:}

			proper subspaces & elements different from $\mi,\ma$ & vertices\\
			\hhline{|-||-||-|}
			
			dimension $k$ of subspace & rank $k$ of element & type $s_k$ of vertex\\
			\hhline{|-||-||-|}

			chains  of proper subspaces & $k$-chains &  $k$-simplices\\
			
			$U_0 \subsetneq U_1 \subsetneq \ldots \subsetneq U_k$&$U_0 < U_1 < \ldots < U_k$&$\set{U_0, U_1, \ldots, U_k}$
			\\
			\hhline{|-||-||-|}
			
			$\dim(V)=n$ & $\rk(\lam)=n$ & $\dim(\si)=n-2$\\
			&& $\rk(\si)=n-1$\\
			\hhline{|-||-||-|}
			
			frames  & Boolean sublattices  & apartments\\
			$\set{L_1, \ldots, L_n}$& $\Braket{L_1, \ldots, L_n}\cong\B_n$ &Coxeter complexes \\[-6pt]
			&& of type $A_{n-1}$\\
			\hhline{|-||-||-|}
			
		\end{tabular}
		\caption{The correspondence of different concepts in a vector space $V$, its associated lattice of linear subspaces $\lam(V)$, and the corresponding order complex $|\lam(V)|$, which is a spherical building.}
		\label{tab:relations_V_lam_buil}
	\end{center}
\end{table}

\cleardoublepage
\part{Non-crossing partitions}\label{part2}

\cleardoublepage

\chapter{Elementary definitions and properties}\label{chap:elem_props_nc}

In this chapter we are laying the foundations for the study of non-crossing partitions. After defining the non-crossing partitions as a subposet of the absolutely ordered Coxeter group, we are highlighting important properties, which are frequently used in the sequel.

From now on, let $(W,S)$ be a finite Coxeter system of rank $n$.

\begin{defi}
	A \emph{standard Coxeter element} of a Coxeter system $(W,S)$ is a product of the simple generators $S$ in any order. A \emph{Coxeter element} is an element of $W$ that is conjugate to a standard Coxeter element.
\end{defi}

A standard Coxeter element has length $n$ by \cref{lem:carter}. Since the absolute length is invariant under conjugation, all Coxeter elements have length $n$ and therefore are maximal elements of the absolute order. Note that in particular, any two standard Coxeter elements are conjugate. 
Hence the Coxeter elements form a single conjugacy class in $W$. In type $A$, all maximal elements are Coxeter elements, but in general, this is \emph{not} the case.  

The \emph{Coxeter number} of a Coxeter group $W$ is the order $h$ of a Coxeter element. This number is well-defined, since any two of them are conjugate, and depends on the type of $W$. \cref{tab:Coxeter_numbers} shows the Coxeter numbers of all finite Coxeter groups.

\renewcommand{\arraystretch}{1,15}
\begin{table}
	\begin{center}
		\begin{tabular}{|r|*{10}{|c}|}
			\hhline{|-||*{10}{-|}}
			type of $W$&$A_n$&$B_n$&$D_n$&$E_6$&$E_7$&$E_8$&$F_4$&$H_3$&$H_4$&$I_2(m)$\\
			\hhline{|-||*{10}{-|}}
			h&$n+1$&$2n$&$2(n-1)$&$12$&$18$&$30$&$12$&$10$&$30$&$m$\\
			\hhline{|-||*{10}{-|}}
		\end{tabular}
		\caption{The Coxeter numbers of finite Coxeter groups are shown here.}
		\label{tab:Coxeter_numbers}
	\end{center}
\end{table}

The notion of a Coxeter element finally allows us to define the central object of this thesis, the non-crossing partitions. 

\begin{defi}
	Let $(W,S)$ be a finite Coxeter system and $\cox \in W$ be a Coxeter element. The \emph{non-crossing partition poset} of $W$ with respect to the Coxeter element $\cox$ is defined to be the subposet
	\[
	\nc(W,\cox) \coloneqq \Set{w\in W\str w \leq \cox}
	\]
	of $W$ endowed with the absolute order.
\end{defi}

\begin{nota}
	If the Coxeter element $c$ is fixed, or the statement is independent from the chosen Coxeter element, we write $\nc(W)$ instead of $\nc(W,\cox)$. If moreover the Coxeter group is chosen arbitrarily, we also write $\nc$ for $\nc(W)$. We refer to $\nc$ as \emph{non-crossing partitions} as well and leave out the term \enquote{poset}. If $W$ is a Coxeter group of type $X_n$, we denote the corresponding non-crossing partitions by $\nc(X_n)$.
\end{nota}

\cref{tab:elements_ncw} depicts the cardinalities of the non-crossing partitions $\nc(W)$ and the corresponding Coxeter groups $W$ for the finite irreducible types. Comparing these numbers shows that the non-crossing partitions form a relatively small part of the Coxeter group when the cardinality of $W$ is large. In low rank examples it may be misleading how large the set of non-crossing partitions inside the Coxeter group is.

\renewcommand{\arraystretch}{1,5}
\begin{table}
	\begin{center}
		\begin{tabular}{|c|c|c|}
			\hhline{|-|-|-|}
			type of $W$ & $\#\nc(W)$ & $\#W$\\
			\hhline{:=:=:=:}
			$A_n$ & $\frac{1}{n+1}\binom{2n}{n}$ & $(n+1)!$ \\
			\hhline{|-|-|-|}
			
			$B_n$  & $\binom{2n}{n}$ & $2^nn!$ \\
			\hhline{|-|-|-|}
			
			$D_n$ & $\frac{3n-2}{n}\binom{2n-2}{n-1}$ & $2^{n-1}n!$ \\
			\hhline{|-|-|-|}
			
			$E_6$ & $833$ & $>10^5$ \\
			\hhline{|-|-|-|}
			
			$E_7$ & $4160$  & $>10^7$ \\
			\hhline{|-|-|-|}
			
			$E_8$ & $25080$ & $>10^9$ \\
			\hhline{|-|-|-|}
			
			$F_4$ & $105$ & $>10^4$ \\
			\hhline{|-|-|-|}
			
			$H_3$ & $32$ & $120$ \\
			\hhline{|-||-|-|}
			
			$H_4$ & $280$ & $>10^5$ \\
			\hhline{|-|-|-|}	
			
			$I_2(m)$  & $m+2$  & $2m$  \\
			\hhline{|-|-|-|}		
		\end{tabular}
		\caption{Cardinalities of the non-crossing partitions and the Coxeter groups for the finite irreducible types.}
		\label{tab:elements_ncw}
	\end{center}
\end{table}

\begin{figure}%
	\begin{center}

		\begin{tikzpicture}[every path/.style={ line join = round}]
		\def\a{1.9cm} 
		\def\b{1.2cm} 

		\node (pg) at (2.5*\a,2.5*\b){$(1\,\;2\,\;3\,\;4)$};
		\foreach \x/\l in {0/{(1\,\;2)(3\,\;4)},
			1/{(1\,\;3\,\;4)},
			2/{(1\,\;2\,\;3)},
			3/{(1\,\;2\,\;4)},
			4/{(2\,\;3\,\;4)},
			5/{(1\,\;4)(2\,\;3)}}
		{
			\node (p\x) at (\x*\a,\b){$\l$};
			\draw (pg) -- (p\x.north);
		}
		
		\node (pk) at (2.5*\a,-2.5*\b){$\id$};
		\foreach \x/\l in {0/{(1\,\;2)},
			1/{(3\,\;4)},
			2/{(1\,\;3)},
			3/{(2\,\;4)},
			4/{(1\,\;4)},
			5/{(2\,\;3)}}
		{
			\node (q\x) at (\x*\a,-\b){$\l$};
			\draw (pk) -- (q\x.south);
		}
		
		\foreach \o/\u in {0/0, 0/1, 1/1, 1/2, 1/4, 2/0, 2/2, 2/5, 3/0, 3/3, 3/4, 4/1, 4/3, 4/5, 5/4, 5/5}
		\draw (p\o.south) -- (q\u.north);

		\end{tikzpicture}

		\caption{The Hasse diagram of $\nc(S_4)$. }
		\label{fig:ncp4_poset}		
	\end{center}
\end{figure}
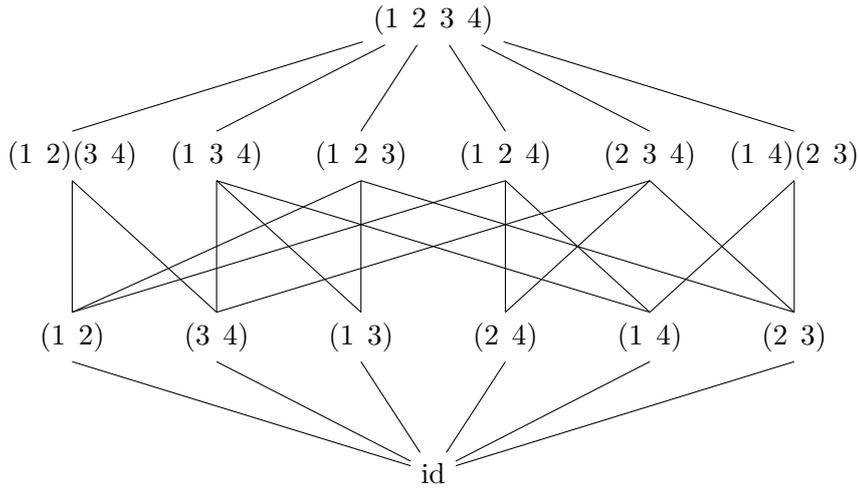

\cref{fig:ncp4_poset} shows the Hasse diagram of the non-crossing partitions $\nc(S_4)$ of the symmetric group. If we take a closer look at the Hasse diagram, three structural properties attract attention:
\begin{enumerate}
	\item For all $t\in T$ we have $t \in \nc(S_4)$.
	\item $\nc(S_4)$ is a lattice.
	\item $\nc(S_4)$ contains numerous Boolean sublattices.
\end{enumerate}
These are not phenomena of $\nc(S_4)$, they hold for general $\nc(W)$. 
We are explaining these properties subsequently.

The fact that all reflections are indeed in $\nc(W)$ for every finite Coxeter group $W$ is shown by using basic properties of the absolute order \cite[Le. 2.6.2]{armstr}. Brady and Watt uniformly show the following \cite{bra_watt_lattice}.

\begin{thm}
	For every finite reflection group $W$ the non-crossing partition poset $\nc(W)$ is a lattice.
\end{thm}

Armstrong shows how the group operation of $W$ interacts with the lattice structure of $\nc(W)$ \cite[Le. 2.6.13]{armstr}.

\begin{lem}
	Let $w\in \nc(W)$ and let $t_1\ldots t_k$ be a reduced decomposition of it. Then $w$ equals the join $t_1\vee \ldots \vee t_k \in \nc(W)$.
\end{lem}

As a consequence we obtain a description of Boolean sublattices in $\nc(W)$.

\begin{cor}\label{cor:construction_boolean_sublatice}
	Let $t_1\ldots t_n$ be a reduced decomposition of the Coxeter element $\cox$. Then the span of the reflections of the reduced decomposition $\braket{t_1, \ldots, t_n}$ in $\nc(W,\cox)$ is a Boolean sublattice of $\nc(W,c)$.
\end{cor}

\begin{proof}
	Every subword of $t_1\ldots t_n$ is an element of $\nc(W)=\nc(W,c)$ by the \cref{subword_prop}. Moreover, for $i_1 < \ldots< i_k$ and $j_1 < \ldots<j_m $ it holds that
	\[
	t_{i_1} \cdot \ldots \cdot t_{i_k} \leq t_{j_1} \cdot \ldots \cdot t_{j_m}
	\]
	if and only if $\set{i_1, \ldots, i_k} \subseteq \set{j_1, \ldots, j_m}$. Hence the map $\braket{t_1, \ldots, t_n} \to \B_n$ defined by
	\[
	t_{i_1} \vee \ldots \vee t_{i_k} \mapsto \set{i_1, \ldots, i_k}
	\]
	is a lattice isomorphism.
\end{proof}

In general, not every Boolean sublattice of $\nc(W)$ arises in this way. For an example of a Boolean sublattice from a different origin, we need the notation of the type $B$ Coxeter group, which will be introduced in \cref{sec:typeB}. 

\begin{exa}
	The element $\cox= [1\,\;2\,\;3]$ is a Coxeter element for type $B_3$. The span of the set $\set{[1], [2], [3]}$ of reflections is a Boolean sublattice in $\nc(B_3)$, but there is no order of the set of reflections that is a reduced decomposition of the Coxeter element $[1\,\;2\,\;3]$. Moreover, note that two different reduced decompositions of the Coxeter element can induce the same Boolean sublattice. This is for instance the case if two reflections in a reduced decomposition commute.
\end{exa}

In this thesis, mainly in \cref{part3}, but also in the next chapter, we work with the order complexes of the non-crossing partition lattices. The order complex $\ocncw$ is a chamber complex, which is homotopy equivalent to a wedge of $(\rk(W)-2)$-spheres. This follows from the more general fact that the non-crossing partition lattices for finite Coxeter groups are shellable \cite{abw}.
The maximal chains of $\nc$, and hence the chambers of $|\nc|$, can be explicitly described and are in bijection with reduced decompositions of the Coxeter element \cite[Cor. 4]{bra_watt_par_ord}.

\begin{lem}\label{lem:construction_chain}
	Let $t_1\ldots t_n$ be a reduced decomposition of the Coxeter element $\cox \in W$. Then $(\id,t_1, t_1\cdot t_2, \ldots, t_1\cdot \ldots \cdot t_n)$ is a maximal chain in $\nc(W)$. In the reverse case, if $(\id, w_1, \ldots, w_n=\cox)$ is a maximal chain in $\nc(W)$, then 
	$w_1 (w_1\inv w_2) \ldots (w_{n-1}\inv w_n)$ is a reduced decomposition of the Coxeter element $\cox$.
\end{lem}

\begin{rem}\label{rem:boolean_sublat_chain}
	Note that the maximal chain arising from a reduced decomposition of a Coxeter element as described in \cref{lem:construction_chain} is contained in the Boolean sublattice of $\nc$ associated to the same reduced decomposition. 
\end{rem}

\cleardoublepage

\chapter{The classical types}\label{chap:classical_types}

In this chapter we study the Coxeter groups of the classical types $A$, $B$ and $D$ and their non-crossing partitions. After investigating the group-theoretic structure of $\nc$, we give pictorial representations of the non-crossing partitions in the different types. The background for Coxeter groups is taken from \cite{bb} and \cite{ath_rei}, and the one for partitions can be found in \cite{staece}. In type $A$, the pictorial representation goes back to Kreweras. In his article \cite{kre} he introduced the \enquote{non-crossing partitions of a cycle} and laid the foundation for the rich study of non-crossing partitions. The pictorial representation of type $B$ is due to Reiner \cite{rei} and the pictorial representation of type $D$ we define here is inspired by the one introduced in \cite{ath_rei}, but has the advantage that different non-crossing partitions are represented by different pictures.

\section{Non-crossing partitions of type $A$}

The \enquote{classical} non-crossing partitions, introduced by Kreweras, which are now seen as the pictorial representations of the type $A$ non-crossing partitions, are not only the origin of the whole study of non-crossing partitions, but they are also a prototype for the other non-crossing partitions of classical type. Hence the study of the type $A$ non-crossing partitions is absolutely essential for the general understanding of non-crossing partitions.
The type $A$ non-crossing partitions, and also the corresponding \enquote{classical} set partitions, are the central objects of study in \cref{part3} of this thesis. For these reasons, we include an expanded discussion of non-crossing partitions of type $A$.

\subsection{The symmetric group}\label{sec:typeA_descr}
The Coxeter group $W$ of type $A_{n-1}$ is the \emph{symmetric group} $S_{n}$, whose elements are the bijective self-maps of $\Set{1, \ldots, n}$ with composition of maps as group operation. We use the cycle notation to express group elements in $S_n$. For instance, for $n \geq 3$ the symbol $(1\,\;2\,\;3)$ denotes the bijection that maps $1\mapsto 2$, $2\mapsto 3$, $3\mapsto 1$ and fixes all other elements. If $z$ is a cycle, we denote by $z\inv$ the cycle such that $zz\inv$ represents the identity map. Moreover, we use the convention that evaluation of compositions is from right to left, that is $(1\,\;2)(2\,\;3)=(1\,\;2\,\;3)$. Equipped with this notation, we choose the standard Coxeter generating set of type $A_{n-1}$ to be
\[
S(A_{n-1}) \coloneqq \Set{ (i\,\;i+1)\str 1 \leq i < n   },
\]
the set of adjacent transpositions. We denote the simple reflection $(i\,\;i+1)$ by $s_i$ for all $1 \leq i < n$. As the standard Coxeter element of type $A_{n-1}$ we choose the product of simple reflections
\[
\cox \coloneqq s_1 s_2 \ldots s_{n-1} = (1\,\;2\ldots n).
\]

\begin{figure}
	\begin{center}
		\begin{tikzpicture}
		\foreach \w in {0,1,...,5}
		\node[kpunkt] (p\w) at (\w+0.5*\w,0){};
		\draw (p0)--(p1)--(p2) (p3)--(p4)--(p5);
		\node at(3.75,0){$\ldots$};
		\node[left of=p0, xshift=-1cm]{$A_{n-1}$};
		\node[below of=p0, yshift = 0.3cm]{$s_1$};
		\node[below of=p1, yshift = 0.3cm]{$s_2$};
		\node[below of=p2, yshift = 0.3cm]{$s_3$};
		\node[below of=p3, yshift = 0.3cm]{$s_{n-3}$};
		\node[below of=p4, yshift = 0.3cm]{$s_{n-2}$};
		\node[below of=p5, yshift = 0.3cm]{$s_{n-1}$};
		\end{tikzpicture}
		\caption{The Coxeter diagram of type $A_{n-1}$ with corresponding standard simple reflections $s_1, \ldots, s_{n-1}$.}
		\label{fig:txpeA_cox_graph}
	\end{center}
\end{figure}
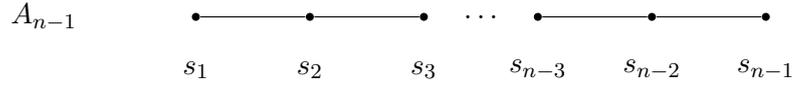

The Coxeter diagram of type $A_{n-1}$ is decpicted in \cref{fig:txpeA_cox_graph}.
The set of all reflections is the union of all conjugates of $S(A_{n-1})$, which is the set 
\[
T(A_{n-1}) \coloneqq \Set{ (i\,\;j) \str 1\leq i < j \leq n  }
\]
of all transpositions of $\Set{1,\ldots, n}$. Note that the symbols $(i\,\;j)$ and $(j\,\;i)$ describe the same maps. We prefer to write $(i\,\;j)$ if $i<j$. We say that two cycles $a,b\in S_n$ are \emph{equivalent}, and write $a\equiv b$, if they represent the same element in $S_n$. For instance, the cycles $(1\,\;2\,\;3)$ and $(2\,\;3\,\;1)$ are equivalent. By abuse of notation, we identify elements of $S_n$ with products of cycles and prefer, most of the time, to use the representatives with the smallest number on the first position. Two cycles $(i_1 \ldots i_k)$ and $(j_1 \ldots j_m)$ are called \emph{disjoint} if the sets $\Set{i_1, \ldots, i_k}$ and $\Set{j_1, \ldots, j_m}$ are disjoint. Every element in the symmetric group $S_n$ has a shortest cycle decomposition consisting of disjoint cycles that is unique up to order and equivalence of cycles. We call this decomposition the \emph{disjoint cycle decomposition}.

Our aim is to analyze which elements of the Coxeter group $W$ are in the interval below the standard Coxeter element $\cox=(1\ldots n)$ in absolute order, that is which elements are in $\nc(S_n)$. For this we need the following definitions, which are the analogs of \cite[Def. 4.4, 4.5]{bra_watt_kpi}. 

\begin{defi}
	Let $i=(i_1 \ldots i_k)$ be a cycle in $S_n$. We say that the cycle $i$ has \emph{consistent orientation} or is \emph{oriented consistently} with respect to the Coxeter element $\cox=(1 \ldots n)\in S_n$, or simply that $i$ has \emph{consistent orientation} or is \emph{oriented consistently} in $S_n$, if $i_1 < i_2 <\ldots < i_k$. A cycle that is equivalent to a cycle with consistent orientation is said to have consistent orientation as well. 
		
	Let $j=(j_1 \ldots j_l)$ be a cycle different from $i$ and suppose that both $i$ and $j$ are oriented consistently. The cycles $i$ and $j$ are called \emph{crossing} if there exist $1\leq a<b<c<d\leq n$ with
	$a=i_\alpha$,  $b=j_\beta$, $c=i_\gamma$ and $d=j_\delta$, 
	or $a=j_\alpha$,  $b=i_\beta$, $c=j_\gamma$ and $d=i_\delta$,
	for some $\alpha,\beta,\gamma,\delta$. Otherwise the cycles are called \emph{non-crossing}. An element in $S_n$ that has a disjoint cycle decomposition consisting of non-crossing cycles only is called \emph{non-crossing}.
\end{defi}

\begin{exa}
	The cycles $(1\,\;2\,\;4)$ and $(3\,\;5)$ are crossing, whereas the cycles $(1\,\;2\,\;6)$ and $(3\,\;5)$ are non-crossing. Hence $(1\,\;2\,\;6)(3\,\;5)$ is non-crossing and  $(1\,\;2\,\;4)(3\,\;5)$ is crossing.
\end{exa}

\begin{rem}
	Having consistent orientation and being non-crossing is a property of elements of $S_n$, not only of the cycles themselves.
\end{rem}

The following proposition is deduced from Theorem 4.1.3 and Lemma 4.1.4 of \cite{armstr} and justifies the notion of \enquote{non-crossing element}.

\begin{prop}\label{prop:typeA_consistent_orientation_nc}
	An element $w$ of $S_n$ is in $\nc(S_n)$ if and only if all cycles of the disjoint cycle decomposition of $w$ are oriented consistently in $S_n$ and are non-crossing.
\end{prop}

We include a sketch of the proof here, because it is instructive and also useful to understand the structure of the non-crossing partitions of type $B$ and $D$.

\begin{proof}[Sketch of proof]
	The idea is to describe the cover relations in the absolute order by \enquote{breaking cycles}. A break of a cycle $(i_1\ldots i_k)$ is obtained by inserting a separation \enquote{$)($} at any position between two numbers, for instance $(i_1\ldots i_m )(i_{m+1}\ldots i_k)$, which gives rise to another element of $S_n$. Recall that there are different, but equivalent, ways to represent an element of $S_n$ by a cycle. This has to be taken into account when breaking the cycles. For instance, let $w$ be the element described by the two equivalent cycles $(1\,\;2\,\;3)\equiv (2\,\;3\,\;1)$. 
	We can break the cycle $(1\,\;2\,\;3)$ into $(1\,\;2)(3)\equiv (1\,\;2)$ and into $(1)(2\,\;3)\equiv(2\,\;3)$.
	Performing breaks on the equivalent cycle $(2\,\;3\,\;1)$ similarly gives the two elements $(2\,\;3)$ and $(3\,\;1)\equiv(1\,\;3)$. Hence \enquote{breaking the cycle} means more precisely \enquote{breaking all cycles in the equivalence class}. 
	
	Now it can be checked that breaking cycles always produces consistently oriented cycles, which are also non-crossing. The other way around, every consistently oriented cycle arises in this way.
\end{proof}

As a direct consequence we get the following.

\begin{cor}\label{cor:length_subcycle}
	Let $a=(i_1 \ldots i_n)$ and $b= (i_{j_1}\ldots i_{j_k})$ be  such that $1 \leq j_1 < \ldots < j_k \leq n$. Then the absolute length of $b\inv a$ equals $n-k$. 
	Conversely, if $\ell(b\inv a)=\ell(a)-\ell(b)$ for any two cycles $a$ and $b$, then there is a cycle $\tilde{b}$ in the equivalence class of $b$ such that $\tilde{b}$ arises from $a$ by removing entries.
\end{cor}

\begin{proof}
	Sending $i_l \mapsto l$ gives a bijection from $\Set{i_1, \ldots, i_n}$ to $\Set{1, \ldots, n}$. Applying \cref{prop:typeA_consistent_orientation_nc} yields the result.
\end{proof}

\subsection{Pictorial representations}
Now we turn to set partitions and their pictorial representations.

\begin{defi}
	Let $B_1, \ldots, B_k$ be disjoint non-empty subsets of $\Set{ 1, \ldots, n }$ such that $ B_1 \cup \ldots \cup B_k = \Set{1, \ldots, n}$. The set $\pi=\Set{ B_1, \ldots, B_k }$ is called a \emph{(set) partition} of $\Set{ 1, \ldots, n }$ and its elements are called \emph{blocks}. A block with exactly one element is called \emph{trivial}. A non-trivial block is called a \emph{base block}.
	Two different blocks $B_i$ and $B_j$ of $\pi$ are called \emph{crossing} if there exist $1\leq a <b<c<d\leq n$ such that $a,c\in B_i$ and $b,d\in B_j$ or vice versa.
	A partition is called \emph{crossing} if it has crossing blocks and \emph{non-crossing} otherwise. The set of partitions of $\Set{ 1, \ldots, n }$ is denoted by $\pn$ and the set of non-crossing partitions is denoted by $\ncpn$.
\end{defi}

The reason for the terms crossing and non-crossing becomes clear in the pictorial representation of a partition.

\begin{defi}
	Let $\pi=\Set{ B_1, \ldots, B_k }$ be a set partition of $\Set{1, \ldots, n}$. 
	We define an embedded graph in the following way.
	Label the vertices of a regular $n$-gon in the plane with $1, 2, \ldots, n$ clockwise in this order. This circularly ordered set is the vertex set of the graph, and the edges are obtained as follows. For every block $B \in \pi$ add the $\#B$-gon on the vertices of $B$, where a $1$-gon is a point and a $2$-gon is an edge. The resulting graph is called the \emph{pictorial representation} of $\pi$. 
\end{defi}

\begin{nota}
	If there is no danger of confusion, we refer to the pictorial representation of a set partition just as \emph{partition}. The symbols $\ncpn$ and $\pn$ may also refer to the sets of pictorial representations of the non-crossing partitions and partitions, respectively.
\end{nota}

If a partition has two crossing blocks, then in its pictorial representation there exist two edges that \emph{cross}, that is they intersect outside their endpoints. The pictorial representations of two set partitions 
are shown in \cref{fig:typeA_pict_rep}.

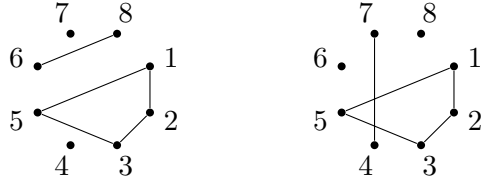
\begin{figure}
	\begin{center}
		\begin{tikzpicture}
		\begin{scope}
		\achtecklab
		\draw(p1)--(p2)(p3)--(p2)(p3)--(p5)(p1)--(p5)(p6)--(p8);
		\end{scope}
		\begin{scope}[xshift= 4cm]
		\achtecklab
		\draw(p1)--(p2)(p3)--(p2)(p3)--(p5)(p1)--(p5)(p7)--(p4);
		\end{scope}
		\end{tikzpicture}
	\end{center}
	\caption{The left-hand side depicts the non-crossing partition $\Set{ \Set{1,2,3,5},\Set{4},\Set{6,8},\Set{7} }$ and the right-hand side shows the crossing partition $\Set{ \Set{1,2,3,5},\Set{4,7},\Set{6},\Set{8} }$.}
	\label{fig:typeA_pict_rep}
\end{figure}

The set of partitions $\pn$ is a graded lattice, where the partial order is given by  \emph{refinement}, that is $\pi_1 \leq \pi_2$ if and only if for every block $B$ in $\pi_1$ there exists a block in $\pi_2$ that contains $B$. Equivalently, this means that $\pi_2$ can be obtained from $\pi_1$ by merging the blocks of $\pi_1$. The \emph{rank} of a partition $\pi$ is defined as $\rk(\pi) \coloneqq n- \#\pi$, where $\#\pi$ denotes the number of blocks of $\pi$. The Hasse diagram of $\p_4$, and $\ncp_4$ inside, is shown in \cref{fig:p4}.

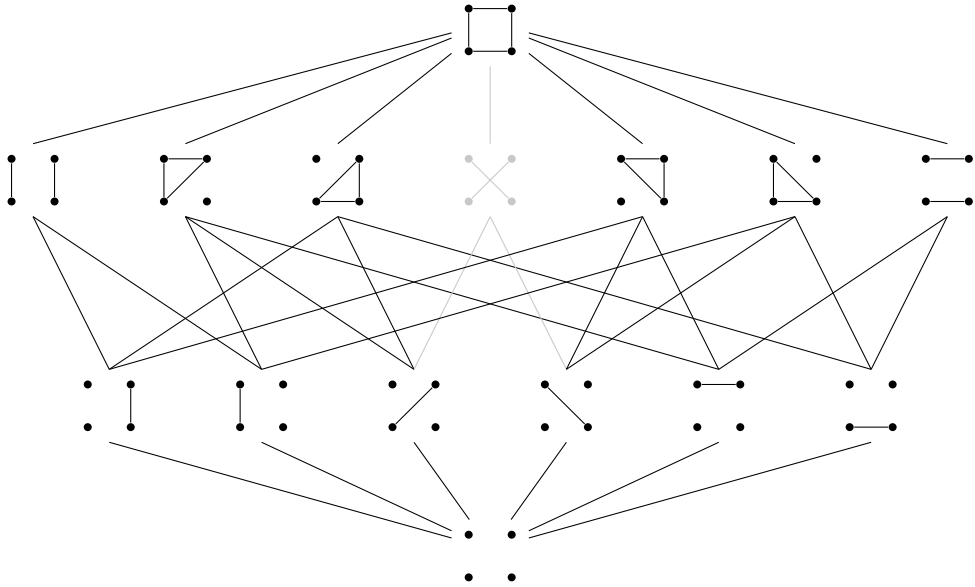
\begin{figure}%
	\begin{center}

		\begin{tikzcd}[column sep=0mm, row sep = large]
		&&&&&&\pfull
		\ar[dllllll,dash, end anchor = north]
		\ar[dllll,dash, end anchor = north]
		\ar[dll,dash, end anchor = north]
		\ar[d,dash, end anchor = north, Lightgray]
		\ar[drr,dash, end anchor = north]
		\ar[drrrr,dash, end anchor = north]
		\ar[drrrrrr,dash, end anchor = north]
		
		&&&&&&\\
		\pezudv 
		\ar[dr,dash, end anchor = north, start anchor = south]
		\ar[drrr,dash, end anchor = north, start anchor = south]
		&&
		\pedv 
		\ar[dr,dash, end anchor = north, start anchor = south]
		\ar[drrr,dash, end anchor = north, start anchor = south]
		\ar[drrrrrrr,dash, end anchor = north, start anchor = south]
		&&
		\pezd 
		\ar[dlll,dash, end anchor = north, start anchor = south]
		\ar[dr,dash, end anchor = north, start anchor = south]
		\ar[drrrrrrr,dash, end anchor = north, start anchor = south]
		&&
		\peduzv
		\ar[dl,dash, end anchor = north, start anchor = south, Lightgray]
		\ar[dr,dash, end anchor = north, start anchor = south, Lightgray]
		&&
		\pezv
		\ar[dr,dash, end anchor = north, start anchor = south]
		\ar[dl,dash, end anchor = north, start anchor = south]
		\ar[dlllllll,dash, end anchor = north, start anchor = south]
		&&
		\pzdv 
		\ar[dr,dash, end anchor = north, start anchor = south]
		\ar[dlll,dash, end anchor = north, start anchor = south]
		\ar[dlllllll,dash, end anchor = north, start anchor = south]
		&&
		\pevuzd
		\ar[dl,dash, end anchor = north, start anchor = south]
		\ar[dlll,dash, end anchor = north, start anchor = south]
		\\[1cm]
		
		&\pez &&
		\pdv &&
		\ped &&
		\pzv &&
		\pev &&
		\pzd&\\
		&&&&&& \begin{tikzpicture}\viereck\end{tikzpicture} 
		\ar[ulllll,dash, end anchor = south]
		\ar[ulll,dash, end anchor = south]
		\ar[ul,dash, end anchor = south]
		\ar[ur,dash, end anchor = south]
		\ar[urrr,dash, end anchor = south]
		\ar[urrrrr,dash, end anchor = south]
		&&&&&&
		\end{tikzcd}
		\caption{The Hasse diagram of the non-crossing partition lattice $\ncp_4$ is shown in black. The gray partition is the only crossing partition in $\p_4$.}%
		\label{fig:p4}%
		
	\end{center}
\end{figure}

Let us now take a closer look at the join and meet operation in $\pn$. 
The \emph{join} of $\pi_1$ and $\pi_2$ is the least partition $\pi_1\vee \pi_2 \in \pn$ such that every block of $\pi_1$ and every block of $\pi_2$ is contained in a block of $\pi_1\vee \pi_2$.
The \emph{meet} of $\pi_1$ and $\pi_2$ is the greatest partition $\pi_1\wedge \pi_2 \in \pn$ such that every block of $\pi_1\wedge \pi_2$ is contained in a block of $\pi_1$ and $\pi_2$.

Kreweras showed \cite{kre} that the set of non-crossing partitions $\ncpn$ is a graded lattice for all $n \in \N$, where the partial order, the rank and the meet is inherited from the lattice $P_n$. The join in $\ncpn$ of two non-crossing partitions $\pi_1$ and $\pi_2$ is the least \emph{non-crossing partition} $\pi_1\vee \pi_2 \in \ncpn$ such that every block of $\pi_1$ and every block of $\pi_2$ is contained in a block of $\pi_1\vee \pi_2$.
Explicitly, the join in $\ncpn$ is obtained from the join in $\pn$ by merging all blocks that cross.

\begin{exa}
	We illustrate the join and meet operations in $\ncp_6$ and $\p_6$ with an example. The corresponding pictures are shown in \cref{fig:typeA_meet_join}. Consider the two non-crossing partitions $\pi_1=\Set{ \Set{1,3,4},\Set{2},\Set{5},\Set{6} }$ and $\pi_2=\Set{ \Set{1},\Set{2,6},\Set{3,4},\Set{5} }$. The meet, which is the same in both $\p_6$ and $\ncp_6$, is the partition that has as only non-trivial block $\Set{3,4}$. The join of $\pi_1$ and $\pi_2$ in $\p_6$ is $\Set{ \Set{1,3,4},\Set{2,6},\Set{5} }$. Because their non-trivial blocks cross, the join of $\pi_1$ and $\pi_2$ in $\ncp_6$ differs from the join in $\p_6$. To obtain the join in $\ncp_6$ from the one in $\p_6$, we join the crossing blocks and get the partition $\Set{ \Set{1,2,3,4,6},\Set{5} }$, which is the join of $\pi_1$ and $\pi_2$ in $\ncp_6$.
	The meet and join operations also can be described pictorially. 
	
	The meet $\pi_1\wedge \pi_2$ is the same in both $\pn$ and $\ncpn$ and is obtained as follows. For every block $B$ in $\pi_1$ and $\pi_2$ add the complete graph on the vertices $B$. Now take the intersection of the two constructed graphs, that is if an edge is contained in both the graph for $\pi_1$ and $\pi_2$, then keep it, otherwise delete it. The resulting graph is a collection of complete graphs. Replacing a complete graph on the vertex set $B \subseteq \Set{1, \ldots, n}$ by the $\#B$-gon on $B$ gives the pictorial representation of $\pi_1\wedge \pi_2$.
	
	The join in $\pn$ of two partitions $\pi_1$ and $\pi_2$ is obtained as follows. Consider the graph that has a $\#B$-gon on each vertex set $B$ in $\pi_1 \cup \pi_2$, which pictorially is the union of the pictorial representations of $\pi_1$ and $\pi_2$. The pictorial representation of the join $\pi_1\vee_{\pn}\pi_2$ has a polygon on the vertices of each connected component of $\pi_1 \cup \pi_2$. 
	
	The join of two non-crossing partitions $\pi_1$ and $\pi_2$ results from their join in $\pn$ by taking the union of all crossing blocks.
\end{exa}

\begin{figure}%
	\begin{center}
		\begin{tikzpicture}
		
		\begin{scope}
		\gsechsecklab
		\draw(p1)--(p3)(p4)--(p3)(p1)--(p4);
		\node at (0,-1.5){$\pi_1$};
		\end{scope}
		
		\begin{scope}[xshift=3cm]
		\gsechsecklab
		\draw(p2)--(p6)(p4)--(p3);
		\node at (0,-1.5){$\pi_2$};
		\end{scope}
		
		\begin{scope}[xshift=6cm]
		\gsechsecklab
		\draw(p4)--(p3);
		\node at (0,-1.5){$\pi_1 \wedge \pi_2$};
		\end{scope}
		
		\begin{scope}[xshift=9cm]
		\gsechsecklab
		\draw(p2)--(p6)(p1)--(p3)(p4)--(p3)(p1)--(p4);
		\node at (0,-1.5){$\pi_1\vee_{\pn}\pi_2$};
		\end{scope}
		
		\begin{scope}[xshift=12cm]
		\gsechsecklab
		\draw(p1)--(p6)(p2)--(p3)(p2)--(p1)(p4)--(p3)(p4)--(p6);
		\node at (0,-1.5){$\pi_1\vee_{\ncpn}\pi_2$};
		\end{scope}
		\end{tikzpicture}
		\caption{From left to right, we see two non-crossing partitions $\pi_1$ and $\pi_2$, their meet $\pi_1 \wedge \pi_2=\pi_1 \wedge_{\pn} \pi_2=\pi_1 \wedge_{\ncpn} \pi_2$ in both $\p_6$ and $\ncp_6$, their crossing join $\pi_1\vee_{\pn}\pi_2$ in $\p_6$, and their non-crossing join $\pi_1\vee_{\ncpn}\pi_2$ in $\ncp_6$.}%
		\label{fig:typeA_meet_join}%
	\end{center}
\end{figure}
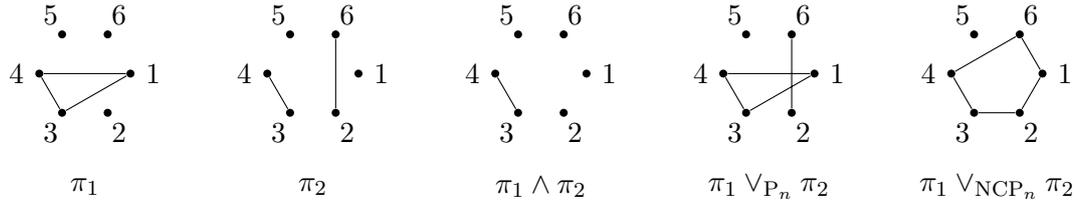

\subsection{Connecting the concepts}

Every element $w \in S_n$ gives rise to a partition $\Set{w}$ of $\Set{1,\ldots,n }$ by defining the blocks of $\Set{w}$ as the orbits of $w$ on $\Set{1, \ldots, n}$. Equivalently, the blocks of $\Set{w}$ are precisely the cycles of $w$, where trivial cycles are considered. Biane \cite[Thm. 1]{biane} and Brady \cite[Le. 3.2]{bra_kpi} show the following.

\begin{thm}\label{thm:typeA_nc_iso}
	The map $\nc(S_n) \to \ncpn$ defined by $w \mapsto \Set{w}$ is a lattice isomorphism.
\end{thm}

The inverse map has a nice description, which is as follows. Let $\pi \in \ncpn$ be a non-crossing partition. Every block of $\pi$ yields a unique consistently oriented cycle. The product of all cycles coming from the blocks of $\pi$ is a permutation and indeed an element of $\nc(S_n)$. The consistent orientation of a cycle corresponding to a block can be read off the pictorial representation by going clockwise around the block.

The lattice isomorphism of \cref{thm:typeA_nc_iso} allows us to pass freely from the group-theoretic non-crossing partitions $\nc(S_n)$ to the set-theoretic non-crossing partitions $\ncpn$, and in particular their pictorial representations, and back. We are using this frequently in the study of non-crossing partitions.

\section{Non-crossing partitions of type $B$}\label{sec:typeB}

The non-crossing partitions of type $B$ are the easiest generalization of the classical non-crossing partitions, since most concepts directly translate from type $A$. The study of the type $B$ analog of both partitions and non-crossing partitions was initiated by Reiner in \cite{rei}. We only treat the non-crossing partitions of type $B$ here.

\subsection{The signed symmetric group}\label{sec:typeB_descr}
The Coxeter group $W(B_n)$ of type $B_n$ is the \emph{signed symmetric group}. Its elements, which are called \emph{signed permutations}, are those bijective self-maps $w$ of $\Set{\pm 1, \ldots, \pm n }$ that satisfy $w(-i)=-w(i)$ for all $-n \leq i \leq n$. The group operation is given by composition of maps.
We identify the signed symmetric group $W(B_n)$ with a subgroup of the symmetric group $S_{2n}$ via the \emph{canonical identification}
\[
\Set{\pm1, \ldots, \pm n} \to \Set{1, \ldots, 2n }, \quad i \mapsto
\begin{cases}
i &\text{if } i>0,\\
n-i &\text{if } i<0. 
\end{cases}
\]
This allows us to use the cycle notation of the symmetric group for the signed symmetric group. For example, we denote the element of $W(B_n)$ that exchanges $1$ and $2$, and necessarily also $-1$ and $-2$,  by $(1\,\;2)(-1\,\;-2)$. But the whole information of this element is already stored in the piece $(1\,\;2)$, because the permutation is signed. 

Brady and Watt introduced the shorthand notations of paired and balanced cycles in \cite{bra_watt_kpi}, which we use here.
Let $i_1, \ldots, i_k \in \Set{\pm 1, \ldots, \pm n }$ be different elements such that $\Set{i_1, \ldots, i_k}$ is disjoint from $\Set{-i_1, \ldots, -i_k}$. We write
\[
\dka i_1 \ldots i_k \dkz \coloneqq (i_1\ldots i_k)(-i_1 \ldots -i_k)
\] 
and call $\dka i_1 \ldots i_k \dkz$ a \emph{paired cycle}, and we write
\[
[i_1 \ldots i_k] \coloneqq (i_1 \ldots i_k\,\; -i_1 \ldots -i_k)
\]
and call $[i_1 \ldots i_k]$ a \emph{balanced cycle}. Note that the notion of equivalent cycles applies here. For instance, the four balanced cycles $\dka i\,\;j\dkz$, $\dka j\,\;i\dkz$, $\dka -i\,\;-j\dkz$ and $\dka -j\,\;-i\dkz$ all describe the same bijection and therefore are equivalent. We prefer to use the symbol $\dka i\,\;j \dkz$ if $0<i<|j|$. However, note that the two balanced cycles $[1\,\;2]$ and $[2\,\;1]$ are \emph{not} equivalent, whereas $[i]\equiv[-i]$ for all $i \in \Set{\pm 1, \ldots, \pm n }$. 
Let $i_1, \ldots, i_k \in \Set{\pm 1, \ldots, \pm n}$ be different elements as well as $j_1, \ldots, j_m \in \Set{\pm 1, \ldots, \pm n}$. Then the balanced or paired cycles corresponding to $\Set{i_1, \ldots, i_k}$ and $\Set{j_1, \ldots, j_m}$ as above are called \emph{disjoint} if and only if the sets $\Set{\pm i_1, \ldots, \pm i_k}$ and $\Set{\pm j_1, \ldots, \pm j_m}$ are disjoint. 
Every element in $W(B_n)$ can be written as a product of disjoint paired and balanced cycles and this product is unique up to order and equivalence of cycles \cite[Prop. 3.1]{bra_watt_kpi}. We call this decomposition the \emph{disjoint cycle decomposition}.

We choose the standard Coxeter generating set of $W(B_n)$ to be 
\[
S(B_n) \coloneqq \Set{ \dka i\,\; i+1 \dkz\str 1\leq i <n } \cup \Set{[n]}
\]
and denote the simple reflections by $s_i\coloneqq \dka i \,\; i+1 \dkz$ for $1\leq i < n$ and $s_n\coloneqq [n]$. The standard Coxeter element of type $B_n$ is defined to be
\[
\cox \coloneqq s_1s_2\ldots s_{n-1}s_n =[1\,\;2\ldots n].
\]
The Coxeter diagram of type $B_n$ is depicted in \cref{fig:typeB_cox_graph} and the set of all reflections is given by
\[
T(B_n) \coloneqq \Set{\dka i\,\;j\dkz \str 1 \leq i < |j| \leq n } \cup \Set{ [i]\str 1\leq i \leq n }.
\]

The absolute length of an element in $W$ can be computed using its disjoint cycle decomposition \cite[Le. 3.3, Le. 3.4]{bra_watt_kpi}.

\begin{figure}
	\begin{center}
		\begin{tikzpicture}
		\foreach \w in {0,1,...,5}
		\node[kpunkt] (p\w) at (\w+0.5*\w,0){};
		\draw (p0)--(p1)--(p2) (p3)--(p4)--(p5);
		\node at(3.75,0){$\ldots$};
		\node[above] at(6.75,0){$4$};
		\node[left of=p0, xshift=-1cm]{$B_{n}$};
		\node[below of=p0, yshift = 0.3cm]{$s_1$};
		\node[below of=p1, yshift = 0.3cm]{$s_2$};
		\node[below of=p2, yshift = 0.3cm]{$s_3$};
		\node[below of=p3, yshift = 0.3cm]{$s_{n-2}$};
		\node[below of=p4, yshift = 0.3cm]{$s_{n-1}$};
		\node[below of=p5, yshift = 0.3cm]{$s_{n}$};
		\end{tikzpicture}
		\caption{The Coxeter diagram of type $B_n$ with corresponding standard simple reflections $s_1, \ldots, s_n$.}
		\label{fig:typeB_cox_graph}
	\end{center}
\end{figure}
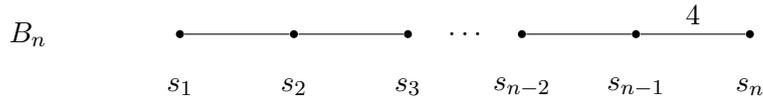

\begin{lem}\label{lem:typeB_length}
	Let $w\in W(B_n)$ and $z_1\ldots z_k$ be its disjoint cycle decomposition. Then the absolute length of $w$ is $\ell(w)= \ell(z_1) + \ldots + \ell(z_k)$, where the length of a paired cycle $\dka i_1 \ldots i_k \dkz$ equals $k-1$ and the length of a balanced cycle $[j_1 \ldots j_m]$ equals $m$. Moreover, every reduced decomposition of a paired cycle involves only reflections of the form $\dka i \,\; j\dkz$, whereas every reduced decomposition of a balanced cycle involves exactly one reflection of the form $[i]$.
\end{lem}

The next aim is to analyze which elements of $W(B_n)$ are non-crossing partitions with respect to the standard Coxeter element $\cox=[1\ldots n]$. A necessary condition on the disjoint cycle decomposition is given by \cite[Cor. 4.3]{bra_watt_kpi}.

\begin{lem}\label{lem:typeB_cycle_decomp_nc}
	Let $w$ be in $\nc(B_n)$. Then its disjoint cycle decomposition either has no or exactly one balanced cycle.
\end{lem}

The following definitions are deduced from \cite[Def. 4.4, 4.5]{bra_watt_kpi}.

\begin{defi}
	A balanced or a paired cycle in $W(B_n)$ is \emph{oriented consistently} or has \emph{consistent orientation} with respect to the Coxeter element $\cox=[1\ldots n]$, or simply is \emph{oriented consistently} or has \emph{consistent orientation} in $W(B_n)$, if their cycles in the image in $S_{2n}$ under the canonical identification are consistently oriented.
\end{defi}

Also the notion of \enquote{non-crossing} in type $B$ is analog to the one in type $A$. 

\begin{defi}
	Two consistently oriented disjoint cycles $a$ and $b$ in $W(B_n)$ are called \emph{crossing} if their canonical images in the symmetric group $S_{2n}$ cross. Otherwise they are called \emph{non-crossing}. An element in $W(B_n)$ that admits a disjoint cycle decomposition consisting of non-crossing cycles is called \emph{non-crossing}.
\end{defi}

\begin{rem}\label{rem:typeB_one_balanced_cycle}
	Note that any two different balanced cycles cross. Hence a non-crossing element in $W(B_n)$ has at most one balanced cycle in its disjoint cycle decomposition. 
\end{rem}

In analogy to type $A$, we characterize the elements in $\nc(B_n)$ in terms of their cycle structure \cite[Prop. 4.6, Prop. 4.7]{bra_watt_kpi}. 

\begin{prop}\label{prop:typeB_consistent_orientation-NC}
	An element $w\in W$ satisfies $w\leq \cox$ if and only if all cycles of the disjoint cycle decomposition of $w$ are non-crossing and oriented consistently.
\end{prop}

Although we defined the notion of crossing and non-crossing cycles in type $B$ in analogy to type $A$, we should not think of the concepts as being \enquote{the same}. In type $A$ the notion of crossing is in some sense a stronger \enquote{obstruction} than the notion of crossing in type $B$. Let us illustrate this with a concrete example.

\begin{exa}\label{exa:crossings_in_A_and_B}
	Let  $a_1=(1\,\;4)$ and $a_2=(2\,\;5)$ be elements of $\nc(S_6)$ and let $b_1=[1]$ and $b_2=[2]$ be the elements in $\nc(B_3)$ corresponding to $a_1$ and $a_2$ under the canonical identification. Then the product $a_1a_2=a_2a_1=(1\,\;4)(2\,\;5)$ has crossing cycles, as well as the product $b_1b_2=b_2b_1=[1][2]$. The smallest element in $\nc(S_6)$ that is greater than both $a_1$ and $a_2$, that is the join of $a_1$ and $a_2$, is $a=(1\,\;2\,\;4\,\;5)$ and has length $\ell(a)=3$, whereas the join of $b_1$ and $b_2$ in $\nc(B_3)$, which is $b=[1\,\;2]$, has length $\ell(b)=2$, since $\dka 1\,\; 2 \dkz [2]$ is a reduced decomposition of $[1\,\;2]$. 
	Hence we see that crossing cycles of two different elements in type $A$ are already an obstruction for having a common cover, whereas in type $B$ this is \emph{not} the case. But note that there obviously is no reduced decomposition of the common cover $[1\,\;2]$ involving both elements $[1]$ and $[2]$. In type $A$ one can easily check that if $w=r \vee t$ is the join of two reflections and $\ell(w)=2$, then $w=rt$ or $w=tr$. In type $B$ this statement is false.
\end{exa}

\subsection{Pictorial representations}

Reiner introduced the type $B$ partitions and their pictorial representations \cite{rei}.

\begin{defi}
	A partition $\pi=\Set{B_1, \ldots, B_k }$ of $\Set{\pm1, \ldots, \pm n }$ is called a \emph{$B_n$-partition} if $\pi$ satisfies 
	\begin{enumerate}
		\item $B\in \pi$ if and only if $-B \in \pi$ and
		\item there is at most one block $B\in\pi$ with $B=-B$.
	\end{enumerate}
	A block $B$ is called \emph{zero block} if $B=-B$, and \emph{non-zero block} otherwise.
	We use the canonical identification of  $\Set{\pm1, \ldots, \pm n }$ with $\Set{1, \ldots, 2n}$ as before. A $B_n$-partition is called \emph{crossing} if its image in $S_{2n}$ is crossing and \emph{non-crossing} otherwise. The notion of crossing and non-crossing blocks translates analogously. We denote the set of non-crossing $B_n$-partitions by $\ncb_n$.
\end{defi}

\begin{rem}
	Note that the images of the non-crossing $B_n$-partitions under the canonical identification are the non-crossing partitions of $\Set{1,\ldots, 2n}$ by definition.
\end{rem}

The pictorial representation of a $B_n$-partition is defined analogously to the pictorial representation of a set partition, again by using the canonical identification of the sets $\Set{\pm1, \ldots, \pm n }$ and $\Set{1, \ldots, 2n}$.

\begin{nota}
	If there is no danger of confusion, we identify $B_n$-partitions with their pictorial representation and refer to a pictorial representation simply as \emph{$B_n$-partition} or \emph{partition of type $B_n$}. In particular, $\ncb_n$ refers to the set of pictorial representations of non-crossing $B_n$-partitions as well.
\end{nota}

Reiner showed that $\ncb_n$ is a graded lattice, ordered by refinement, where the rank of $\pi \in \ncb_n$ is given by $\rk(\pi)=n-\lfloor\frac{\#\pi}{2}\rfloor$ \cite[Prop. 2]{rei}. Note that the rank function on $\ncb_n$ differs from the rank function of set partitions. Whereas for instance the $B_6$-partition $\Set{\Set{1,-1},\Set{2,3},\Set{-2,-3}}\in \ncb_3$ has rank $2$, its image $\Set{\Set{1,4},\Set{2,3},\Set{5,6}} \in \ncp_6$ under the canonical identification has rank $3$. 

To describe the meet and join operations in $\ncb_n$ we identify the non-crossing $B_n$-partitions with their canonical images in $\ncp_{2n}$, perform the meet or join on the images in $\ncp_{2n}$, and use the identification once again to retrieve a non-crossing $B_n$-partition, which is the meet or join of the two $B_n$-partitions we started with. 
Note that this procedure does not contradict the observations from above or \cref{exa:crossings_in_A_and_B}.
The Hasse diagram of the lattice $\ncb_3$ is shown in \cref{fig:typeB_hasse_diagram_B4}.

\begin{figure}%
	\begin{center}		
		\begin{tikzcd}[column sep=4mm]
		&&&& \pBfull 
		\ar[dllll, dash,end anchor = north]
		\ar[dlll, dash,end anchor = north]
		\ar[dll, dash,end anchor = north]
		\ar[dl, dash,end anchor = north]
		\ar[d, dash,end anchor = north]
		\ar[dr, dash,end anchor = north]
		\ar[drr, dash,end anchor = north]
		\ar[drrr, dash,end anchor = north]
		\ar[drrrr, dash,end anchor = north]
		&&&&\\
		
		\pBZed
		\ar[d,dash, start anchor = south, end anchor = north]
		\ar[drrr,dash, start anchor = south, end anchor = north]
		\ar[drrrrr,dash, start anchor = south, end anchor = north]
		\ar[drrrrrrr,dash, start anchor = south, end anchor = north] & 
		\pBezud 
		\ar[d,dash, start anchor = south, end anchor = north]
		\ar[dl,dash, start anchor = south, end anchor = north]& 
		\pBezmd 
		\ar[d,dash, start anchor = south, end anchor = north]
		\ar[dl,dash, start anchor = south, end anchor = north]
		\ar[drrrrr,dash, start anchor = south, end anchor = north]& 
		\pBZez 
		\ar[d,dash, start anchor = south, end anchor = north]
		\ar[dll,dash, start anchor = south, end anchor = north]
		\ar[drrr,dash, start anchor = south, end anchor = north]
		\ar[drrrrr,dash, start anchor = south, end anchor = north]& 
		\pBeuzd 
		\ar[d,dash, start anchor = south, end anchor = north]
		\ar[dl,dash, start anchor = south, end anchor = north]& 
		\pBezd 
		\ar[d,dash, start anchor = south, end anchor = north]
		\ar[dl,dash, start anchor = south, end anchor = north]
		\ar[dllll,dash, start anchor = south, end anchor = north]& 
		\pBZzd 
		\ar[d,dash, start anchor = south, end anchor = north]
		\ar[dll,dash, start anchor = south, end anchor = north]
		\ar[dllll,dash, start anchor = south, end anchor = north]
		\ar[dllllll,dash, start anchor = south, end anchor = north]& 
		\pBemduz
		\ar[d,dash, start anchor = south, end anchor = north]
		\ar[dl,dash, start anchor = south, end anchor = north] & 
		\pBemzmd
		\ar[d,dash, start anchor = south, end anchor = north]
		\ar[dl,dash, start anchor = south, end anchor = north]
		\ar[dllll,dash, start anchor = south, end anchor = north]\\[2cm]
		
		\pBd & \pBez & \pBzmd & \pBe & \pBzd & \pBed & \pBz & \pBemd & \pBemz\\
		&&&& \begin{tikzpicture} \sechseck\end{tikzpicture} 
		\ar[ullll, dash,end anchor = south]
		\ar[ulll, dash,end anchor = south]
		\ar[ull, dash,end anchor = south]
		\ar[ul, dash,end anchor = south]
		\ar[u, dash,end anchor = south]
		\ar[ur, dash,end anchor = south]
		\ar[urr, dash,end anchor = south]
		\ar[urrr, dash,end anchor = south]
		\ar[urrrr, dash,end anchor = south]&&&&
		\end{tikzcd}
		\caption{The Hasse diagram of $\ncb_3$.}%
		\label{fig:typeB_hasse_diagram_B4}%
	\end{center}
\end{figure}
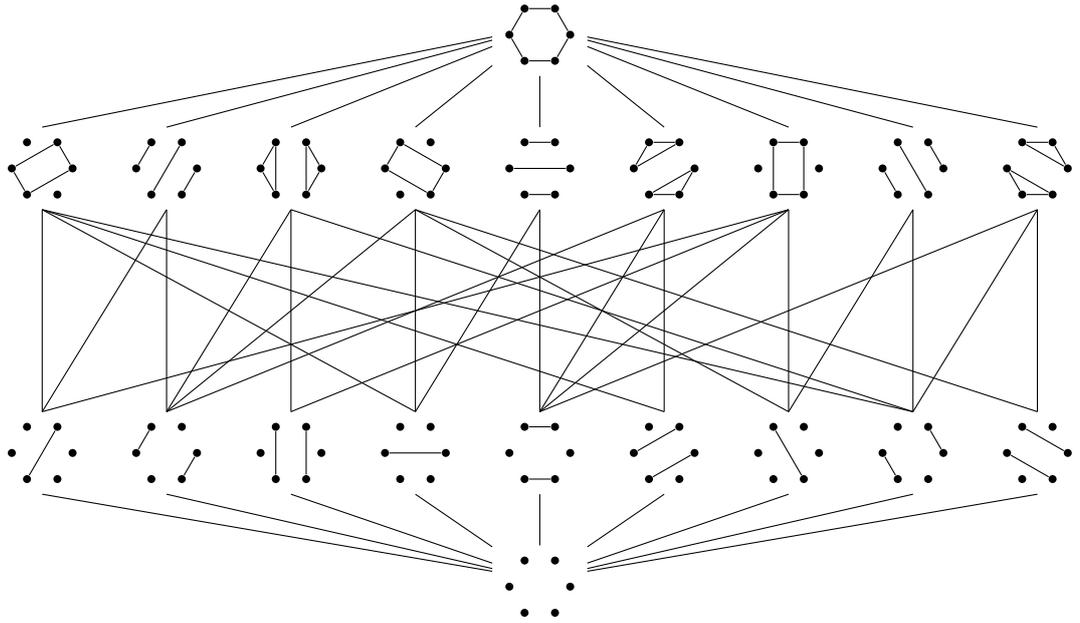

\subsection{Connecting the concepts}

Fortunately the two concepts of non-crossing partitions of type $B_n$ are compatible \cite[Thm. 4.9]{bra_watt_kpi}.
Recall that for a permutation $w$, the partition $\Set{w}$ has as non-trivial blocks the non-trivial cycles of $w$. For instance, the element $w=\dka 1\,\;2\dkz\dka3\,\;4\dkz$ in $\nc(B_4)$ defines the corresponding $B_4$-partition $\Set{w}=\Set{\Set{1,2},\Set{-1,-2},\Set{3,4},\Set{-3,-4}}$.

\begin{thm}\label{thm:typeB_nc_iso}
	The map $\nc(B_n) \to \ncb_n$ with $w \mapsto \Set{w}$ is a lattice isomorphism.
\end{thm}

In analogy to type $A$, we can construct an element in $\nc(B_n)$ for every non-crossing $B_n$-partition. The cycles are given by the blocks, and the order is obtained from ordering going clockwise around each block. This procedure yields a cycle with consistent orientation.

\begin{exa}
	Let us illustrate the isomorphism and derived observations with an example. Consider the signed permutation $w$ given by $[2\,\;5]\dka 3\,\;4 \dkz $. This is indeed an element of $\nc(B_5)$ because the two cycles do not cross. The corresponding $B_5$-partition, depicted in \cref{fig:typeB_example}, is given by $\Set{\Set{1},\Set{-1},\Set{2,5,-2,-5},\Set{3,4},\Set{-3,-4}}$.
	The balanced cycle is mapped to a zero block and the paired cycle is mapped to a pair of blocks. Note that the pictorial representation is invariant under rotation by 180 degrees. The reason is the condition that the elements of $W$ are \emph{signed} permutations. Changing the sign of the vertices of the polygon corresponds to a rotation by 180 degrees. Since singed permutations are \enquote{invariant} under sign change, the resulting pictorial representations are invariant under rotation by 180 degrees.
\end{exa}

\begin{figure}
	\begin{center}
		\begin{tikzpicture}
		\foreach \w in {1,...,10} 
		\node (p\w) at (-\w  * 360/10 +36   : 9mm) [kpunkt] {};
		\foreach \w in {1,...,5} 
		\node (q\w) at (-\w  * 360/10 +36 : 13mm) {$\w$};
		\foreach \w [evaluate = \w as \j using int(5-\w)]in {6,...,10} 
		\node (q\w) at (-\w  * 360/10 +36 : 13mm) {$\j$};
		\draw(p2)--(p5)(p5)--(p7)(p7)--(p10)(p10)--(p2)(p3)--(p4)(p8)--(p9);
		\end{tikzpicture}
		\caption{The pictorial representation of the non-crossing $B_5$-partition corresponding to the signed permutation $[2\,\;5]\dka 3\,\;4 \dkz$ is depicted here.}
		\label{fig:typeB_example}
	\end{center}
\end{figure}
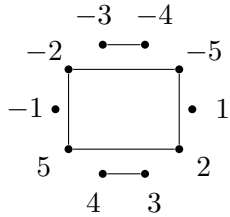

\section{Non-crossing partitions of type $D$}

The non-crossing partitions of type $D$ are the most involved ones among the classical types. The fist attempt of defining a pictorial version of type $D$ \cite{rei} did not fit to the group-theoretic definition \cite{bra_watt_kpi}, which was established later. A pictorial version fitting to the group theory is introduced in \cite{ath_rei}. A modified version with slightly different conventions is given in \cite{bg}. We introduce a new one here. We start with the description of the Coxeter group of type $D$.

\subsection{The oriflamme group}\label{sec:typeD_descr}

The Coxeter group $W(D_n)$ of type $D_n$ is the index 2 subgroup of the signed symmetric group $W(B_n)$ generated by the reflections of the form $\dka i\,\;j \dkz$ for $1\leq i < |j| \leq n$. Since $W(D_n)$ is a subgroup of the signed symmetric group, we can use the same terminology and notation as for the Coxeter group of type $B$. In particular, the group $W(D_n)$ can be canonically identified with a subgroup of the symmetric group $S_{2n}$. Hence every element of $W(D_n)$ has a disjoint cycle decomposition, which is unique up to order and equivalence of cycles.

The Coxeter group of type $D_n$ is also called \emph{oriflamme group}, because its Coxeter diagram, shown in \cref{fig:typeD_cox_graph}, resembles the medieval flag \enquote{Oriflamme} \cite{gar}.

\begin{rem}\label{rem:typeD_even_number_bal_cyc}
	Every cycle decomposition of an element in $W(D_n)$ either has no or an even number of balanced cycles, which can be seen as follows. Every generator $t=\dka i\,\; j\dkz$, seen as self-bijection of $\Set{\pm1, \ldots, \pm n }$, changes an even number of signs, that is the number of negative elements of  $\Set{t(1), \ldots, t(n)}$ is even. Consequently, every element of $W(D_n)$ has an even number of sign changes. But a balanced cycle changes exactly one sign, hence there are none or an even number of balanced cycles present in cycle decompositions of elements of the oriflamme group.
\end{rem}

The set of reflections of $W(D_n)$ is given by 
\[
T(D_n) \coloneqq \Set{ \dka i\,\; j\dkz \str 1\leq i < |j| \leq n }
\]
and we choose the standard Coxeter generating set of type $D_n$ to be
\[
S(D_n) \coloneqq \Set{ \dka i\,\; i+1 \dkz \str 1\leq i <n } \cup \Set{ \dka -(n-1)\,\; n\dkz }.
\] 
The simple reflections are denoted by $s_i \coloneqq \dka i\,\; i+1 \dkz$ for $1\leq i < n$ and $s_n \coloneqq \dka -(n-1)\,\; n\dkz$. We choose the standard Coxeter element in type $D_n$ to be the product of simple reflections
\[
\cox \coloneqq s_1s_2\ldots s_{n-1}s_n =[1\ldots n-1][n].
\]
Note that this Coxeter element is different from the standard Coxeter element in \cite{bra_watt_kpi}, but is is the same as the Coxeter element in \cite{ath_rei}, which is our main source for the type $D$. The results from \cite{bra_watt_kpi} hold in analogous ways. 

Moreover, note that we chose to write $s_n=\dka -(n-1)\,\;n \dkz$ instead of $\dka n-1\,\; -n\dkz$. When we are working in type $D$, we always prefer to have $n$ instead of $-n$ present in the cycle notation. This convention becomes clear when we define the new pictorial representation of non-crossing partitions of type $D$ in \cref{sec:typeD_pict} below. 

\begin{figure}
	\begin{center}
		\begin{tikzpicture}
		\foreach \w in {0,1,...,5}
		\node[kpunkt] (p\w) at (1.5*\w,0){};
		\node[kpunkt] (p6) at (6,1.5){};
		\draw (p0)--(p1)--(p2) (p3)--(p4)--(p6) (p4)--(p5); 
		\node at(3.75,0){$\ldots$};
		\node[left of=p0, xshift=-1cm]{$D_{n}$};
		\node[below of=p0, yshift = 0.3cm]{$s_1$};
		\node[below of=p1, yshift = 0.3cm]{$s_2$};
		\node[below of=p2, yshift = 0.3cm]{$s_3$};
		\node[below of=p3, yshift = 0.3cm]{$s_{n-3}$};
		\node[below of=p4, yshift = 0.3cm]{$s_{n-2}$};
		\node[below of=p5, yshift = 0.3cm]{$s_{n-1}$};
		\node[ right of=p6, xshift=-0.3cm]{$s_{n}$};
		\end{tikzpicture}
		\caption{The Coxeter diagram of type $D_n$ with corresponding standard simple reflections $s_1, \ldots, s_n$.}
		\label{fig:typeD_cox_graph}
	\end{center}
\end{figure}
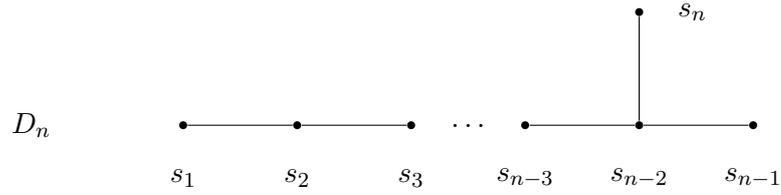

Unfortunately, it is not so easy to describe which elements of the oriflamme group are in $\nc(D_n)$, that is for which $w\in W(D_n)$ we have $w \leq \cox$. A necessary condition on the cycle structure is given by \cite[Cor. 4.3]{bra_watt_kpi}.

\begin{lem}\label{lem:typeD_cycle_decomp_nc}
	Let $w\in W(D_n)$ be such that $w \leq c$. Then the disjoint cycle decomposition of $w$ either has no balanced cycles or exactly two balanced cycles. In the second case, one of these balanced cycles is $[n]$.
\end{lem}

The absolute length of an element in $W(D_n)$ can be computed using its cycle decomposition. 

\begin{lem}\label{lem:typeD_length}
	The absolute length in $W(D_n)$ of a paired cycle $\dka i_1 \ldots i_k \dkz$ equals $k-1$ and the length of the product of two balanced cycles $[j_1 \ldots j_m][j_{m+1}]$ for  $1\leq m <n$ is $m+1$.
	If $w\in W(D_n)$ is a product of disjoint paired cycles $p_1\ldots p_k$ for $0\leq k\leq n$, then the absolute length of $w$ is 
	\[
	\ell(w)= \ell(p_1) + \ldots + \ell(p_k).
	\]
	If $w\in W(D_n)$ is a product $p_1\ldots p_k b_1 b_2$ of disjoint cycles, where $p_1, \ldots, p_k$ are paired with $0\leq k \leq n-2$, and $b_1$ and $b_2$ are balanced with $1\leq m <n$, then the absolute length of $w$ is  
	\[
	\ell(w)= \ell(p_1) + \ldots + \ell(p_k) + \ell(b_1b_2).
	\]
\end{lem}

\begin{proof}
	The absolute length of a paired cycle $\dka i_1 \ldots i_k \dkz$ equals $k-1$ and involves only reflections of the form $\dka i\,\;j\dkz$ with $i,j\in\Set{i_1, \ldots, i_k}$ by \cref{lem:typeB_length}. This shows the first case. Now consider the second case. It follows from a direct computation that $[j_1\ldots j_m][j_{m+1}]=\dka j_1\ldots j_{m+1}\dkz\dka -j_m\,\; j_{m+1}\dkz$ and that $[j_1\ldots j_m][j_{m+1}]$ admits a reduced decomposition with $m+1$ reflections. Since the paired and balanced cycles are disjoint, the elements in the reduced decomposition of the pair of balanced cycles are not involved in the reduced decompositions of the paired cycles. Hence the sum of the length of the product of paired cycles and the length of the product of balanced cycles equals the length of $w$. 
\end{proof}

\begin{rem}\label{rem:typeBD_length_bal_cyc}
	If $b$ is a balanced cycle with entries in $\Set{\pm 1, \ldots, \pm(n-1) }$, then the absolute length of $b$ in $W(B_{n-1})$ plus one equals the absolute length of $b[n]$  in $W(D_n)$, that is $\ell_D(b[n])=\ell_B(b)+1$ by \cref{lem:typeB_length}.
\end{rem}

Let us illustrate with an example why \enquote{consistent orientation} of the cycles of a reduced decomposition with the Coxeter element $(1\,\;2 \ldots 2n)$ of the canonical image in $S_{2n}$ is \emph{not} the right concept for cycles involving $n$, as it was in type $B$. The right notion of \enquote{consistent orientation} is slightly more subtle. We introduce it in \cref{defi_typeD_consistent_orientation} below.

\begin{exa}\label{exa:typeD_not_consistent_orientation}
	Let $w = \dka -1\,\;2\,\;4 \dkz = \dka -1\,\;2 \dkz \dka 2\,\;4 \dkz$ be an element of the oriflamme group $W(D_4)$. The canonical image of $w$ in $S_8$ is $(2\,\;4\,\;5)(1\,\;6\,\;8)$, which consists of consistently oriented and non-crossing cycles. Hence it is in $\nc(S_8)$ by \cref{prop:typeA_consistent_orientation_nc}. In order to check whether $w \leq \cox$ or not, note that the length of $w\inv \cox = [1][2\,\;3\,\;4]$ equals $4$ by \cref{lem:typeD_length}, but $\ell(\cox)-\ell(w)=2$. Hence $w$ is \emph{not} in $\nc(D_4)$ as one might have suspected. 
	Also it does not suffice that the cycle $\dka -1\,\;2\,\;4 \dkz$ is consistently oriented in $W(B_3)$ after removing $4$, because $\dka -1\,\;2 \dkz$ \emph{is} consistently oriented, but $\dka -1\,\;2\,\;4 \dkz$ is \emph{not} in $\nc(D_4)$. To finish this example we note that the element $w\inv=\dka 2\,\;-1\,\;4\dkz$ is in $\nc(D_4)$, since $w\cox=\dka 1\,\;4 \dkz \dka 2\,\;3 \dkz$ has length $2$. 
\end{exa}

\begin{defi}\label{defi_typeD_consistent_orientation}
	Let $b$ be a balanced cycle $[i_1\ldots i_k]$. Then $b$ is called \emph{consistently oriented} with respect to the Coxeter element $[1\ldots n-1][n]\in W(D_n)$, or simply \emph{consistently oriented} in $W(D_n)$, if  $-n < i_j < n$ for all $1\leq j \leq k$ and $b$ is consistently oriented in $W(B_{n-1})$. 
	
	Let $p$ be a paired cycle $\dka i_1 \ldots i_k \dkz$ such that $-n < i_j < n$ for all $1\leq j \leq k$. Then $p$ is called \emph{consistently oriented} with respect to the Coxeter element $[1\ldots n-1][n]\in W(D_n)$, or simply \emph{consistently oriented} in $W(D_n)$, if $p$ is consistently oriented in $W(B_{n-1})$.

	Let $q = \dka i_1 \ldots i_k\,\; n \dkz$ be a paired cycle with $-n < i_j < n$ for all $1\leq j \leq k$. 
	There are two cases to consider. If $i_1 > 0$, then $q$ is called \emph{consistently oriented} with respect to the Coxeter element $[1\ldots n-1][n]\in W(D_n)$, or \emph{consistently oriented} in $W(D_n)$, if there exists some $l\leq k$ such that $0<i_1< \ldots< i_l$ and $0>i_{l+1}> \ldots> i_k > -i_1$.
	If $i_1 < 0$, then $q$ is called \emph{consistently oriented} with respect to the Coxeter element $[1\ldots n-1][n]\in W(D_n)$, or \emph{consistently oriented} in $W(D_n)$, if there is some $l\leq k$ 
	such that $0>i_1 > \ldots > i_l$ and $0<i_{l+1} < \ldots < i_k < -i_1$.

	A cycle that is equivalent to a consistently oriented cycle is called \emph{consistently oriented} or said to have \emph{consistent orientation}.
\end{defi}

\begin{rem}\label{rem:typeD_defi_consistent_orientation}
	Although the definition of the consistent orientation of a paired cycle involving $n$ is not in terms of consistent orientation of a cycle in $W(B_{n-1})$, as it is in the other cases, we can however connect the concepts of consistent orientation in Coxeter groups of types $B$ and $D$. Note that a cycle $q$ involving $n$ that is consistently oriented in $W(D_n)$ in general is \emph{not} consistently oriented in $W(B_n)$. This was highlighted in \cref{exa:typeD_not_consistent_orientation}.
	\begin{enumerate}
		\item If $q$ is a consistently oriented paired cycle involving $n$, then the paired  cycle $p$ arising from $q$ by removing the entry $n$ is a consistently oriented paired cycle in $W(B_{n-1})$, and hence also consistently oriented in $W(D_n)$.
		
		\item If $p$ is a consistently oriented cycle in $W(B_{n-1})$, then there is an element $\dka i_1 \ldots i_k\dkz $ in the equivalence class of $p$ such that $\dka i_1 \ldots i_k \,\; n \dkz$ is consistently oriented with respect to the Coxeter element $[1\ldots n-1][n]$ of $W(D_n)$.
	\end{enumerate}
\end{rem}

\begin{prop}
	Let $w \in W(D_n)$ be such that its disjoint cycle decomposition either consists of exactly one non-trivial paired cycle $p$ or a product of a balanced cycle $b$ not involving $\pm n$ with the balanced cycle $[n]$. Then $w$ is an element of $\nc(D_n)$ if and only if $p$, or $b$ respectively, is consistently oriented with respect to the Coxeter element $[1 \ldots n-1][n]$ of $W(D_n)$.
\end{prop}

\begin{proof}
	
	We consider the three cases that
	\begin{enumerate}
		\item $p$ is a paired cycle not involving $\pm n$, 
		\item $b$ is a balanced cycle not involving $\pm n$, and
		\item $p$ is a paired cycle involving $\pm n$.
	\end{enumerate}
	
	Recall that paired and balanced cycles that do not involve $\pm n$ give rise to  elements of both $W(D_n)$ and $W(B_{n-1})$. 
	Moreover, we can interpret a paired cycle not involving $\pm n$ both as an element in $W(D_n)$ and in $W(B_{n-1})$.
	Since reduced decompositions of paired cycles in $W(B_{n-1})$ only involve reflections given by paired cycles by \cref{lem:typeB_length}, the length of the element given by a paired cycle in $W(B_{n-1})$ equals the length of it in $W(D_{n-1})$. We denote the absolute length in $W(B_{n-1})$ by $\ell_B$ and the absolute length of $W(D_n)$ by $\ell_D$.
	
	Consider case a) and let $p$ be a paired cycle not involving $\pm n$. Let $w\in W(D_n)$ be the element having $p$ as disjoint cycle decomposition and length $\ell_D(w)=k$. 
	
	Suppose that $p$ is consistently oriented. Then $w\in W(B_{n-1})$ and, by definition of consistent orientation in type $D$, $p$ is consistently oriented with respect to the Coxeter element $\cox_B=[1\ldots n-1]$ of $W(B_{n-1})$. By \cref{prop:typeB_consistent_orientation-NC} consistent orientation of a single cycle implies that $w\in\nc(B_n)$, which means by definition of the absolute order that $\ell_B(p\inv\cox_B)= n-1-k$. We have to prove that $w\in\nc(D_n)$, meaning that $w\leq \cox_D=[1\ldots n-1][n]=\cox_B[n]$. Hence we have to show that $\ell_D(w\inv \cox_D)=n-k$. 
	Let $v=w\inv\cox_B$, which is an element of $\nc(B_{n-1})$ by \cref{obs:abs_ord}. Then the disjoint cycle decomposition of $v$ involves exactly one balanced cycle, since by \cref{rem:typeB_one_balanced_cycle} it has at most one balanced cycle. If there were no balanced cycle, then the product of paired cycles $wv$ changes an even number of signs, and the balanced cycle $\cox_B$, which equals the product $wv$, changes an odd number of signs as self-bijection of $\Set{\pm 1, \ldots, \pm (n-1) }$, which cannot happen. Let $v=v_1v_2$ be such that $v_1$ is a product of paired cycles and $v_2$ is a balanced cycle. Because $w\inv\cox_D = v[n]$, we have
	\begin{align*}
	\ell_D(w\inv\cox_D)&=\ell_D(v_1v_2[n])=\ell_D(v_1)+\ell_D(v_2[n])
	=\ell_B(v_1)+\ell_B(v_2)+1\\
	&=\ell_B(v)+1=\ell_B(w\inv\cox_B)=n-1+k+1=n-k,
	\end{align*}
	where we use \cref{lem:typeB_length}, \cref{lem:typeD_length} and \cref{rem:typeBD_length_bal_cyc} to compute the lengths of balanced and paired cycles in $W(B_{n-1})$ and $W(D_n)$.
	
	For the other direction, let $w$ be an element of $\nc(D_n)$. We have to prove that $p$ is consistently oriented with respect to the Coxeter element $\cox_D$, which is by definition that it is, seen as element of $W(B_{n-1})$, consistently oriented with respect to $\cox_B$. By \cref{prop:typeB_consistent_orientation-NC}, $p$ being consistently oriented is equivalent to $w\in \nc(B_{n-1})$. We show $w\in \nc(B_{n-1})$ by verifying that $\ell_B(w\inv\cox_B)=n-1-k$.
	Because $w\in \nc(D_n)$, it holds that $\ell_D(w\inv\cox_D)=n-k$ and that $w\inv\cox_D \in \nc(D_n)$ by \cref{obs:abs_ord}. Let $u=w\inv\cox_D$. Since the cycle $p$ does not involve $\pm n$, we know that $u(n)=-n$ and $u(-n)=n$. Hence the disjoint cycle decomposition of $u$ in $W(D_n)$ has $[n]$ as a cycle and therefore there is exactly one other balanced cycle in the cycle decomposition, which we call $u_2$, by \cref{lem:typeD_cycle_decomp_nc}. We write $u=u_1u_2[n]$ with $u_1$ being a product of disjoint paired cycles. By \cref{lem:typeB_length}, \cref{lem:typeD_length} and \cref{rem:typeBD_length_bal_cyc} we get
	\begin{align*}
	\ell_D(u)=\ell_D(u_1)+\ell_D(u_2[n]) = \ell_B(u_1) + \ell_B(u_2) + 1 =\ell_B(u_1u_2) + 1,
	\end{align*}
	which shows $\ell_B(u_1u_2)=n-k-1$. But $u_1u_2=u[n]=w\inv\cox_D[n]=w\inv\cox_B$ and hence $\ell_B(w\inv\cox_B)=n-k-1$.
	
	Now consider case b) and let $b$ be a balanced cycle not involving $\pm n$. Let $w$ be the element of $W(D_n)$ having the disjoint cycle decomposition $b[n]$ and let $\tilde{w}$ be the element of $W(B_{n-1})$ having the disjoint cycle decomposition $b$. Suppose that $b=[i_1 \ldots i_k]$, hence $\ell_D(w)=k+1$ and $\ell_B(\tilde{w})=k$ by \cref{lem:typeB_length} and \cref{lem:typeD_length}.
	
	Let $b$ be consistently oriented with respect to the Coxeter element $\cox_D=[1 \ldots n-1][n]$, which means by definition that $b$ is consistently oriented with respect to the Coxeter element $\cox_B=[1\ldots n-1]$ of $W(B_{n-1})$. Applying \cref{prop:typeB_consistent_orientation-NC} yields that $\tilde{w}$ is an element of $\nc(B_{n-1})$, that is $\ell_B(\cox_B) = \ell_B(\tilde{w}) + \ell_D(\tilde{w}\inv \cox_B)$ or, equivalently, that $\ell_D(\tilde{w}\inv \cox_B) = n-1-k$. 
	We have to prove that $w \in \nc(D_n)$. We do so by showing that $\ell_D(\cox_D) = \ell_D(w) + \ell_D(w\inv \cox_D)$, which is equivalent to $\ell_D(w\inv \cox_D)= n-k-1$.
	Note that $w\inv \cox_D=b\inv[n][1\ldots n-1][n]=b\inv [1\ldots n-1] = \tilde{w}\inv\cox_B$. Since $\tilde{w}\inv\cox_B$ is in $\nc(B_{n-1})$, its disjoint cycle decomposition either has no or exactly one balanced cycle by \cref{lem:typeB_cycle_decomp_nc}. Because the element $w\inv\cox_D$ in $W(D_n)$ has the same disjoint cycle decomposition, there is no balanced cycle by \cref{rem:typeD_even_number_bal_cyc}. So we have shown that $\ell_D(w\inv\cox_D)=\ell_B(\tilde{w}\inv \cox_B)=n-k-1$.
	
	To prove the other direction suppose that $w$ is an element of $\nc(D_n)$. We have to show that $b$ is consistently oriented in $W(D_n)$. As is case a), this means that $b$ is consistently oriented in $W(B_{n-1})$, which is equivalent to $\tilde{w}\in\nc(B_{n-1})$ by \cref{prop:typeB_consistent_orientation-NC}. We prove that $\tilde{w}\in \nc(B_{n-1})$ by showing that $\ell_B(\tilde{w}\inv\cox_B)=\ell_B(\cox_B)-\ell_B(\tilde{w})=n-1-k$. We use a similar argument as above. 
	Since $w\inv\cox_D$ is in $\nc(D_n)$, and $\tilde{w}\inv\cox_B$ and $w\inv\cox_D$ have the same disjoint cycle decomposition, this cycle decomposition only has paired cycles and hence $\ell_B(\tilde{w}\inv\cox_B)=\ell_D(w\inv\cox_D)$. 
	Again using the assumption $w\in\nc(D_n)$ yields $\ell_D(w\inv\cox_D)=\ell_D(\cox_D)-\ell_D(w)=n-k-1$.
	
	Let us consider the last case c). Let $q=\dka i_1\ldots i_k\,\; n\dkz$ be a paired cycle and let $w\in W(D_n)$ be the element whose disjoint cycle decomposition is $q$. The length of $w$ is $\ell_D(w)=k$.
	
	Suppose that $q=\dka i_1 \ldots i_k \,\;n\dkz$ is consistently oriented with respect to the Coxeter element $\cox_D =[1\ldots n-1][n]$ of $W(D_n)$. 
	We have to prove that $w \in \nc(D_n)$, which means $\ell_D(w\inv\cox_D)=n-k$. Let $p=\dka i_1 \ldots i_k \dkz$. Then $q=p\, \dka i_k \,\; n \dkz$ and  we can express $w\inv\cox_D$ as $q\inv\cox_D = \dka i_k\,\;n\dkz p\inv \cox_B[n]$.  Since the absolute length is invariant under conjugation, we hence get that $\ell_D(w\inv\cox_D)=\ell_D(p\inv \cox_B[n]\dka i_k\,\;n\dkz)$. Note that $[n]\dka i_k\,\;n\dkz=[n \,\; i_k]$. We show that $\ell_D(p\inv \cox_B[n\,\; i_k])=n-k$. 
	
	Suppose that $i_k < 0$. Then
	\[
	c_B[n \,\;i_k]=  \dka 1 \,\; 2 \ldots -\!i_k\,\; n \,\; i_k\!-1 \,\; i_k\!-2 \ldots -\!(n-1) \dkz\eqqcolon p_1,
	\]
	hence the disjoint cycle decomposition of $c_B[n \,\;i_k]$ consists of a single paired cycle, which we call $p_1$. Since $q=p\dka i_k\,\; n\dkz$ is consistently oriented in $W(D_n)$, $p$ arises from $p_1$ by removing entries. Hence \cref{cor:length_subcycle} implies that the length computes as $\ell_D(p\inv \cox_B[n\,\; i_k])=\ell_D(p_1)-\ell_D(p)=n-1-(k-1)=n-k$.
	
	Now suppose that $i_k>0$. We get that
	\[
	c_B[n \,\;i_k]=  \dka 1 \,\; 2 \ldots i_k\,\; -\!n \,\; -\!i_k\!-\!1 \,\; -\!i_k\!-\!2 \ldots -\!(n-1) \dkz \eqqcolon p_2
	\]
	and see that the cycle $\dka i_{l+1}\ldots i_k\,\;i_1\ldots i_l\dkz$, which is equivalent to $p$, arises from $p_2$ by deleting entries, because $q=p\dka i_k\,\; n\dkz$ is consistently oriented in $W(D_n)$. Hence $\ell_D(p\inv \cox_B[n\,\; i_k])=n-k$, again by \cref{cor:length_subcycle}.
	
	Now let $q=\dka i_1\ldots i_k\,\; n\dkz$ be contained in $\nc(D_n)$ and let $t=\dka -i_k \,\; n\dkz$.
	First we show that $qt = [i_1 \ldots i_k][n]$ is in $\nc(D_n)$.
	For this let $u=q\inv \cox$, which is in $\nc(D_n)$ by \cref{obs:abs_ord} and has length $\ell(u)=n-k$, since $q \in \nc(D_n)$. Let $z_1 \ldots z_m$ be the disjoint cycle decomposition of $u$. Then one of the cycles, say $z_1$, has a reduced decomposition starting with $t = \dka n \,\; -i_k \dkz$, since $q\inv(\cox(n))=q\inv(-n)=-i_k$. Therefore, also $u$ has a reduced decomposition starting with the reflection $t$ and the length of $tu$ equals $\ell(tu)=\ell(u)-1 = n-k-1$. Because $tu=tq\inv \cox = (qt)\inv \cox$, it follows that $\ell((qt)\inv \cox)=n-k-1=\ell(\cox)-\ell(qt)$ and hence $qt=[i_1\ldots i_k][n] \in \nc(D_n)$. 
	By the above case b) we know that $[i_1\ldots i_k]$ is consistently oriented with respect to the Coxeter element $[1\ldots n-1][n]$ in $W(D_n)$. This means that $[i_1 \ldots i_k]$ is consistently oriented in $W(B_n)$ and arises from a zero block of the corresponding pictorial representation by going around it clockwise. This implies that we either have $0 < i_1 <\ldots <i_l $ for some $l$ and $0>i_{l+1} > \ldots > i_k > -i_1$, or $0 > i_1 > \ldots > i_l$ for some $l$ and $0< i_{l+1} < \ldots < i_k < -i_1$. This shows that $q$ is consistently oriented in $\nc(D_n)$.
\end{proof}

When a paired cycle does not involve $\pm n$, then the correspondence of type $B$ and type $D$ is pleasant.

\begin{cor}\label{cor:typeBD_elements_ncd_as_ncb}
	Let $p$ be a paired cycle with entries in $\Set{\pm 1,\ldots, \pm (n-1)}$ and let $w\in W(D_n)$ and $v\in W(B_{n-1})$ be the elements with disjoint cycle decomposition $p$. Then $w$ is in $\nc(D_n)$ if and only if $v$ is in $\nc(B_{n-1})$.
\end{cor}

\subsection{Pictorial representations}\label{sec:typeD_pict}

We start with introducing a set-theoretic notion of $D_n$-partitions \cite{ath_rei}.

\begin{defi}
	A $D_n$-partition is defined to be a $B_n$-partition with the additional property that the zero block, if present, has at least four elements. More precisely, a partition $\pi=\Set{B_1, \ldots, B_k }$ of $\Set{\pm1, \ldots, \pm n }$ is called a \emph{$D_n$-partition} if $\pi$ satisfies 
	\begin{enumerate}
		\item $B\in \pi$ if and only if $-B \in \pi$, 
		\item there is at most one block $B\in\pi$ with $B=-B$, called \emph{zero block}, and
		\item if $B$ is a zero block, then $\#B\geq 4$.
	\end{enumerate}
	If a $D_n$-partition has the additional property that the zero block, if present, contains the pair $\Set{-n,n}$, then we call it \emph{pure} $D_n$-partition. 
\end{defi}

In the following, we are only interested in pure $D_n$-partitions.

Condition a) guarantees that non-zero blocks come in pairs $\pm B \coloneqq \Set{B, -B}$, called \emph{block pairs}. By $k\in \pm B$ we mean $k\in B$ or $k \in -B$.  If $n\in \pm B$, we use the convention that $n\in B$. 

Before we define what a non-crossing $D_n$-partition is, we introduce a new pictorial representation for the class of pure $D_n$-partitions. The main advantage of this pictorial representation compared to the one defined in \cite{ath_rei} is that every vertex only has one label. Moreover, now different partitions have different pictorial representations. Hence one can reconstruct the underlying $D_n$-partitions by looking at the graph only, once a fixed labeling of the vertices is chosen. 

\begin{defi}
	Let $\pi=\Set{ B_1, \ldots, B_k }$ be a pure $D_n$-partition.
	We define an embedded graph in the following way.
	Label the vertices of a regular $2(n-1)$-gon in the plane with the numbers 
	\[
	1, 2, \ldots, n-1, -1, -2, \ldots, -(n-1)
	\]
	clockwise in this order and add the midpoint as an additional vertex and label it with $n$. This circularly ordered set with midpoint is the vertex set of the graph, and the edges are obtained as follows.
	\begin{enumerate}
		\item If $B$ is a zero block, then add the $(\#B-2)$-gon on the vertices $B\setminus \Set{-n,n}$. If $\#B=4$, then the $2$-gon is a line passing through the midpoint.
		\item If $\pm B$ is a block pair with $n \notin \pm B$, then add \emph{two} $\#B$-gons, one on the vertices of $B$ and one on the vertices of $-B$.
		\item If $\pm B$ is a block pair with $n \in \pm B$, then add \emph{one} $\#B$-gon on the vertices of $B$. Recall that by convention we have $n \in B$.
	\end{enumerate}
	The resulting graph is called the \emph{pictorial representation} of $\pi$. 
\end{defi}

\begin{rem}
	\hfill
	\begin{enumerate}
		\item 
		If a block of a $D_n$-partition neither contains $-n$ nor $n$, then it can be naturally interpreted as a block of a $B_{n-1}$ partition. 
		\item 
		If a block pair contains $\pm n$, then only \emph{one} block is present in the pictorial representation. Nevertheless, we also refer to it as \enquote{block pair} and keep in mind that there is a second, but invisible, block.
	\end{enumerate}
\end{rem}

Equipped with the pictorial representation of pure $D_n$-partitions, we can now define crossing blocks. The definition is in such a way that two blocks of a pure $D_n$-partition are called crossing if there are edges of the blocks in the pictorial representation that cross. It might be helpful to consider \cref{fig:typeD_crossing_blocks} for the different types of crossing blocks in the set-theoretic definition. 

\begin{figure}
	\begin{center}
		\begin{tikzpicture}
		\begin{scope}[xshift = 0cm]
		\achteckmp
		\draw(p1)--(p2)(p5)--(p2)(p5)--(p6)(p1)--(p6)(p4)--(p7)(p3)--(p8);
		\node at (-1.3,0){a)};
		\end{scope}
		
		\begin{scope}[xshift = 4cm]
		\achteckmp
		\draw(p1)--(p3)(p4)--(p2)(p5)--(p7)(p8)--(p6);
		\node at (-1.3,0){b)};
		\end{scope}
		
		\begin{scope}[xshift = 8cm]
		\achteckmp
		\draw[Lightgray](p4)--(p5)(p8)--(p5)(p8)--(p1);
		\draw(p1)--(p4)(p4)--(p0)(p0)--(p1)(p2)--(p7)(p3)--(p6);
		
		\node at (-1.3,0){c)};
		\end{scope}
		\end{tikzpicture}
		\caption{a) A zero block crossing a non-zero block pair.\\
			b) Two crossing block pairs that do not contain $n$.\\
			c) A block pair containing $n$ crosses another block pair. The gray line indicates the zero block $Z$ mentioned in \cref{defi:typeD_ncp}.}
		\label{fig:typeD_crossing_blocks}
	\end{center}
\end{figure}
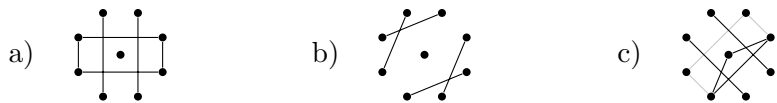

\begin{defi}\label{defi:typeD_ncp}
	A pure $D_n$-partition is called \emph{non-crossing} $D_n$-partition if it has no crossing blocks. Otherwise it is called \emph{crossing}. 
	The following list defines when two blocks are \emph{crossing}.
	\begin{enumerate}
		\item If $Z$ is a zero block and $B$ a non-zero block, then they are \emph{crossing} if $Z\setminus\Set{-n,n}$ and $B$ are crossing as $B_{n-1}$-partitions.		
		
		\item Two blocks $B,B'$ that neither contain $-n$ nor $n$ are \emph{crossing} if they cross as blocks of $B_{n-1}$-partitions.
		
		\item Let $\pm B$ and $\pm B'$ be two different block pairs of $\pi$ such that $n\in \pm B$ and set $Z\coloneqq (-B \cup B)\setminus \Set{-n,n}$. Then $\pm B$ and $\pm B'$ are \emph{crossing} if and only if the zero block $Z$ and $\pm B'$ are crossing.
	\end{enumerate}
	The set of non-crossing $D_n$-partitions is denoted by $\ncd_n$. 
\end{defi}

\begin{rem}
	We leave out the adjective \enquote{pure} in front of the non-crossing $D_n$-partition, since they are pure by definition.
\end{rem}

The notions of \enquote{crossing} in type $B$ and $D$ are not compatible, which is best seen in an example.

\begin{exa}\label{exa:crossing_and_not}
	Consider the set partitions $\pi_1$ and $\pi_2$ of the set $\Set{\pm1, \pm 2, \pm 3, \pm 4 }$ given by 
	\begin{align*}
	\pi_1&=\Set{\Set{1,-3},\Set{-1,3},\Set{2,-4},\Set{-2,4}} \text{ and }\\ \pi_2&=\Set{\Set{1,-2},\Set{-1,2},\Set{3,4},\Set{-4,-4}}.
	\end{align*}
	Then $\pi_1$ interpreted as $D_4$-partition is non-crossing, whereas $\pi_1$ seen as $B_4$-partition has crossing blocks. The other way around, the partition $\pi_2$ is crossing when considered as $D_4$-partition, and it is non-crossing when interpreted as $B_4$ partition. The pictorial representations of $\pi_1$ and $\pi_2$, both as $B_4$- and $D_4$-partitions, are shown in \cref{fig:typeBD_crossing_and_not}. 
\end{exa}

\begin{figure}
	\begin{center}
		\begin{tikzpicture}
		\begin{scope}
		\foreach \w in {1,...,6} 
		\node (p\w) at (-\w * 360/6 +60  : 8mm) [kpunkt] {};
		\node[kpunkt](p0) at (0,0){};
		\node[align = center] at (0,-2){$\pi_1$ as\\ $D_4$-partition};
		\draw(p1)--(p6)(p3)--(p4)(p0)--(p5);
		\end{scope}
		
		\begin{scope}[xshift = 3.5cm]
		\foreach \w in {1,...,8} 
		\node (p\w) at (-\w  * 360/8 +67.5   : 8mm) [kpunkt] {};
		\node[align = center] at (0,-2){$\pi_1$ as\\ $B_4$-partition};
		\draw(p1)--(p7)(p2)--(p8)(p3)--(p5)(p4)--(p6);
		\end{scope}
		
		\begin{scope}[xshift = 7cm]
		\foreach \w in {1,...,6} 
		\node (p\w) at (-\w * 360/6 +60  : 8mm) [kpunkt] {};
		\node[kpunkt](p0) at (0,0){};
		\node[align = center] at (0,-2){$\pi_2$ as\\ $D_4$-partition};
		\draw(p1)--(p5)(p2)--(p4)(p0)--(p3);
		\end{scope}
		
		\begin{scope}[xshift = 10.5cm]
		\foreach \w in {1,...,8} 
		\node (p\w) at (-\w  * 360/8 +67.5   : 8mm) [kpunkt] {};
		\node[align = center] at (0,-2){$\pi_2$ as\\ $B_4$-partition};
		\draw(p1)--(p6)(p2)--(p5)(p3)--(p4)(p7)--(p8);
		\end{scope}
		\end{tikzpicture}
		\caption{The respective pictorial representations of the two set partitions $\pi_1$ and $\pi_2$ from \cref{exa:crossing_and_not} as $D_4$- and $B_4$-partitions.}
		\label{fig:typeBD_crossing_and_not}
	\end{center}
\end{figure}
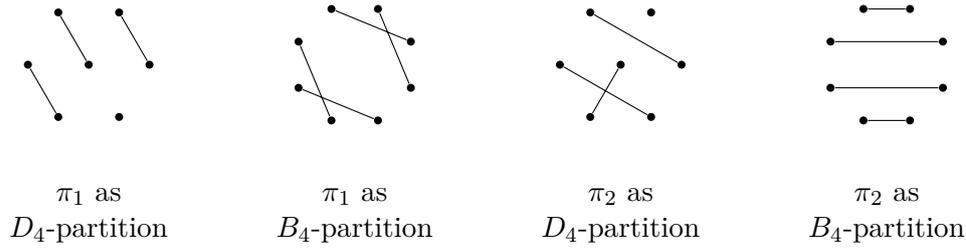 

\begin{nota}
	As in the other types before, we identify the pictorial representations of non-crossing $D_n$-partitions with the set-theoretic non-crossing $D_n$-partitions and use the term \emph{partition} for the pictorial representation as well if there is no danger of confusion. In particular, we refer to the set of pictorial representations of non-crossing $D_n$-partitions as $\ncd_n$. 
\end{nota}

\begin{rem}
	We shortly remark on the pictorial representation of Athanasiadis and Reiner, which we call \emph{AR-representation} here. 
	The AR-representation of a not necessarily pure $D_n$-partition is a graph on the same vertex set as our pictorial representation, where the midpoint gets a \enquote{flexible} labeling with both $-n$ and $n$. Depending on the $D_n$-partition, the labels of the midpoint are at different positions. The edges result from adding the convex hulls of \emph{all} blocks. See \cite[Sec. 3]{ath_rei}.
	One can easily pass from one representation to the other by either removing or adding blocks if necessary and adjusting the labeling. 
\end{rem}

Athanasiadis and Reiner \cite[Prop. 3.1]{ath_rei} show that $\ncd_n$, ordered by refinement, is a graded lattice. The rank function is given by $\rk(\pi)=n-\lfloor\frac{\#\pi}{2}\rfloor$ for $\pi\in\ncd_n$, where $\#\pi$ denotes the number of blocks of the set-theoretic non-crossing partition in $\ncd_n$. The Hasse diagram of the lattice $\ncd_4$ is shown in \cref{fig:typeD_hasse_diagram_D4}.

\begin{figure}%
	\begin{center}

		\begin{tikzpicture}
		\def\a{1.9cm} 
		\def\b{1.2cm} 
		
		\node (pg) at (2.5*\a,2.5*\b){\scalebox{1}{\begin{tikzpicture}\mpviereck \draw[ ] (p1)--(p2)--(p3)--(p4)--(p1);\end{tikzpicture}}};
		\foreach \x/\l in {0/\begin{tikzpicture}\mpviereck \draw[ ](p1)--(p0)--(p3) ;\end{tikzpicture}, 
			1/\begin{tikzpicture}\mpviereck \draw[ ](p0)--(p2)--(p3)--(p0) ;\end{tikzpicture},
			2/\begin{tikzpicture}\mpviereck \draw[ ] (p0)--(p1)--(p4)--(p0);\end{tikzpicture},
			3/\begin{tikzpicture}\mpviereck \draw[ ] (p0)--(p2)--(p1)--(p0);\end{tikzpicture},
			4/\begin{tikzpicture}\mpviereck \draw[ ] (p0)--(p4)--(p3)--(p0);\end{tikzpicture},
			5/\begin{tikzpicture}\mpviereck \draw[ ] (p2)--(p0)--(p4);\end{tikzpicture}}
		{
			\node (k\x) at (\x*\a,\b){\scalebox{1}{\l}};
			\draw (pg) -- (k\x.north);
		}

		\node (pk) at (2.5*\a,-2.5*\b){\scalebox{1}{\begin{tikzpicture}\mpviereck\end{tikzpicture}}};
		\foreach \x/\l in {0/\begin{tikzpicture}\mpviereck \draw[ ] (p1)--(p0);\end{tikzpicture},
			1/\begin{tikzpicture}\mpviereck \draw[ ] (p3)--(p0);\end{tikzpicture},
			2/\begin{tikzpicture}\mpviereck \draw[ ] (p1)--(p4)(p2)--(p3);\end{tikzpicture},
			3/\begin{tikzpicture}\mpviereck \draw[ ] (p1)--(p2)(p3)--(p4);\end{tikzpicture},
			4/\begin{tikzpicture}\mpviereck \draw[ ] (p2)--(p0);\end{tikzpicture},
			5/\begin{tikzpicture}\mpviereck \draw[ ] (p4)--(p0);\end{tikzpicture}}
		{
			\node (m\x) at (\x*\a,-\b){\scalebox{1}{\l}};
			\draw (pk) -- (m\x.south);
		}
		
		\foreach \o/\u in {0/0, 0/1, 1/1, 1/2, 1/4, 2/0, 2/2, 2/5, 3/0, 3/3, 3/4, 4/1, 4/3, 4/5, 5/4, 5/5}
		\draw (k\o.south) -- (m\u.north);
	\end{tikzpicture}

	\caption{The Hasse diagram of $\ncd_4$.}%
	\label{fig:typeD_hasse_diagram_D4}%
\end{center}
\end{figure}
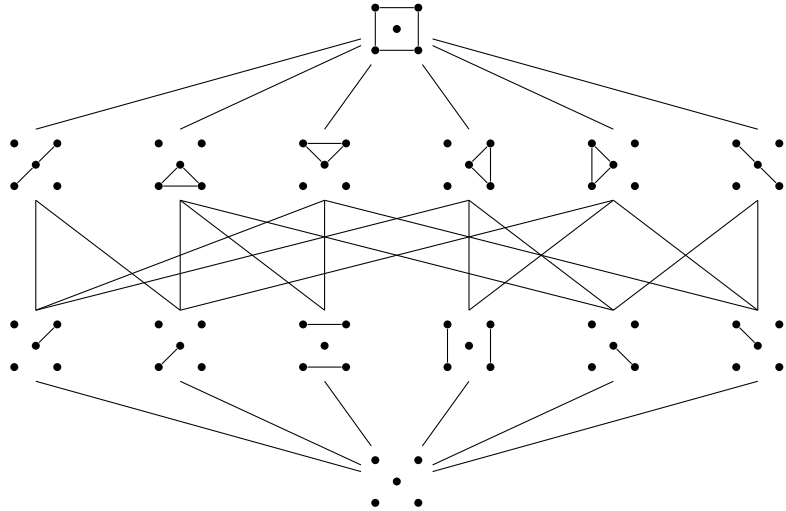

\subsection{Connecting the concepts}

In order to show that $\nc(D_n)$ and $\ncd_n$ are isomorphic lattices, we need to construct a $D_n$-partition from an element of $\nc(D_n)$ \cite[Thm. 1.1]{ath_rei}. 
\begin{defi}
Let $p=\dka i_1\ldots i_m \dkz$ be a paired cycle. We define $\Set{\!\Set{ p }\!}$ to consist the two blocks $\Set{i_1, \ldots, i_m},\Set{-i_1, \ldots, -i_m}$ corresponding to the elements of the two cycles of $p$.
If $b_1b_2=[j_1\ldots j_l][j_{l+1}]$ is a pair of balances cycles,
we define $\Set{\!\Set{ b_1b_2 }\!}$ to be the block $\Set{j_1, \ldots, j_{l+1}, -j_1, \ldots, -j_{l+1}}$ that consists of the union of elements of the two balanced cycles $b_1$ and $b_2$.
Let $w$ be in $\nc(D_n)$. Then the $D_n$-partition arising from $w$ is defined to be the $D_n$-partition $\Set{\!\Set{ w }\!}$ that has a non-trivial block $\Set{\!\Set{p}\!}$ for every paired cycle $p$ in the disjoint cycle decomposition of $w$ and a block $\Set{\!\Set{b_1b_2}\!}$ if the pair of balanced cycles $b_1b_2$ is present in the disjoint cycle decomposition. 
\end{defi}

Note that if $w\in \nc(D_n)$, then $\Set{\!\Set{ w }\!}$ is a pure $D_n$-partition, since by \cref{lem:typeD_cycle_decomp_nc} one of the balanced cycles has to be $[n]$.
The following is \cite[Thm. 1.1]{ath_rei}.

\begin{thm}\label{thm:typeD_nc_iso}
The map $\nc(D_n) \to \ncd_n$ with $w \mapsto \Set{\!\Set{ w }\!}$ is a lattice isomorphism.
\end{thm}

The main part of the proof is to understand the cover relations in both $\nc(D_n)$ and in $\ncd_n$ and show how they correspond. Inverse induction on the rank shows that the \enquote{breaking} of either cycles or blocks, which exactly correspond to cover relations, is compatible with applying the map $w \mapsto  \Set{\!\Set{ w }\!}$ and its inverse $\pi \mapsto w(\pi)$, see \cref{defi:typeD_inverse_of_iso}.

We now describe how the inverse map is obtained. This is a little vague in \cite{ath_rei}, but knowing the precise structure of the cycles is crucial for \cref{cha:typeD_auto_group}. Private communication with Reiner clarified the impreciseness of the inverse map.

The idea of the inverse map is similar to those of types $A$ and $B$. Choose a non-crossing $D_n$-partition. For every block pair choose one block in the pictorial representation, go around it clockwise and obtain a paired cycle. The procedure for the zero block is similar. You go around the zero block clockwise and obtain a consistently oriented balanced cycle, which you multiply with the balanced cycle $[n]$. Then you multiply all cycles obtained in this way to get an element of $W(D_n)$.
The extra multiplication of $[n]$ to the balanced cycle obtained from the zero block represents the fact that every zero block has to contain the elements $-n$ and $n$, since non-crossing partitions are pure by definition. But the midpoint is not visible when going around the zero block, so it has to be added afterwards. The precise definition is as follows.

\begin{defi}\label{defi:typeD_inverse_of_iso}
Let $\pi=\Set{\pm B_1, \ldots, \pm B_k}$ or $\pi=\Set{\pm B_1, \ldots, \pm B_k,Z}$ be a non-crossing $D_n$-partition, either with or without zero block $Z$. For every non-trivial block pair $\pm B_i$ let $p_i$ be the consistently oriented paired cycle on the elements of $B_i$. 
If $\pi$ has a zero block $Z$, let $b$ be the balanced cycle on the elements of $Z\setminus \Set{-n,n}$ with consistent orientation. Then the element of $\nc(D_n)$ induced by the non-crossing $D_n$-partition $\pi$ is $w(\pi)$, given by the disjoint cycle decomposition $p_1 \ldots p_k $, or, if a zero block is present, $p_1 \ldots p_k b [n]$.
\end{defi} 

\begin{rem}
In the definition we use the term \enquote{the} consistently oriented paired cycle on the elements of $B_i$, although there is not a unique paired cycle with these properties. Nevertheless, there is a unique \emph{equivalence class} that satisfies the conditions. Hence, we may choose any of its representatives. The analog is meant when saying \enquote{the} balanced cycle.
\end{rem}

\cleardoublepage

\chapter{Embedding the non-crossing partitions}\label{chap:embedding}

In this chapter we consider embeddings of the non-crossing partitions into lattices of linear subspaces of different vector spaces. If the vector space is the Euclidean space, then this embedding already appeared in \cite{bra_watt_kpi,bra_watt_par_ord}, although in a different guise, see also \cite{armstr}. If $W$ is a crystallographic Coxeter group, then we showed in joint work with Schwer that $\nc(W)$ embeds into a finite lattice of linear subspaces in Proposition 3.14 of \cite{hs}. This proposition is a prequel of \cref{thm:embedding_Vp_VZ} presented here. Since the published proof of \cite[Prop. 3.14]{hs} is sketchy, we replace it by a new one and obtain the stronger result of \cref{thm:embedding_Vp_VZ}. 

Since the order complex of a lattice of linear subspaces is a spherical building, we obtain an embedding of the non-crossing partition complex for all finite Coxeter groups into spherical buildings. For the classical non-crossing partitions, this embedding was observed in \cite{bra-mcc} and plays the central role in \cite{hks}. We investigate the simplicial version of the embedding of $|\nc(W)|$ and its consequences in \cref{sec:simpl_emb}. Moreover, we demonstrate concrete embeddings for the classical type in \cref{sec:emb_classical}. The interpretation of $|\nc(W)|$ as a chamber subcomplex of a spherical building is essential for the rest of this thesis.\medskip

Let $(W,S)$ be a finite Coxeter system. Recall that $V$ denotes the $\R$-vector space with basis $\set{\alpha_s\str s\in S}$  
and that the moved space of an element $w\in W$ is the subspace $\mov(w) =\Braket{\alpha_{t_1}, \ldots, \alpha_{t_k}} \subseteq V$ whenever $t_1\ldots t_k$ is a reduced decomposition of $w$ by \cref{lem:basis_of_moved_space}. By $\lam(V)$ we denote the lattice of linear subspaces of $V$ introduced in \cref{sec:lattice_vec_space}.
The following theorem is a combination of results of Brady and Watt \cite{bra_watt_kpi,bra_watt_par_ord}. See also \cite[Thm. 2.4.9]{armstr}.

\begin{thm}\label{thm:brady_watt_embedding}
	Let $W$ be a finite Coxeter group of rank $n$. Then the map 
	\begin{align*}
	\emb \colon \nc(W) \to \lam(V), \quad w \mapsto \mov(w)
	\end{align*}
	is a rank-preserving poset embedding.
\end{thm}

\begin{proof}
	The map $\emb$ is well-defined \cite[Sec. 2]{bra_watt_kpi}, rank-preserving \cite[Prop. 2.2]{bra_watt_kpi} and injective \cite[Thm. 1]{bra_watt_par_ord}.	
\end{proof}

\begin{lem}\label{lem:join_and_em_interchange}
	Let $w\in \nc(W)$ and let $t_1\ldots t_k$ be a reduced decomposition of $w$. Then the image of $w$ under the embedding $\emb \colon \nc(W) \to \lam(V)$ is $\emb(w)=\emb(t_1) \vee \ldots \vee \emb(t_k)$. 
\end{lem}

\begin{proof}
	By \cref{lem:basis_of_moved_space} a basis for $\emb(w)=\mov(w)$ is given by the set of positive roots $\Set{\alpha_{t_1}, \ldots, \alpha_{t_k}}$ corresponding to the reflections $t_i$ for $1\leq i \leq k$. In particular, a basis for $\emb(t_i)$ is $\Set{\alpha_{t_i}}$ and hence $\emb(w)=\emb(t_1)+\ldots+\emb(t_k)=\emb(t_1)\vee \ldots \vee \emb(t_k)$.
\end{proof}

This lemma shows that the values $\emb(t)=\mov(t)=\langle \alpha_t \rangle$ in $\lam(V)$ for reflections $t \in T$ completely determine the values of $\emb(w)$ for all $w\in \nc(W)$. In general, it holds that $\emb(t \vee t') \neq \emb(t)\vee \emb(t')$, as the next example demonstrates.

\begin{exa}
	Let $n\geq 4$ and $t,t' \in T \subseteq \nc(S_n)$ be given by $(1\,\;3)$ and $(2\,\;4)$, respectively. Then $t\vee t'=(1\,\;3) \vee (2\,\;4) = (1\,\;2\,\;3\,\;4)$ and $\dim\emb(t\vee t' )=4$. But $\emb(t) \vee \emb(t')$ has dimension $2$, hence $\emb(t \vee t') \neq \emb(t)\vee \emb(t')$.
\end{exa}

\section{The crystallographic case}\label{sec:cryst_emb}

In this section we show that if $W$ is a finite crystallographic Coxeter group, then $\nc(W)$ can be embedded into a lattice of linear subspaces of certain finite vector spaces.

Let $W$ be a finite crystallographic Coxeter group of rank $n$ with Coxeter generators $S =\Set{s_1, \ldots, s_n}$. Let $\Phi = \Set{a_t \str t\in T}$ be a crystallographic root system of $W$ with simple roots $\alpha_{s_1}, \ldots, \alpha_{s_n}$. We use the same symbols for the roots as before, although they do not need to have length one. By $V$ we denote the $\R$-vector space with basis $\Set{\alpha_s \str s\in S}$, as introduced in \cref{sec:geom_rep}. Recall that $\sigma\colon W \to \gl(V)$ is the faithful geometric representation of $W$. 
Since $\Phi$ is crystallographic, every root is an integer linear combination of simple roots. This means that the action of $W$ on $V$ stabilizes the free $\Z$-module $\VZ\coloneqq \bigoplus_{s\in S}\Z \,\alpha_s \subseteq V$ and hence induces a faithful action of $W$ on $\VZ$.

Let $p$ be a prime. By $\Vp$ we denote the vector space over $\F_p$ with formal basis $\Set{\pal_s \str s \in S }$. The natural map $\Z \to \F_p$, given by $z \mapsto z \pmod p$, naturally induces a projection map $\prho\colon\VZ \to \Vp$ such that $\alpha_s \mapsto \pal_s$ for $s\in S$. The image of a positive root $\alpha\in \VZ$ by is denoted by $\pal$ and called \emph{$p$-root}. Moreover, $\prho$ induces a map  from matrices with integer entries to matrices with entries in $\F_p$, which we denote by $\prho_M$.

\begin{defi}
	Let $w\in W$ and $p$ be a prime. The \emph{integral moved space} of $w$ is defined as $\mov_\Z(w) \coloneqq \mov(w)\cap \VZ \subseteq \VZ$, and the \emph{$p$-moved space} of $w$ is the subspace $\pmov(w) \coloneqq \prho(\mov_\Z(w)) \subseteq \Vp$.
\end{defi}

The notion of $p$-moved space is only interesting for certain primes. For instance, if $v_1,v_2\in \VZ$ are linearly independent, then also $v_1+v_2$ and $v_1-v_2$ are linearly independent. But their projections to $V^{(2)}$ are linearly dependent. We are interested in primes $p$ such that linear independence of vectors is preserved under the projection $\prho$.

\begin{defi}
	Let $\Phi$ be a crystallographic root system for the Coxeter system $(W,S)$. A prime $p$ is called \emph{compatible} with $\Phi$ if for every basis $B \subseteq \Phi^+ \subseteq \VZ$ of $V$ consisting of positive roots the projection $\prho(B) \subseteq \Vp$ is a basis of $\Vp$. 
\end{defi}

Note that in the definition of \enquote{compatible prime} we require that every basis $B \subseteq \Phi^+$ of the \emph{whole vector space} $V$, and not only the bases of $\VZ$, gets mapped to a basis of $\Vp$ under the projection. The reason in the following lemma, which otherwise would not be true in general.

\begin{lem}\label{lem:p-projection_injective}
	If $p$ is a prime compatible with the root system $\Phi$, then the induced projection map $\prho\colon \Phi^+ \to \Vp$ is injective.
\end{lem}

\begin{proof}
	The statement is clear if $W$ has rank $1$, since there is only one positive root. Hence we assume that $\rk(W) \geq 2$.
	
	We show that any pair of different positive roots $\alpha, \beta \in \Phi^+$ can be extended to a basis $B\subseteq \Phi^+$ of $V$. If $p$ is compatible with $\Phi$, then $\prho(B)$ is a basis of $\Vp$ by definition and in particular, $\pal \neq \beta\hp$.
	
	First assume that $\alpha$ is a simple root. If $\beta$ is a simple root as well, then $\Set{\alpha_s\str s\in S}$ is a basis of $V$ containing both $\alpha$ and $\beta$. If $\beta$ is not a simple root, then it can be written as an integer linear combination $\beta = \sum_{i=1}^{n}a_i \alpha_{s_i}$, where $a_j \neq 0$ for at least two values of $j$. Choose one of those indices $j$ such that $\alpha_{s_j}\neq \alpha$ and obtain a linear combination 
	\[
	\alpha_{s_j}=\sum_{i=1,\, i\neq j}^{n}\frac{a_i}{a_j}\alpha_{s_i}-\frac{1}{a_j}\beta.
	\]
	Thus the set $B=\set{\alpha_{s_i}\str 1\leq i \leq n, i \neq j} \cup \set{\beta}$ of positive roots is a basis of the vector space $V$.
	
	If $\alpha$ is not a positive root, then there exists some $w\in W$ such that $w\alpha$ is positive, since $W$ acts transitively on the set of roots. If $w\beta$ is not positive, we replace it by $-w\beta$ and the procedure from above yields a basis $B$ of $V$ containing the roots $w\alpha$ and $\pm w\beta$. Since $w$ acts on $V$ by linear transformations, the set $w\inv B$ is a basis of $V$. If $w\beta$ is positive, then $B$ is a basis containing $\alpha$ and $\beta$. Otherwise, replace $-\beta$ by $\beta$ in $B$ to obtain another basis of $V$, which then contains $\alpha$ and $\beta$.
\end{proof}

\begin{lem}\label{lem:comp_primes}
	Every crystallographic root system has infinitely many compatible primes.
\end{lem}

\begin{proof}
	Let $B \subseteq \Phi^+$ be a basis of $V$ consisting of positive roots. Let $B'$ be a matrix whose columns are the basis vectors in $B$. For a prime $p$, the projection $\prho(B)$ is a basis of $\Vp$ if and only if $\det(\prho_M(B')) \neq 0$, that is the set of projected basis vectors is linearly independent in $\Vp$. 
	Since $\Phi^+$ is a finite set, only finitely many bases $B$ of $V$ are contained in $\Phi^+$. 
	Consequently, there are only finitely many primes $p$ such that $\det(\prho_M(B')) = 0$. These are exactly those primes that are not compatible with $\Phi$.
\end{proof}

We can describe bases of $p$-moved spaces for compatible primes $p$ in analogy to \cref{lem:basis_of_moved_space}.

\begin{lem}\label{lem:basis_p_mov_space}
	Let $p$ be a prime compatible with the root system $\Phi$ and let  $t_1\ldots t_k$ be a reduced decomposition of ${w\in W}$. Then the set of $p$-roots $\Set{\pal_{t_1}, \ldots, \pal_{t_k}}$ is a basis of the $p$-moved space $\pmov(w)$. In particular, $\dim(\pmov(w)) = \ell(w)$.
\end{lem}

\begin{proof}
	Let $p$ be a compatible prime and $t_1\ldots t_k$  a reduced decomposition of $w\in W$. By the \cref{prefix} we can extend  the reduced decomposition of $w$ to a reduced decomposition $t_1\ldots t_k \ldots t_n$ of the Coxeter element $\cox$. By \cref{lem:basis_of_moved_space} the set of positive roots $A=\Set{\alpha_{t_1}, \ldots, \alpha_{t_k}}$ is a basis of $\mov(w)$. Moreover, it is contained in the basis $\Set{\alpha_{t_1}, \ldots, \alpha_{t_n}}=B\subseteq \Phi^+$ of $V$. Note that both bases are contained in $\VZ$ and hence they are bases of the integral moved spaces $\mov_\Z(w)$ and $\VZ$, respectively. By the assumption of $p$ being compatible with the root system, the basis $B$ is projected to the basis $\prho(B)$ of $\Vp$. Further we have that $\prho(A)\subseteq \prho(B)$ and hence the image $\prho(A) \subseteq \pmov(w)$ is a set of $k$ linearly independent elements of $\Vp$. 
	Since $\dim(\mov(w))=k$, every set of $k+1$ elements in $\mov(w)$ is linearly dependent. Consequently, every set of $k+1$ elements in $\pmov(w)$ is linearly dependent, which means that $\dim(\pmov(w))\leq k$. Hence $\prho(A)=\Set{\pal_{t_1}, \ldots, \pal_{t_k}}$ is a basis of $\pmov(w)$.
\end{proof}

In summary, we can associate to every element $w\in W$ its moved space, its integral moved space, and its $p$-moved space for a compatible prime $p$. If $t_1\ldots t_k$ is a reduced decomposition of $w$, these spaces are given by
\begin{align*}
\mov(w)&=\bigoplus_{i=1}^{k} \R\,\alpha_{t_i} \subseteq V,\\
\mov_\Z(w)&=\bigoplus_{i=1}^{k} \Z\,\alpha_{t_i} \subseteq \VZ \ \text{ and }\\
\pmov(w)&=\bigoplus_{i=1}^{k} \F_p\,\pal_{t_i} \subseteq \Vp.
\end{align*}
In particular, the $\R$-span of the integral moved space equals the moved space. The following is the main theorem of this chapter.

\begin{thm}\label{thm:embedding_Vp_VZ}
	Let $(W,S)$ be a crystallographic Coxeter system and $p$ a compatible prime. Then the map
	\begin{alignat*}{2}
	\pemb\colon &\nc(W) \to \lam(\Vp), \ \ &&w \mapsto \pmov(w)
	\end{alignat*}
	is a rank-preserving poset embedding.
\end{thm}

\begin{proof}
	Let $p$ be a prime compatible with the crystallographic root system corresponding to the Coxeter system $(W,S)$.
	Let $v,w\in \nc(W)$ be such that $v \leq w$. By the \cref{prefix} there is a reduced decomposition $t_1\ldots t_i \ldots t_k$ of $w$ such that $t_1 \ldots t_i$ is a reduced decomposition of $v$. The corresponding sets of positive roots, $\Set{\alpha_{t_1}, \ldots, \alpha_{t_i} }$ and $\Set{\alpha_{t_1}, \ldots, \alpha_{t_k} }$, are bases of the integral moved spaces of $v$ and $w$, where the former one is contained in the latter one, hence $\mov_\Z(v)\subseteq \mov_\Z(w)$. Similarly, we have that a basis for the $p$-moved space $\pmov(v)$ is given by $\Set{\pal_{t_1}, \ldots, \pal_{t_i} }$, which in turn is contained in $\Set{\pal_{t_1}, \ldots, \pal_{t_k} }$, a basis for the $p$-moved space $\pmov(w)$, by \cref{lem:basis_p_mov_space}. Hence $\pmov(v) \subseteq \pmov(w)$ and $\pemb$ is a order-preserving poset map. By the same lemma it follows that $\pemb$ is rank-preserving, since $\rk_\nc(w)=\ell(w)$ coincides with $\rk_\lam(\pemb(w))=\dim(\pemb(w))$.
	
	To show the injectivity of $\pemb$, suppose that $v,w\in \nc(W)$ are different elements. We show that their $p$-moved spaces are different as well.
	Since $\emb$ is injective by \cref{thm:brady_watt_embedding}, we know that the moved spaces $\mov(v)$ and $\mov(w)$ are different. Suppose without loss of generality that there exists a basis of $\mov(w)$ consisting of positive roots such that there is a basis element $\alpha$ that is not contained in $\mov(v)$. Otherwise, interchange the roles of $v$ and $w$. Then also $\alpha\notin\mov_\Z(v) \subseteq \VZ$, since $\alpha \in \VZ$ and $\mov_\Z(v)=\mov(v) \cap \VZ$. This implies that $\mov_\Z(v)\neq \mov_\Z(w)$. Further, the projection $\prho\colon \Phi^+ \to \Vp$ is injective by \cref{lem:p-projection_injective} and hence its restrictions to $\mov_\Z(v) \cap \Phi^+$ and $\set{\alpha} \cup (\mov_\Z(v) \cap \Phi^+)$ are injective as well. This means that $\alpha \notin \pmov(v)$, that is $\pmov(v)\neq \pmov(w)$. Hence the map $\pemb$ is injective.
\end{proof}

The analog of \cref{lem:join_and_em_interchange} holds for the embedding $\pemb$.

\begin{lem}\label{lem:pemb_join_interchange}
	Let $t_1\ldots t_k$ be a reduced decomposition of $w\in W$ and $p$ a compatible prime. Then the join in $\nc(W)$ commutes with the embedding $\pemb$, that is we have
	\[
	\pemb(w)=\pemb(t_1)\vee \ldots \vee \pemb(t_k).
	\]
\end{lem}

\begin{proof}
	By \cref{lem:basis_p_mov_space}, the set $\Set{\pal_{t_1}, \ldots, \pal_{t_k}}$ is a basis for $\pemb(p)=\pmov(w)$. Hence $\pemb(w)=\pemb(t_1)+ \ldots + \pemb(t_k)$ and since the join in $\lam$ is given by the sum of subspaces, we get that $\pemb(w)=\pemb(t_1)\vee \ldots \vee \pemb(t_k)$.
\end{proof}

\section{Simplicial version}\label{sec:simpl_emb}

In this section we consider the embeddings of the non-crossing partition lattices from a simplicial perspective by using their order complexes. Recall that to every lattice $L$, there is an associated simplicial complex, its order complex $|L|$, which we introduced in \cref{sec:order_complex}. The vertex set of $|L|$ consists of the elements of $L$ except for the minimum and the maximum, and chains in $L$ give rise to simplices in $|L|$. Recall that the order complex of the lattice of linear subspaces of some vector space is a spherical building, and that the order complex of $\nc(W)$ is a chamber complex. Thus the following theorem is the translation of the Theorems \ref{thm:brady_watt_embedding} and \ref{thm:embedding_Vp_VZ} into the simplicial context. 

\begin{thm}\label{thm:simplicial_embedding}
	Let $W$ be a finite Coxeter group of rank $n$. The order complex of the non-crossing partition lattice $\ocncw$ is isomorphic to a finite chamber subcomplex of the infinite spherical building $|\lam(V)|$ of type $A_{n-1}$. Moreover, if $W$ is crystallographic and $p$ is a compatible prime, then the complex $\ocncw$ is isomorphic to a chamber subcomplex of the finite spherical building $|\lam(\Vp)|$ of type $A_{n-1}$ as well.
\end{thm}

The interpretation of the complex of non-crossing partitions as a chamber subcomplex of a spherical building leads to a natural adaption of terms used for buildings.
Recall that $\ocnc$ is a chamber complex, hence its maximal simplices are called chambers. The set of chambers of the complex $\ocnc$ is denoted by $\C(\nc)$. 

\begin{defi}
	A chamber subcomplex $A \subseteq \ocnc$ is called \emph{apartment} if it is isomorphic to a Coxeter complex of type $A_{n-1}$.
\end{defi}

Since the order complex of a Boolean lattice of rank $n$ is a Coxeter complex of type $A_{n-1}$, every apartment in $\ocnc$ corresponds to a Boolean lattice in $\nc$ and vice versa. In \cref{cor:construction_boolean_sublatice} we already saw how we can construct some of the Boolean lattices in $\nc$ as spans of reflections of a reduced decomposition of the Coxeter element. Such a reduced decomposition also gives rise to a maximal chain in $\nc$ and hence a chamber in $\ocnc$. 

\begin{nota}
	Let $t_1\ldots t_n$ be a reduced decomposition of the Coxeter element. We denote the apartment corresponding to the Boolean sublattice $\Braket{t_1, \ldots, t_n}$ by $A(t_1\ldots t_n)$. The chamber with vertices $\Set{t_1, t_1t_2, \ldots,  t_1\ldots t_{n-1}}$ is denoted by $C(t_1\ldots t_n)$. \cref{tab:relations_nc_order_complex} summarizes the correspondences.
\end{nota}	

\renewcommand{\arraystretch}{1,5}
\begin{table}
	\begin{center}
		\begin{tabular}{|c||c|}
			\hhline{|-||-|}
			
			lattice $\nc = \nc(W)$ & simplicial complex $|\nc|$ \\
			\hhline{:=::=:}
			
			maximal chain & chamber \\[-6pt]
			
			$\id < t_1 < t_1t_2 < \ldots < t_1\ldots t_n$& $C(t_1\ldots t_n)$\\ 
			\hhline{|-||-|}
			
			Boolean sublattice & apartment \\[-6pt]
			$\Braket{t_1, \ldots, t_n}$& $A(t_1\ldots t_n)$\\
			\hhline{|-||-|}
		\end{tabular}
		\caption{To every reduced decomposition $t_1\ldots t_n$ of the Coxeter element  we associate a maximal chain and Boolean sublattice in $\nc$, as well as a chamber and an apartment in the complex of non-crossing partitions $\ocnc$.}
		\label{tab:relations_nc_order_complex}
	\end{center}
\end{table}

\begin{rem}\label{rem:chambers_bijection}
	The maximal chains of $\nc$ are in bijection with the reduced decompositions of the Coxeter element $\cox$ by \cref{lem:construction_chain}. Consequently, the chambers of $\ocnc$ are in bijection with reduced decompositions of the Coxeter element as well. 
\end{rem}

\begin{lem}\label{lem:red_decomp_chambers_apartm}
	If $\tau$ is a reduced decomposition of the Coxeter element $c$, then $C(\tau)$ is a chamber of the apartment $\A(\tau)$.
\end{lem}

\begin{proof}
	By \cref{rem:boolean_sublat_chain} the maximal chain in $\nc$ corresponding to $\tau$ is contained in the Boolean sublattice arising from the same reduced decomposition $\tau$. Passing to the order complexes preserves the containment and hence $C(\tau)\in \C(A(\tau))$.
\end{proof}

For type $A$, the following is Proposition 8.6 of \cite{bra-mcc}.

\begin{cor}\label{cor:union_of_aptms}
	The complex $\ocnc$ is a union of apartments.
\end{cor}

\begin{proof}
	Since $\ocnc$ is a chamber complex, it is the union of its chambers, which are in bijection with reduced decompositions of the Coxeter element $\cox$. By \cref{lem:red_decomp_chambers_apartm} every chamber is contained in the apartment corresponding to the same reduced decomposition, hence $\ocnc$ is contained in a union of apartments. But the apartments are subcomplexes of $\ocnc$, hence equality holds.
\end{proof}

Note that in general, not every apartment of $|\nc(W)|$ is given by a reduced decomposition of the Coxeter element. However, every apartment of $\on$ can be described by a reduced decomposition of the Coxeter element, see Lemmas \ref{lem:aptms} and \ref{lem:good_labeling}.

Investigations of the apartments in type $B_3$ suggest that the apartments correspond to reduced decompositions of maximal elements. This can be seen for instance for the element $[1][2][3]\in W(B_3)$.

\begin{oq}
	Is there a group-theoretic description of \emph{all} apartments in the non-crossing partition complex $|\nc(W)|$? For instance, does every reduced decomposition of a \emph{maximal element} of the absolute order on $W$ give rise to an apartment in $|\nc(W)|$?
\end{oq}

\subsection{The Hurwitz graph}

As an application of the simplicial version of the embedding we get a lower bound on the radius of the Hurwitz graph for all finite Coxeter groups. Adin and Roichman show that the radius of the Hurwitz graph for $S_n$ equals $\binom{n-1}{2}$ \cite[Thm. 10.3]{ar} and ask for the radius and diameter for a finite Coxeter group. In joint work with Schwer, we contribute the following \cite[Thm. 5.4]{hs}.

\begin{thm}\label{thm:hurwitz}
	The radius of the Hurwitz graph $H(W)$ of a finite Coxeter group $W$ of rank $n$ is bounded from below by $\binom{n}{2}$.
\end{thm}

Since the symmetric group $S_n$ is a Coxeter group of rank $n-1$, the lower bound established in the above theorem in fact equals the radius of the Hurwitz graph $H(S_n)$.

Generalizing Definition 2.2 of \cite{ar}, which defines $H(S_n)$, to finite Coxeter groups yields the following. 

\begin{defi}
	For a finite Coxeter group $W$, its Hurwitz graph $H(W)$ is the graph whose vertices are the maximal chains of $\nc(W)$, and two vertices are connected by an edge if and only if they differ in exactly one element.	
\end{defi}

Recall that maximal chains in $\nc(W)$ are in bijection with chambers of the complex $|\nc(W)|$ of non-crossing partitions. The graph with vertex set $\C(\nc)$ and edges between two chambers if and only if they are adjacent is called \emph{chamber graph} of $\nc(W)$. Note that it is isomorphic to the Hurwitz graph $H(W)$. 
In the sequel, we identify $H(W)$ with the chamber graph of $\nc(W)$.

For a finite graph $G$ with vertex set $V$, the \emph{eccentricity} of a vertex $v\in V$ is the maximal distance of $v$ to any other vertex in $G$, that is $\ecc(v) \coloneqq \max\set{\di(v,w)\str w\in V}$, where $\di$ denotes the combinatorial distance in $G$. The \emph{radius} of $G$ is the minimal of all possible eccentricities, that is $\rad(G) \coloneqq \min\set{\ecc(v)\str v\in V}$, and the \emph{diameter} of $G$ is the maximal one, that is $\diam(G) \coloneqq \max\set{\ecc(v)\str v\in V}$. A lower bound on the eccentricity hence gives a lower bound on the radius and trivially on the diameter as well.

\begin{proof}[Proof of \cref{thm:hurwitz}]
	Let $W$ be a finite Coxeter group of rank $n$. We show that the eccentricity of each vertex in the Hurwitz graph is bounded from below by $\binom{n}{2}$. Translated to the simplicial context this means that for every chamber $C$ of $\nc=\nc(W)$ there exists a chamber $C' \in \C(\nc)$ such that $\di(C,C')\geq \binom{n}{2}$. 
	
	By \cref{lem:red_decomp_chambers_apartm}, every chamber $C$ is contained in an apartment, which is isomorphic to the Coxeter complex $\Sigma$ of type $A_{n-1}$. In $\Sigma$, every chamber has a unique chamber of maximal distance $\binom{n}{2}$, which we discussed in \cref{sec:cox_cplx}. Hence the eccentricity of each vertex in $H(W)$ has to be at least $\binom{n}{2}$.
\end{proof}

\begin{rem}
	The proof of \cite{ar} that $\rad(H(S_n))=\binom{n-1}{2}$ can be simplified using the simplicial perspective combined with the fact that $\ncp_n$ is supersolvable \cite[Thm. 4.3.2]{her}. The translation of the supersolvability of $\ncp_n$ into the simplicial perspective is that there exists a chamber $C \in \C(\ncp_n)$ such that the union of all apartments containing the chamber $C$ is already the whole complex $|\ncp_n|$. With this interpretation it is clear that $\rad(H(S_n))=\binom{n-1}{2}$. The simplicial interpretation of supersolvability in type $A$ plays an important role in \cref{part3} of this thesis. We characterize the chambers $C$ that can be used to construct $|\ncp_n|$ in \cref{prop:univ_chamber_union}.
\end{rem}

\section{Embedding the classical types}\label{sec:emb_classical}

In this section we explicitly construct embeddings of the non-crossing partitions of the classical types into finite spherical buildings. Unlike before we are now interested in concrete versions of roots with respect to the standard basis of $\R^n$. We start with crystallographic root system as described in \cite{bour} and then project the coefficients to finite fields. In this section, the standard basis vectors of $\R^n$ are denoted by $\eps_i$ and those of $\F_p^n$ by $e_i$ for all $n \in \N$ and $1\leq i \leq n$.  In joint work with Schwer, we considered concrete embeddings of $\nc(A_n)$ and $\nc(B_n)$ in Chapter 4 of \cite{hs}. For the pictorial version of type $A$, embeddings of $\ncpn$ already existed before, see \cref{rem:old_embedding_ncpn} below for the details. Moreover, for each of the classical types we classify the compatible primes.

\subsection{Type $A$}\label{sec:typeA_embedding}

Let $W$ be the Coxeter group of type $A_n$, that is the symmetric group $S_{n+1}$, with the set of adjacent transpositions as Coxeter generators as described in \cref{sec:typeA_descr}. As before, let $\Phi \subseteq V = \bigoplus_{i=1}^{n} \R \,\alpha_{s_i}$ be a crystallographic root system of $W$. By \cite[Chap. VI.4.7]{bour}, the assignment $\alpha_{s_i}\mapsto \eps_i - \eps_{i+1} \in \R^{n+1}$ for all $1\leq i \leq n$ induces an angle-preserving isomorphism $V \to E \subseteq \R^{n+1}$, where $\R^{n+1}$ is equipped with the standard Euclidean structure and 
\[  
E = \Set{(x_1, \ldots, x_n, x_{n+1})\in \R^{n+1}\str x_{n+1}=-\sum_{i=1}^{n}x_i}
\]
is the $n$-dimensional subspace of $\R^{n+1}$ whose elements have coordinates that sum up to zero.
It is more convenient to work in $\R^n$ instead of $E$ inside $\R^{n+1}$. We identify $E$ with $\R^n$ via the map $E\to \R^n$ that sends $(x_1, \ldots, x_{n+1}) \mapsto (x_1, \ldots, x_n)$. Note that this identification is \emph{not} angle-preserving. But this is not an obstacle, because the embedding of the non-crossing partitions into a lattice of linear subspaces only sees the vector space structure, whereas the geometry on it is not relevant for this procedure. We denote the identification of $V$ with $\R^n$ by 
\[
\iota\colon V \to \R^n, 
\quad \alpha_{s_i} \mapsto 
\begin{cases}
\eps_i-\eps_{i+1}& \text{ if }i<n,\\
\eps_n& \text{ if }i=n
\end{cases}
\]
and the image $\iota(\alpha)$ of a root $\alpha \in \Phi$ by $\tal$. Although $\tphi\coloneqq \iota(\Phi)$ is not a root system, as it does not respect the angles, we call its elements \emph{roots} nevertheless. Note that the elements of $\tphi$ have integer coordinates, hence $\tphi \subseteq \Z^n$. The set of \emph{positive roots} is given by
\[
\tphi^+ \coloneqq \iota(\Phi^+) =\Set{\eps_i-\eps_j\str 1\leq i<j\leq n} \cup \Set{\eps_i \str 1\leq i \leq n}.
\]
Note that the $\Z$-span of $\tphi$ is $\Z^n$, hence the analog of $\VZ$ in the present setting is $\Z^n$.
We denote the canonical projection from $\Z^n$ to $\F_p^n$ by $\prho$ for every prime $p$.
As before, we say that a prime $p$ is \emph{compatible} with $\tphi$ if for every basis $B \subseteq \tphi^+ \subseteq \Z^n$ of $\R^n$ the image $\prho(B)$ under the projection is a basis of $\F_p^n$.

\begin{lem}\label{lem:typeA_comp_primes}
	Every prime is compatible with $\tphi$.
\end{lem}

\begin{proof}
	Let $B=\set{b_1, \ldots, b_n}\subseteq \tphi^+ \subseteq \Z^n$ be a basis of $\R^n$ and let $B'$ the matrix with columns $b_1, \ldots, b_n$. We show that $\det(B')=\pm 1$. Then $\det(\prho_M(B'))=\pm 1$ for all primes $p$ and consequently, $\prho(B)$ is a basis of $\F_p^n$. Thus $p$ is compatible with $\tphi$.
	
	There is an element $b\in B$ of the form $b=\eps_i$ for some $1\leq i \leq n$. Otherwise, the coordinates of all elements in the span of $B$ would sum up to zero, which contradicts that $B$ is a basis of $\R^n$.  
	Without loss of generality suppose that $b_n=\eps_i$. Then the determinant of $B'$ equals, up to sign, the determinant of the matrix $B''$ that arises from $B'$ by removing the $n^{\text{th}}$ column and $\ith$ row by Laplace's formula. The vectors of the matrix $B''$ form a basis of $\R^{n-1}$, which allows us to apply the argument from above to $B''$. Repeatedly applying the same procedure yields after a total of $n-1$ steps that $\det(B')=\pm\det(1)=\pm 1$.
\end{proof}

The notions of the different moved spaces translate to the setting in $\R^n$ as well. We denote $\tmov(w)\coloneqq \iota(\mov(w))$. If $t_1\ldots t_k$ is a reduced decomposition of $w\in W$, then 
\begin{align*}
\tmov(w)&=\bigoplus_{i=1}^{k} \R\,\tal_{t_i} \subseteq V.
\end{align*}
Analogously, we get for the integral moved space $\tmov_\Z(w)$ and the $p$-moved space $\tpmov(w)$ for a prime $p$ that
\begin{align*}
\tmov_\Z(w)=\bigoplus_{i=1}^{k} \Z\,\tal_{t_i} \subseteq \VZ \quad \text{ and } \quad
\tpmov(w)=\bigoplus_{i=1}^{k} \F_p\,\tpal_{t_i} \subseteq \Vp,
\end{align*}
where $\tpal \coloneqq \prho(\tal)$ for a positive root $\tal \in \tphi$.
Translating \cref{thm:embedding_Vp_VZ} to the setting of type $A_n$ with $\tphi \subseteq \R^n$ yields, together with \cref{lem:typeA_comp_primes}, the following.

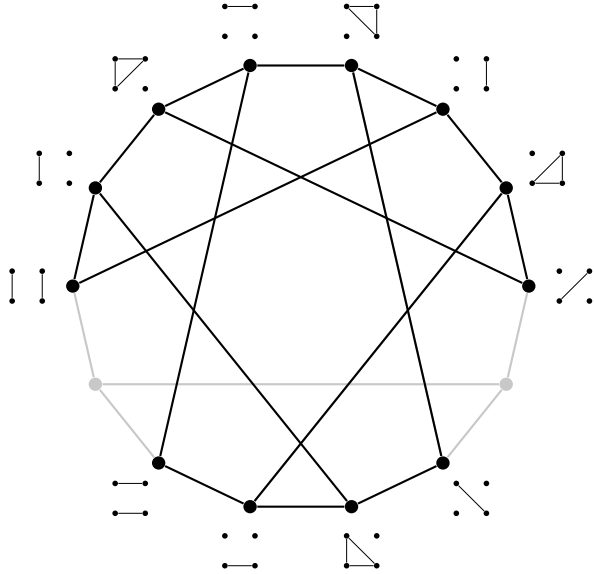
\begin{figure}
	\begin{center}
		\begin{tikzpicture}
		\def\r{30mm}
		\foreach \w in {1,6}
		\node (e\w) at (-\w * 360/14 : \r) [mpunkt,Lightgray] {};
		
		\foreach \w in {2,3,4,5,7,8,9,10,11,12,13,14}
		\node (e\w) at (-\w * 360/14 : \r) [mpunkt] {};
		
		\foreach \w/\s in {
			2/\pzv,
			3/\pzdv,
			4/\pzd,
			5/\pevuzd,
			7/\pezudv,
			8/\pdv,
			9/\pedv,
			10/\pev,
			11/\pezv,
			12/\pez,
			13/\pezd,
			14/\ped}
		\node (f\w) at (-\w * 360/14  : \r + \r/5)  {\scalebox{0.7}{\s}};

		\foreach \i [evaluate=\i as \j using int(\i+1)] in {1,5,6}
		\draw[thick,Lightgray] (e\i) -- (e\j);
		\draw[thick,Lightgray] (e1) -- (e14);

		\foreach \i [evaluate=\i as \j using int(\i+5)] in {1}
		\draw[thick,Lightgray] (e\i) -- (e\j) ;
		\foreach \i [evaluate=\i as \j using int(\i+9)] in {2,4}
		\draw[thick,Lightgray] (e\i) -- (e\j);
		
		\foreach \i [evaluate=\i as \j using int(\i+1)] in {2,3,4,7,8,9,10,11,12,13}
		\draw[thick] (e\i) -- (e\j);

		\foreach \i [evaluate=\i as \j using int(\i+5)] in {3,5,7,9}
		\draw[thick] (e\i) -- (e\j) ;
		\foreach \i [evaluate=\i as \j using int(\i+9)] in {2,4}
		\draw[thick] (e\i) -- (e\j);
		
		\end{tikzpicture}
		\caption{The complex $|\ncp_4|$ inside the spherical building $|\lam(\F_2^3)|$.}%
		\label{fig:ncp4_in_building}%
	\end{center}
\end{figure}

\begin{prop}\label{prop:typeA_emb}
	For every prime $p$ and every positive integer $n$, the map
	\[
	\pemb\colon \nc(A_n) \to \lam(\F_p^n), \quad w\mapsto \tpmov(w)\subseteq \F_p^n
	\]
	is a rank-preserving poset embedding.
\end{prop}

Pre-composition with the isomorphism  $\ncp_n \to \nc(S_n)$ from \cref{thm:typeA_nc_iso} yields an embedding of the classical non-crossing partitions into finite spherical buildings. \cref{fig:ncp4_in_building} shows the order complex of $\ncp_4$ as a subcomplex of the spherical building $|\lam(\F_2^3)|$. The labeling is chosen in such a way that it coincides with the labeling with linear subspaces of the spherical building in \cref{fig:A2-F_2^3}. \cref{part3} further investigates the simplicial structure of $\on$ using the embedding $\on \to |\lam(\F_2^{n-1})|$.

\begin{rem}\label{rem:old_embedding_ncpn}
	Brady and McCammond remarked that $\ncp_n$ embeds into $\lam(\F_2^n)$ \cite[Rem. 8.5]{bra-mcc}. This was generalized to an embedding into $\lam(V)$ for an $(n-1)$-dimensional subspace $V \subseteq \F^{n}$ for an arbitrary field $\F$ in \cite[Le. 2.24]{hks}. If $\F$ is finite, then, after an appropriate identification of $V$ with $\F_p^{n-1}$, their embeddings coincide with our embeddings $\ncp_n \overset{\cong}{\longrightarrow} \nc(A_{n-1})\hookrightarrow \lam(\F_p^{n-1})$. We consider the embedding $\ncp_n \hookrightarrow \lam(\vs)$ in more detail in \cref{sec:emb_typeA} and give a purely pictorial description of it.
\end{rem}

\subsection{Type $B$}

Let $W$ be the Coxeter group of type $B_n$ with the set of Coxeter generators $S$ as described in \cref{sec:typeB_descr}. Let $\Phi \subseteq \bigoplus_{i=1}^{n}\R\,\alpha_{s_i}$ be a crystallographic root system of $W$ as before. The assignment
\[
\alpha_{s_i} \mapsto 
\begin{cases}
\eps_i - \eps_{i+1}& \text{ if }i<n,\\
\eps_n & \text{ if } i=n
\end{cases}
\]
induces an angle-preserving isomorphism $\iota\colon V \to \R^n$ with $\R^n$ being equipped with the standard Euclidean structure. Then the image of the root system $\Phi \subseteq V$ is 
\[
\tphi \coloneqq \iota(\Phi) = \Set{\pm \eps_i \pm \eps_j \str 1 \leq i<j \leq n} \cup \Set{\pm \eps_i \str 1\leq i \leq n},
\]
which is a crystallographic root system of $W$ inside $\R^n$ by \cite[Chap. VI.4.5]{bour}. 
Note that the $\Z$-span of $\tphi$ is $\Z^n$ and that the set of positive roots of $\tphi$ is given by 
\[
\tphi^+ \coloneqq \iota(\Phi^+)= \Set{\eps_i \pm \eps_j \str 1 \leq i<j \leq n} \cup \Set{\eps_i \str 1\leq i \leq n}.
\]
We continue to denote the canonical projection $\Z^n \to \F_p^n$ by $\prho$ for every prime $p$. It is immediate that the prime $2$ is \emph{not} compatible with $\tphi$, because the roots $\eps_1 + \eps_2$ and $\eps_1 - \eps_2$, which can be contained in a common basis $B \subseteq \tphi$ of $\R^n$, get mapped to the same element by $\rho^{(2)}$.

\begin{lem}\label{lem:typeB_comp_prime}
	Every prime $p \neq 2$ is compatible with $\tphi$.
\end{lem} 

\begin{proof}
	Let $B \subseteq \tphi^+$ be a basis of $\R^n$ and $B'$ a matrix whose columns are the basis vectors of $B$ in an arbitrary order. We show that if $\det(B')\neq \pm 1$, then the only prime factor of $\det(B')$ is $2$. Then $\det(\prho_M(B'))\neq 0$ for all primes $p \neq 2$ and hence $\prho(B) \subseteq \F_p^n$ is a basis, which means that $p$ is compatible with $\tphi$.
	
	Note that every positive root in $\tphi$ has exactly one or two non-zero entries. We
	arrange the columns of $B'$ in such a way that all diagonal entries are non-zero.
	Hence $B$ is almost in row echelon form, since all diagonal entries are non-zero and there is at most one other non-zero entry in each column. Using the Gaussian elimination, at most one elementary row operation has to be performed for every column. Since all non-zero entries of $B'$ are $1$ or $-1$, the only entries that can appear on the diagonal of the row echelon form of $B'$ are $\pm 1$ and $\pm 2$. Hence $\det(B')$ is a product of elements in $\set{\pm 1, \pm 2}$.
\end{proof}

We use the same notation for the images of moved spaces under the identification map $\iota$ as in the previous section. The translation of \cref{thm:embedding_Vp_VZ} to type $B$, together with the previous lemma, is the following.

\begin{prop}
	Let $p$ be a prime with $p \neq 2$. For every positive integer $n$, the map
	\[
	\pemb\colon \nc(B_n) \to \lam(\F_p^n), \quad w\mapsto \tpmov(w)\subseteq \F_p^n
	\]
	is a rank-preserving poset embedding.
\end{prop}

Composition with the isomorphism $\ncb_n \to \nc(B_n)$ from \cref{thm:typeB_nc_iso} gives an embedding of the pictorial representations of the non-crossing partitions of type $B$ into a finite spherical building. The embedding of $|\ncb_3|$ into the spherical building associated to $\F_3^3$ is shown in \cref{fig:ncb_3-embedding}.

\begin{figure}
	\begin{center}
		\begin{tikzpicture}
		\def\r{30mm}
		\foreach \w in {1,2,...,26}
		\node (e\w) at (-\w * 360/26 : \r) [mpunkt] {};
		\foreach \m in {4,6,7,9,13,14,25,26}
		\node at (e\m) [mpunkt, Lightgray] {};
		
		\foreach  \i [evaluate=\i as \j using int(\i+1)] in
		{4,6,7,9,13,14,25,3,5,12,24,8}
		\draw[thick,Lightgray] (e\i) -- (e\j);
		\draw[thick,Lightgray] (e1) -- (e26);
		\foreach \i [evaluate=\i as \j using int(\i+5)] in {1,7,9,13,21}
		\draw[thick,Lightgray] (e\i) -- (e\j) ;
		\draw[thick,Lightgray] (e25) -- (e4);   
		\foreach \i [evaluate=\i as \j using int(\i+9)] in {4,6,14,16}
		\draw[thick,Lightgray] (e\i) -- (e\j);
		\foreach \i [evaluate=\i as \j using int(\i+17)] in {7,9}
		\draw[thick,Lightgray] (e\i) -- (e\j);

		\foreach \i [evaluate=\i as \j using int(\i+1)] in
		{1,2,10,11,15,16,17,18,19,20,21,22,23}
		\draw[thick] (e\i) -- (e\j);
		\foreach \i [evaluate=\i as \j using int(\i+5)] in {3,5,11,15,17,19}
		\draw[thick] (e\i) -- (e\j) ;
		\draw[thick] (e23) -- (e2);   
		\foreach \i [evaluate=\i as \j using int(\i+9)] in {2,8,10,12}
		\draw[thick] (e\i) -- (e\j);
		\foreach \i [evaluate=\i as \j using int(\i+17)] in {1,3,5}
		,        \draw[thick] (e\i) -- (e\j);   
		
		\foreach \w/\d in {    1/\pBeuzd,
			2/\pBzd,
			3/\pBemzmd,
			5/\pBezud,
			8/\pBemz,
			10/\pBd,
			11/\pBZzd,
			12/\pBzmd,
			15/\pBemduz,
			16/\pBz,                       
			17/\pBZez,
			18/\pBe,
			19/\pBZed,
			20/\pBemd,
			21/\pBezmd,
			22/\pBez,
			23/\pBezd,
			24/\pBed}
		\node (e\w) at (-\w * 360/26 : \r + \r/5) {\scalebox{0.7}{\d}};
		\end{tikzpicture}
		\caption{The complex $|\ncb_3|$ inside the spherical building $|\lam(\F_3^3)|$.}%
		\label{fig:ncb_3-embedding}%
	\end{center}
\end{figure}

\subsection{Type $D$}

Let $W$ be the Coxeter group of type $D_n$ with Coxeter generating set $S$ introduced in \cref{sec:typeD_descr}. Let $\Phi \subseteq V=\bigoplus_{i=1}^n \R\, \alpha_{s_i}$ be a crystallographic root system for it. The assignment 
\[
\alpha_{s_i} \mapsto 
\begin{cases}
\eps_i - \eps_{i+1}& \text{ if }i<n,\\
\eps_{n-1} + \eps_n & \text{ if } i=n
\end{cases}
\]
induces an angle-preserving isomorphism $\iota\colon V \to \R^n$, where $\R^n$ is equipped with the standard Euclidean structure.
The image of the root system $\Phi$ under $\iota$ is again a crystallographic root system $\tphi \subseteq \R^n$ of $W$, which is explicitly given by
\[
\tphi \coloneqq \iota(\phi) = \Set{\pm \eps_i \pm \eps_j \str 1\leq i<j \leq n}
\]
by \cite[Chap. VI.4.8]{bour}.
Note that the $\Z$-span $\braket{\tphi}$ of $\tphi$ in $\R^n$ is \emph{not} $\Z^n$, since for instance the vector $\eps_1$ is not contained in $\braket{\tphi}$. We denote the canonical projection $\braket{\tphi} \subseteq \Z^n \to \F_p^n$ by $\prho$ for every prime $p$. The set of positive roots is given by
\[
\tphi^+\coloneqq \iota(\Phi^+)= \Set{\eps_i \pm \eps_j \str 1\leq i<j \leq n}.
\]

Similar to type $B$, the prime $2$ is \emph{not} compatible with $\tphi$. But, in analogy to type $B$ as well, this is the only exception.

\begin{lem}
	Every prime $p \neq 2$ is compatible with $\tphi$.
\end{lem}

\begin{proof}
	We have to show that for every basis $B\subseteq \tphi^+$ of $\R^n$ the projection $\prho(B)$ is a basis of $\F_p^n$ for every prime $p \neq 2$. Since all roots have exactly two non-zero entries, we can use the same argumentation as in \cref{lem:typeB_comp_prime}.
\end{proof}

As before, we use the notation $\tpmov$ for the $p$-moved spaces that arises from a moved space $\iota(\mov)\subseteq \R^n$. Together with the previous lemma, the interpretation of \cref{thm:embedding_Vp_VZ} in type $D$ is the following.

\begin{prop}
	Let $p$ be a prime with $p \neq 2$. For every positive integer $n$, the map
	\[
	\pemb\colon \nc(D_n) \to \lam(\F_p^n), \quad w\mapsto \tpmov(w)\subseteq \F_p^n
	\]
	is a rank-preserving poset embedding.
\end{prop}

Composition of this embedding with the isomorphism of \cref{thm:typeD_nc_iso} yields an embedding of the pictorial representations of $D_n$-partitions into a finite spherical building. \cref{fig:ncd_3-embedding} shows the order complex of $\ncd_3$ embedded into the spherical building associated to $\F_2^3$. Note that since $W(A_3) \cong W(D_3)$, the order complex of $\ncd_3 \cong \ncp_4$ embeds into the spherical building associated to $\F_2^3$ as well.   

\begin{figure}%
	\begin{center}
		\begin{tikzpicture}
		\def\r{30mm}
		\foreach \w in {1,2,...,26}
		\node (e\w) at (-\w * 360/26 : \r) [mpunkt,Lightgray] {};

		\foreach \m in {14,...,25}
		\node at (e\m) [mpunkt] {};

		\foreach \i [evaluate=\i as \j using int(\i+1)] in
		{25,1,2,3,4,5,6,7,8,9,10,11,12,13}
		\draw[thick,Lightgray] (e\i) -- (e\j);
		\draw[thick,Lightgray] (e1) -- (e26);
		\foreach \i [evaluate=\i as \j using int(\i+5)] in {1,3,5,7,9,11,13,21}
		\draw[thick,Lightgray] (e\i) -- (e\j) ;
		\draw[thick,Lightgray] (e25) -- (e4);
		\draw[thick,Lightgray] (e23) -- (e2);   
		\foreach \i [evaluate=\i as \j using int(\i+9)] in {2,4,...,12}
		\draw[thick,Lightgray] (e\i) -- (e\j);
		\draw[thick,Lightgray] (e26) -- (e9);
		\foreach \i [evaluate=\i as \j using int(\i+17)] in {1,3,5,7,9}
		\draw[thick,Lightgray] (e\i) -- (e\j);
		
		\foreach  \i [evaluate=\i as \j using int(\i+1)] in {14,...,24}
		\draw[thick] (e\i) -- (e\j);
		\foreach \i [evaluate=\i as \j using int(\i+5)] in {15,17,19}
		\draw[thick] (e\i) -- (e\j) ;   
		\foreach \i [evaluate=\i as \j using int(\i+9)] in {14,16}
		\draw[thick] (e\i) -- (e\j);
		
		\foreach \w/\d in {    14/\begin{tikzpicture}
			\mpviereck \draw(p1)--(p4)(p3)--(p2);
			\end{tikzpicture},
			15/\begin{tikzpicture}
			\mpviereck \draw (p2)--(p3)--(p0)--(p2);
			\end{tikzpicture},
			16/\begin{tikzpicture}
			\mpviereck \draw (p3)--(p0);
			\end{tikzpicture},
			17/\begin{tikzpicture}
			\mpviereck \draw (p4)--(p3)--(p0)--(p4);
			\end{tikzpicture},
			18/\begin{tikzpicture}
			\mpviereck \draw (p4)--(p3) (p1)--(p2);
			\end{tikzpicture},
			19/\begin{tikzpicture}
			\mpviereck \draw (p2)--(p1)--(p0)--(p2);
			\end{tikzpicture},
			20/\begin{tikzpicture}
			\mpviereck \draw (p2)--(p0);
			\end{tikzpicture},
			21/\begin{tikzpicture}
			\mpviereck \draw (p2)--(p0)--(p4);
			\end{tikzpicture},
			22/\begin{tikzpicture}
			\mpviereck \draw (p0)--(p4);
			\end{tikzpicture},
			23/\begin{tikzpicture}
			\mpviereck \draw (p1)--(p4)--(p0)--(p1);
			\end{tikzpicture},
			24/\begin{tikzpicture}
			\mpviereck \draw (p0)--(p1);
			\end{tikzpicture},                       
			25/\begin{tikzpicture}
			\mpviereck \draw (p3)--(p0)--(p1);
			\end{tikzpicture}}
		\node (e\w) at (-\w * 360/26 : \r + \r/5) {\scalebox{0.7}{\d}};
	\end{tikzpicture}
	\caption{The complex $|\ncd_3|$ inside the spherical building $|\lam(\F_3^3)|$.}%
	\label{fig:ncd_3-embedding}%
\end{center}
\end{figure}
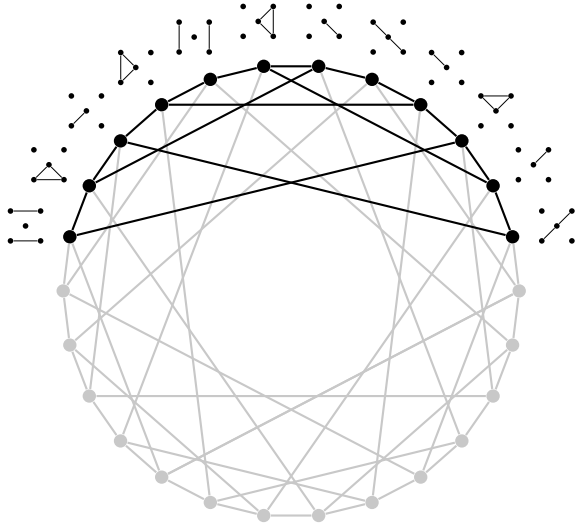

\cleardoublepage

\chapter{Automorphisms and anti-automorphisms}\label{chap:autos}

In this chapter we investigate the lattice automorphisms and anti-automorphisms of the non-crossing partition lattices. For type $A$, the automorphism group is known to be a dihedral group \cite{biane}. We compute the automorphism group of $\nc(B_n)$ for all $n$, and for $\nc(D_n)$ for $n\neq 4$ and show that they are dihedral groups. Moreover, we show that in the classical types, all automorphisms extend uniquely to the lattice of linear subspaces $\lam$ they are embedded in, where for type $D$ we have to further assume that $n\neq 4$. For the classical types, we show that a class of anti-automorphisms extends to the lattice $\lam$ as well. 
Further we show that for an arbitrary Coxeter group, there exists a dihedral group of automorphisms whose elements extend uniquely to lattice automorphisms of $\lam$. We use the poset-theoretic language in this chapter, but the results can be translated to the simplicial setting as well. For instance, the type-preserving automorphisms of the non-crossing partitions of classical type extend uniquely to type-preserving simplicial automorphisms of the spherical building.

For the rest of this chapter, let $W$ be a finite Coxeter group of rank $n$ and $V$ a $n$-dimensional vector space, which always can be chosen to be isomorphic to $\R^n$ or, if we are in the setting of a crystallographic Coxeter group, it can be chosen to be isomorphic to $\F_p^n$ for a compatible prime $p$ as well. We denote the respective embeddings of $\nc(W)$ from \cref{thm:brady_watt_embedding} and \cref{thm:embedding_Vp_VZ} by $\emb\colon \nc(W) \to \lam(V)$, independently of the chosen vector space $V$.

\section{Automorphisms}
The aim of this section is to describe lattice automorphisms of the non-crossing partitions. Recall that a lattice automorphism of a lattice $L$ is a order-preserving bijection $L \to L$ whose inverse is order-preserving as well. The group of automorphisms of the lattice $L$ is denoted by $\aut(L)$. We follow \cite{armstr} and consider a dihedral group of automorphisms. We show that these automorphisms extend uniquely to automorphisms of the lattice of linear subspaces in which the non-crossing partitions are embedded.

\subsection{Bipartitions of Coxeter elements}

In this section we introduce the notion of a bipartition of a Coxeter element, which proves useful for defining a group of lattice automorphisms of $\nc(W)$.
In the literature the terms \enquote{bipartite Coxeter element} \cite{armstr} and  \enquote{chromatic pair} \cite{bes} are used.

\begin{defi}
	A product $\l\r$ of elements $\l,\r\in W$ is called \emph{bipartition} of the Coxeter element $\cox$ if $\cox=\l\r$ and $\l^2=\r^2=\id$. The elements $\l$ and $\r$ are called \emph{left part} and \emph{right part} of the bipartition, respectively.
\end{defi}

\begin{lem}\label{lem:cox_element_has_bipartition}
	Every Coxeter element of a finite Coxeter group has a bipartition.
\end{lem}

\begin{proof}
	Let $W$ be a finite Coxeter group with fixed simple system $S=\Set{s_1, \ldots, s_n}$ and Coxeter element $\cox$. Then the Coxeter diagram is a tree and hence there exists a partition of the vertices $S$ into two subsets $S_1$ and $S_2$ of non-adjacent vertices. This means that all elements in $S_1$ commute, and the same is true for the elements in $S_2$. There are exactly two choices of sets $S_1$ and $S_2$. Let us fix one of them. Let $\l'$ be the product of the elements of $S_1$ and let $\r'$ be the product of elements of $S_2$. 
	Then $\cox'=\l'\r'$ is a product of all the simple reflections and hence a standard Coxeter element. Since any two Coxeter elements are conjugate, there is a $w\in W$ such that $w\cox'w\inv=\cox$. Choose such an element $w$ and set $\l\coloneqq w\l'w\inv$ and $\r\coloneqq w \r w\inv$. Then $c=\l\r$ and $\l^2=\r^2=\id$, since all respective factors of $\l$ and $\r$ commute.
\end{proof}

From now on, we fix a bipartition $\l\r$ of the Coxeter element $\cox\in W$ and call it the \emph{standard bipartition} of $\cox$. 

\begin{exa}
	Let us construct bipartitions of the Coxeter element $\cox=(1\,\;2\,\;3\,\;4\,\;5)$ in the symmetric group on five elements. We follow the recipe provided by the proof of \cref{lem:cox_element_has_bipartition}. Let $S=\Set{(1\,\;2),(2\,\;3),(3\,\;4),(4\,\;5)}$ be the fixed simple generating set. One bipartition of the vertices of the Coxeter diagram of $S_5$ consists of the two sets $S_1=\Set{(1\,\;2),(3\,\;4)}$ and $S_2=\Set{(2\,\;3),(4\,\;5)}$. Multiplication of the elements of $S_1$ and of $S_2$ gives elements $\l'=(1\,\;2)(3\,\;4)$ and $\r'=(2\,\;3)(4\,\;5)$, whose product is the standard Coxeter element $\cox'=(1\,\;2\,\;4\,\;5\,\;3)$. The element $w=(3\,\;5\,\;4)$ conjugates $\cox'$ to $\cox$, that is $\cox=w\cox'w\inv$.
	Hence $\l = w \l' w\inv = (1\,\;2)(3\,\;5)$ 
	is the left part and $\r=w \r' w\inv =(2\,\;5)(3\,\;4)$
	is the right part of the bipartition $\l\r=(1\,\;2)(3\,\;5)\cdot(2\,\;5)(3\,\;4)$ of the Coxeter element $\cox=(1\,\;2\,\;3\,\;4\,\;5)$. Note that the element $(3\,\;5\,\;4)$ that conjugates $\cox'$ to $\cox$ is not uniquely determined. For the element $v=(1\,\;5\,\;3\,\;4\,\;2)$ it is also true that $v\cox' v\inv=\cox$. Hence we get another bipartition of $\cox$ with left part $a=v\l'v\inv=(1\,\;5)(2\,\;4)$
	and right part $b=v\r'v\inv=(1\,\;4)(2\,\;3)$, 
	that is $\cox= ab =(1\,\;5)(2\,\;4)\cdot(1\,\;4)(2\,\;3)$.
	
	Let us fix the bipartition $\l\r= (1\,\;2)(3\,\;5)\cdot(2\,\;5)(3\,\;4)$ as standard bipartition of the Coxeter element $\cox$.
	We compute another bipartition of $\cox$ by using a different method as above. Note that $\cox=\cox\l\cdot\l$ and $(\cox\l)^2=\l\r\l\l\r\l =\id$ as well as $\l^2=\id$. Hence we have a bipartition $\l_1\r_1$ with left element $\l_1=\cox\l=(1\,\;3)(4\,\;5)$ and right element $\r_1=\l=(1\,\;2)(3\,\;5)$. 
	Further bipartitions of $\cox$ are given by $(1\,\;3)(4\,\;5)\cdot(1\,\;2)(3\,\;5)$ and $(2\,\;5)(3\,\;4)\cdot(1\,\;5)(2\,\;4)$.
	
\end{exa}

The method of producing new bipartitions of a Coxeter element out of the standard partition is not a phenomenon of the symmetric group. It works in the general setting.

\begin{lem}\label{lem:further_bipartitions}
	Let $W$ be a finite Coxeter group and $\cox=\l\r$ the standard bipartition of the Coxeter element $\cox$. Then $\cox^{k}\l \cdot \cox^{k-1}\l$ is a bipartition of $\cox$ for all $k\in \Z$.
\end{lem}

\begin{proof}
	We have to show that $\cox^k\l \cdot \cox^{k-1}\l=\cox$ and that $(\cox^k\l)^2=\id$ for $k\in \Z$. 
	Note that $\cox^k\l=(\l\r)^k\l=\l(\r\l)^k=\l\cox^{-k}$. Hence $\cox^k\l \cdot \cox^{k-1}\l=\cox^k\l \cdot \l \cox^{k-1} = \cox$. Using the same substitution we get  $(\cox^k\l)^2=\cox^k\l\cdot\l\cox^{-k} =\id$.
\end{proof}

Recall that we denote the order of the Coxeter element $c\in W$ by $h$.

\begin{cor}
	There are at least $h$ different bipartitions of a Coxeter element. 
\end{cor}

\begin{proof}
	Let $0\leq i,j <h$ be different numbers. Then $\cox^i$ is not equal to $\cox^j$, since the order of $\cox$ is $h$ and $i,j<h$. Hence the left parts of the bipartitions $\cox^{i}\l \cdot \cox^{i-1}\l$ and $\cox^{j}\l \cdot \cox^{j-1}\l$ are different elements of $W$, which means that the bipartitions themselves are different. 
\end{proof}

\begin{rem}
	Since $\cox^i\l = \r \cox^j$ for $i+j-1\equiv 0 \pmod h$, we can express the bipartitions of a 
	Coxeter element as described in \cref{lem:further_bipartitions} also as $\r\cox^j \cdot \r\cox^{j-1}$ for $j\in \Z$.
\end{rem}

\subsection{A dihedral group of automorphisms}
In this section we consider a subgroup of the automorphism group $\aut(\nc)$ of the non-crossing partitions $\nc=\nc(W)$. This group appears in \cite[Chap. 3.4.6]{armstr}.

\begin{defi}
	Let $\l$ be the left and $\r$ be the right part of the standard bipartition $\cox=\l\r$. We set 
	$\autl(w) = \l w\inv\l$ 
	and $\autr(w) = \r w\inv \r$ for $w\in\nc$.	
\end{defi}

\begin{lem}\label{lem:autl_autr_are_autos}
	The assignments $w \mapsto \autl(w)$ and $w \mapsto \autr(w)$ for $w\in \nc$ define two lattice automorphisms $\autl\colon \nc \to \nc$ and $\autr\colon \nc \to \nc$, respectively. If moreover $W$ has rank at least $2$, then $\autl$ and $\autr$ have order two.
\end{lem}

The fact that $\autl$ and $\autr$ are lattice automorphisms of $\nc$ is \cite[Le. 3.4.12]{armstr}. For the sake of completeness we include a complete proof of the above lemma here.

\begin{proof}
	We only consider $\autl$ as the proof for $\autr$ is analogously. 
	In order to show that $\autl$ is an automorphism, we show that $\autl$ maps $\nc$ to itself and preserves the order. Since $\autl\inv=\autl$, the assertion follows.
	
	Let $v,w\in \nc$ be such that $v\leq w$. By the \cref{subword_prop} we have that $v\inv \leq w\inv\leq \cox\inv=\r\l$. Conjugation with $\l$ yields $\l v\inv \l \leq \l w\inv\l\leq \l\r\l\l=\cox$ since conjugation is an automorphism of the absolute order on $W$.
	
	Since $\autl\inv=\autl$ we know that $ \autl$ has at most order $2$. For a contradiction, suppose that $\autl=\id$. Then we have $\autl(\r)=\r$ and hence $\l\r\l=\r$, which is equivalent to $\cox^2=\id$. But the Coxeter element has order two if and only if $W$ has rank $1$, which can be checked by looking at \cref{tab:Coxeter_numbers}.
\end{proof}

\begin{defi}
	Let $\D \leq \aut(\nc)$ be the subgroup of the automorphism group $\aut(\nc)$ generated by the two automorphisms $\autl$ and $\autr$.
\end{defi}

Recall that the Coxeter number $h$ denotes the order of the Coxeter element $\cox$.

\begin{lem}\label{lem:dihedral_auto_group}
	If $W$ has rank at least 2, then the order $d$ of composition $\autl \circ \autr$ is a divider of the Coxeter number $h$. Moreover, the group of automorphisms $\D$ is a dihedral group of order $2d \leq 2h$.
\end{lem}

\begin{proof}
	Let $W$ be of rank at least $2$. By \cref{lem:autl_autr_are_autos} the two generators $\autl$ and $\autr$ of $\D$ have order two. 
	The composition $\autl \circ \autr$ is conjugation with the Coxeter element $\cox=\l\r$. Since the order of $\cox$ is $h$, we have $(\autl\circ \autr)^h=\id$ and the order $d$ of the composition $\autl\circ \autr$ divides $h$. Hence $\autl$ and $\autr$ generate a dihedral group of order $2d$.
\end{proof}

\begin{rem}
	Note that Armstrong states after Lemma $3.4.12$ in \cite{armstr} that $\D$ \emph{has} order $2h$. This is true for example in type $A$ \cite{biane}, but it is false in types $B$ and $D$. We treat the automorphisms of $\nc(B_n)$ and $\nc(D_n)$ in Sections \ref{cha:typeB_auto_group} and \ref{cha:typeD_auto_group}, respectively. There we compute the order of $\D$ explicitly and in particular provide examples that in these cases, $\D$ has order different from $2h$.
\end{rem}

For the case of $S_n$, Biane showed that $\D$ is the full automorphism group of $\nc(S_n)$ and that $\aut(\nc(S_n)) \cong D_n$ \cite{biane}. This means in particular that every automorphism of $\nc(S_n)$ is an automorphism of the pictorial representations. \cref{fig:auto_NCP4} convinces that the automorphism group of $\ncp_4$ is indeed isomorphic to the symmetry group pf the square, that is the dihedral group $D_4$. We compute the automorphism groups of $\nc(B_n)$ and $\nc(D_n)$ in \cref{thm:typeB_full_aut} and \cref{thm:typeD_full_aut}, respectively.

\begin{figure}
	\begin{center}
		\begin{tikzpicture}
		\def\r{1.8} 
		\foreach \x/\d in {0/l,1/m,2/r}   
		\foreach \y/\e in {0/u,2/o}
		\node (\d\e) at (\r*\x,\r*\y) [mpunkt] {};
		\node (l) at (0,\r) [mpunkt] {};
		\node (r) at (2*\r,\r) [mpunkt] {};
		\draw (lu) -- (mu) -- (ru) -- (r) -- (ro) -- (mo) -- (lo) -- (l)
		-- (lu) -- (ro) (lo) -- (ru) (mo) -- (mu) (l) -- (r);
		\foreach \m/\n in {ru/lo, mu/mo, lu/ro, l/r}
		\node (\m,\n) at ($(\m)!0.25!(\n)$) [mpunkt] {};
		\def\a{0mm} 
		\foreach \m/\n/\p/\d in {    ru/lo/above right/\scalebox{0.7}{\ped},
			mu/mo/below right/\scalebox{0.7}{\pevuzd},
			lu/ro/below right/\scalebox{0.7}{\pzv},
			l/r/below/\scalebox{0.7}{\pezudv}}
		\node[\p = \a] (\m,\n) at ($(\m)!0.25!(\n)$) {\d};
		\def\b{2mm} 
		\foreach \m/\p/\d in {    ro/above right/\scalebox{0.7}{\pezv},
			r/right/\scalebox{0.7}{\pez},
			ru/below right/\scalebox{0.7}{\pezd},
			mu/below/\scalebox{0.7}{\pzd},
			lu/below left/\scalebox{0.7}{\pzdv},
			l/left/\scalebox{0.7}{\pdv},
			lo/above left/\scalebox{0.7}{\pedv},
			mo/above/\scalebox{0.7}{\pev}}
		\node[\p = \b] at (\m) {\d};
		\end{tikzpicture}
		\caption{The complex $|\ncp_4|$ in a square-shape.}
		\label{fig:auto_NCP4}
	\end{center}
\end{figure}
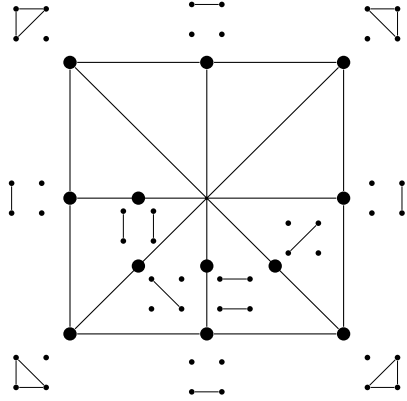

\begin{nota}
	We denote by $d$ the order of the automorphism $\autl\circ \autr\colon \nc \to \nc$, which is conjugation with the Coxeter element $\cox$.
\end{nota}

\begin{lem}
	Every automorphism $\auto\colon \nc\to\nc$ in $\D$ is of one of the following forms:
	\begin{enumerate}
		\item $\auto\colon w \mapsto \cox^k w \cox^{-k}$ for $0\leq k < d$ 
		\item $\auto\colon w \mapsto \cox^k\l w\inv \l\cox^{-k} = \cox^k\l w\inv \cox^k\l$ for $0\leq k < d$ 
	\end{enumerate}
\end{lem}

\begin{proof}
	Since $\ord(\autl\circ\autr)=d$, it holds that $(\autl\circ\autr)^i=(\autr\circ\autl)^j$ for $i+j=d$ and further $(\autl\circ\autr)^k\circ\autl=(\autr\circ\autl)^m\circ\autr$ for $k+m=d-1$. Hence every automorphism in the dihedral group $\D$ is of the form $(\autl\circ\autr)^i$ or $(\autl\circ\autr)^k\circ\autl$ for $0\leq i,k <d$. Inserting the definitions of $\autl$ and $\autr$ yields that the automorphisms in $\D$ are given by conjugation with $\cox^k$ or conjugation of the inverse by $\cox^k\l$ for $0\leq k<d$.
\end{proof}

If an automorphism $\auto\in\D$  is of type \emph{a)}, it is called \emph{rotating}, and if it is of type \emph{b)}, it is called \emph{reflecting}.

\subsection{Extending the automorphisms}\label{sec:induced_autos}

The aim of this section is to show that the automorphisms in $\D$ of $\nc=\nc(W)$ extend uniquely to automorphisms of a lattice of linear subspaces $\lam=\lam(V)$. In this section, the vector space $V$ defined over $\F$ can either be isomorphic to $\R^n$, or in the case that $W$ is equipped with a crystallographic root system $\Phi$, we can choose $V$ to be isomorphic to $\F_p^n$ for a prime compatible to $\Phi$ by \cref{thm:embedding_Vp_VZ}. The map $\emb \colon\nc \to \lam$, which maps an element to its moved space, or $p$-moved space, respectively, is a rank-preserving poset embedding by Theorems \ref{thm:brady_watt_embedding} and \ref{thm:embedding_Vp_VZ}.
From a given automorphism $\auto$ of $\nc$ we construct a vector space automorphism of $V$. This in turn induces an automorphism of the lattice $\lam$, which restricts to the automorphism $\auto$ from $\nc$ we started with.\\

Recall that in \cref{lem:join_and_em_interchange} and \cref{lem:pemb_join_interchange} we showed that the embedding of $\nc$ into $\lam$ is uniquely determined by the images of the reflections $T$ in $W$.
Since $\nc$ and $\lam$ are atomic lattices, the whole automorphism is determined by the images of the elements of rank $1$ by \cref{lem:lattice_auto_join_interchange}. For the non-crossing partition lattice, an even stronger statement is true.

\begin{lem}\label{lem:ncauto_red_decomp}
	Let $\auto\in \D$ be an automorphism of $\nc$. If $t_1\ldots t_k$ is a reduced decomposition of $w\in \nc$, then $\ncauto(t_1) \ldots  \ncauto(t_k)$ or $\ncauto(t_k) \ldots  \ncauto(t_1)$ is a reduced decomposition of $\auto(w)$. The first case occurs if $\ncauto$ is rotating and the second case if $\ncauto$ is reflecting.
\end{lem}

\begin{proof}
	If $\auto$ is a rotating automorphism then
	\[
	\auto(w)=\cox^k w \cox^{-k} =\cox^k t_1\ldots t_k \cox^{-k}=\cox^k t_1 \cox^{-k} \ldots \cox^k t_k \cox^{-k} = \auto(t_1) \ldots \auto(t_k)
	\]
	for some $0 \leq k < h$.		
	Let $\auto$ be a reflecting automorphism. Then it holds that
	\[
	\auto(w)= \cox^k\l w\inv \l\cox^{-k} = \cox^k\l t_k \ldots t_1 \l\cox^{-k} = \cox^{k}\l t_k\l\cox^{-k} \ldots \cox^{k}\l t_1\l\cox^{-k}=\auto(t_k) \ldots \auto(t_1)
	\] 
	for some $0 \leq k < h$, where $\l$ is the left part of the standard bipartition $\l\r$ of the Coxeter element $\cox$.

\end{proof}

In the following, let $\auto \in \D$ be a fixed automorphism of $\nc$.

\begin{lem}\label{lem:defi_lamauto}
	Let $r_1, \ldots, r_n\in T$ be reflections such that $r_1\ldots r_n$ is a reduced decomposition of the Coxeter element $\cox$. Then the assignment  $\alpha_{r_i} \mapsto \alpha_{\ncauto(r_i)}$ for $1 \leq i \leq n$ induces a vector space automorphism $\vecauto\colon V\to V$.
\end{lem}

\begin{proof}
	We have to prove that both sets $\Set{\alpha_{r_1}, \ldots, \alpha_{r_n}}$ and $\Set{\alpha_{\ncauto(r_1)}, \ldots, \alpha_{\ncauto(r_n)}}$ are bases of the vector space $V$.
	Since $r_1\ldots r_n$ is a reduced decomposition of the Coxeter element, it follows from \cref{lem:carter} that the first set is a basis of $V$.
	From the same lemma it follows that $\Set{\alpha_{\ncauto(r_1)}, \ldots, \alpha_{\ncauto(r_n)}}$ is a basis of $V$ since $\ncauto(r_1)\ldots \ncauto(r_n)$ or $\ncauto(r_n)\ldots \ncauto(r_1)$ is a reduced decomposition of the Coxeter element $\cox$ by \cref{lem:ncauto_red_decomp}.
\end{proof}

Let $\vecauto\colon V \to V$ be the vector space automorphism defined by $\alpha_{r_i} \mapsto \alpha_{\ncauto(r_i)}$ for $1 \leq i \leq n$
as in \cref{lem:defi_lamauto}.

\begin{lem}\label{lem:lamauto_alpha_t}
	For  all $t \in T$ it holds that $\vecauto( \alpha_t ) =  \alpha_{\ncauto(t)}$.
\end{lem}

\begin{proof}
	Recall from \cref{sec:geom_rep} that the Coxeter group $W$ acts on $V$ as a reflection group by linear transformations. In particular, $W$ acts by linear transformations on the set of roots by $w(\alpha_t)=\alpha_{wtw\inv}$ for $t \in T$. Since $t\inv=t$, the reflecting automorphisms act, as well as the rotating ones, by conjugation on the set of reflections $T$. Let $w\in W$ be the element such that $\auto(t)=wtw\inv$ for $t\in T$. This means that $w=\cox^k$ or $w=\cox^k\l$ for some $0\leq k < h$.
	Since $\Set{\alpha_{r_1}, \ldots, \alpha_{r_n}}$ is a basis of $V$, for every $t\in T$ the corresponding root $\alpha_t$  can be expressed as $\alpha_t = \sum_{i=1}^{n} a_i \alpha_{r_i}$ with appropriate $a_i \in \F$. Hence we get
	\begin{align*}
	\vecauto(\alpha_t) 
	&= \vecauto\left(\textstyle\sum_{i=1}^{n} a_i \alpha_{r_i}\right)
	= \textstyle\sum_{i=1}^{n} a_i \vecauto(\alpha_{r_i})
	= \textstyle\sum_{i=1}^{n} a_i \alpha_{\ncauto(r_i)}\\
	&= \textstyle\sum_{i=1}^{n} a_i \alpha_{w r_iw\inv}
	= \textstyle\sum_{i=1}^{n} a_i w(\alpha_{r_i})
	= w\left(\textstyle\sum_{i=1}^{n} a_i \alpha_{r_i}\right)\\
	&= w(\alpha_t)
	= \alpha_{w t w\inv}
	= \alpha_{\ncauto(t)}.			
	\end{align*}
\end{proof}

With the above lemma, $\vecauto$ induces a unique rank-preserving lattice automorphism 
\begin{align*}
\lamauto\colon\lam(V) &\to \lam(V)\\
U &\mapsto \vecauto(U)
\end{align*}
where $\vecauto(U) = \Set{\vecauto(u)\str u\in U}$.

\begin{rem}\label{rem:lamauto_join}
	Let $U \subseteq V$ be a subspace with basis $\Set{u_1, \ldots, u_k}$. Since the join in $\lam(V)$ is given by the sum of subspaces, we get that $U= \langle u_1\rangle \vee \ldots \vee\langle u_k \rangle$ as well as $\lamauto(U)=\langle \vecauto(u_1), \ldots, \vecauto(u_k)\rangle=\langle \vecauto(u_1)\rangle \vee \ldots \vee \langle \vecauto(u_k)\rangle$.
\end{rem}

\begin{thm}\label{thm:autos_commute}
	Every lattice automorphism in the dihedral group of automorphisms $\D$ of $\nc$ extends uniquely to a lattice automorphism of $\lam$.
\end{thm}

\begin{proof}
	First we show that every automorphism in $\D$ of $\nc$ extends to an automorphism of the lattice $\lam=\lam(V)$ and then that this automorphism is unique.
	
	Let $\ncauto\colon \nc \to \nc$ be an automorphism of $\nc$ with $\ncauto\in\D$, let $\emb\colon \nc \to \lam$ be the embedding from \cref{thm:brady_watt_embedding} or \cref{thm:embedding_Vp_VZ}, respectively, defined by $\emb(t)= \langle \alpha_t \rangle$, and let $\lamauto\colon \lam \to \lam$ be the lattice automorphism induced from 
	the vector space automorphism $\vecauto\colon V\to V$ with $\vecauto(\alpha_t) = \alpha_{\ncauto(t)}$
	for a reflection $t\in T\subseteq \nc$. This is well-defined by \cref{lem:lamauto_alpha_t}. We show that
	\[
	\lamauto \circ f= f \circ \ncauto.
	\]	
	
	Let $w\in \nc$ be arbitrary with reduced decomposition $t_1\ldots t_k$. Using Lemmas \ref{lem:join_and_em_interchange}  and  \ref{lem:pemb_join_interchange}, we can write the image of $w$ under the embedding $\emb$ as $\emb(w)=f(t_1)\vee \ldots \vee f(t_k)$. Since also the lattice automorphism $\lamauto$ has the property that it interchanges with the join by \cref{lem:lattice_auto_join_interchange}, we have 
	\[
	\lamauto(\emb(w)) = \lamauto(\emb(t_1)) \vee \ldots \vee \lamauto(\emb(t_k)).
	\]
	On the other side, we also have that $\ncauto(w)=\ncauto(t_1)\vee \ldots \vee \ncauto(t_k)$ by \cref{lem:lattice_auto_join_interchange} since $\ncauto$ is an automorphism of the lattice $\nc$. Applying Lemmas \ref{lem:join_and_em_interchange} and \ref{lem:pemb_join_interchange} once again yields
	\[
	\emb(\ncauto(w))= \emb(\ncauto(t_1)) \vee \ldots \vee \emb(\ncauto(t_k)).
	\]
	
	So the equality $\lamauto(\emb(w))=\emb(\ncauto(w))$ holds for all $w\in \nc$ if it holds for the reflections $t\in T$. Using $\vecauto(\alpha_t) 
	= \alpha_{\ncauto(t)}$ for $t \in T$ and the definitions of the maps $\lamauto$ and $\emb$, we get that
	\[
	\lamauto(f(t)) 
	= \lamauto(\langle\alpha_t\rangle) 
	= \langle \vecauto(\alpha_t) \rangle
	= \langle\alpha_{\ncauto(t)}\rangle 
	= f(\ncauto(t))
	\]
	for all $t\in T$ and hence $\lamauto \circ f= f \circ \ncauto$.
	
	To show that the extension is unique, we show that the only automorphism that extends the identity on $\nc$ is the identity on $\lam$. Let $\rho\colon \lam \to \lam$ be an automorphism of $\lam$ that extends the identity, that is we have $\rho \circ \emb = \emb$. In particular, it holds that 
	\[
	\rho(\langle \alpha_t \rangle) = \langle \alpha_t \rangle
	\]
	for all $t\in T$. Let $r_1, \ldots, r_n \in T$ be such that $\cox=r_1\ldots r_n$. Then $\Set{\alpha_{r_1}, \ldots, \alpha_{r_n} }$ is a basis of $V$ such that the linear spans of the basis vectors are fixed by $\rho$. By \cref{lem:lamauto_induces_vecautos} there exist $\lambda_1, \ldots, \lambda_n\in \F\setminus\Set{0}$ such that the assignment $\alpha_i\mapsto \lambda_i\alpha_i$ induces a vector space automorphism $\vecauto$ that extends to the lattice automorphism $\rho$.
	By showing that $\lambda_1=\ldots=\lambda_n$, which means that $(\lambda_1, \ldots, \lambda_n)=\lambda_1\cdot(1,\ldots,1)$, we prove $\rho=\id_\lam$.
	Let $t\in T$ be arbitrary. Since $\Set{\alpha_{r_1}, \ldots, \alpha_{r_n} }$ is a basis of $V$, there exist $a_i\in \F$ for $1 \leq i \leq n$ such that $\alpha_t=\sum_{i=1}^{n}a_i\alpha_{r_i}$. Since $\rho(\langle \alpha_t \rangle)=\langle \alpha_t \rangle$ by assumption, there exists a $\lambda \in \F\setminus\Set{0}$ such that $\vecauto(\alpha_t)=\lambda\alpha_t$. Hence we get on the one hand that 
	\begin{align*}
	\vecauto(\alpha_t) 
	=\lambda\alpha_t 
	=\lambda\sum_{i=1}^{n}a_i\alpha_{r_i}
	=\sum_{i=1}^{n}\lambda a_i\alpha_{r_i},
	\end{align*}
	and on the other hand that
	\begin{align*}
	\vecauto(\alpha_t) 
	= \vecauto\left(\sum_{i=1}^{n}a_i\alpha_{r_i}\right)
	= \sum_{i=1}^{n}a_i\vecauto(\alpha_{r_i})
	= \sum_{i=1}^{n}a_i\lambda_i\alpha_{r_i},
	\end{align*}
	which implies that $\lambda=\lambda_1=\ldots=\lambda_n$.
\end{proof}

The translation of \cref{thm:autos_commute} into the simplicial setting is the following. Note that every lattice automorphism of a graded lattice $L$ gives rise to a type-preserving simplicial automorphism of its order complex $|L|$. Hence the lattice automorphism $\aut(L)$ group can be naturally identified with the type-preserving simplicial automorphism group $\aut(|L|)$. 

\begin{thm}
	Every type-preserving simplicial automorphism in $\D \leq \aut(\nc)$ extends uniquely to a type-preserving simplicial automorphism of the simplicial building $|\lam(V)|$.
\end{thm}

\section{Automorphisms of type $B$}\label{cha:typeB_auto_group}

The aim of this section is to show that the automorphism group $\aut(\nc(B_n))$ of the non-crossing partition lattice of type $B_n$ is isomorphic to the dihedral group $\D$ of automorphisms. 
First we show in \cref{lem:typeB_dihedral_autos} that the dihedral group of automorphisms $\D=\langle \autl, \autr \rangle$ has order $2n=h$ and then we show in \cref{thm:typeB_full_aut} that this is already the full group of lattice automorphisms of $\nc(B_n)$.

Recall that the standard Coxeter element in type $B_n$ is given by the balanced cycle $[1\ldots n]$ and has order $h=2n$. 

\begin{lem}\label{lem:typeB_dihedral_autos}
	The dihedral group of lattice automorphisms $\D=\langle \autl, \autr \rangle$ of the lattice of non-crossing partitions of type $B_n$ has order $h=2n$.
\end{lem}

\begin{proof}
	By \cref{lem:dihedral_auto_group} we have to show that the automorphism $\ncauto$ of $\nc(B_n)$ given by $\ncauto \coloneqq\autl \circ \autr\colon w \mapsto \cox w \cox\inv$ is of order $d=n$.
	First we show that for arbitrary $w\in W$ we have that that $\ncauto^n(w)=w$. 
	It holds that 
	\[
	\cox^n=[1\ldots n]^n= [1][2]\ldots[n]
	\]
	and that $\cox^{-n}=\cox^n$, as $\ord(\cox)=2n$. Let $i\in \Set{\pm 1,\ldots, \pm n}$ be arbitrary. Then $\cox^n(i)=-i$ and hence we get that
	\[
	\ncauto^n(w)(i)=(\cox^n w \cox^n)(i)=(\cox^n w)(-i) = \cox^n(w(-i))=\cox^n(-w(i))=w(i),
	\]
	where $w(-i)=-w(i)$, because $w$ is a signed permutation.
	
	Moreover, since $[1\ldots n]^k(1)=k+1$ for all $1\leq k < n$, we get with the same computation as above that $\ncauto^k(w)(1)\neq w(1)$ for all $1\leq k < n$ and hence $\ncauto^k\neq \id$.
\end{proof}

In the following example we construct the automorphisms $\autl$ and $\autr$ for $\nc(B_4)$ by first constructing a standard bipartition of the Coxeter element.

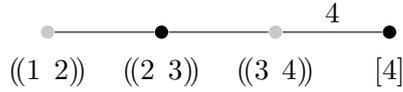
\begin{figure}
	\begin{center}
		\begin{tikzpicture}
		\node[mpunkt, Lightgray] (p1) at (1.5,0){};
		\node[mpunkt] (p2) at (3,0){};
		\node[mpunkt, Lightgray] (p3) at (4.5,0){};
		\node[mpunkt] (p4) at (6,0){};
		\node[above] (b) at (5.25,0){$4$};
		\node[below = 2mm] at (p1) {$\dka 1\,\; 2\dkz$};
		\node[below = 2mm] at (p2) {$\dka 2\,\; 3\dkz$};
		\node[below = 2mm] at (p3) {$\dka 3\,\; 4\dkz$};
		\node[below = 2mm] at (p4) {$[4]$};
		\draw (p1)--(p2)--(p3)--(p4);
		\end{tikzpicture}
		\caption{The Coxeter diagram of type $B_4$ with a bipartition of the vertices corresponding to the bipartition $\dka 1\,\;2\dkz\dka3\,\;4\dkz \cdot \dka 2\,\;3 \dkz [4]$ of the Coxeter element $[3\,\;1\,\;2\,\;4] \in W(B_4)$.}
		\label{fig:coxeter_graph_b4}
	\end{center}
\end{figure}

\begin{exa}\label{exa:typeB_autos_pict}
	In this example, we consider the non-crossing partitions of type $B_4$ with Coxeter element $\cox=[1\,\;2\,\;3\,\;4]$. The standard reduced decomposition of $\cox$ is given by $\dka 1\,\; 2\dkz \dka 2\,\;3\dkz \dka 3\,\;4\dkz [4]$, and the corresponding Coxeter diagram is the line shown in \cref{fig:coxeter_graph_b4}. Hence, the two preliminary left and right parts are $\l'=\dka 12\dkz \dka 34\dkz $ and $\r'=\dka 23\dkz [4]$. Since $\cox'=\l'\r'=[3\,\;1\,\;2\,\;4]$ and the element $\dka 1\,\;2\,\;3\dkz $ conjugates $\cox'$ to $\cox$, conjugation of $\l'$ and $\r'$ with $\dka 1\,\;2\,\;3\dkz $ yields the left part $\l=\dka 1\,\;4\dkz \dka 2\,\;3\dkz $ and the right part $\r=\dka 1\,\;3\dkz [4]$ of the standard Coxeter element $\cox=\l\r$.
	
	The next step is to construct and understand the automorphisms $\autl$ and $\autr$ associated to $\l$ and $\r$, respectively.
	The symmetries of the pictorial representations of $\l$ and $\r$, shown in \cref{fig:b4_left_and_right_part}, already suggest which reflections of the octagon are represented by $\autl$ and $\autr$. Since the edges of $\l$ must be fixed by the corresponding automorphism $\autl$, the axis of reflection of $\autl$ is parallel to the parallel edges of $\l$. The same is true for $\r$ and $\autr$. Note that the axes of reflection of $\autl$ and $\autr$ have angle $\frac{\pi}{8}$, hence they should generate the symmetry group of the octagon, which is the dihedral group $D_8$ of order $16$. At first glance, this conflicts the statement of \cref{lem:typeB_dihedral_autos}, where we showed that the dihedral group of automorphisms generated by $\autl$ and $\autr$ is a dihedral group $D_4$ of order $8$.
	
	Let us take a closer look at the pictorial representations of a type $B$ non-crossing partition. By definition they are invariant under 180 degree rotation. Hence the automorphism $(\autl \circ \autr)^4$, which is rotation by 180 degrees, is the identity on $\nc(B_4)$, although it is \emph{not} the identity in the symmetry group of the octagon. This implies that every automorphism in $\D=\langle \autl, \autr \rangle$ is represented by exactly \emph{two} elements of the symmetry group of the octagon. In the concrete case of $\autr$, the second axis of reflection is given by the line that goes through $2$ and $-2$, which is exactly the former axis rotated by $180$ degrees. So we only need half of the automorphisms of the symmetry group of the octagon to describe the automorphisms of $\nc(B_4)$ and the considerations of the dihedral group acting on the type $B$ non-crossing partitions fits to the statement of \cref{lem:typeB_dihedral_autos}. The observations of this example for the concrete case of $\nc(B_4)$ are also true in an analog way for all $\nc(B_n)$.
\end{exa}

\begin{figure}
	\begin{center}
		\begin{tikzpicture}
		\gachteck	\draw (p1) -- (p4) (p2) -- (p3) (p8) -- (p5) (p7) -- (p6);
		\node[below = 3mm] at ($(p3)!0.5!(p4)$) {$l$};
		
		\begin{scope}[xshift = 3cm]
		\gachteck \draw (p1) -- (p3) (p8) -- (p4) (p7) -- (p5);
		\node[below = 3mm] at ($(p3)!0.5!(p4)$) {$r$};
		\end{scope}
		
		\begin{scope}[xshift = 6cm]
		\gachteck \draw (p1) -- (p4) (p2) -- (p3) (p8) -- (p5) (p7) -- (p6) (p1) -- (p3) (p8) -- (p4) (p7) -- (p5);
		
		\coordinate (lo) at (45:13mm); \coordinate (lu) at (45:-13mm);
		\draw[Orange, thick, dotted] (lu) -- (lo);
		\coordinate (ro) at ($(p8) + (67.5:8mm)$); \coordinate (ru) at ($(p4) + (247.5:5mm)$);
		\draw[Orange, thick, dotted] (ro) -- (p8) -- (p4) -- (ru);
		
		\draw[Orange,<->] (ro) + (2.5mm,-1.4mm)  arc (-12.5:147.5:3mm) node[midway, above] {$\varphi_r$}; 
		\draw[Orange,<->] (lo) + (1.5mm,-3mm)  arc (-33.5:127.5:3mm) node[midway, right] {$\varphi_l$};

		\def\r{5mm}
		\draw[Orange, thick] (45:\r) -- (0,0) -- (67.5:\r); 
		\clip (0,0) -- (45:6mm) -- (67.5:6mm) -- cycle;
		\draw[Orange, thick] (0,0) circle (\r);
		
		\end{scope}
		\end{tikzpicture}
		\caption{The left part $\l=\dka 1\,\;4\dkz \dka 2\,\;3\dkz $ and the right part $\r=\dka 1\,\;3)[4]$ of $\cox=[1\,\;2\,\;3\,\;4]$ are shown here. The axes of reflection for the corresponding automorphisms $\autl$ and $\autr$ are depicted in orange on the right.}
		\label{fig:b4_left_and_right_part}
	\end{center}
\end{figure}
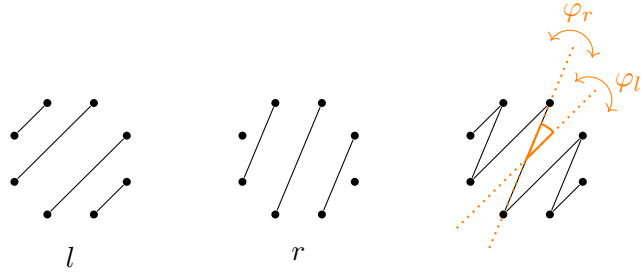

\begin{thm}\label{thm:typeB_full_aut}
	The automorphism group $\aut(\nc(B_n))$ of the non-crossing partition lattice $\nc(B_n)$ is the dihedral group $\D$ of lattice automorphisms of order $h=2n$. 
\end{thm}

The rest of this section is devoted to prove this theorem. The proof uses ideas similar to the ones of proof of Theorem 2 in \cite{biane}.  Before we can prove it, we need several lemmas.
We start with defining the type of rank $1$ and $2$ elements of $\nc(B_n)$. In this section, the letters $p$ and $q$ are reserved for elements of rank $1$, and we use $x$, $y$ and $z$ to describe elements of rank $2$.

Let $p\in\nc(B_n)$ be of rank $1$. The \emph{type} of $p$ is defined as
\[
\typ(p)=
\begin{cases}
j-i+1& \text{ if }p\equiv\dka i\,\; j\dkz \text{ for }0<i<j,\\
n-i-j+1& \text{ if }p\equiv\dka i\,\; j\dkz \text{ for }0<-i<j,\\
-1 &\text{ if } p\equiv[i].
\end{cases}
\]
Note that $2 \leq \typ(p)\leq n$ if $p$ is a paired cycle and $\typ(p)=-1$ if $p$ is a balanced cycle. 

If we pass to the pictorial representation, the type of an element describes how \enquote{diagonal} the edge or the pair of edges lies inside the $2n$-gon, since it measures the distance of the endpoints of the edges. Hence having type $k > 0$ means \enquote{looking like} $\dka 1\,\; k \dkz$ and having negative type corresponds to being a zero block.

Now let $x\in \nc(B_n)$ be an element of rank $2$. Then either, it consists of one balanced cycle, one paired cycle, a balanced and a paired cycle, or two paired cycles. 
In the pictorial representation, this corresponds to having one, two, three or four blocks. The \emph{type} of $x$ records this number and is defined as 
\[
\Typ(x)=
\begin{cases}
1& \text{ if } x=[a\,\;b],\\
2& \text{ if } x=\dka a\,\; b\,\; c\dkz,\\
3& \text{ if } x=\dka a\,\;b\dkz [c],\\
4& \text{ if } x=\dka a\,\; b\dkz \dka c\,\; d\dkz
\end{cases}
\]
for appropriate $a,b,c,d \in \Set{\pm1, \ldots, \pm n }$. 

Recall that $\ab(p)$ is the cover set, that is the set of all elements that cover $p$, and $\bel(x)$ is the covered set, that is the set of all elements that are covered by $x$.

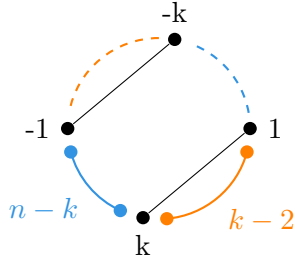
\begin{figure}%
	\begin{center}
		\begin{tikzpicture}
		\def\r{1.2} 
		\def\w{35} 
		
		\node[mpunkt] (1) at (0:\r) {};
		\node[mpunkt] (-1) at (180:\r) {};
		\node[mpunkt] (-k) at (80:\r) {};
		\node[mpunkt] (k) at (260:\r) {};
		
		\foreach \n/\pos in {1/right,-1/left,k/below,-k/above}
		\node[\pos=1mm] at (\n) {\n};
		
		\draw (-1) -- (-k) (k) -- (1);
		\begin{scope}[Blue,thick]
		\node[mpunkt] at (195:\r) {};
		\node[mpunkt] at (245:\r) {};
		\draw (195:\r) arc (195:245:\r) node[midway, below left] {$n-k$};
		\draw[dashed] (10:\r) arc (10:70:\r);
		\end{scope}
		
		\begin{scope}[Orange,thick]
		\node[mpunkt] at (275:\r) {};
		\node[mpunkt] at (345:\r) {};
		\draw (275:\r) arc (275:345:\r) node[midway, below right] {$k-2$};
		\draw[dashed] (90:\r) arc (90:170:\r);
		\end{scope}
		\end{tikzpicture}
		\caption{There are $k-2$ vertices below the edge $\Set{1,k}$, shown in orange, and $n-k$ vertices above it, which are shown in blue. The dashed lines indicate that every vertex has an opposite.}%
		\label{fig:typeB_above_below_vertices}%
	\end{center}
\end{figure}

\begin{lem}\label{lem:typeB_number_rk2_elements_types}
	Let $p\in \nc(B_n)$ be an element of rank 1. If $\typ(p)=k>0$, then $\ab(p)$ has \begin{itemize}
		\item 1 element of type $1$,
		\item $2n-k-2$ elements of type $2$,
		\item $n-k$ elements of type $3$, and
		\item $\frac{1}{2}(k-2)(k-3)+(n-k)(n-k-1)$ elements of type $4$.
	\end{itemize}
	If $\typ(p)=-1$, then $\ab(p)$ has
	\begin{itemize}
		\item $n-1$ elements of type $1$, and
		\item $\frac{1}{2}(n-1)(n-2)$ elements of type $3$.
	\end{itemize}
\end{lem}

\begin{figure}
	\begin{center}
		\begin{tikzpicture}
		
		\grauKreis
		\draw[Blue, thick] (-1) -- (k) (-k) -- (1);
		\node at (0,-1.4) {type 1};

		\begin{scope}[xshift = 4cm]
		\grauKreis
		\node(-a) [Orange, mpunkt] at (150:\r) {};
		\node(a) [Orange, mpunkt] at (150:-\r) {};
		\draw[Orange, thick] (-1) -- (-a) -- (-k) (1) -- (a) -- (k);
		
		\begin{scope}[xshift = 3 cm]
		\grauKreis
		\node(-a)[Blue, mpunkt] at (25:\r) {};
		\node(a) [Blue, mpunkt] at (25:-\r) {};
		\draw[Blue, thick] (k) -- (a) -- (1) (-1) -- (-a) -- (-k);
		\node at (0,-1.4) {type 2};
		\begin{scope}[xshift = 3 cm]
		\grauKreis
		\node(-a)[Blue, mpunkt] at (25:\r) {};
		\node(a) [Blue, mpunkt] at (25:-\r) {};
		\draw[Blue, thick] (-1) -- (a) -- (-k) (1) -- (-a) -- (k);
		\end{scope}
		\end{scope}	
		\end{scope}
		
		\begin{scope}[yshift = - 3.1 cm]
		\grauKreis
		\node(-a)[Blue, mpunkt] at (25:\r) {};
		\node(a) [Blue, mpunkt] at (25:-\r) {};
		\draw[Blue, thick] (a) -- (-a);
		\node at (0,-1.4) {type 3};
		
		\begin{scope}[xshift = 4cm]
		\grauKreis
		\node(-a) [Orange, mpunkt] at (150:\r) {};
		\node(a) [Orange, mpunkt] at (150:-\r) {};
		\node (-b) [Orange, mpunkt] at (80:\r) {};
		\node (b) [Orange, mpunkt] at (80:-\r) {};
		\draw[Orange, thick] (-a) -- (-b) (a) -- (b);
		
		\begin{scope}[xshift = 3 cm]
		\grauKreis
		\node(-a) [Blue, mpunkt] at (45:\r) {};
		\node(a) [Blue, mpunkt] at (45:-\r) {};
		\node(-b) [Blue, mpunkt] at (15:\r) {};
		\node(b) [Blue, mpunkt] at (15:-\r) {};
		\draw[Blue, thick] (a) -- (b) (-a) -- (-b);
		\node at (0,-1.4) {type 4};
		\begin{scope}[xshift = 3 cm]
		\grauKreis
		\node(-a) [Blue, mpunkt] at (45:\r) {};
		\node(a) [Blue, mpunkt] at (45:-\r) {};
		\node(-b) [Blue, mpunkt] at (15:\r) {};
		\node(b) [Blue, mpunkt] at (15:-\r) {};
		\draw[Blue, thick] (a) -- (-b) (-a) -- (b);
		\end{scope}
		\end{scope}	
		\end{scope}	
		\end{scope}
		\end{tikzpicture}
		\caption{The different ways to obtain a covering element of a rank $1$ element of type $k$ in $\ncb_n$ are shown here.}%
		\label{fig:typeB_cover_blocks}%
	\end{center}
\end{figure}
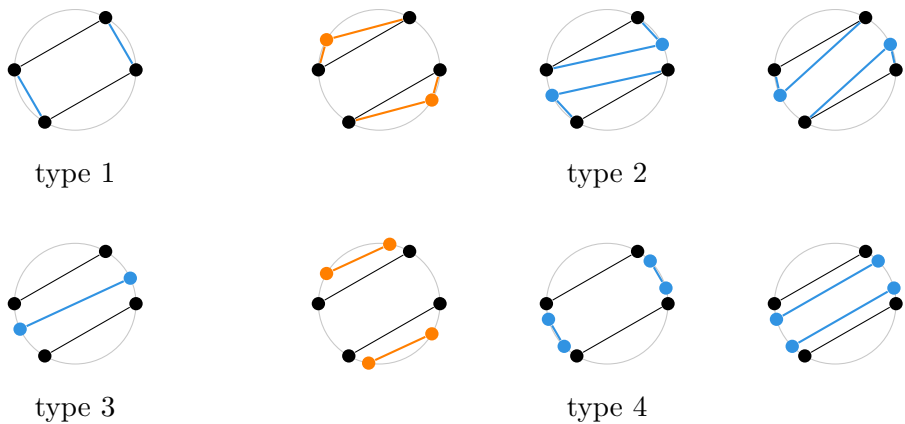

\begin{proof}
	Let $p\in \nc(B_n)$ be of rank $1$ and type $k$. We may assume without loss of generality that $p=\dka 1\,\; k\dkz$ if $\typ(p)>0$ and $p=[1]$ if $\typ(p)=-1$. By \cref{thm:typeB_nc_iso} we furthermore may pass to the pictorial representation. Hence, $p$ corresponds to the pair of edges $\Set{1, k}\Set{1,-k}$, or the edge $\Set{1,-1}$, respectively.
	Recall that the type of $p$ equals, up to sign, the number of non-trivial blocks in the pictorial representation of $p$. 
	In this proof, we use the same symbols for elements in $\nc(B_n)$ and their pictorial representations in $\ncb_n$.
	
	A vertex $a\in \Set{ 1, \ldots,  n}$ is called \emph{below} the edge $\Set{1,k}$ if $1 < a <k$, and it is called \emph{above} if $k<a\leq n$. In \cref{fig:typeB_above_below_vertices} the elements below are shown in orange and the elements above $\Set{1,k}$ in blue.	
	
	In the following, we count the possibilities to construct covers of $p$ with the desired number on non-trivial blocks. The different ways to do so are shown in \cref{fig:typeB_cover_blocks}.
	
	First we consider the case of the paired cycle $p=\dka 1\,\; k\dkz$.
	
	Let $x\in \ab(p)$ with $\Typ(x)=1$. Then $p$ has exactly one zero block with four elements, which has to contain $\Set{1,k,-1,-k}$, hence there is no choice.
	
	Let $x\in \ab(p)$ with $\Typ(x)=2$. Since the number of non-trivial blocks of $p$ and $x$ is the same, the only way to obtain $x$ from $p$ is to add one vertex to each of the edges of $p$. Since the elements in $\ncb_n$ are invariant under rotation of $\pi$, it is enough to consider the possibilities to add a vertex to the edge $\Set{1,k}$. There are $k-2$ ways to add a vertex below and $n-k$ ways to add a vertex above. But also the negatives of the vertices above can be added, so we get in total $k-2+2(n-k)=2n-k-2$ elements in $\ab(p)$ of type $2$.
	
	Let $x\in \ab(p)$ with $\Typ(x)=3$. The number of non-trivial blocks increases by one if we compare $p$ to $x$. Since the number of non-trivial blocks of $x$ is odd, there has to be a zero block present in $x$. In $p$ there is no zero block, so $x$ consists of the edges of $p$ with a zero block added above the edge $\Set{1,k}$. There are $n-k$ ways to do so.
	
	Let $x\in \ab(p)$ with $\Typ(x)=4$. Since $x$ has two blocks more than $p$, we have to add a pair of edges to $p$ that does not cross $p$ in order to construct $x$. The first possibility is to choose two vertices below and connect them. There are $\binom{k-2}{2}$ ways to do so. The second possibility is to choose an edge $\Set{a, b}$ among the vertices above, which yields $\binom{n-k}{2}$ different partitions. But instead of adding the pair $\Set{a,b}\Set{-a,-b}$, we can also add the pair $\Set{a,-b}\Set{-a,b}$. In total we get $\binom{k-2}{2}+2\binom{n-k}{2} =\frac{1}{2}(k-2)(k-3)+(n-k)(n-k-1)$ elements of type $4$ that cover $p$.
	
	Now consider the balanced cycle $p=[1]$. Since $p$ is a zero edge, there are no vertices above and the number of blocks of elements covering $p$ has to be $1$ or $3$. Hence we have to consider all cases with orange edges in \cref{fig:typeB_cover_blocks}. 
	
	Let $x\in \ab(p)$ with $\Typ(x)=1$. There are $n-1$ vertices below the edge $\Set{1,-1}$ and every vertex, as well as its negative, can be added to the block of $p$ to obtain an element with one block.
	
	Let $x\in \ab(p)$ with $\Typ(x)=3$. We increase the number of blocks by $2$ when constructing the partition $x$. Hence we have to add a pair of edges, which can be freely chosen from the $n-1$ elements below $\Set{1,-1}$. Hence there are $\binom{n-1}{2}$ different covering elements of type $3$.
\end{proof}

\begin{lem}\label{lem:typeB_rk2_type1_autos_preserve}
	If $x\in \nc(B_n)$ has rank $2$ and type $\Typ(x)=1$ or $\Typ(x)=2$, then for all automorphisms $\ncauto \in \aut(\nc(B_n))$ it holds that $\Typ(x)=\Typ(\ncauto(x))$.
\end{lem}

\begin{proof}
	The cardinality of $\bel(x)$ is preserved under all automorphisms. The covered sets for the four different types of rank $2$ elements are given by
	\begin{align*}
	&\bel([a\,\;b])=\Set{[a],[b],\dka a\,\; b\dkz, \dka a\,\;-b\dkz },\\
	&\bel(\dka a\,\;b\,\;c\dkz)=\Set{\dka a\,\; b\dkz, \dka b\,\;c\dkz, \dka a\,\; c\dkz },\\
	&\bel(\dka a\,\;b\dkz[c])=\Set{\dka a\,\; b\dkz, [c] },\text{ and}\\
	&\bel(\dka a\,\; b\dkz \dka c\,\;d\dkz)=\Set{ \dka a\,\; b\dkz, \dka c\,\;d\dkz }
	\end{align*}
	for appropriate $a,b,c,d\in \Set{\pm1, \ldots, \pm n}$.
	Hence we have that $T(x)=1$ if and only if $\#\bel(x)=4$, and that $T(x)=2$ if and only if $\#\bel(x)=3$. Consequently, an element of type $1$ gets mapped to an element of type $1$, and an element of type $2$ gets mapped to one of type $2$.
\end{proof}

Note that the lemma only states that the type of rank $2$ elements is preserved whenever the type equals $1$ or $2$. If $x,y\in \nc(B_n)$ are both of rank $2$ and type $k$, then in general there does \emph{not} exist an automorphism $\ncauto\in \aut(\nc(B_n))$ such that $x=\ncauto(y)$. The reason is that the type of a rank $2$ element only reflects its coarse structure, namely the cycle structure, but it does not encode how these cycle are built, that is what the types of the covered rank $1$ elements are.

\begin{lem}\label{lem:typeB_nc_auto_preserves_type_rk_1}
	The type of rank $1$ elements of $\nc(B_n)$ is preserved under all lattice automorphisms of $\nc(B_n)$.
\end{lem}

\begin{proof}
	Let $\ncauto\in \aut(\nc(B_n))$ be an automorphism and $p\in \nc(B_n)$ be an element of rank $1$. 
	By \cref{lem:cov_set_autos} the cardinalities of cover sets and covered sets are preserved under lattice automorphisms. In particular, for all $x\in \ab(p)$ it holds that 
	\[
	\#\Set{x\in \ab(p)\str \Typ(x)=2 } = \#\Set{y\in \ab(\ncauto(p))\str \Typ(y)=2 }
	\]
	because type $2$ elements of rank $2$ are mapped to type $2$ elements by \cref{lem:typeB_rk2_type1_autos_preserve}.
	The number of elements that cover $p$ and are of type $2$ is counted in \cref{lem:typeB_number_rk2_elements_types} and depends on the type $\typ(p)$ of $p$ only. If $\typ(p)=k>0$, then $\#\Set{x\in \ab(p)\str \Typ(x)=2}=2n-k-2$, and if $\typ(p)=-1$, then $\ab(p)$ does not contain any element of type $2$.
	This implies that the type of rank $1$ elements has to be preserved under automorphisms.
\end{proof}

\begin{proof}[Proof of \cref{thm:typeB_full_aut}]
	The strategy of the proof is the following. We show that an arbitrary automorphism of $\nc(B_n)$ is, after passing to the isomorphic lattice $\ncb_n$, a composition of symmetries of the $2n$-gon up to rotation by 180 degrees. Since there are $2n$ such symmetries and the order of the subgroup $\D \leq \aut(\nc(B_n))$ equals $2n$ by \cref{lem:typeB_dihedral_autos}, it follows that $\D = \aut(\nc(B_n))$. Since a lattice automorphism of a graded lattice is completely determined by the images of elements of rank $1$ by \cref{lem:lattice_auto_join_interchange}, it is enough to consider the images of the rank $1$ elements. In this proof we are passing freely from $\nc(B_n)$ to $\ncb_n$ and vice versa by using \cref{thm:typeB_nc_iso}. 
	
	Let $\ncauto\in \aut(\nc(B_n))$ be an arbitrary automorphism of $\nc(B_n)$. First we show that, after possibly composing $\ncauto$ with automorphisms of $\D$, $\ncauto$ fixes all elements of rank $1$ and type $2$. With induction we show that all elements of rank $1$ are fixed, which shows that $\ncauto=\id_{\nc(B_n)}$ and proves the claim.
	
	Let $p\in \nc(B_n)$ be an element of rank $1$ and type $2$, for instance $p=\dka 1\,\; 2\dkz$.
	By \cref{lem:typeB_nc_auto_preserves_type_rk_1} the type of $\ncauto(p)$ is also two. Hence there exists a rotating automorphism such that its composition with $\ncauto$ fixes $p$. In the pictorial representation this is the rotation that moves $\ncauto(p)$ back to $p$. We assume without loss of generality that $\ncauto$ already fixes $p$.

	The next goal is to show that $\ncauto$ fixes all elements of rank $1$ and type $2$. For this, we take a closer look at $\ab(p)$ and show that two of its elements of type $2$ get fixed by $\ncauto$. 
	From the proof of \cref{lem:typeB_rk2_type1_autos_preserve} it follows that there exist exactly two elements $x,y\in \ab(p)$ of type $2$ with the property that $\bel(x)$ and $\bel(y)$ contain exactly two elements of type $2$.
	In the concrete case of $p=\dka 1\,\;2\dkz$, the elements are $x=\dka 1\,\; 2\,\; 3\dkz$ and $y=\dka 1\,\;2\,\;-n\dkz$. Then we get that
	\begin{align*}
	&\bel(x)=\Set{p=\dka 1\,\;2\dkz, \dka 2\,\;3\dkz, \dka 1\,\;3\dkz }\text{ and }\\
	&\bel(y)=\Set{p=\dka 1\,\;2\dkz, \dka 2\,\;-n\dkz, \dka 1\,\;-n\dkz }.
	\end{align*}
	
	Recall that the type of rank $1$ elements is invariant under all automorphisms of $\nc(B_n)$ by \cref{lem:typeB_nc_auto_preserves_type_rk_1}. Hence it holds for all $z\in \nc(B_n)$ with $\rk(z)=2$ that $\bel(z)$ contains two elements of type $2$ if and only if $\bel(\ncauto(z))$ contains two elements of type $2$. 
	
	The elements in $\ab(p)$ with the property that they have type $2$ and that their covered set contains exactly two elements of type $2$ are exactly $x$ and $y$. Since $\ncauto(p)=p$ we have $\ab(\ncauto(p))=\ab(p)$ and the elements of $\ab(\ncauto(p))$ with the above property are exactly $x$ and $y$. Hence $\ncauto(\Set{x,y})=\Set{x,y}$. 
	There are two cases: either $\ncauto$ fixes both $x$ and $y$, or it exchanges the two elements.
	In the second case, the reflecting automorphism in $\D$ that fixes $p$ exchanges $x$ and $y$. The composition of $\ncauto$ with this reflecting automorphism fixes both $x$ and $y$. So we may assume without loss of generality that the automorphism $\ncauto$ fixes both $x$ and $y$.
	
	Now we are ready to show that $\ncauto$ fixes all rank $1$ elements of type $2$. 
	Since $\ncauto$ fixes $p$ and $x$, the other type $2$ element in $\bel(x)$, call it $q$, also gets fixed. In the concrete example of $p=\dka 1\,\;2 \dkz$ we have $q=\dka 2\,\; 3\dkz$. 
	Now we are in the same situation as above, where $q$ now plays the role of $p$ from above: We know that $q$ gets fixed by $\ncauto$ and that there exist exactly two elements in $\ab(q)$ of type $2$ that cover exactly two elements of type $2$. One of these elements is, to be consistent with the example, $x$, from which we already know that it gets fixed by $\ncauto$. Hence, the type $2$ elements in $\ab(q)$ both get fixed. Successively applying the same argument to the element circularly next to $q$ in the pictorial representation yields that all rank $1$ elements of type $2$ in $\nc(B_n)$ get fixed by the automorphism $\ncauto$.
	
	The final step is to show that $\ncauto$ fixes all elements of rank $1$. We proceed by induction on the type, where we regard the type $-1$ elements as \enquote{type $n+1$} elements for the proof. This makes sense if one passes to the pictorial representation: in a rank $1$ element of type $k>0$, there are $k-2$ vertices between the endpoints of the edge, and if the type equals $-1$ there are $n-1$ vertices between the endpoints.
	We already showed that the statement is true for type $2$. Let us suppose that it is true for all types up to $k-1\leq n$.  
	Let $p$ and $q$ be elements of rank $1$ with $\typ(p)=2$ and $\typ(q)=k-1$ so that they do not commute. 
	To make the argument more transparent, we use without loss of generality the concrete elements $p=\dka 1\,\;2\dkz$ and $q=\dka 2\,\;k \dkz$. Then the join of $p$ and $q$ is  $x=\dka 1\,\;2\,\;k\dkz$ with $\Typ(x)=2$, and $\bel(x)=\Set{p,q,\dka 1\,\;k\dkz}$. We have that $\ncauto(x)=x$, because $x=p \vee q$ and $p$ and $q$ get fixed by $\ncauto$. Hence $\bel(x)$ is also invariant under $\ncauto$, which implies that $\ncauto$ fixes the element $\dka 1\,\;k\dkz$ of type $k$.
\end{proof}

\begin{cor}
	Every automorphism of $\nc(B_n)$ can be realized as a symmetry of the $2n$-gon of the pictorial representations $\ncb_n$.
\end{cor}

\section{Automorphisms of type $D$}\label{cha:typeD_auto_group}

The aim of this section is to show that the automorphism group of $\nc(D_n)$ is isomorphic to the symmetry group of the $2(n-1)$-gon for $n\neq 4$. In particular we show that $\aut(\nc(D_n)) \cong \D$ when $n$ is odd and that $\D \leq \aut(\nc(D_n))$ is an index $2$ subgroup for even $n$. In the case that $n\neq 4$, we give a group-theoretic description of $\aut(\nc(D_n))$ and show that every automorphism of $\aut(\nc(D_n))$ extends uniquely to an automorphism of $\lam$, provided that $\nc(D_n) \to \lam$ is an embedding. In addition, we give an explicit example of an automorphism of $\nc(D_4)$ that cannot be realized as a symmetry of the hexagon.

Recall that the standard Coxeter element of $W(D_n)$ is given by $\cox = [1\,\; 2\ldots n-1][n]$ and that its order equals $h=2(n-1)$.

\begin{lem}\label{lem:typeD_order_dih_auto_group}
	The dihedral group of automorphisms $\D=\langle \autl, \autr \rangle \subseteq \aut(\nc(D_n))$ has order $2h=4(n-1)$ if $n$ is odd, and it has order $h=2(n-1)$ if $n$ is even.
\end{lem}

\begin{proof}
	We show that the automorphism $\ncauto\coloneqq\autl \circ \autr$ of $\nc(D_n)$ has order $2(n-1)$ if $n$ is odd, and order $(n-1)$ if $n$ is even. Since $\D$ is a dihedral group, and $\autl$ and $\autr$ have order $2$, the statement follows.
	By the proof of \cref{lem:typeB_dihedral_autos} it holds that 
	\[
	\cox^{n-1}=[1\ldots n-1]^{n-1}[n]^{n-1}=[1][2]\ldots[n-1][n]^{n-1}
	\]
	for the Coxeter element $\cox=[1\ldots n-1][n]$ of $W(D_n)$. Hence for even $n$ we get that $\cox^{n-1}=[1][2]\ldots [n-1][n]$. In \cref{lem:typeB_dihedral_autos} we have shown that $\ncauto^{n-1}=\id_{\nc(D_n)}$ and moreover $\ncauto^k\neq\id_{\nc(D_n)}$ for $1\leq k <n-1$. Hence it holds that $\ord(\ncauto)=n-1$. 
	
	If $n$ is odd, then $\cox^{n-1}=[1]\ldots[n-1]$ and for all elements $w\in \nc(D_n)$ it holds that $\ncauto^{n-1}(w)(1)=w(1)$, but $\ncauto^{n-1}\neq \id_{\nc(D_n)}$, since $\ncauto^{n-1}([n])(n)=-n$. Further note that 
	$\ncauto^{k}(w)(1)\neq w(1)$ for all $1\leq k < n-1$, which we showed in \cref{lem:typeB_dihedral_autos}. Combining this we get for all $1\leq k < n-1$ that
	\[
	\ncauto^{k+n-1}(w)(1)=\ncauto^{k}(\ncauto^{n-1}(w)(1))=\ncauto^k(w)(1) \neq w(1)
	\]
	and therefore $\ncauto^k\neq \id_{\nc(D_n)}$ for all $1\leq k < 2(n-1)$. Since the order of the Coxeter element $\cox$ equals $2(n-1)$, it follows that $\ord(\ncauto)=2(n-1)$.
\end{proof}

If $n$ is even, then $\D$ is not the full automorphism group of the pictorial representations of non-crossing $D_n$-partitions. This can be seen in the following example for $\nc(D_4)$ and can be easily generalized for $\nc(D_n)$.

\begin{figure}
	\begin{center}
		\begin{tikzpicture}
		\node[mpunkt, Lightgray] (p1) at (1.5,0){};
		\node[mpunkt] (p2) at (3,0){};
		\node[mpunkt,Lightgray] (p3) at (4.5,0){};
		\node[mpunkt,Lightgray] (p4) at (3,1.5){};
		\node[below = 2mm] at (p1) {$\dka 1\,\; 2\dkz$};
		\node[below = 2mm] at (p2) {$\dka 2\,\; 3\dkz$};
		\node[below = 2mm] at (p3) {$\dka 3\,\; 4\dkz$};
		\node[right = 2mm] at (p4) {$\dka -3\,\; 4\dkz$};
		\draw (p1)--(p2)--(p3)(p2)--(p4);
		\end{tikzpicture}
		\caption{The Coxeter diagram of type $D_4$ with a bipartition of the vertices corresponding to the bipartition $\dka 1\,\;2\dkz\dka3\,\;4\dkz\dka-3\,\;4\dkz\cdot\dka2\,\;3\dkz$ of the Coxeter element $\cox' = [3\,\;1\,\;2][4] \in W(D_4)$.}
		\label{fig:coxeter_graph_d4}
	\end{center}
\end{figure}

\begin{exa}\label{exa:typeD_bipart_and_autos}
	In this example we first construct a bipartition $\l\r$ for the standard Coxeter element $[1\,\;2\,\;3][4]=\dka 1\,\;2\dkz\dka2\,\;3\dkz\dka3\,\;4\dkz\dka-3\,\;4\dkz$ of $W(D_4)$. Then we use this to compute the corresponding automorphisms $\autl$ and $\autr$ of $\nc(D_4)$ and interpret the action of the group of automorphisms $\D$ in the pictorial representations. 
	
	The Coxeter diagram of type $D_n$ is shown in \cref{fig:coxeter_graph_d4}. We partition the simple reflections into the two sets $\Set{\dka 1\,\;2\dkz, \dka3\,\;4\dkz, \dka-3\,\;4\dkz}$ and $\Set{\dka2\,\;3\dkz}$, which yield a reduced decomposition $\dka 1\,\;2\dkz\dka3\,\;4\dkz\dka-3\,\;4\dkz\cdot\dka2\,\;3\dkz$ of the Coxeter element $\cox' = [3\,\;1\,\;2][4]$. Conjugation of $\cox'$ with $w=(1\,\;2\,\;3)$ is $\cox=w\cox'w\inv$ and therefore, a bipartition of the standard Coxeter element is given by $\l\cdot\r=\dka2\,\;3\dkz [1][4]\cdot\dka1\,\;3\dkz$. Let this be the standard bipartition of $[1\,\;2\,\;3][4]$. The pictorial representations of $\l$ and $\r$ are shown in \cref{fig:typeD_left_right_part}. 
	
	The geometry of the pictorial representations of the left part $\l$ and the right part $\r$ of the standard bipartition of $\cox$ hints at how the automorphisms $\autl$ and $\autr$ act on the pictorial representations. If $\autl$ acts as an element of the symmetry group of the hexagon, then it has to fix the edges $\Set{1,4}$ and $\Set{-1,4}$, and the pair of edges corresponding to the paired cycle $\dka 2\,\;3\dkz$. The only element in the symmetry group of the hexagon that has these properties is the reflection with axis through $1$ and $-1$. An easy computation shows that the automorphism $\autl$ indeed acts as this reflection. Similarly, the automorphism $\autr$ acts as the reflection with axis going through $2$ and $-2$. 
	
	Recall that in \cref{exa:typeB_autos_pict} we saw that every automorphism of $\nc(B_n)$ can be represented by \emph{two} elements of the symmetry group of the $2n$-gon. In $\nc(D_n)$, this is not true, as the above analysis of $\autl$ showed. Every symmetry of the hexagon induces an automorphism of $\ncd_n$ by acting on the underlying $2(n-1)$-gon. But unfortunately, not every such automorphism can be described by an element of $\D$ if $n$ is even. In our example, the axes of reflection of $\autl$ and $\autr$ form an angle of $\frac{\pi}{3}$, which is depicted in \cref{fig:typeD_left_right_part}. This means that they generate a dihedral group of order $6$, which is the symmetry group of a triangle. Hence the automorphism $\autl\circ\autr$ corresponds to a rotation by $\frac{2\pi}{3}$, and the rotation by 180 degrees is \emph{not} given by any element of $\D$. Note that rotation by $180$ degrees corresponds to changing the sign of all vertices but $n$.
\end{exa}

\begin{figure}
	\begin{center}
		\begin{tikzpicture}
		\sgsechseckmp \draw (p1) -- (p0) -- (p4) (p5) -- (p6) (p2) -- (p3);
		\node[below = 3mm] at ($(p3)!0.5!(p2)$) {$l$};
		
		\begin{scope}[xshift = 3cm]
		\sgsechseckmp \draw (p1) -- (p3) (p4) -- (p6);
		\node[below = 3mm] at ($(p3)!0.5!(p2)$) {$r$};
		\end{scope}
		
		\begin{scope}[xshift = 6cm]
		\sgsechseckmp \draw (p1) -- (p3) (p4) -- (p6) (p1) -- (p0) -- (p4) (p5) -- (p6) (p2) -- (p3);
		
		\coordinate (r) at ($(p2) - (120:5mm)$);
		\coordinate (l) at ($(p1) + (5mm,0)$);
		
		\draw[Orange, dotted, thick] (l) -- (p1) -- (p0) -- (p4) -- ($(p4) - (5mm,0)$);
		\draw[Orange, dotted, thick] (r) -- (p2) -- (p0) -- (p5) -- ($(p5) + (120:5mm)$);
		
		\draw[Orange,<->] (r) + (-2mm,-2mm)  arc (260:340:5mm) node[midway, below right] {$\varphi_r$}; 
		\draw[Orange,<->] (l) + (1.5mm,-3.5mm)  arc (315:405:5mm) node[midway, right] {$\varphi_l$}; 
		
		\def\r{3mm} 
		\draw[Orange, thick] (\r,0) -- (p0) -- (120:-\r); 
		\clip (0,0) -- (120:-6mm) -- (1,0) -- cycle;
		\draw[Orange, thick] (p0) circle (\r);
		\end{scope}
		\end{tikzpicture}
		\caption{The left part $\l= \dka2\,\;3\dkz [1][4]$ and the right part $\r=\dka1\,\;3\dkz$ of the standard Coxeter element $\cox=[1\,\;2\,\;3][4]\in W(D_4)$ are shown here. The axes of reflection for the corresponding automorphisms $\autl$ and $\autr$ are depicted in orange on the right.}
		\label{fig:typeD_left_right_part}
	\end{center}
\end{figure}
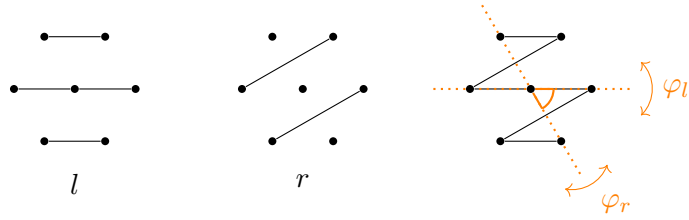

\begin{lem}\label{lem:typeD_autn}
	The assignment $w \mapsto [n]w[n]$ for $w\in W(D_n)$ induces a lattice automorphism $\autn\colon \nc(D_n) \to \nc(D_n)$ of order $2$ if $n>1$. Moreover, $\autn$ is an element of $\D$ if and only if $n$ is odd.
\end{lem}

\begin{proof}
	The Coxeter element $\cox=[1\ldots n-1][n]$ is mapped to itself by $\autn$ and conjugation is an automorphism of the absolute order on $W(D_n)$. This means for $v,w\in \nc(D_n)$ that $v \leq w \leq \cox$ if and only if $[n]v[n] \leq [n]w[n] \leq \cox$, which shows that $\autn\colon \nc(D_n) \to \nc(D_n)$ is an automorphism. Since $\autn(\dka 1 \,\; n \dkz)=\dka -1\,\;n\dkz$ and $[n]^2=\id$, it holds that $\autn \neq \id$ and $\autn^2 = \id$ and thus $\ord(\autn)=2$.
	
	If $n$ is odd, then $\cox^{n-1}=[1]\ldots[n-1]$, and the automorphism $\auto$ that is given by conjugation with $\cox^{n-1}$ equals $\autn$, hence $\autn \in \D$.
	This is true since in a cycle decomposition, $\auto$ changes the sign of all entries \emph{except} of $\pm n$, and $\autn$ \emph{only} changes the sign of $\pm n$. But since every element of $\nc(D_n)$ is a signed permutation, these two kinds of sign changes yield equivalent cycles.
	
	Now suppose that $n$ is even and $\autn\in \D$. Since $\autn$ is given by conjugation, it has to be of the form $(\autl\circ\autr)^k\in\D$ for some $k$. In the proof of \cref{lem:typeD_order_dih_auto_group} we computed the powers of $(\autl\circ\autr)$, which are all different from $\autn$. Hence $\autn \notin \D$.
\end{proof}

\begin{defi}
	We define $\autn\colon \nc(D_n) \to \nc(D_n)$ to be the lattice automorphism of $\nc(D_n)$ given by $w \mapsto [n]w[n]$.
\end{defi}

\begin{cor}
	The automorphism $\autn$ of $\nc(D_n)$ corresponds to rotation by $180$ degrees in the pictorial representation $\ncd_n$.
\end{cor}

\begin{proof}
	In the underlying labeled $2(n-1)$-gon of the pictorial representation of non-crossing $D_n$-partitions, a sign change corresponds to the rotation by 180 degrees. In particular, the midpoint gets fixed. 
	The image of a cycle under $\autn$ is the cycle with the sign of $\pm n$, or equivalently, all signs except that of $\pm n$, changed. Cycles in a disjoint cycle decomposition get mapped to blocks on the same set of elements when passing to the pictorial representation using \cref{thm:typeD_nc_iso}. Hence $\autn$ acts as rotation by $180$ degrees. 
\end{proof}

\begin{prop}\label{prop:typeD_de_dihedral_group}
	The group $\De \leq \aut(\nc(D_n))$ generated by the automorphisms $\autl$ and $\autr \circ \autn$ is a dihedral group of order $2h=4(n-1)$. If $n$ is even, then $\D \leq \De$ is a subgroup of index 2 and if $n$ is odd, then $\De=\D$.
\end{prop}

\begin{proof}
	We show that $\ord(\autr \circ \autn)= 2$ and $\ord(\autl\circ \autr \circ \autn)=h$, which implies that $\De$ has the group presentation
	\[
	\left\langle \autl, \autr \circ \autn \mid \autl^2, (\autr \circ \autn)^2, (\autl\circ \autr \circ \autn)^h \right\rangle
	\]
	and is a dihedral group of order $2h$.
	
	First we show that $\autn$ fixes $\l$ and $\r$. Since
	\[
	[n]\l[n]\cdot [n]r[n] = [1\ldots n-1 ][n]=\cox,
	\]
	we see that either both or none of $\l$ and $r$ commute with $[n]$. Suppose that they do not commute with $[n]$. Then in particular, they do not contain balanced cycles in their disjoint cycle decompositions. This means that both $\l$ and $\r$ have a reflection involving $\pm n$, say $\dka i\,\; n\dkz$ and $\dka j\,\; n\dkz$, respectively. Then $i\neq j$, since $\cox$ has length $n$, and also $i\neq -j$, because otherwise the disjoint cycle decomposition of $\cox$ would contain $[j][n]$. Since $\l$ and $\r$ have order $2$, all paired cycles in their disjoint cycle decomposition commute. In particular, no cycle different from $\dka i\,\;n\dkz$ and $\dka j\,\;n\dkz$ may involve $\pm n$. But then $\dka i\,\; n\dkz$ and $\dka j\,\; n\dkz$ are the only cycles in $\l\r$ involving $n$, which means that $n$ does not get mapped to $-n$ by $\l\r=\cox$, which is a contradiction. Hence both $\l$ and $\r$ commute with $[n]$.
	
	It holds that $\autr \circ \autn \neq \id$ and since $\autn(\r)=[n]\r[n]=\r$, it follows that
	\[
	(\autr \circ \autn)^2(w)=\r[n]\r[n]w[n]\r[n]\r =\r^2 w\r^2 = w
	\]
	and hence $\ord(\autr \circ \autn)=2$. Recall that $h=2(n-1)$. Since $\autl$ and $\autr$ commute with $\autn$, we have that 
	\[
	(\autl\circ \autr \circ \autn)^h=(\autl\circ \autr)^{2(n-1)} \circ \autn^{2(n-1)} = \id,
	\]
	since $\ord(\autl\circ \autr)$ equals $n-1$ if $n$ is even and $2(n-1)$ if $n$ is odd, and $\ord(\autn)=2$.
	
	If $n$ is odd, then $(\autl\circ \autr)^{k}=\autn$ for $1\leq k < h$ if and only if $k=n-1$, what we showed in the proof of \cref{lem:typeD_order_dih_auto_group}. But $\autn^{n-1}=\id$, so the order of $\autl\circ \autr\circ \autn$ is $h$ for odd $n$. 
	
	If $n$ is even, then the order of $\autl\circ \autr$ equals $n-1$ and since $\autn^{n-1}=\autn$, it follows that the order of $\autl\circ \autr\circ\autn$ equals $h$ for even $n$. 
	
	In particular, we showed that $(\autl\circ \autr \circ \autn)^{n-1}=\autn$, which implies that
	\[
	\autr = \autr \circ \autn \circ \autn =  \autr \circ \autn \circ (\autl\circ \autr \circ \autn)^{n-1}
	\]
	is an element of $\D$. Hence the group $\D$ generated by $\autl$ and $\autr$ is a subgroup of $\De$. If $n$ is even, then $\#\D=2(n-1)=h$, and if $n$ is odd, then $\#\D=2h$ by \cref{lem:typeD_order_dih_auto_group} and therefore $\D\leq\De$ has index 2 for even $n$, and $\D=\De$ for odd $n$.
\end{proof}

The rest of the section is devoted to prove that the automorphism group of $\nc(D_n)$ equals the dihedral group $\De$ if $n \neq 4$.

\begin{thm}\label{thm:typeD_full_aut}
	The automorphism group $\aut(\nc(D_n))$ of the non-crossing partition lattice $\nc(D_n)$ is the dihedral group $\De$ of lattice automorphisms of order $2h=4(n-1)$ if $n \neq 4$. If $n=4$, then $\De$  is a proper subgroup of  $\aut(\nc(D_4))$.
\end{thm}

The idea of the proof is the same as in type $A$ and $B$, but the technical details are more involved. We start with defining types for elements of rank $1$ and $2$. As before, we use $p$ and $q$ for rank $1$ elements and $x$, $y$ and $z$ for elements of rank $2$.

Let $p\in\nc(D_n)$ be of rank $1$. The \emph{type} of $p$ is defined as
\[
\typ(p)=
\begin{cases}
j-i+1& \text{ if }p\equiv\dka i\,\; j\dkz \text{ for }0<i<j<n,\\
n-i-j& \text{ if }p\equiv\dka i\,\; j\dkz \text{ for }0<-i<j<n,\\
n &\text{ if } p\equiv\dka i \,\;n \dkz \text{ for } |i|<n.
\end{cases}
\]
Note that $2 \leq \typ(p)\leq n$ holds for all reflections $p$, and if $p$ does not involve $n$, then the type of $p\in\nc(D_n)$ equals the type of $p\in \nc(B_{n-1})$. As before, the type measures how \enquote{diagonal} the edge corresponding to the reflection lies in the $2(n-1)$-gon. Type $n$ corresponds to an edge adjacent to the midpoint.

If $x\in \nc(D_n)$ is an element of rank $2$, then the \emph{type} of $x$ is defined as
\[
\Typ(x)=
\begin{cases}
-1 &\text{ if }x= [a][n],\\
1 &\text{ if }x= \dka a\,\;b\,\;n\dkz,\\
2&\text{ if }x= \dka a\,\;b\,\;c\dkz, \\
3&\text{ if }x= \dka a\,\;b\dkz\dka c\,\;d\dkz,\\
4&\text{ if }x=  \dka a\,\;b\dkz\dka c\,\;n\dkz
\end{cases}
\]
for appropriate $a,b,c,d\in\Set{ \pm 1,\ldots, \pm (n-1) }$. The type of a rank $2$ element records, up to sign, the number of non-trivial blocks in the pictorial representation. \cref{fig:typeD_types_rank_2} shows the different types of rank $2$ elements in $\ncd_5$.

\begin{obs}\label{obs:typeD_card_bel_rk_2}
	Let $x\in\nc(D_n)$ be of rank $2$. Then the cardinality of the covered set of $x$ depends on the type of $x$ and is
	\[
	\#\bel(x)=
	\begin{cases}
	2&\text{ if }T(x)=-1,3,4,\\
	3&\text{ if }T(x)=1,2,
	\end{cases}
	\] 
	which can be checked by looking at \cref{fig:typeD_types_rank_2}.
\end{obs}

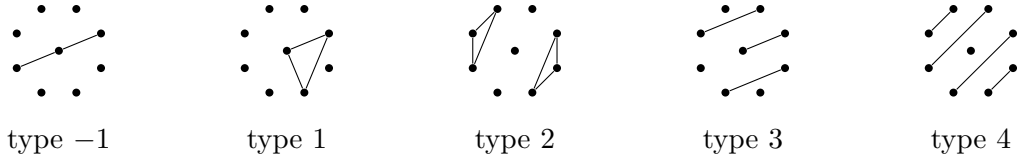
\begin{figure}%
	\begin{center}
		\begin{tikzpicture}
		\begin{scope}[xshift = 0cm]
		\achteckmp
		\draw(p1)--(p0)(p5)--(p0);
		\node at (0,-1.2) {type $-1$};
		\end{scope}
		
		\begin{scope}[xshift = 3cm]
		\achteckmp
		\draw(p1)--(p3)(p3)--(p0)(p1)--(p0);
		\node at (0,-1.2) {type $1$};
		\end{scope}
		
		\begin{scope}[xshift = 6cm]
		\achteckmp
		\draw(p1)--(p3)(p3)--(p2)(p1)--(p2)(p5)--(p6)(p6)--(p7)(p7)--(p5);
		\node at (0,-1.2) {type $2$};
		\end{scope}
		
		\begin{scope}[xshift = 9cm]
		\achteckmp
		\draw(p1)--(p0)(p2)--(p4)(p8)--(p6);
		\node at (0,-1.2) {type $3$};
		\end{scope}
		
		\begin{scope}[xshift = 12cm]
		\achteckmp
		\draw(p1)--(p4)(p2)--(p3)(p5)--(p8)(p6)--(p7);
		\node at (0,-1.2) {type $4$};
		\end{scope}
		\end{tikzpicture}
		\caption{Five non-crossing $D_5$-partitions of rank $2$ with increasing type from left to right.}%
		\label{fig:typeD_types_rank_2}%
	\end{center}
\end{figure}

\begin{lem}\label{lem:typeD_number_rk2_elements_types}
	Let $p\in\nc(D_n)$ be of rank $1$. If $p$ has type $\typ(p)=k<n$, then $\ab(p)$ has
	\begin{itemize}
		\item no elements of type $-1$,
		\item $2$ elements of type $1$,
		\item $2n-k-4$ elements of type $2$,
		\item $2(n-k-1)$ elements of type $3$, and
		\item $\binom{k-2}{2}+2\binom{n-k-1}{2}$ elements of type $4$.
	\end{itemize}
	In total, the cover set has cardinality $\#\ab(p)=\frac{1}{2}(2n^2+3k^3-4nk+2n-5k+1)$.
	
	If $p$ has type $\typ(p)=n$, then $\ab(p)$ has
	\begin{itemize}
		\item $1$ element of type $-1$,
		\item $2(n-2)$ elements of type $1$, 
		\item $\binom{n-2}{2}$ elements of type $3$, and
		\item no elements of type $2$ and $4$.
	\end{itemize}
	In total, the cover set has cardinality $\#\ab(p)=\frac{1}{2}n(n-1)$.
\end{lem}

\begin{proof}
	We proceed in the same way as in the proof of \cref{lem:typeB_number_rk2_elements_types} and compute the possibilities to construct a cover of the particular type. In this proof, we are using the same symbols for elements in $\nc(D_n)$ and their pictorial representation in $\ncd_n$. 
	
	Let $p\in \nc(D_n)$ be of rank $1$ and type $k<n$. We may assume that $p=\dka 1\,\;k\dkz$. There are $k-2$ vertices below the edge $e\coloneqq\Set{1,k}$ and $n-1-k$ vertices above $e$, which is depicted in \cref{fig:typeD_above_below_vertices}.
	
	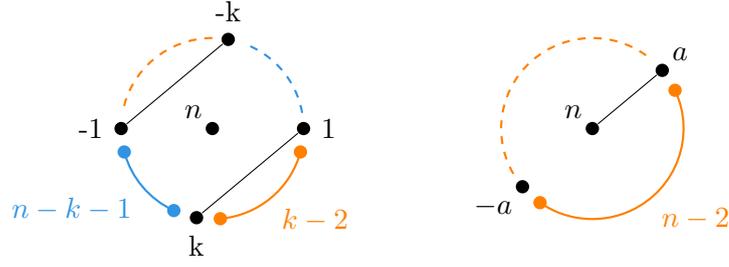
\begin{figure}%
		\begin{center}
			\begin{tikzpicture}
			\def\r{1.2} 
			\def\w{35} 
			
			\node[mpunkt] (O) at (0,0) {};		
			
			\node[mpunkt] (1) at (0:\r) {};
			\node[mpunkt] (-1) at (180:\r) {};
			\node[mpunkt] (-k) at (80:\r) {};
			\node[mpunkt] (k) at (260:\r) {};
			
			\foreach \n/\pos in {1/right,-1/left,k/below,-k/above}
			\node[\pos=1mm] at (\n) {\n};
			\node[above left] at (O) {$n$};
			
			\draw (-1) -- (-k) (k) -- (1);
			\begin{scope}[Blue,thick]
			\node[mpunkt] at (195:\r) {};
			\node[mpunkt] at (245:\r) {};
			\draw (195:\r) arc (195:245:\r) node[midway, below left] {$n-k-1$};
			\draw[dashed] (10:\r) arc (10:70:\r);
			\end{scope}
			
			\begin{scope}[Orange,thick]
			\node[mpunkt] at (275:\r) {};
			\node[mpunkt] at (345:\r) {};
			\draw (275:\r) arc (275:345:\r) node[midway, below right] {$k-2$};
			\draw[dashed] (90:\r) arc (90:170:\r);
			\end{scope}
			
			\begin{scope}[xshift = 5cm]		
			\node[mpunkt] (O) at (0,0) {};
			
			\node[mpunkt] (P) at (40:\r) {};
			\draw (O) -- (P);
			\node[mpunkt] (Q) at (220:\r) {};
			
			\node[above left] at (O) {$n$};
			\node[above right] at (P) {$a$};
			\node[below left] at (Q) {$-a$};
			
			
			\begin{scope}[Orange,thick]
			\node[mpunkt] at (385:\r) {};
			\node[mpunkt] at (235:\r) {};
			\draw (385:\r) arc (385:235:\r) node[midway, below right] {$n-2$};
			\draw[dashed] (210:\r) arc (210:50:\r);
			\end{scope}
			\end{scope}
			\end{tikzpicture}
			\caption{Available vertices for constructing covers for type $k<n$ are shown on the left, and for type $n$ on the right. Vertices above are shown in blue and vertices below in orange. The dashed lines indicate that every vertex has an opposite.}%
			\label{fig:typeD_above_below_vertices}%
		\end{center}
	\end{figure}
	
	It is clear that there cannot be an element of type $-1$ covering $p$.
	
	Let $x\in \ab(p)$ be of type $1$. Then $x$ has to be a cycle involving $n$, for which we have exactly two choices.
	
	To construct a cover of $p$ with two non-trivial blocks, we have to join a vertex different from the midpoint to the edge $e$. Either we can join a vertex below, or we can join a vertex above to $e$ or $-e$. In total we have $k-2+2(n-k-1)=2n-k-4$ possibilities.
	
	If $x\in \ab(p)$ has type $3$, then it is the join of $p$ with an edge adjacent to the midpoint not crossing $p$. There are $2(n-k-1)$ possibilities, since a  vertex above or its negative can be connected to the midpoint.
	
	To construct a cover of $p$ of type $4$, we have to add a block pair that does not cross $p$. Either we connect two vertices below, for which we have $\binom{k-2}{2}$ ways, or we choose two vertices among the vertices above $e$, for which we have $\binom{n-k-1}{2}$ ways. If we have chosen $a$ and $b$ from above, then we can form the block pair $\pm \Set{a,b}$ or the block pair $\pm \Set{-a, b}$. Hence, there are $\binom{k-2}{2}+2\binom{n-k-1}{2}$ elements of type $4$ in $\ab(p)$.
	
	Now let $p$ be of type $n$, see also \cref{fig:typeD_above_below_vertices}. We may assume that $p=\dka 1\,\;n\dkz$. 
	
	Since rank $2$ elements of type $2$ and $4$ do not involve the midpoint $n$, such elements cannot cover $p$.
	
	There is exactly one element that covers $p$ and has type $-1$, namely $[1][n]$. 
	
	To construct a covering element with exactly one non-trivial block, we can choose any vertex except $1$,$-1$ and $n$, and join it to $p$. In total, there are $2(n-2)$ possibilities to construct a cover of type $1$.
	
	An element $x\in \ab(p)$ of type $3$ arises from $p$ by adding a pair of edges $\pm f$ that does not cross $p$. In order to be non-crossing, the vertices for the edge $f$ have to be chosen among $2, 3, \ldots, n-1$, for which there are $\binom{n-2}{2}$ choices.
\end{proof}

\begin{lem}\label{lem:typeD_type_preserving_rank_2}
	Let $x\in\nc(D_n)$ be of rank $2$ and type $\Typ(x)=1$ or $\Typ(x)=2$. If $n\neq 4$, then for all automorphisms $\auto \in \aut(\nc(D_n))$ it holds that $\Typ(\auto(x))=\Typ(x)$.
\end{lem}

\begin{proof}
	Let $x,y\in \nc(D_n)$ be of rank $2$ with $\Typ(x)=1$ and $\Typ(y)=2$ such that $\auto(x)=y$ for an automorphism $\auto\in\aut(\nc(D_n))$. We show that then $n=4$, which implies that the types of rank $2$ elements are preserved for $n \neq 4$. 
	
	Let the elements covered by $x$ be 
	$
	\bel(x)=\Set{a,b,c}
	$
	such that $\typ(a)=\typ(b)=n$ and $\typ(c)=k<n$, and the elements covered by $y$ be
	$
	\bel(y)=\Set{d,e,f}
	$
	with $\typ(d)=m<n$, $\typ(e)=l<n$ and $\typ(f)=m+l-1\coloneqq r<n$. This is the generic situation and the pictorial representations are shown in \cref{fig:typeD_proof_type_preserving_rank_2}.
	
	\begin{figure}
		\begin{center}
			\begin{tikzpicture}
			\def\r{0.7cm} 
			
			\foreach \x/\y/\n in 	{1/0/2,
				2/0/3,
				2/1.5/1,
				3/0/4,
				4.5/0/6,
				5.5/0/7,
				5.5/1.5/5,
				6.5/0/8}
			\coordinate (n\n) at (2.2*\x,2*\y);
			
			\foreach \n in {1,...,8}{
				\draw[Lightgray] (n\n) circle (\r);
			}
			\foreach \n in {2,3,4} \draw[shorten <=\r*1.7, shorten >= \r*1.6,Lightgray] (n1) -- (n\n);
			
			\foreach \n in {6,7,8} \draw[shorten <=\r*1.7, shorten >= \r*1.6,Lightgray] (n5) -- (n\n);
			
			\foreach \n in {1,...,8} \coordinate (1\n) at ($(n\n)+(0:\r)$) {}; 
			\foreach \n in {1,...,8} \coordinate (k\n) at ($(n\n) + (270:\r)$) {}; 
			\foreach \n in {1,...,8} \coordinate (m\n) at ($(n\n) + (300:\r)$) {}; 
			\foreach \n in {1,...,8} \coordinate (r\n) at ($(n\n) + (240:\r)$) {}; 
			\foreach \n in {1,...,8} \coordinate (l\n) at ($(n\n) + (180:\r)$) {}; 
			\foreach \n in {1,...,8} \coordinate (ol\n) at ($(n\n) + (120:\r)$) {}; 
			\foreach \n in {1,...,8} \coordinate (or\n) at ($(n\n) + (60:\r)$) {}; 
			
			\coordinate (oben4) at ($(n4)+(90:\r)$) {};
			\coordinate (links4) at ($(n4)+(180:\r)$) {};
			\node[kpunkt] at (oben4){};
			\node[kpunkt] at (links4){};
			\draw (oben4) -- (links4);
			
			\foreach \l in {1,n,k} \node[kpunkt] at (\l1) {};
			\draw (n1)--(11)--(k1)--(n1);
			\foreach \l in {n,1} \node[kpunkt] at (\l2) {}; \draw (n2)--(12);
			\foreach \l in {n,k} \node[kpunkt] at (\l3) {}; \draw (n3)--(k3);
			\foreach \l in {k,1} \node[kpunkt] at (\l4) {}; \draw(k4)--(14);
			\foreach \l in {1,m,r,l,ol,or} \node[kpunkt] at (\l5) {};
			\draw (r5)--(m5)--(15)--cycle; \draw (ol5)--(or5)--(l5)--cycle;
			\foreach \l in {1,m,l,ol} \node[kpunkt] at (\l6) {}; \draw (m6)--(16); \draw (l6)--(ol6);
			\foreach \l in {r,m,ol,or} \node[kpunkt] at (\l7) {}; \draw (m7)--(r7); \draw (ol7)--(or7);
			\foreach \l in {1,r,l,or} \node[kpunkt] at (\l8) {}; \draw (r8)--(18); \draw (l8)-- (or8);
			
			\foreach \n in {1,2,3} \node[above] at (n\n) {$n$};
			\foreach \n in {1,2,4,5,6,8} \node[right] at (1\n) {$1$};
			\foreach \n in {1,3,4} \node[below] at (k\n) {$k$};
			\foreach \n in {5,6,7} \node[below right] at (m\n) {$m$};
			\foreach \n in {5,7,8} \node[below left] at (r\n) {$r$};
			
			\foreach \n/\lab in {2/a,3/b,4/c,6/d,7/e,8/f}
			\node at ($(n\n) + (0,-\r*2.5)$) {$\lab$};
			\foreach \n/\lab in {1/x,5/y} 
			\node at ($(n\n) + (-\r*2,0)$) {$\lab$};
			\end{tikzpicture}
			\caption{Generic rank $2$ elements $x$ and $y$ of type $1$ and $2$, respectively, and the corresponding covered elements $a$, $b$ and $c$, and $d$, $e$ and $f$. }
			\label{fig:typeD_proof_type_preserving_rank_2}
		\end{center}
	\end{figure}
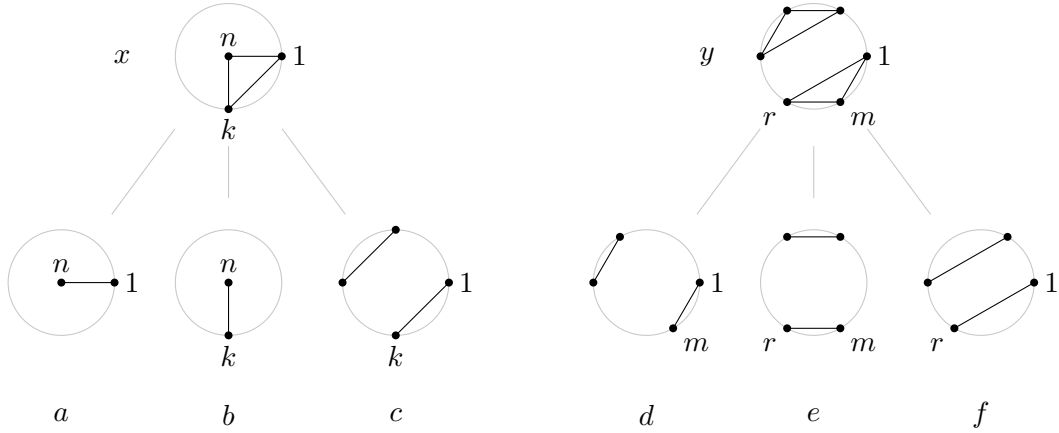
	
	Since $\auto(x)=y$ and therefore $\bel(\auto(x))=\bel(y)$, it holds that 
	\[
	\Set{\auto(a),\auto(b),\auto(c)}=\Set{d,e,f}.
	\]
	To find out how $\auto$ maps $\bel(x)$ to $\bel(y)$, note that a necessary condition is that the cover set of the image has the same cardinality as the cover set of the preimage. By \cref{lem:typeD_number_rk2_elements_types} we have that a rank $1$ element of type $n$ is covered by $\frac{1}{2}n(n-1)$ elements, whereas  an element of type $i<n$ is covered by $\frac{1}{2}(2n^2+3i^3-4ni+2n-5i+1)$ elements. Since 
	\[
	\frac{1}{2}(2n^2+3i^3-4ni+2n-5i+1) = \frac{1}{2}n(n-1)
	\]
	if and only if $i=\frac{n+2}{3}\in\Z$, we know that two elements of $\bel(y)$ have type $\frac{n+2}{3}$. Since the types of $d$ and $e$ are strictly less than the type of $f$, we have that $\typ(d)=\typ(e)=\frac{n+2}{3}$ and $\typ(f)=\frac{2n+1}{3}$. Further it holds that $\auto(c)=f$ and, without loss of generality, that $\auto(a)=d$ and $\auto(b)=e$.
	
	After having identified that $a$ gets mapped to $d$, we now take a closer look at the cover sets of $a$ and $d$. Every rank $2$ element covers either $2$ or $3$ reflections. Since cardinalities of covered sets are preserved under automorphisms, it holds that
	\[
	\#\Set{z\in\ab(a)\str \#\bel(z)=3} = \#\Set{z\in\ab(d)\str \#\bel(z)=3}.
	\]
	
	By \cref{obs:typeD_card_bel_rk_2} the elements of rank $2$ that cover $3$ elements have type $1$ and $2$. Hence we have to count the elements of type $1$ and $2$ in $\ab(a)$ and $\ab(d)$, respectively. Using \cref{lem:typeD_number_rk2_elements_types} yields that
	$\#\Set{z\in\ab(a)\str \#\bel(z)=3} = 2(n-2)$
	and that 
	$\#\Set{z\in\ab(d)\str \#\bel(z)=3} =\frac{5n-8}{3}$. Hence we have $6(n-2)=5n-8$, which is equivalent to $n=4$.  
\end{proof}

\begin{lem}\label{lem:typeD_nc_auto_preserves_type_rk_1}
	The type of rank $1$ elements of $\nc(D_n)$ is preserved under all lattice automorphisms of $\nc(D_n)$ for $n\neq 4$.
\end{lem}

\begin{proof}
	Let $p\in\nc(D_n)$ be a reflection and let $\auto\in\aut(\nc(D_n))$ be an automorphism. We show that the rank of $\auto(p)$ equals the rank of $p$. By \cref{lem:typeD_type_preserving_rank_2} it holds for $n\neq 4$ that $x\in \ab(p)$ has type $2$ if and only if $\auto(x)\in\ab(\auto(p))$ has type $2$. Hence we get the equality
	\[
	\#\Set{x\in\ab(p)\str \Typ(x)=2 } = \#\Set{y\in\ab(\auto(p))\str \Typ(y)=2 },
	\]
	which is equivalent to $2n-4-\rk(p)=2n-4-\rk(\auto(p))$ for $\rk(p)<n$ by \cref{lem:typeD_number_rk2_elements_types}. From the same lemma it follows for $\rk(p)=n$ that $\#\Set{x\in\ab(p)\str \Typ(x)=2 } = 0$. This implies that $\typ(\auto(p))=\typ(p)$.
\end{proof}

\begin{lem}\label{lem:typeD_exo_auto}
	If $n=4$ then there exists an involutive automorphism $\exo \in \aut(\nc(D_4))$ that is not an element of $\De$ and does not preserve types.
\end{lem}

\begin{proof}
	In \cref{exa:typeD_bipart_and_autos} we showed that for $n=4$, the generators of $\D$ act as reflections on the hexagon. Together with $\autn$, which is rotation by 180 degrees by \cref{lem:typeD_autn}, they generate the group $\De$. Hence $\De$ is the symmetry group of the hexagon. In particular, every element in $\De$ preserves the type.
	
	The exotic automorphism $\exo$ is explicitly given by the following pairs of reflections that are exchanged by $\exo$:
	\begin{align*}
	&\dka 1\,\;2\dkz \leftrightarrow \dka 1\,\;4\dkz, \dka 2\,\;3\dkz \leftrightarrow \dka 3\,\;4\dkz, \dka 1\,\;-3\dkz \leftrightarrow \dka-2\,\;4\dkz,
	\dka -1\,\;4\dkz \leftrightarrow \dka-3\,\;4\dkz,\\
	&\dka 2\,\;4\dkz \leftrightarrow \dka 2\,\;4\dkz, \dka 2\,\;-3\dkz \leftrightarrow \dka 1\,\;-2\dkz, \dka 1\,\;3\dkz \leftrightarrow \dka 1\,\;3\dkz.
	\end{align*}
	By \cref{lem:lattice_auto_join_interchange} the automorphism $\exo$ of the graded lattice $\nc(D_4)$ is completely determined by the images of rank $1$. That this assignment is indeed an automorphism can be easily checked by looking at \cref{fig:typeD_exotic_auto}, whose columns consist of those pairs of elements of $\nc(D_4)$ that are exchanged by $\exo$. Note that the order of reflections in the first column of \cref{fig:typeD_exotic_auto} corresponds to the order of reflections in the above list. \cref{fig:typeD_exotic_auto} also shows that the types of neither rank 1 nor rank 2 elements are preserved. Also in rank $3$ the \enquote{shapes} are not preserved under $\exo$, as zero blocks are mapped to non-zero blocks.
\end{proof}

\begin{figure}
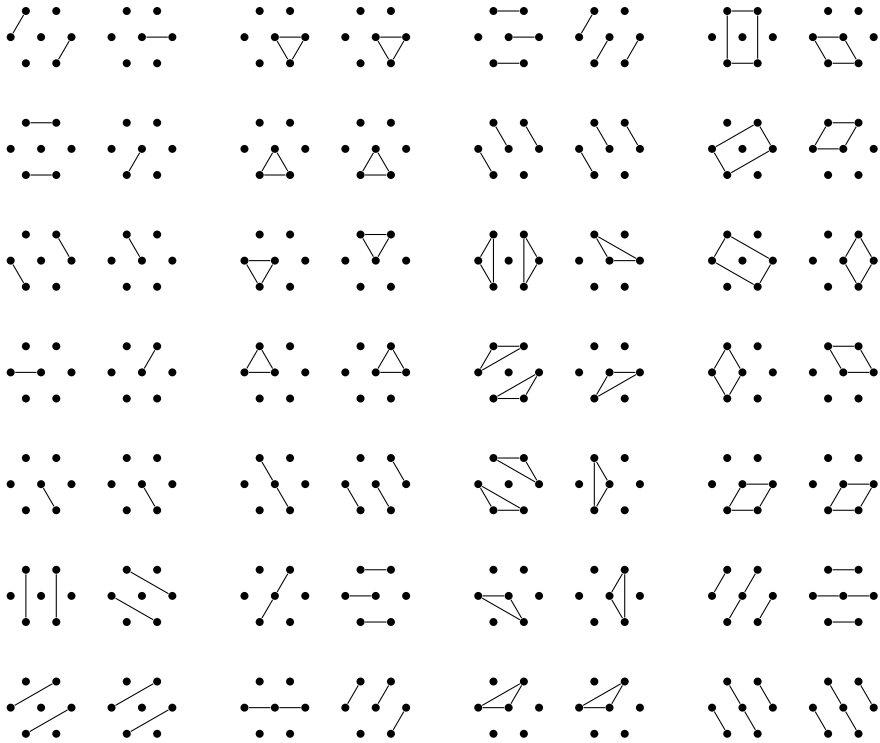
%
	\begin{center}
		\renewcommand{\arraystretch}{3.1}
		\begin{tabular}{ccccccccccc}
			\pDez & \pDev && \pDezv & \pDezv && \pDevuzd & \pDdvuez && \pDZzd & \pDzdmev\\
			
			\pDzd & \pDdv && \pDzdv & \pDzdv && \pDmzvuemd & \pDmzvuemd && \pDZed & \pDmemzmdv\\
			
			\pDemd & \pDmzv && \pDdmev & \pDmzmdv && \pDezmd & \pDmzev && \pDZez & \pDmdezv\\
			
			\pDmev & \pDmdv && \pDmemzv & \pDmdev && \pDezd & \pDedv && \pDdmemzv & \pDmzmdev\\
			
			\pDzv & \pDzv && \pDzmz & \pDzvuemd && \pDemzmd & \pDdmzv && \pDezdv & \pDezdv\\
			
			\pDzmd & \pDemz && \pDdmd & \pDmevuzd && \pDzmev & \pDmdzv && \pDdmduez & \pDemeuzd\\
			
			\pDed & \pDed && \pDeme & \pDmdvuez && \pDmemdv & \pDmemdv && \pDzmzuemd &\pDzmzuemd
			
		\end{tabular}
	\end{center}
	\caption{These pairs of non-crossing $D_4$-partitions are exchanged by the exotic involutive automorphism $\exo$.}%
	\label{fig:typeD_exotic_auto}%
\end{figure}

\begin{oq}
	What is $\aut(\nc(D_4))$? 
\end{oq}

\begin{rem}
	Although the exotic automorphism $\exo$ is not given by a reflection in the hexagon, it seems that is has something to do with the reflection with axis through $-2$ and $2$, call it $\auto$. This becomes in particular clear when considering the fixed points of both $\exo$ and $\auto$. All fixed points of $\exo$ are among the fixed points of $\auto$. Moreover, every fixed point of $\auto$ gets mapped  by $\exo$ to another of fixed point of $\auto$.
	We are not aware of any group-theoretic description of this exotic automorphism.
\end{rem}

\begin{proof}[Proof of \cref{thm:typeD_full_aut}]
	The strategy for the proof in the case $n\neq 4$ is the same as for type $B$ in \cref{thm:typeB_full_aut}. We show that every automorphism in $\auto\in\aut(\nc(D_n))$ can be expressed as a composition of symmetries of the $2(n-1)$-gon. This implies that $\aut(\nc(D_n))$ equals the symmetry group of the $2(n-1)$-gon, which is a group of order $4(n-1)$. Since the subgroup $\De \leq \aut(\nc(D_n))$ also has order $4(n-1)$ by \cref{lem:typeD_order_dih_auto_group}, it follows that $\aut(\nc(D_n))=\De$. We regard elements of $\nc(D_n)$ as non-crossing $D_n$-partitions in $\ncd_n$ and vice versa by using \cref{thm:typeD_nc_iso}.
	
	Let $\auto\in\aut(\nc(D_n))$ be an arbitrary automorphism. We show that, after composing $\auto$ with symmetries of the $2(n-1)$-gon, this automorphism fixes all rank $1$ elements. By \cref{lem:lattice_auto_join_interchange} this shows that $\auto$ is the identity on $\nc(D_n)$. 
	
	Since $\auto$ preserves the rank of reflections by \cref{lem:typeD_nc_auto_preserves_type_rk_1}, we can use the same argumentation as in the proof of \cref{thm:typeB_full_aut} to show that $\auto$, or possibly a composition of $\auto$ with symmetries, fixes all rank $1$ elements of type $2$. 
	Further, $\auto$ fixes all reflections with type different from $n$ by using the same induction argument as for type $B_n$ in \cref{thm:typeB_full_aut}. Hence we can suppose that all reflections of type less than $n$ are fixed by the automorphism $\auto$.
	
	What is left to show is that $\auto$ fixes the reflections of type $n$. For this, consider the two reflections $a=\dka 1\,\;n\dkz$ and $b=\dka 2\,\;n\dkz$ of type $n$. Their join is the rank $2$ element $x=\dka 1\,\;2\,\;n \dkz$ of type $1$. Note that $p=\dka 1\,\;2\dkz$ is also covered by $x$. Since $p$ is fixed by $\auto$ by assumption and $\Typ(\auto(x))=1$ by \cref{lem:typeD_type_preserving_rank_2}, it follows that $\auto(x)=x$ or $\auto(x)=-x$, where $ -x \coloneqq \dka -1\,\;-2\,\;n \dkz$. In general, if $z$ is a paired cycle involving $n$, we denote by $-z$ the cycle that arises from $z$ by changing all signs of entries different from $n$. 
	We can write
	\[
	\auto(x) = \auto(p) \vee \auto(b) = p \vee \auto(b)
	\]
	and it follows that $\auto(b)\in \Set{-a,a,-b,b}$. The same considerations are true for the cycle $y=\dka 2\,\;3\,\;n \dkz$ with $c=\dka 3\,\;n\dkz$ and this implies that $\auto(b) \in \Set{-b,b,-c,c}$. Hence it holds that $\auto(\dka i\,\;n\dkz)\in \Set{\dka- i\,\;n\dkz, \dka i\,\;n\dkz}$ for all $i\in \Set{\pm1, \ldots, \pm (n-1)}$ by circularly applying the same argument to all type $n$ reflections. 
	
	Suppose that $\auto(b)=b$. Then $x=\auto(x)=p\vee \auto(a)$ and this means that $\auto(a)=a$, since $-a$ is not covered by $x$ in $\nc(D_n)$. Applying repeatedly the argument to the circular next rank $2$ block and its covered reflections of type $n$, for instance $y$ with reflections $b$ and $c$, yields that all reflections of type $n$ are fixed. 
	
	Now suppose that $\auto(b)=-b$. Analogous to the former case, we get that $-x=p\vee \auto(a)$ and hence $\auto(a)=-a$. Again by repeatedly applying the procedure to the circularly next block yields that $\auto(\dka i\,\; n\dkz)=\dka -i\,\;n\dkz$ for all $i\in \Set{\pm 1,\ldots, \pm (n-1)}$. The composition of $\auto$ with the rotation by 180 degrees then fixes all type $n$ reflections and acts trivially on all cycles not involving $n$.
	
	If $n=4$, then exists an exotic automorphism $\exo\in\aut(\nc(D_n))$, which is not an element of $\De$ by \cref{lem:typeD_exo_auto}.
\end{proof}

\begin{cor}
	Every automorphism of $\nc(D_n)$ can be realized as a symmetry of the $2(n-1)$-gon of the pictorial representations $\ncd_n$ if and only if $n\neq 4$.
\end{cor}

\subsection{Extending the automorphisms}
In \cref{sec:induced_autos} we showed that all automorphisms of the group $\D \leq \aut(\nc)$ extend to lattice automorphisms of $\lam$ whenever $\emb\colon\nc \to \lam$ is an embedding. In type $D$ we get the following.

\begin{lem}
	The automorphism $\autn \in \aut(\nc(D_n))$ extends uniquely to a lattice automorphism of $\lam$ if $\emb\colon\nc(D_n) \to \lam$ is a poset embedding.
\end{lem}

\begin{proof}
	We have to check that the statements of \cref{sec:induced_autos} are also true for the automorphism $\autn$. 
	The analog of \cref{lem:ncauto_red_decomp} is true, since for every reduced decomposition $t_1\ldots t_k$ for $w\in\nc(D_n)$ we have that $\autn(t_1)\ldots \autn(t_k)$ is a reduced decomposition of $\autn(w)$. Consequently, also \cref{lem:defi_lamauto} holds for $\autn$. Using these statements, the proofs of \cref{lem:lamauto_alpha_t} and \cref{thm:autos_commute} work analogously for $\autn$ and hence $\autn$ extends uniquely to a lattice automorphism of $\lam$.
\end{proof}

\begin{cor}
	Every automorphisms of $\aut(\nc(D_n))$ extends uniquely to a lattice automorphism of $\lam$ for $n\neq 4$. If $n=4$, then every automorphisms of $\De\leq\aut(\nc(D_n))$ extends uniquely to a lattice automorphism of $\lam$.
\end{cor}

\begin{proof}
	The group $\De$ is generated by $\autl$ and $\autr \circ \autn$. Since $\autl, \autr\in\D$ extend uniquely to $\lam$ by \cref{thm:autos_commute}, it follows by the above lemma that all automorphisms of $\De$ extend uniquely to $\lam$.
\end{proof}

\begin{oq}
	Does the exotic automorphism $\zeta$ of $\nc(D_4)$ from \cref{lem:typeD_exo_auto} extend to the ambient lattice $\lam$?
\end{oq}

\section{Anti-automorphisms}

The aim of this section is to investigate the anti-automorphisms of the non-crossing partition lattices. Recall that a lattice anti-automorphism of a lattice $L$ is a order-reversing bijection $L \to L$ whose inverse is order-reversing as well, and a skew-automorphism is either an automorphism or an anti-automorphism. Moreover, we construct a dihedral group of skew-automorphisms. This group extends to a group of automorphisms of the ambient lattice $\lam$.

Biane shows that the group of skew-automorphisms of $\ncp_n$ is isomorphic to the symmetry group of the $2n$-gon \cite{biane}. For $n=4$, this can easily be veryfied by looking at \cref{fig:anti_auto_NCP4}, which shows the order complex of $\ncp_4$ in an octagonal version.

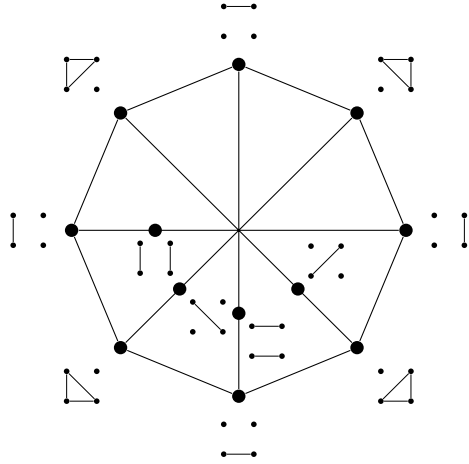
\begin{figure}
	\begin{center}
		\begin{tikzpicture}
		\def\r{2.2}
		\foreach \w in {1,...,8}
		\node (q\w) at (-\w  * 360/8 + 45   : \r) [mpunkt] {};
		\draw (q1) -- (q2) -- (q3) -- (q4) -- (q5) -- (q6) -- (q7) -- (q8)
		-- (q1) (q1) --(q5) (q2) -- (q6) (q3) -- (q7) (q4) -- (q8);
		for-Schleife:
		\foreach \m/\n in {2/6,3/7,4/8,5/1}
		\node (\m\n) at ($(q\m)!0.25!(q\n)$) [mpunkt] {};
		\def\a{0mm} 
		\foreach \m/\n/\p/\d in {    2/6/above right/\scalebox{0.7}{\ped},
			3/7/below right/\scalebox{0.7}{\pevuzd},
			4/8/below right/\scalebox{0.7}{\pzv},
			5/1/below/\scalebox{0.7}{\pezudv}}
		\node[\p] (\m\n) at ($(q\m)!0.25!(q\n)$) {\d};
		\def\b{2mm} 
		\foreach \m/\p/\d in {    1/right/\scalebox{0.7}{\pez},
			2/below right/\scalebox{0.7}{\pezd},
			3/below/\scalebox{0.7}{\pzd},
			4/below left/\scalebox{0.7}{\pzdv},
			5/left/\scalebox{0.7}{\pdv},
			6/above left/\scalebox{0.7}{\pedv},
			7/above/\scalebox{0.7}{\pev},
			8/above right/\scalebox{0.7}{\pezv}}
		\node[\p=\b] at (q\m) {\d};
		\end{tikzpicture}
		\caption{The complex $|\ncp_4|$ in an octagon-shape.}
		\label{fig:anti_auto_NCP4}
	\end{center}
\end{figure}

\subsection{A dihedral group of skew-automorphisms}
We construct a dihedral group of skew-automorphisms of $\nc=\nc(W)$ for all finite Coxeter groups $W$ here.

\begin{defi}
	For $w\in \nc$ we set $\ncaa(w) \coloneqq w\inv\cox$ for the Coxeter element $\cox\in W$.
\end{defi}

The following is proven in Lemma 2.5.4 of \cite{armstr}, but for the sake of completeness, and since $\ncaa$ plays an important role in the sequel, we include a proof here.

\begin{lem}\label{lem:anti-auto}
	The map $\ncaa\colon \nc\to\nc$ defined by $w\mapsto\ncaa(w)= w\inv\cox$ is a lattice anti-automorphism of $\nc$.
\end{lem}

\begin{proof}
	We show that $\ncaa$ maps $\nc$ bijectively to itself and that it reverses the order.
	Let $w\in \nc$ be arbitrary. Then $\ncaa(w)=w\inv\cox \leq \cox$ by \cref{obs:abs_ord}, which implies that $\ncaa$ is well-defined. We define a map $\ncaa\inv\colon \nc\to\nc$ by $w\mapsto \cox w\inv$. This map is well-defined since $w\leq \cox$ implies $\cox w\inv \leq \cox$ by \cref{obs:abs_ord} and hence $\cox w\inv \in \nc$. The maps $\ncaa$ and $\ncaa\inv$ are inverses
	and hence $\ncaa$ is a bijection. 
	For the last step, let $v,w\in \nc$. We show that $v\leq w$ if and only if $\ncaa(w)\leq \ncaa(v)$.  
	First suppose that $v\leq w$. Then $w\inv\cox \leq v\inv \cox$ holds by \cref{obs:abs_ord}, which means $\ncaa(w)\leq \ncaa(v)$. 
	For the other direction let $v$ and $w$ be such that $\ncaa(w)\leq \ncaa(v)$. Then, again by \cref{obs:abs_ord}, we have that $\ncaa(v)\inv\cox \leq \ncaa(w)\inv\cox$.
	Inserting the definition of $\ncaa$ yields that $\cox\inv v \cox\leq\cox\inv w \cox$.
	By definition of the absolute order, this means that
	\[
	\ell(\cox\inv w \cox) = \ell(\cox\inv v \cox) + \ell((\cox\inv v\cox)\inv \cox\inv w \cox) = \ell(\cox\inv v \cox) + \ell(\cox\inv v\inv  w \cox).
	\] 
	But the absolute length is invariant under conjugation so we get $\ell(w)=\ell(v)+\ell(v\inv w)$, which means that $v \leq w$.
\end{proof}

Recall that the dihedral group $\D = \langle \autl, \autr\rangle \leq \aut(\nc)$ has order $2d$ by \cref{lem:dihedral_auto_group}. 

\begin{lem}\label{lem:skew-autos_nc}
	The maps $\ncaa$ and $\autl$ generate a group of skew-automorphisms of $\nc$ that is isomorphic to a dihedral group of order $4d$.
\end{lem}

\begin{proof}
	We show that $\ncaa$ has order $2d$ and that $(\autl\circ\ncaa)^2=\id$. Let $w\in W$ be arbitrary. Then $\ncaa(\ncaa(w))=\ncaa(w\inv \cox)=(w\inv\cox)\inv\cox=\cox\inv w \cox$ and hence $\ncaa^2$ is conjugation with $\cox\inv$, which has order $d$ by \cref{lem:dihedral_auto_group}. Therefore the order of $\ncaa$ is $2d$. Moreover, it holds that
	\begin{align*}
	\autl(\ncaa(\autl(\ncaa(w))))
	&=\autl(\ncaa(\autl(w\inv\cox)))
	=\autl(\ncaa(\l\cox\inv w\l))
	=\autl(\l w\inv \cox \l \cox)\\
	&=\autl(\l w\inv \l\r\cdot \l \cdot\l\r)
	=\autl(\l w\inv \l)
	=\l^2 w \l^2
	=w
	\end{align*}
	and hence $(\autl\circ\ncaa)^2=\id$.
\end{proof}

\begin{nota}
	We denote by $\DD$ the dihedral group of skew-automorphisms of $\nc$ generated by $\autl$ and $\ncaa$. By $\autd$ we denote the anti-automorphism $\autl\circ\ncaa$. 
\end{nota}

\begin{cor}\label{cor:nc_autos_extend}
	The group $\DD$ is generated by the two elements $\autl$ and $\autd$.
\end{cor}

\begin{proof}
	We already know that $\DD$ is a dihedral group of order $4d$ and that $\autl$ has order $2$. Note that 
	$\autd^2=\autl\circ\ncaa\circ\autl\circ\ncaa=\id$ by the proof of  \cref{lem:skew-autos_nc} and that $\autl \circ \autd = \ncaa$. Since $\ncaa$ and $\autl$ generate $\DD$, also $\autl$ and $\autd$ generate $\DD$.
\end{proof}

Recall that $\ncaa^2(w)=\ncaa(w\inv\cox)=(w\inv\cox)\inv\cox=\cox\inv w\cox$ for $w\in\nc$, that is $\ncaa^2$ is conjugation with $\cox\inv$ and hence an automorphism of $\nc$ in $\D$ of order $h$.

\begin{lem}
	The dihedral group of automorphisms $\D$ is a subgroup of index $2$ in $\DD$.
\end{lem}

\begin{proof}
	Since $\D$ is generated by $\autl$ and $\autr$, we show that $\autr=\ncaa^2\circ \autl \in \DD$. Let $w\in W$ be arbitrary. Then 
	$\ncaa^2(\autl(w))
	=\ncaa^2(\l w\inv \l)
	=\cox\inv \l w\inv \l \cox
	=\r\l(\l w\inv \l)\l\r
	=\r w\inv\r
	=\autr(w)$. Hence $\D \leq \DD$ is a subgroup of index $\nicefrac{\#\DD}{\#\D}=\nicefrac{4d}{2d}=2$ of $\DD$. 
\end{proof}

\begin{lem}
	Every anti-automorphism $\antiauto\colon \nc\to\nc$ in $\DD$ is of one of the following forms:
	\begin{enumerate}
		\item $\antiauto\colon w \mapsto \cox^{-k} w\inv \cox^{k+1}$ for $0\leq k < d$,
		\item $\antiauto\colon w \mapsto \cox^{k+1}\l w \l\cox^k$ for $0\leq k < d$.
	\end{enumerate}
\end{lem}

\begin{proof}
	From the group presentation
	\[
	\DD=\left\langle \autl,\ncaa \mid \autl^2, \ncaa^{2d}, (\autl\circ \ncaa)^2 \right\rangle
	\]
	of $\DD$ as dihedral group we know that every skew-automorphism $\antiauto\in \DD$ is of the form $\autl^i\circ\ncaa^j$ for some $i\in\Set{0,1}$ and $0\leq j <d$. The number $j$ determines if $\antiauto$ is an automorphism or an anti-automorphism. If $j$ is odd, then $\antiauto$ is order-reversing and hence an anti-automorphism. If $j$ is even, then $\antiauto$ is order-preserving and hence an automorphism. This means that for an anti-automorphism we either have $\antiauto(w)=\ncaa^{2k+1}(w)=\cox^{-k}w\inv \cox^{k+1}$ or $\antiauto(w)=\autl\circ\ncaa^{2k+1}(w)=\autl(\cox^{-k}w\inv \cox^{k+1})=\l\cox^{-k-1}w \cox^k\l=\cox^{k+1}\l w \cox^k\l$ for $w\in\nc$ and $0\leq k <d$. 
\end{proof}

\begin{rem}
	The anti-automorphisms of the form $\ncaa^k$, which are exactly the ones of item \emph{a)} in the previous lemma, appear in \cite{armstr} under the name \emph{Kreweras-type} anti-automorphisms of $\nc$.
\end{rem}

\begin{oq}
	What is the group of skew-automorphisms of $\nc(W)$? Is it isomorphic to the dihedral group $\DD$?
\end{oq}

\subsection{Anti-automorphisms of the lattice of linear subspaces}

In order to show that the skew-automorphisms of $\DD$ extend to lattice skew-automorphisms of $\lam$ in the classical types, we have to study the anti-automorphisms of the lattice of linear subspaces first. For this we need some theory of bilinear forms and follow \cite{lorenz}.

Let $V$ be a finite dimensional vector space defined over the field $\F$ and $\bil\colon V\to V$ a bilinear form. The two maps
\begin{align*}
&\bo\colon V\to V\st,\quad x\mapsto (\bo(x)\colon v\mapsto \bil(x,v)),\\
&\bt\colon V\to V\st, \quad x\mapsto (\bt(x)\colon v\mapsto \bil(v,x))
\end{align*}
are linear maps from $V$ to its dual space $V\st$, and $\bo$ is an isomorphism if and only if $\bt$ is one. In this case, the bilinear form $\bil$ is called \emph{non-degenerate}.

For a subspace $U\subseteq V$ let $r_U\colon V\st \to U\st$ be the map on the dual spaces induced from the restriction to $U$. We define the maps $\bou\coloneqq r_U\circ \bo$ and $\btu\coloneqq r_U\circ \bt$, which are linear maps $V \to U\st$.

The \emph{right complement} of $U$ in $V$ with respect to $\bil$ is defined as
\[
U\com \coloneqq \ker(\btu)=\Set{v\in V\str \bil(u,v)=0 \text{ for all }u\in U }
\]
and the \emph{left complement} of $U$ in $V$ with respect to $\bil$ as
\[
\com U\coloneqq \ker(\bou) =\Set{v\in V\str \bil(v,u)=0 \text{ for all }u\in U }.
\]
Both $U\com$ and $\com U$ are subspaces of $V$ with dimension $\dim(V)-\dim(U)$ \cite[Thm. VI.F11]{lorenz}.
Note that in general we have $(U\com)\com \neq U$, but $(\com U)\com = {\com(U\com)}= U$ holds true \cite[Thm. VI.F11]{lorenz}.

We are mainly interested in the right complement of a subspace and therefore use the term \enquote{complement} as synonym for \enquote{right complement}.

\begin{prop}
	For any non-degenerate bilinear form $\bil$ on $V$ the map 
	\[
	\aab\colon \lam\to\lam, \quad U\mapsto U\com,
	\]
	which maps a subspace to its complement, is a lattice anti-automorphism of $\lam$.
\end{prop}

\begin{proof}
	It is clear that $\aab$ is well-defined. To see that it is a bijection, define the map $\aab\inv\colon \lam\to\lam$ by $U\mapsto {\com U}$. Since $(\com U)\com = {\com(U\com)}= U$, the map $\aab\inv$ is indeed the inverse of $\aab$. To see that $\aab$ and $\aab\inv$ are order-reversing let $U,W\subseteq V$ be subspaces with $U \subseteq W$. If $\bil(w,v)=0$ for $v\in V$ and all $w\in W$, then also $\bil(u,v)=0$ for all $u\in U\subseteq W$ and hence $W\com \subseteq U\com$. Analogously we have $\com W \subseteq {\com U}$.
\end{proof}

The bilinear form $\bil$ is non-degenerate if its matrix representations are invertible. For a fixed basis of $V$, every matrix in $\gl_n(\F)$ induces an anti-automorphism of $V$.

\subsection{Extending the anti-automorphisms}

In order to extend the anti-automorphisms $\antiauto \in \DD$ of $\nc$ to $\lam$, we need to find a bilinear form $\bil$ on $V$ such that $\emb\circ\antiauto=\aab\circ\emb$. Recall that every anti-automorphism in $\DD$ is a combination of the anti-automorphism $\ncaa\colon \nc \to \nc$,  $w\mapsto w\inv\cox$, and the automorphism $\autl\colon \nc\to\nc$, $w\mapsto \l w\inv \l$. We already know that the automorphisms of $\nc$ that are contained in the subgroup $\D$  extend to $\lam$ by \cref{thm:autos_commute}. Hence it is enough to show that the anti-automorphism $\ncaa$ also extends to $\lam$ in order to show that we can extend all anti-automorphisms in $\DD$ of $\nc$ to $\lam$, that is that the diagram
\[
\begin{tikzcd}
\nc \rar["\ncaa"] \dar["\emb"'] & \nc \dar["\emb"]\\
\lam \ar[r,dashrightarrow,"\aab"] & \lam
\end{tikzcd}
\]
commutes. In order to do so, we construct appropriate bilinear forms for the classical types.

\begin{defi}
	Let $s,t\in T$ be different reflections. We say that $s$ \emph{subordinates} $t$ if $st \in \nc$. By $\subo(t)$ we denote the set of reflections that subordinate $t$, that is
	\[
	\subo(t)\coloneqq \Set{s \in T\str st \in \nc }.
	\]	
\end{defi}

In general, there are reflections $s,t\in T$ such that neither $st \in \nc$ nor $ts\in \nc$. An example is given by the reflections $[1]$ and $[2]$ in the Coxeter group of type $B_n$. In the symmetric group $S_n$ we have for all reflections $s,t \in T$ that $st \in \nc(S_n)$ or $ts \in \nc(S_n)$.

\begin{lem}
	Let $s,t\in T$ be reflections. Then $s$ subordinates $t$ if and only if there exists a reduced decomposition $t_1\ldots t_n$ of the Coxeter element $\cox$ such that $s=t_i$ and $t=t_j$ for some $i<j$.
\end{lem}

\begin{proof}
	Suppose that $t_1\ldots s \ldots t \ldots t_n$ is a reduced decomposition of $\cox$. Since $st$ is a subword of the reduced decomposition of $\cox$ we have by the \cref{subword_prop} that $st \leq c$, which means $st\in \nc$.
	
	Now suppose that $w\coloneqq st\in \nc$ with $s\neq t$. Then by definition $w\leq\cox$ and hence $\ell(w)+\ell(w\inv \cox)=\ell(\cox)$. Choose a reduced decomposition $t_3\ldots t_n$ of $w\inv\cox$. Then we can write $\cox$ as a product of $n$ reflections, namely $\cox=stt_3\ldots t_n$, and so $\cox$ has a reduced decomposition where $s$ appears left of $t$.
\end{proof}

Recall that we denote by $\alpha_t$ the positive root corresponding to a reflection $t$ and that $\emb(t)=\mov(t)=\Braket{\alpha_t} \subseteq V$. The following criterion allows us to verify that a given bilinear form gives rise to an extension of $\ncaa$.

\begin{thm}\label{thm:criterion_extension_anti-auto}
	The anti-automorphism $\ncaa$ of $\nc$ extends to an anti-automorphism of $\lam$ if there exists a bilinear form $\bil$ on $V$ such that 
	\[
	\bil(\alpha_s,\alpha_t)=0
	\]
	for all $t\in T$ and $s\in \subo(t)$. In this case, it holds that $\aab$ is an anti-automorphism extending $\ncaa$, that is $\aab \circ \emb = \emb \circ \ncaa$.
\end{thm}

\begin{proof}
	Suppose that there is a bilinear form $\bil$ on $V$ such that $\bil(\alpha_s,\alpha_t)=0$ for all $t\in T$ and $s\in \subo(t)$. We show that $\aab(\emb(w))=\emb(\ncaa(w))$ for all $w\in \nc$.  
	First we consider the case $t\in T$ and then extend it to an arbitrary $w\in \nc$.
	
	Let  $t_1\ldots t_n$ be a reduced decomposition of the Coxeter element $\cox$.
	By assumption we have that $\bil(\alpha_{t_1},\alpha_{t_i})=0$ for all $2\leq i \leq n$. Let $U \coloneqq \emb(t_1)$. Note that $\dim(U)=1$ and hence $\dim(U\com)=n-1$. Since $\alpha_{t_i} \in U\com$ for all $2\leq i\leq n$ and since the $\alpha_{t_i}$ are linearly independent by \cref{lem:carter}, the set $\Set{\alpha_{t_2}, \ldots, \alpha_{t_n}}$ is a basis of $U\com$, that is
	\[
	\aab(\emb(t_1))=U\com=\langle \alpha_{t_2}, \ldots, \alpha_{t_n}\rangle = \emb(t_2) \vee \ldots \vee \emb(t_n).
	\]
	
	Let us consider $\emb(\ncaa(t_1))$. Since $\ncaa(t_1)=t_1\cox=t_1\cdot t_1 t_2\ldots t_n$, the decomposition $t_2\ldots t_n$ of $\ncaa(t_1)$ is reduced. Hence we get
	\[
	\emb(\ncaa(t_1))=\emb(t_2\ldots t_n)=\emb(t_2) \vee \ldots \vee \emb(t_n)
	\]
	by \cref{lem:join_and_em_interchange}. For every reflection $t\in T$ there is a reduced decomposition of the Coxeter element $\cox$ having $t$ as the first letter by the \cref{subword_prop}. So we have $\aab(\emb(t))=\emb(\ncaa(t))$ for all reflections $t\in T$.
	
	Now let $w\in \nc$ be arbitrary and let $t_1\ldots t_k$ be a reduced decomposition of it. Moreover, let $t_{k+1}, \ldots, t_n$ be reflections such that $t_1\ldots t_n=\cox$. 
	Note that $t_{k+1}\ldots t_n$ is a reduced decomposition of $w\inv\cox=\ncaa(w)$. Hence we have the following two equations
	\begin{align*}
	\emb(\ncaa(w))&=\emb(t_{k+1}\ldots t_n)=\emb(t_{k+1}) \vee \ldots \emb(t_n) \text{ and}\\
	\aab(\emb(w))&=\aab(\emb(t_1)\vee \ldots \vee \emb(t_k)) =\aab(\emb(t_1)) \wedge \ldots \wedge \aab(\emb(t_k)),
	\end{align*}
	where the second line holds by \cref{lem:lattice_antiauto_join_interchange}. Note that both $\emb(\ncaa(w))$ and $\aab(\emb(w))$ have dimension $n-k$, since $\emb$ is rank-preserving, and $\ncaa$ and $\aab$ are rank-reversing.
	We show that 
	\[
	\aab(\emb(t_1))\wedge\ldots\wedge\aab(\emb(t_k)) = \emb(t_{k+1})\vee \ldots \vee \emb(t_n)
	\]
	holds by showing that for every $1\leq i \leq n$ there is a basis of $\aab(\emb(t_i))$ that contains the vectors $\alpha_{t_{i+1}}, \ldots, \alpha_{t_n}$. Then for $\aab(\emb(t_1)), \ldots, \aab(\emb(t_k))$ there exist bases that all contain the vectors $\alpha_{t_{k+1}}, \ldots, \alpha_{t_n}$. Since the meet in $\lam$ is intersection of subspaces, and $\aab(\emb(w))$ has dimension $n-k$, it follows that $\Set{\alpha_{t_{k+1}}, \ldots, \alpha_{t_n}}$ is a basis of $\aab(\emb(w))$. Hence $\aab(\emb(w))=\emb(t_{k+1})\vee \ldots \vee \emb(t_n) = \emb(\ncaa(w))$.
	
	It remains to show that there exists a basis of $\aab(\emb(t_i))$ that contains the vectors $\alpha_{t_{i+1}}, \ldots, \alpha_{t_n}$. We successively perform left-shifts in the reduced decomposition of $\cox$ until $t_i$ is the first letter as follows. First shift $t_2$ to the left, then $t_3$ and so on, until you shift $t_i$ to the foremost left. In more detail, the following steps are performed:
	\begin{align*}
	&t_1\, t_2 \,t_3 \,t_4\ldots t_i\, t_{i+1} \ldots t_n,\\
	&t_2\, t_1^{t_2}\, t_3 \,t_4\ldots t_i\, t_{i+1} \ldots t_n,\\
	& t_2 \,t_3\, t_1^{t_3t_2}\, t_4\ldots t_i \, t_{i+1}\ldots t_n,\\
	&t_3\, t_2^{t_3}\,  t_1^{t_3t_2}\, t_4\ldots t_i\, t_{i+1} \ldots t_n,\\
	&\qquad\qquad \vdots\\
	&t_i \, t_{i-1}^{t_i} \, t_{i-2}^{t_{i-1} t_i}\ldots t_1^{t_i\ldots t_2}\, t_{i+1} \ldots t_n,
	\end{align*}
	where by $a^b$ we denote conjugation with $b$, that is $a^b=bab\inv$. By the \cref{shifting_prop}, applying shifts does not change the element itself, only the reduced decomposition. 
	In the reduced decomposition, 
	shifting an element from position $j$ to the left only affects the reflections on the positions less than $j$.
	
	Hence $t_i \, t_{i-1}^{t_i} \, t_{i-2}^{t_{i-1} t_i}\ldots t_1^{t_i\ldots t_2}\, t_{i+1} \ldots t_n$ is a reduced decomposition of $\cox$ involving the reflections $t_{i+1},\ldots, t_n$ with $t_i$ on the first position. We showed above that in this case, $\aab(f(t_i))$ has the basis 
	\[
	\Set{\alpha_{ t_{i-1}^{t_i}}, \ldots, \alpha_{ t_1^{t_i\ldots t_2}},  \alpha_{t_{i+1}}, \ldots, \alpha_{t_n}},
	\]
	that is we found a basis of $\aab(\emb(t_i))$ involving the vectors $\alpha_{t_{i+1}}, \ldots, \alpha_{t_n}$.
\end{proof}

\begin{thm}\label{thm:anti-autos_extend}
	If $W$ is a Coxeter group of  classical type $A$, $B$, or $D$, then all lattice skew-automorphisms in $\DD$ of $\nc(W)$ extend to lattice skew-automorphisms of $\lam$.
\end{thm}

The proof is case-by-case for the different types, but the strategy to show that the anti-automorphism $\ncaa$ extends is the same. First we analyze for which pairs of transpositions $s,t\in T$ the product $st$ is in $\nc$. Then we show that for the bilinear form $\bil$, which is different for the different types, we get $\bil(\alpha_s,\alpha_t)=0$. Using \cref{thm:criterion_extension_anti-auto} implies the result. To ease notation we write $\bil(s,t)$ for $\bil(\alpha_s,\alpha_t)$. 

Recall that $V$ denotes a vector space such that $\emb\colon \nc\to\lam(V)$ is the poset embedding from \cref{thm:brady_watt_embedding} or \cref{thm:embedding_Vp_VZ}, respectively.
Let $e_i$ be the standard basis vectors of $V$ for $1\leq i \leq \dim(V)$. By $u\tr$ we denote the transpose of a vector $u\in V$ and $u\tr \cdot v$ is the usual matrix multiplication.

\begin{prop}\label{prop:typeA_anti-auto_extends}
	The bilinear form  $\bila$ on the $(n-1)$-dimensional vector space $V$ given by 
	\[
	\bila(e_i,e_j)=
	\begin{cases}
	1 &\ i\leq j,\\
	0 &\ i > j
	\end{cases}
	\]
	for $1\leq i,j < n$ satisfies $\bila(s,t)=0$ for all $s,t\in T$ with $st\in\nc(S_n)$.
\end{prop}

\begin{proof}
	
	Recall that for a transposition $s=(x\,\;y)\in T(S_n)$ with $x<y$ the corresponding root is $\alpha_s=e_x-e_y$ with the convention that $e_n=0$ in the $(n-1)$-dimensional vector space $V$. The evaluation of $\bila$ on $s$ and $v\in V$ is
	\[
	\bila(s,v)=\left(\sum_{i=x}^{y-1} e_i\tr\right)\cdot v.
	\]
	
	We show that for every reduced decomposition $st$ of an element in $\nc(S_n)$ we have $\bila(s,t)=0$. An element of rank two in $\nc(S_n)$ can be of two types. Either, it is a $3$-cycle, or it is a product of two commuting transpositions.
	
	Let $st$ be a $3$-cycle. The possibilities to decompose $st$ are $st=(ab)(bc)$ for $a<b<c$, $b<c<a$ or $c<a<b$. 
	If $a<b<c$, we have 
	\[
	\bila(s,t)=\left(\sum_{i=a}^{b-1} e_i\tr\right)\cdot (e_b - e_c)
	=0,
	\]
	since $b,c > b-1$. If $b<c<a$ the evaluation is 
	\[
	\bila(s,t)=\left(\sum_{i=b}^{a-1} e_i\tr\right)\cdot (e_b - e_c)
	= \sum_{i=b}^{a-1} e_i\tr\cdot e_b - \sum_{i=b}^{a-1} e_i\tr\cdot e_c
	= e_b\tr\cdot e_b - e_c\tr\cdot e_c
	= 0.
	\]
	For the last case $c<a<b$ we get
	\[
	\bila(s,t)=\left(\sum_{i=a}^{b-1} e_i\tr\right)\cdot (e_c - e_b)
	= 0,
	\]
	because $c<a$ and $b>b-1$. 
	
	Now let $st$ be a product of two commuting transpositions $(a\,\;b)(c\,\;d)$. The possibilities for $a,b,c,d$ such that $st\in\nc(S_n)$ are $a<b<c<d$ and $b<c<d<a$. 
	Suppose that $a<b<c<d$. Then 
	\[
	\bila(s,t)=\left(\sum_{i=a}^{b-1} e_i\tr\right)\cdot (e_c - e_d)
	=0,
	\]
	since $c,d > b-1$. In the case of $b<c<d<a$ we get
	\[
	\bila(s,t)=\left(\sum_{i=b}^{a-1} e_i\tr\right)\cdot (e_c - e_d)
	= \sum_{i=b}^{a-1} e_i\tr\cdot e_c - \sum_{i=b}^{a-1} e_i\tr\cdot e_d
	= e_c\tr\cdot e_c - e_d\tr\cdot e_d
	= 0.
	\]
\end{proof}

The proofs of the following two propositions are the analogous computations to the one in type $A$ and are relocated in \cref{chap:typeB_details} and \cref{chap:typeD_details}, respectively.

\begin{prop}\label{prop:typeB_anti-auto_extends}
	The bilinear form $\bilb$ on the $n$-dimensional vector space $V$ given by 
	\[
	\bilb(e_i,e_j)=
	\begin{cases}
	1 &\ i\leq j,\\
	-1  &\ i > j
	\end{cases}
	\]
	for $1\leq i,j \leq n$ satisfies $\bilb(s,t)=0$ for all $s,t\in T$ with $st\in\nc(B_n)$.
\end{prop}

\begin{prop}\label{prop:typeD_anti-auto_extends}
	The bilinear form $\bild$  on the $n$-dimensional vector space $V$ given by 
	\[
	\bild(e_i,e_j)=
	\begin{cases}
	1 &\ i\leq j<n,\ i=j=n,\\
	-1 &\ j < i < n,\\
	0 &\ \text{otherwise}
	\end{cases}
	\]
	for $1\leq i,j \leq n$ satisfies $\bild(s,t)=0$ for all $s,t\in T$ with $st\in\nc(D_n)$.
\end{prop}

\begin{proof}[Proof of \cref{thm:anti-autos_extend}]
	In order to show that an anti-automorphism of $\nc(W)$ extends to an anti-automorphism of $\lam(V)$ we use the criterion from \cref{thm:criterion_extension_anti-auto}. Hence we have to show that there exists a bilinear form $\bil$ on $V$ such that $\bil(s,t)=0$ for all $t \in T$ and $s\in \subo(t)$. For type $A$, we proved the existence of such a bilinear form in \cref{prop:typeA_anti-auto_extends}, for type $B$ in \cref{prop:typeB_anti-auto_extends} and for type $D$ in \cref{prop:typeD_anti-auto_extends}.
\end{proof}

The translation on \cref{thm:anti-autos_extend} into the simplicial setting is the following. Recall that we identified the lattice skew-automorphisms with the corresponding simplicial automorphisms of the order complex.

\begin{thm}
	Let $W$ be a finite Coxeter group of type $A$, $B$ or $D$. If $W$ is of type $D$, further assume that $\rk(W)\neq 4$. Then every type-reversing simplicial automorphism in $\DD \leq \aut(\nc(W))$ extends to a type-reversing simplicial automorphism of the simplicial building $|\lam(V)|$.
\end{thm}

\cleardoublepage
\part{Geometry of the classical non-crossing partitions}\label{part3}

\cleardoublepage

\chapter{Motivation from geometric group theory}\label{chap:motivation}

The motivation for studying the geometric properties of the non-crossing partition complex originates from geometric group theory, which is concerned with the interaction of algebraic properties of groups and geometric properties of spaces. The aim of this chapter is to explain the role of the non-crossing partitions in the context of geometric group theory and to provide the necessary background from geometry. Finally, we outline a proof strategy that could show that the complex of classical non-crossing partitions is $\cato$, and hence that the braid groups are $\catz$ groups.\medskip

\emph{Given an algebraically defined class of groups, does every group in this class act nicely on a space of non-positive curvature?}

This is a typical question in geometric group theory. The class of groups that is linked with the classical non-crossing partitions is the class of braid groups. A \emph{braid group on $n$ strands} is a group $B_n$ that has a presentation of the form
\[
\Braket{s_1, \ldots, s_{n-1} | s_is_{i+1}s_i=s_{i+1}s_is_{i+1},\ s_is_j=s_js_i \text{ if }|i-j|>1}.
\]
The presentation of the braid group resembles the Coxeter presentation of the symmetric group. The only difference is that generators in $S_n$ have order $2$, whereas in $B_n$ they are of infinite order. 
Brady showed in \cite{bra_kpi} that the braid groups admit a presentation with the non-crossing partitions as formal generators. This initiated the study of braid groups with non-crossing partitions as a tool. 

The right notion of non-positive curvature in the setting of geodesic metric spaces is the concept of so-called $\catz$ spaces. Closely related to them are $\cato$ spaces, which have curvature bounded from above by one. The precise definitions are given below in \cref{def:catoz}.
We say that a group is a \emph{$\catz$ group} if there exists a $\catz$ space on which it acts properly, cocompactly and by isometries. The question we are concerned with is:

\emph{Are braid groups $\catz$ groups for all $n$?}

There is a positive answer to this question for small $n$. Besides the trivial cases for $n \leq 4$, it is proven for $n=5$ in \cite{bra-mcc}, and for $n=6$ in \cite{hks}. We do not give an answer for any new ranks in this thesis, but we provide a uniform strategy for all $n$ with some intermediate results for the following conjecture \cite[Conj. 8.4]{bra-mcc}.

\begin{curvcon}\label{conj:ncpn_cat1}
	The spherical complex $\on$ is $\cato$ for every $n$.
\end{curvcon}

The following result from Brady and McCammond demonstrates that a positive answer to the \cref{conj:ncpn_cat1} gives a positive answer to the question whether braid groups are $\catz$ \cite[Prop. 8.3]{bra-mcc}. 

\begin{prop}
	If the spherical complex $\on$ is $\cato$, then the braid group on $n$ strands is a $\catz$ group. 
\end{prop}

\section{Geodesic metric spaces}

In this section, we shortly review the basic concepts of geodesic metric spaces that are needed in the sequel. In particular, we define a spherical metric on simplicial complexes and introduce $\catz$ and $\cato$ spaces. We follow Bridson and Haefliger \cite{bh}. 

Let $(X,\di)$ be a metric space. A \emph{geodesic} in $X$ is an isometry $\gamma\colon I \to X$, where $I$ is a closed interval $I=[0,l]\subseteq \R$ and $\R$ is equipped with the Euclidean standard metric. The geodesic $\gamma$ \emph{joins} $x$ and $y$, or \emph{is from $x$ to $y$}, if $\gamma(0)=x$ and $\gamma(l)=y$. If we want to emphasize the endpoints, we also write $\gamma_{xy}$. In particular it holds that $\di(x,y)=l$. By abuse of notation, we refer to the image $\gamma(I)\subseteq X$ as \emph{geodesic} $\gamma$ as well. 
We say that $X$ is a \emph{geodesic space} if every two points in $X$ are joined by a geodesic. It is called \emph{uniquely geodesic} if there is exactly one geodesic joining $x$ and $y$ for all $x,y\in X$. For an $r>0$, the metric space $X$ is called \emph{$r$-(uniquely) geodesic} if there exists a (unique) geodesic joining $x$ and $y$ for all $x,y\in X$ with $\di(x,y)<r$.
Note that in a geodesic space, the distance of any two points is realized by a geodesic joining them, that is the metric is a length metric \cite[Chap. I.1]{bh}.

\begin{exa}
	The Euclidean space $\R^n$ is a uniquely geodesic space for every dimension. If we endow the sphere $S^n$ with the standard round metric, then $S^n$ is a geodesic space for all $n\geq 1$. Moreover, $S^n$ is $\pi$-uniquely geodesic, but not $r$-uniquely geodesic for $r>\pi$.
\end{exa}

In the sequel, the space $\R^n$ is always equipped with the standard Euclidean metric and the spheres $S^n$ with the standard round metric.

\subsection{Metrizing simplicial complexes}\label{sec:metric}

In \cref{sec:simpl_compl} we already saw how to endow an abstract simplicial complex with a Euclidean metric on every simplex to obtain its affine realization. Now we define a spherical version of it by endowing every simplex with a spherical metric \cite[Chap. I.7, I.7A]{bh}. This enables us to turn Coxeter complexes, and hence also non-crossing partition complex $\on$ and spherical buildings, into complete geodesic spaces.

\begin{defi}
	A \emph{spherical $n$-simplex} is the convex hull of $n+1$ points $v_1, \ldots, v_{n+1}$ in general position that are contained in a open hemisphere of $S^m$ for $m \geq n$. In particular, it is a complete geodesic space with the metric induced from $S^m$. A \emph{face} of a spherical $n$-simplex is the convex hull of a subset of its vertices. 
\end{defi}

Note that every spherical simplex has an underlying abstract simplex. Moreover, there is a unique map from a spherical simplex $\bar{A}$ to the affine realization of the underlying abstract simplex $A$ that respects the abstract structure. 
This map is denoted by $f_A\colon \bar{A} \to A$.

\begin{defi}\label{defi:metric}
	Let $\si$ be a finite simplicial complex and $\S$ a set of spherical simplices. An assignment $f\colon \si \to \S$ of $k$-simplices to spherical $k$-simplices induces a \emph{spherical realization} of $\si$ if 
	for every $A\in \si$ and every face $A'\subseteq A$, the map $f_A\inv \circ f_{A'}$ is an isometry of $\bar{A}'$ onto a face of $\bar{A}$. A \emph{spherical complex} is a simplicial complex in its spherical realization.
\end{defi}

The metric on the spherical simplices naturally induces a length metric on the spherical complex. With this metric, finite spherical complexes are complete geodesic spaces \cite[Thm. I.7.19]{bh}. Note that the affine realization and the spherical realization are homeomorphic as topological spaces. This means that the topological properties such as homotopy type are the same for affine and spherical realizations. 

One class of simplicial complexes of our particular interest are Coxeter complexes. We turn them into spherical complexes as follows.
The Coxeter complex $\Sigma$ of a finite Coxeter group $W$ is isomorphic to the triangulated unit sphere, seen as simplicial complex, that arises from the geometric representation of $W$.  Hence every simplex of $\Sigma$ naturally is endowed with a spherical metric and the spherical realization of $\Sigma$ is isometric to the unit sphere. See \cref{rem:simplicial_structure_sphere} and \cref{rem:cox_cplx_homeo_sphere} for the details.

\begin{rem}
	The action of $W$ on $\R^n$, and hence on the unit sphere, via the geometric representation is by isometries and is compatible with the combinatorial action of $W$ on its Coxeter complex by left-multiplication. Because this action is simply transitive on chambers, all chambers of the spherical complex $\Sigma$ are isometric. 
\end{rem}

\begin{rem}\label{rem:metric}
	The construction of the natural metric on a Coxeter complex induces metrics on simplicial complexes that are unions of Coxeter complexes as well. Both non-crossing partition complexes and spherical buildings are unions of Coxeter complexes, hence they are naturally endowed with a metric. Further, they are complete geodesic spaces. 
\end{rem}

\begin{nota}
	By abuse of notation, we identify the abstract simplicial complexes with their spherical realizations in the sequel. Whenever we are speaking about a metric on one of these complexes, we mean the natural spherical metric induced from the Coxeter complex. If we want to emphasize the metric, we speak of \emph{spherical} complexes.
\end{nota}

Brady and McCammond use a different approach to define a metric on the non-crossing partition complex $\on$ in \cite{bra-mcc}. Their metric naturally arises from so-called orthoschemes. 
On the apartments it coincides  with the metric on the Coxeter complex we defined above \cite[Def. 4.6, Ex. 5.2]{bra-mcc}. Hence their metrics on $\on$ and the spherical building associated to vector spaces coincide with the metric we defined here.

\subsection{Curvature}

Intuitively speaking, a geodesic space $X$ is $\catz$ if every triangle is not thicker than a comparison triangle in the Euclidean space $\R^2$, which has constant sectional curvature $0$. Analogously, triangles in $\cato$ spaces are not allowed to be thicker as their comparison triangles in the sphere $S^2$, a space of constant sectional curvature $1$. 

Let $(X,\di)$ be a metric space and $(Y,\bar{\di})$ either $\R^2$ or $S^2$, equipped with the respective standard metrics. A \emph{triangle} in $X$ with vertices $x,y,z\in X$ is the union $\si$ of three geodesics $\gamma_{xy}$, $\gamma_{yz}$ and $\gamma_{xy}$. Note that in general, a triangle is not uniquely determined by its vertices, as there may be different geodesics joining two of its vertices.  A \emph{comparison triangle} for a triangle $\si$ as above is a triangle $\bar{\si}$ in $Y$ with vertices $\bar{x}$, $\bar{y}$ and $\bar{z}$ such that $\di(a,b)=\bar{\di}(\bar{a},\bar{b})$ for all $a,b\in \Set{x,y,z}$. If $Y=\R^2$, then such a comparison triangle always exists. In the case of $Y=S^2$, a comparison triangle exists if the perimeter of $\si$ satisfies $\di(x,y)+\di(y,z)+\di(x,z)<2\pi$.
A point $\bar{p} \in \bar{\si}$ is a \emph{comparison point} for $p\in \gamma_{ab}$, with $a,b\in \Set{x,y,z}$ if $\di(a,p)=\bar{\di}(\bar{a}, \bar{p})$ holds. A thin triangle and its comparison triangle in $\R^2$ with comparison points are shown in \cref{fig:thin_triangle}.

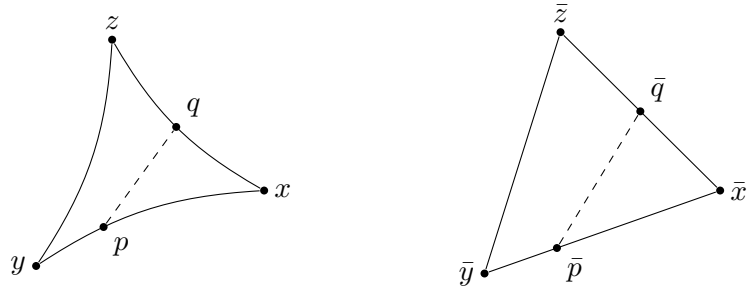
\begin{figure}
	\begin{center}
		\begin{tikzpicture}
		\node[kpunkt] (x) at (0,0){};
		\node[kpunkt] (y) at (-3,-1){};
		\node[kpunkt] (z) at (-2,2){};
		\draw  (x)  to [bend right=15] node(p)[pos = 0.7 ,kpunkt] {}(y) ;
		\draw  (y)  to [bend right=15](z) ;
		\draw  (z)  to [bend right=15] node(q)[pos = 0.5 ,kpunkt] {} (x) ;
		\node[right] at (x){$x$};
		\node[left] at (y){$y$};
		\node[above] at (z){$z$};
		\node[below right] at (p){$p$};
		\node[above right] at (q){$q$};
		\draw[dashed] (p) -- (q);
		
		\node[kpunkt] (a) at (6,0){};
		\node[kpunkt] (b) at (2.9,-1.1){};
		\node[kpunkt] (c) at (3.9,2.1){};
		\draw  (a)  to  node(d)[pos = 0.7 ,kpunkt] {}(b) ;
		\draw  (b)  to  (c) ;
		\draw  (a)  to  node(e)[pos = 0.5 ,kpunkt] {}(c) ;
		
		\node[right] at (a){$\bar{x}$};
		\node[left] at (b){$\bar{y}$};
		\node[above] at (c){$\bar{z}$};
		\node[below right] at (d){$\bar{p}$};
		\node[above right] at (e){$\bar{q}$};
		\draw[dashed] (d) -- (e);
		\end{tikzpicture}
	\end{center}
	\caption{A thin triangle is depicted on the left and its comparison triangle in $\R^2$ on the right. The dashed line indicates that the distance of $p$ and $q$ on the left is less than the distance of their comparison points $\bar{p}$ and $\bar{q}$ on the right.} 
	\label{fig:thin_triangle}
\end{figure}

\begin{defi}\label{def:catoz}
	Let $(X,\di)$ be a geodesic space and $\si \subseteq X$ a triangle. Let $\bar{\si}_0$ be a comparison triangle for $\si$ in $(\R^2,\bar{\di})$, and $\bar{\si}_1$ a comparison triangle for $\si$ in $(S^2,\bar{\di})$, provided the diameter of $\si$ is less than $2\pi$. For $\kappa \in \Set{0,1}$ we say that $\si$ \emph{satisfies the $\mathrm{CAT}(\kappa)$ inequality} if for all $p,q\in \si$ and all comparison points $\bar{p}, \bar{q} \in \bar{\si}_\kappa$ it holds that
	\[
	\di(p,q) \leq \bar{\di}(\bar{p},\bar{q}).
	\]
	The space $X$ is called a \emph{$\catz$ space} if all its triangles satisfy the $\catz$ inequality, and $X$ is called a \emph{$\cato$ space} if all its triangles of perimeter less than $2\pi$ satisfy the $\cato$ inequality. A space $X$ is called \emph{locally $\mathrm{CAT}(\kappa)$} for $\kappa\in\Set{0,1}$ if every point in $X$ has a neighborhood that is $\mathrm{CAT}(\kappa)$ with the induced metric.
\end{defi}

\begin{rem}
	The notion of $\mathrm{CAT}(\kappa)$ space exists for all $\kappa \in \R$, where the comparison space $\R^2$ or $S^2$ is replaced by a two-dimensional Riemannian manifold of constant sectional curvature $\kappa$. 	
\end{rem}

Intuitively, if a triangle satisfies the $\mathrm{CAT}(\kappa)$ inequality, then the \enquote{interior} of it is \enquote{filled} and does not contain any holes. If there were a hole, the geodesic connecting two comparison points would have to make a detour around that hole and would become longer than the one in the comparison point. Vaguely formulated, finding out whether a space is $\mathrm{CAT}(\kappa)$ means finding out that every hole is filled.

\subsubsection{Coxeter complexes and buildings}

Coxeter complexes of finite Coxeter groups are, endowed with the spherical metric, homeomorphic to unit spheres, hence they are $\cato$ spaces. Spherical buildings are $\cato$ spaces as well \cite[Thm. II.10A.4]{bh}. Coxeter complexes of affine Coxeter groups, which are certain  infinite Coxeter groups, are $\catz$ spaces, as they are homeomorphic to the Euclidean space. Analogously, affine buildings, whose apartments are affine Coxeter complexes, are $\catz$ spaces \cite[Thm. II.10A.4]{bh}.
The crucial step in the proof that these buildings are $\cato$ or $\catz$, respectively, is that any two points are contained in a common apartment. This allows to directly prove the $\cato$ or $\catz$ property.

The property that any two points are contained in a common apartment is not longer true if we consider a complex that consists only of a part of the apartments of the building, such as the non-crossing partition complex, and curvature criteria are needed.

\subsubsection{Curvature criteria}

In general it is hard to find out whether a space is $\catz$ or $\cato$, also in the local versions, and general criteria are very rare. The situation is a little bit better if the metric space is a simplicial complex. We focus on criteria for $\cato$ here, but they have analogs for $\catz$ spaces as well.  

Our strategy that might prove the \cref{conj:ncpn_cat1} is based on the following criterion for spherical complexes \cite[Thm II.5.4]{bh}.

\begin{unique}\label{unique}
	Let $X$ be a finite spherical complex. Then $X$ is a $\cato$ space if and only if $X$ is $\pi$-uniquely geodesic.
\end{unique}

The next criterion is from Bowditch \cite[Thm. 3.1.2]{bow}, see also \cite[Le. 3.6]{bra-mcc}. Before we can state it, we need the notion of a shrinkable loop.
Let $X$ be a complete metric space and $\gamma\colon S^1 \to X$ be a rectifiable loop. Then $\gamma$ is called \emph{shrinkable} if there exists a null-homotopy $h\colon I\times S^1 \to X$ of $\gamma$ such that $h(t)\colon S^1 \to X$ is a rectifiable loop for all $t\in I$ and additionally, $\ell(h(t))\leq \ell(h(t_0))$ for all $t_0,t\in I$ with $t\geq t_0$. The space $X$ is called \emph{shrinkable} if there exists a contraction that shrinks every loop in $X$.

\begin{bowditch}\label{bow}
	Let $X$ be a finite, locally $\cato$ spherical complex. Then $X$ is a $\cato$ space if and only if every loop in $X$ of length less than $2\pi$ is shrinkable.
\end{bowditch}

Bowditch's criterion requires the spherical complex to be locally $\cato$. The following criterion is due to Gromov and for instance stated in \cite[Thm. II.5.2]{bh}.

\begin{gromov}\label{gromov}
	Let $X$ be a finite spherical complex. Then $X$ is locally $\cato$ if and only if for all vertices $v\in X$ the spherical link complex $\lk(v,X)$ is $\cato$.
\end{gromov}

In the link criterion, the link complex is endowed with a metric, which arises as follows. Let $X$ be a spherical simplicial complex of dimension $n$ and $v\in X$ be a vertex. Then the link $\lk(v,X)$ is isomorphic to the triangulated sphere, seen as simplicial complex, whose simplicial structure arises from intersecting $X$ with a small $(n-1)$-dimensional sphere in $X$ around $v$. Hence the link naturally is a spherical complex. See \cite[Chap. I.7]{bh} for the details.

\section{An approach towards the Curvature Conjecture}\label{sec:strategy}

The aim of this section is to outline a strategy that might prove the \cref{conj:ncpn_cat1}. 
Recall that $\on$ is isomorphic to a chamber subcomplex of a finite spherical building $\si$ of type $A_{n-1}$ by \cref{thm:embedding_Vp_VZ}. The interaction of the simplicial structures of $\on$ and $\si$ is outlined in \cref{sec:simpl_structure}.

By the \cref{unique}, the spherical complex $\on$ is $\cato$ if for every pair of points $x,y\in \on$ that have distance less than $\pi$, there exists a unique geodesic connecting $x$ and $y$. If we want to use this criterion, we have to find out whether geodesics of \enquote{close} points are unique. For this it is helpful to know that the geodesics joining $x$ and $y$ are contained in a smaller subcomplex, which is established by the following \cite[Le. 3.1]{marq}.

\begin{lem}\label{lem:geod_contained_in_conv}
	Let $\Sigma$ be the spherical realization of a Coxeter complex, $x,y\in \Sigma$ two points and $C,D\in\C(\Sigma)$ chambers such that $x\in C$ and $y\in D$. Then every geodesic joining $x$ and $y$ is contained in a minimal gallery connecting $C$ and $D$.
\end{lem}

Since every minimal gallery connecting $C$ and $D$ in $\Sigma$ is contained in the convex hull $\conv_\Sigma(C,D)$, every geodesic joining $x\in C$ and $y\in D$ is contained in the convex hull of $C$ and $D$. Note that the above lemma is only stated for Coxeter complexes, but the proof, which is based on \cite[Prop. 12.25]{ab}, also works for $\on$ instead of $\Sigma$. 

Now let $C,D\in \Cn$ be two chambers and let $x,y\in \on$ be two points 
such that $x\in C$ and $y\in D$. 
Since all geodesics joining $x$ and $y$ are contained in $\convn(C,D)$, it is enough to find out if $\convn(C,D)$ is $\pi$-uniquely geodesic. Since $\convn(C,D)$ is, as a chamber subcomplex of $\on$, a finite spherical complex, the \cref{unique} states that $\convn(C,D)$ is $\pi$-uniquely geodesic if and only if $\convn(C,D)$ is $\cato$. So we have to find out whether $\convn(C,D)$ is $\cato$.

\subsubsection{The first cases}

There are two cases for which we know that $\convn(C,D)$ is $\cato$.
The first is if $C$ and $D$ are contained in a common apartment $A$ of $\on$, which is $\cato$. 
The second case is that $C$ and $D$ have non-empty intersection. Then $\convn(C,D)$ is a subcomplex of $\sta(C\cap D, \on)$ and therefore contractible, hence $\cato$.\medskip 

So we may assume that $C$ and $D$ have empty intersection and that they are not contained in a common apartment in $\on$.
Now the strategy splits into two different cases. 
Either, the gallery distance of $C$ and $D$ in $\on$ \emph{equals} the distance measured in the building $\si$, or \emph{not}. 

\subsubsection{The equal-distance case}

First, let us consider the case $\din(C,D)=\dig(C,D)$, that is the minimal galleries in $\on$ have the same lengths as the ones in $\si$. Hence $\convn(C,D)$ is contained in $\convg(C,D)$ and in particular in an apartment $A$ of the building. Note that $A$ is not an apartment of $\on$, hence $\convn(C,D) \neq A$. 
At this point we want to use the \cref{bow} and show that $\convn(C,D)$ is shrinkable.

\subsubsection{The different-distance case}

Now let us consider the case $\din(C,D)>\dig(C,D)$, which means that the minimal galleries in $\on$ that connect $C$ and $D$ are longer than the minimal galleries in $\si$ connecting the two chambers. Consequently, $\convn(C,D) \not\subseteq \convg(C,D)$ and hence $\convn(C,D)$ is not contained in any apartment of the building $\si$. 
We distinguish the cases whether $C$ and $D$ are opposite in the building or not, that is whether $\dig(C,D)$ equals $\binom{n-1}{2}$ or is less. In both cases, we consider concrete examples in \cref{sec:dif_dis}, which may help to prove the general case. 

If $C$ and $D$ are opposite in the building, then the convex hull may not be simply connected. In the example, we use this to show that any loop that passes through different \enquote{strands} of the convex hull has length $\geq \pi$.

If $C$ and $D$ are not opposite in the building, we believe the convex hull is shrinkable, which is supported by a detailed discussion of an example.

\cleardoublepage

\chapter{Simplicial structure}\label{sec:simpl_structure}

In this chapter we describe the simplicial structure of the non-crossing partition complex $\on$ and the partition complex $\op$ as a chamber subcomplex of the spherical building associated to $\vs$. We interpret apartments and chambers graphically and show some basic structural properties of the so-called universal and base chambers in $\on$ and $\op$, respectively. Finally, we further investigate the Coxeter complex of type $A$. All in all, this chapter provides the needed background for \cref{chap:dist}.

For the rest of this part, we make the following conventions. We suppose that $n$ is an arbitrary positive integer. An element $\pi \in \ncpn$ is in its pictorial representation if not stated otherwise. The group-theoretic non-crossing partitions $\nc(S_n)$ are with respect to the Coxeter element $\cox=(1\,\;2\ldots n)$ of $S_n$. Moreover, we choose the isomorphism of the two lattices $\iota \colon \ncpn \to \nc(S_n)$ that is the inverse of the isomorphism from \cref{thm:typeA_nc_iso}. Finally, we set $\si \coloneqq \lam(\vs)$.

\section{Embeddings}

In this section we illustrate the pictorial version of the embedding of $\ncpn$ into the lattice $\lam(\vs)$. From this we derive a compatible embedding of the partition lattice $\pn \to \lam(\vs)$.

\subsection{Pictorial non-crossing partitions}\label{sec:emb_typeA}
Let us quickly recall the group-theoretic embedding $\emb\colon\nc(S_n) \to \lam(\vs)$ from \cref{sec:typeA_embedding}. By $e_i$ we denote the $\ith$ standard basis vector of $\vs$ for $1\leq i <n$. The assignment of reflections to one-dimensional subspaces $\emb\colon T \to \lam(\vs)$ given by
\[
(i\,\;j)\mapsto
\begin{cases}
\Braket{e_i + e_j}  & \text{ if }j<n,\\
\Braket{e_i}&\text{ if }j=n
\end{cases}
\]
induces the embedding $\emb\colon \ncpn \to \lam(\vs)$ via
\[
t_1\ldots t_k \mapsto \emb(t_1) \vee \ldots \vee \emb(t_k)
\]
by \cref{lem:pemb_join_interchange}, provided that $t_1\ldots t_k$ is a reduced decomposition. Pre-composition of $f$ with the isomorphism $\iota \colon \ncpn \to \nc(S_n)$ gives rise to an embedding $\ncpn \to \lam(\vs)$. The following is devoted to establish an explicit description of this embedding in pictorial terms. 

Recall that the rank $1$ elements of $\ncpn$ are in bijection with the edges of the regular $n$-gon in the plane.

\begin{con}\label{conv:edges}
	We identify the edge set on the vertices of the $n$-gon with the set of rank $1$ elements of $\ncpn$ and denote it by $\E$. An edge is either denoted by $\Set{i,j}$ or $(i,j)$. If we use the notation $(i,j)$, then we assume that $i<j$.
\end{con}

The restriction of $\iota\colon \ncpn \to \nc(S_n)$ to $\E$ gives rise to the bijection
\[
\iota \colon\E \to T,\quad (i,j) \mapsto (i\,\;j),
\]
which leads to an injective map 
\[
\emb_\E\coloneqq \emb\circ\iota\colon\E \to \lam(\vs), \quad (i,j) \mapsto 
\begin{cases}
\Braket{e_i + e_j}  & \text{ if }j<n,\\
\Braket{e_i}&\text{ if }j=n.
\end{cases}
\]

In order to get an embedding of $\ncpn$ that is completely described in pictorial terms, we have to characterize the set of edges that correspond to reduced decompositions. In order to do so, we need some terminology from graph theory.

\begin{defi}
	A \emph{tree} is a connected graph without cycles and a \emph{forest} is a graph whose connected components are trees. 	
	By a \emph{non-crossing forest} we mean an embedded forest on a subset of the labeled vertices of the $n$-gon without crossing edges. A tree is called \emph{spanning} if it connects all vertices of the $n$-gon. A \emph{labeling} of a graph with $k$ edges and edge set $A$ is a bijective map $\lambda\colon \Set{1, \ldots, k} \to A$. A tree together with a labeling is called \emph{labeled tree}.
\end{defi}

The following lemma classifies reduced decompositions in terms of non-crossing trees \cite[Thm. 2.2, Le. 2.5]{gy}.

\begin{lem}\label{lem:nc_spanning_tree_red_decomp}
	Let $\Set{a_1, \ldots, a_{n-1}}\subseteq \E$ be a set of edges on the vertices of the $n$-gon. Then there exists a permutation $\sigma \in S_n$ such that $\iota(a_{\sigma(1)})\ldots \iota(a_{\sigma({n-1})})$ is a reduced decomposition of the Coxeter element if and only if $a_1, \ldots, a_{n-1}$ are the edges of a non-crossing tree.
\end{lem}

\begin{cor}\label{cor:nc_forest}
	A set of edges $A\subseteq \E$ corresponds to a reduced decomposition of an element $w\in S_n$ if and only if $A$ is the edge set of a non-crossing forest.
\end{cor}

\begin{proof}
	Let $A$ be a edge set of a non-crossing forest $F$ on the vertices of a regular $n$-gon. Then $F$ can be extended to a non-crossing spanning tree, which corresponds to a reduced decomposition $\tau$ of the Coxeter element by the above lemma. The reflections of $\tau$ that correspond to edges in $A$ form, as a subword of a reduced decomposition, a reduced decomposition by the \cref{subword_prop}. 
	
	Now let $\tau=t_1\ldots t_k$ be a reduced decomposition. By the \cref{prefix}, $\tau$ can be extended to a reduced decomposition $\tau'$ of the Coxeter element. This reduced decomposition $\tau'$ gives rise to a non-crossing spanning tree by the above lemma. The reflections of $\tau$ hence correspond to a subset $A$ of edges of a non-crossing spanning tree, which form a non-crossing forest.
\end{proof}

\begin{nota}
	Let $A\subseteq \E$ be the edge set of an embedded forest on the vertices of the regular $n$-gon. By abuse of notation we refer to $A$ as forest and omit both the term \enquote{edge set of} and its vertex set, which we always assume to be the vertices of the regular $n$-gon in the plane. 
\end{nota}

\begin{defi}
	Let $\pi \in \ncpn$ be a non-crossing partition. A \emph{non-crossing spanning forest for $\pi$} is a forest $A(\pi)\subseteq \E$ such that $\bigvee A(\pi) = \pi$ in $\ncpn$. 
\end{defi}

By \cref{cor:nc_forest} the forest $A$ is a non-crossing forest of $\pi \in \ncpn$ if and only if it corresponds to a reduced decomposition of $\iota(\pi)$ under the isomorphism $\iota$. Just as reduced decompositions of an element $w\in \nc(S_n)$ are not uniquely determined in general, non-crossing forest are not unique as well. \cref{fig:spanning_forests} shows two non-crossing forests for the same non-crossing partition and two corresponding reduced decompositions. Note that to a spanning non-crossing forest there may correspond different reduced decompositions. The reason is that the forest only determines the \emph{set} of reflections of the reduced decomposition, not their order. \cref{lem:good_labeling} shows which orders of edges correspond to reduced decompositions.

\begin{figure}
	\begin{center}
		\begin{tikzpicture}
		\foreach \w in {1,...,6}
		\coordinate(e\w) at (\w*60 : 8mm);
		\fill[Superlightgray] (e1)--(e2)--(e3)--(e4)--(e6)--(e1);
		\foreach \w in {1,...,6}
		\node[kpunkt] at (e\w){};
		\draw[thick] (e1)--(e6)(e1)--(e4)(e2)--(e4)(e2)--(e3);
		
		\begin{scope}[xshift = 4cm]
		\foreach \w in {1,...,6}
		\coordinate(e\w) at (\w*60 : 8mm);
		\fill[Superlightgray] (e1)--(e2)--(e3)--(e4)--(e6)--(e1);
		\foreach \w in {1,...,6}
		\node[kpunkt] at (e\w){};
		\draw[thick] (e1)--(e2)(e6)--(e4)(e2)--(e4)(e2)--(e3);
		\end{scope}
		\end{tikzpicture}
	\end{center}
	\caption{Two non-crossing spanning forests of the same non-crossing partition are shown here. A reduced decomposition corresponding to the left forest is $(3\,\;6)(1\,\;6)(4\,\;5)(3\,\;5)$, and $(1\,\;3)(4\,\;5)(3\,\;5)(5\,\;6)$ corresponds to the right forest.} 
	\label{fig:spanning_forests}
\end{figure}
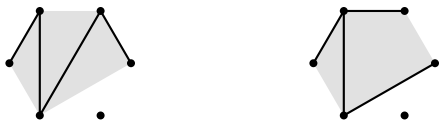

\begin{thm}\label{thm:ncpn_pict_embedding}
	The map $\emb\colon \ncpn \to \lam(\vs)$ given by
	\[
	\pi = \bigvee A(\pi) \mapsto \bigvee_{a\in A(\pi)} \emb_\E(a)
	\]
	is a rank-preserving poset embedding.
\end{thm}

\begin{proof}
	By \cref{cor:nc_forest} every non-crossing forest gives rise to a reduced decomposition of an element of $\nc(S_n)$ and vice versa. 
	Suppose that $\Set{a_1, \ldots, a_k}$ is a non-crossing spanning forest of $\pi\in\ncpn$, then, without loss of generality, $\iota(a_1)\ldots \iota(a_k)$ is a reduced decomposition of $\iota(\pi)$. By \cref{lem:pemb_join_interchange}, $\emb(\iota(\pi))$ is given by $\emb(\iota(a_1))\vee\ldots \vee\emb(\iota(a_k))$, which equals $\emb_\E(a_1)\vee\ldots\vee \emb_\E(a_k)$ by definition of $\emb_\E$. Hence the map is well-defined and the remaining statement follows, since pre-composition of the rank-preserving poset embedding from \cref{thm:embedding_Vp_VZ} with the lattice isomorphism $\iota$ is again a rank-preserving poset embedding.
\end{proof}

The above considerations lead to some observations.
First, we can use edges and forests to describe apartments and chambers in $\on$.
Second, the embedding of $\ncpn$ based on the embedding of edges gives rise to an embedding of the partition lattice $\pn$. The embedding of the partition lattice is described in the next section and the pictorial description of apartments and chambers is the subject of \cref{sec:aptms_cham_nc_trees}. 

\subsection{Pictorial partitions}

In \cite[Le. 2.24]{hks} it is shown that the partition lattice $\pn$ embeds into $\lam(V)$ for an $(n-1)$-dimensional subspace of $\F^n$ for an arbitrary field $\F$. 
The aim of this section is to explicitly define an embedding of $\pn$ in pictorial terms such that its restriction to $\ncpn$ is the embedding from \cref{thm:ncpn_pict_embedding}.

\begin{defi}
	Let $\pi\in \pn$ be a partition and $a=(i,j)$ an edge in $\E$. We say that $a$ is an \emph{edge of the partition} $\pi$ if there exists a block $B\in \pi$ such that $i,j\in B$.
	A \emph{spanning forest} of $\pi$ is a forest $A(\pi)\subseteq \E$ such that $\bigvee A(\pi) = \pi$ in $\pn$. 
\end{defi}

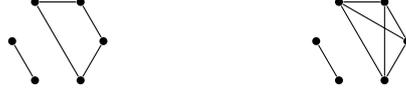
\begin{figure}
	\begin{center}
		\begin{tikzpicture}
		\begin{scope}
		\gsechseck
		\draw(p1)--(p2)(p5)--(p2)(p5)--(p6)(p1)--(p6)(p3)--(p4);
		\end{scope}
		\begin{scope}[xshift=4cm]
		\gsechseck
		\draw(p1)--(p2)(p5)--(p2)(p5)--(p6)(p1)--(p6)(p3)--(p4)(p1)--(p5)(p6)--(p2);
		\end{scope}
		\end{tikzpicture}
		\caption{A pictorial representation of a partition is shown on the left and its corresponding set of edges on the right.}
		\label{fig:edges_of_a_partition}
	\end{center}
\end{figure}

The edges of a partition are not the edges of the pictorial representation graph. But we get all edges of a partition from its pictorial representation by adding a complete graph on the vertex set of each of its blocks. \cref{fig:edges_of_a_partition} shows a partition and its edges. Note that a spanning forest of a partition $\pi$ necessarily consists of edges of $\pi$. The same is true for non-crossing spanning forests of non-crossing partitions. 

\begin{lem}\label{lem:rank_forest}
	Let $\pi \in \pn$ be a partition and $A(\pi) \subseteq \E$ a spanning forest of it. Then the rank of $\pi$ equals the cardinality of $A$, that is $\rk(\pi)=\#A(\pi)$.
\end{lem}

\begin{proof}
	The rank of a partition $\pi\in \pn$ is defined as $\rk(\pi)=n-\#\pi$. We show that $\#A(\pi)+\#\pi=n$. For this note that for every block $B\in \pi$ there are $\#B-1$ edges contained in $A(\pi)$. Using this we get that
	\begin{align*}
	\#A(\pi)+\#\pi=\sum_{B\in \pi}(\#A(B)+1)=\sum_{B\in \pi}(\#B-1+1)=\sum_{B\in \pi}\#B = n.
	\end{align*}
\end{proof}

\begin{lem}\label{lem:cycles_dependent}
	Let $A=\Set{a_1, \ldots, a_k} \subseteq \E$ be a set of edges and $v_1, \ldots, v_k \in \vs$ such that $\Braket{v_i} = \emb_\E(a_i)$ for $1\leq i \leq k$. Then $A$ contains a cycle if and only if $\Set{v_1, \ldots, v_k}$ is linearly dependent. 
\end{lem}

\begin{proof}
	Let $A=\Set{a_1, \ldots, a_k} \subseteq \E$ be a set of edges and $v_1, \ldots, v_k \in \vs$ be such that $\Braket{v_i} = \emb_\E(a_i)$. 
	
	First suppose that $A$ contains cycle. Without loss of generality we assume that all edges of $A$ are contained in the cycle. Then the edges are of the form $a_j=\Set{i_j,i_{j+1}}$ for $1\leq j <k$ and $a_k=\Set{i_k,i_1}$ for a subset $I\coloneqq \Set{i_1, \ldots, i_k} \subseteq \Set{1, \ldots, n}$. Then $\emb_\E(a_j)=e_{i_j}+e_{i_{j+1}}$, where we interpret $e_{n+1}$ as zero if present.  Note that every vertex $i_j\in I$ is contained in exactly two edges. Hence  $\sum_{j=1}^k v_j = 0$, since every vector $e_{i_j}$ appears twice in the sum. 
	Consequently, $\Set{v_1, \ldots, v_k}$ is a linearly dependent set.
	
	For the other direction suppose that $\Set{v_1, \ldots, v_m}$ is a linearly dependent set. Without loss of generality we assume that it is minimal, that is, each of its subsets is linear independent. We show that $A$ is a cycle.
	It holds that $\sum_{j=1}^k v_j = 0$, because $\Set{v_1, \ldots, v_m}$ is linearly dependent and minimal with this property.
	Since the vectors $v_m$ are the sum of exactly one or two standard basis vectors of $\vs$ by the definition of $\emb$, every standard basis vector that appears once in the sum already has to appear twice.
	Hence every endpoint of an edge $a_j$ has to be contained in exactly one other edge, which means that $A$ is a cycle.
\end{proof}

\begin{cor}\label{cor:basis_edges}
	Let $\pi\in \pn$ be a partition and $A=\Set{a_1, \ldots, a_k}$ a spanning forest of it. Then the set of non-zero vectors of $\emb(a_1), \ldots, \emb(a_k)$ is a basis of $\emb(\pi) \in \lam(\vs)$.
\end{cor}

\begin{thm}\label{thm:pn_pict_embedding}
	The map $\emb\colon \pn \to \lam(\vs)$ given by
	\[
	\pi = \bigvee A(\pi) \mapsto \bigvee_{a\in A(\pi)} \emb_\E(a)
	\]
	is a rank-preserving poset embedding.
\end{thm}

\begin{proof}
	We have to prove that $\emb$ is a poset map that is rank-preserving and injective. To show that $\emb$ is a poset map, let $\pi_1,\pi_2\in \pn$ be such that $\pi_1\leq \pi_2$. Then there exists a spanning forest $A_1$ of $\pi_1$ that contains a spanning forest $A_2$ of $\pi_2$. Hence it holds that  $\bigvee_{a\in A_1} \emb_\E(a) \leq \bigvee_{a\in A_2} \emb_\E(a)$ and $\emb$ is a poset map.

	If $A$ is a spanning forest of $\pi$, then the vectors corresponding to $\emb_\E(a)$ for $a\in A$ are linearly independent by \cref{lem:cycles_dependent}, hence a basis of $\emb(A)$. With \cref{lem:rank_forest} it follows that 
	\[
	\rk(\pi)=\#A=\dim(\emb(\pi)) = \rk(\emb(\pi))
	\]
	and therefore $\emb$ is rank-preserving.
	
	In order to show that $\emb$ is injective, suppose that $\pi_1,\pi_2\in \pn$ are such that $\pi_1\neq \pi_2$. Hence there exists an edge $a$ of $\pi_1$ that is not an edge of $\pi_2$. We show that $\emb(a)$ is not contained in $\emb(\pi_2)$, which proves $\emb(\pi_1)\neq \emb(\pi_2)$. Suppose to the contrary that $\emb(a)\subseteq \emb(\pi_2)$. If $A=\Set{a_1, \ldots, a_k}$ is a spanning forest for $\pi_2$, then there exists a subset $M \subseteq \Set{1, \ldots, k}$ such that $v=\sum_{m\in M}v_m$, where by $v_m$ we denote the unique vector spanning $\emb_\E(a_m)$ and $\Braket{v}=\emb_\E(a)$. Since $\Set{v_1, \ldots, v_m,v}$ is linearly dependent, the corresponding edges form a cycle by \cref{lem:cycles_dependent}, that is the edge $a$ is in the same connected component as $\Set{a_i\str i\in M }\subseteq A$. 
	This means that $a\in \bigvee_{m \in M} a_m \in \pi_2$ by the definition of the join in $\pn$, and hence that $a$ is an edge of $\pi_2$, which is a contradiction.
\end{proof}

\begin{rem}
	Each connected component of a spanning forest of a partition can be chosen to be non-crossing. Hence, the restriction of the above embedding to $\ncpn$ coincides with the embedding $\emb\colon \ncpn\to\lam(\vs)$ from \cref{thm:ncpn_pict_embedding}. This justifies the use of the common letter $\emb$ describing the embeddings.
\end{rem}

\subsection{Chamber subcomplexes of the building}

Being equipped with the poset-theoretic pictorial versions of the embedding of $\ncpn$ and $\pn$, we focus on the simplicial point of view from now on. 
The aim of this section is to settle notation and give some basic descriptions
of the simplicial complexes $\on$ and $\op$ inside the spherical building $\si = |\lam(\vs)|$. 

The poset-theoretic embeddings of $\ncpn$ and $\pn$ into the lattice of linear subspaces of $\vs$ from \cref{thm:ncpn_pict_embedding} and \cref{thm:pn_pict_embedding} translate into the following simplicial version by realizing the lattices as order complexes.

\begin{thm}
	The simplicial complexes $\on$ and $\op$ are isomorphic to chamber complexes of $\si$.
\end{thm}

\begin{con}\label{con:subcompl}
	We identify $\on$ and $\op$ with their images in $\si$, that is we see them as \emph{subcomplexes} of $\si$. Moreover, we identify $\ncpn$ and $\pn$ with their images in $\lam(\vs)$ as well.
\end{con}

\begin{rem}
	Note that $\on$ and $\op$ are induced subcomplexes of $\si$, as they arise as order complexes of subposets	of $\lam(\vs)$. This means in particular that whether a simplex of $\si$ is contained in one of the induced subcomplexes is completely determined by its vertex set. If all vertices are (non-crossing) partitions, then the simplex they span is in $\on$ or $\op$, respectively.
\end{rem}

In order to create a more convenient way to study the building $\si$ and its subcomplexes, we make a couple of conventions. 

\begin{con}\label{con:general}
	From now on, the following conventions become effective.
	\begin{enumerate}
		\item The embeddings of $\ncpn$ and $\pn$ into $\lam(\vs)$ are rank-preserving, hence the notion of \enquote{rank} is well-defined and we use the term \enquote{rank of a vertex} without further specification. 
		Recall that in \cref{sec:sph_build}, we defined the type of a vertex in the barycentric subdivision of the boundary on a simplex as the cardinality of its label. This type of a vertex, seen as a vertex of an apartment, equals its rank in the corresponding Boolean lattice and hence in the building $\si$. Because of this, we use terms \enquote{type of a vertex} and \enquote{rank of a vertex} interchangeably.
		
		\item The identification of $\on$ and $\op$ with subcomplexes of $\si$ leads to two kinds of vertex labels. Either, the label can be seen as a partition, or as a subspace of $\vs$. Depending on the circumstances, we address the label of a vertex as a partition or a subspace, and sometimes omit the term \enquote{label of} as well.
		
		\item Recall that the one-dimensional subspaces of $\vs$, which are the rank $1$ vertices of $\si$, contain exactly one non-zero vector. If there is no danger of confusion, we refer to a rank $1$ vertex as a vector.
		
		\item We fix the set $\Set{e_1, \ldots, e_{n-1}}$, consisting of the standard basis vectors of $\vs$, as the standard basis of $\vs$. This allows us to use coordinates of vectors.
		
		\item The notions of meet and join in a lattice do not have analogs in the simplicial version. Whenever we write $u\vee v$ for vertices $u,v\in \si$, we mean the vertex that corresponds to the element that arises by taking the join in $\lam(\vs)$. The convention for the meet is analogously. Although the vertices may correspond to the subcomplex $\on$ or $\op$, the meet and join operations are always the ones from the lattice $\lam(\vs)$ if not explicitly stated otherwise.  
	\end{enumerate} 
\end{con}

\subsubsection{Vertex enumeration}
In order to imagine how large the complexes $\on$, $\op$ and $\si$ are in comparison, we enumerate the number of its rank $k$ vertices. \cref{tab:numbers} and \cref{fig:numbers} show the relative number of vertices of $\ncpn$ and $\pn$ inside $\lam(\vs)$ for small $n$.

The number of non-crossing partitions in $\ncpn$ with $k$ blocks, that is of rank $n-k$, is given by the \emph{Narayana number} \cite[Le. 3.1]{edel}, which is
\[
N(n,k) \coloneqq \frac{1}{n} \binom{n}{k} \binom{n}{k-1}.
\]
The total number of non-crossing partitions in $\ncpn$ is counted by the \emph{Catalan number} \cite[Cor. 4.2]{kre} and is given by
\[
C_n \coloneqq \frac{1}{n+1}\binom{2n}{n}.
\]

The number of partitions of $\pn$ with $k$ blocks is counted by the \emph{Stirling number of the second kind} \cite[p. 73f.]{staece}, which is
\[
S(n,k)\coloneqq \frac{1}{k!}\sum_{i=0}^{k}(-1)^{k-i}\binom{k}{i}i^n. 
\]
The total number of partitions $\pn$ is given by the \emph{Bell number}, which is 
\[
B_n \coloneqq \sum_{k=1}^{n} S(n,k).
\]

Finally, the number of $k$-dimensional subspaces of $\F_2^n$ is given by the Gaussian binomial coefficient 
\[
\binom{n}{k}_2 \coloneqq \prod_{i=0}^{k-1} \frac{2^n-2^i}{2^k-2^i}.
\]
\begin{table}
	\begin{center}
		\begin{tabular}{|r||r|r|r|}
			\hhline{|-||-|-|-|}
			n & $\#\ncpn$ & $\#\pn$ &$\#\lam(\vs)$\\
			\hhline{:=::=:=:=:}
			1	&	1	&	1	&	1	\\
			2	&	2	&	2	&	2	\\
			3	&	5	&	5	&	5	\\
			4	&	14	&	15	&	16	\\
			5	&	42	&	52	&	67	\\
			6	&	132	&	203	&	374	\\
			7	&	429	&	877	&	2\,825	\\
			8	&	1\,430	&	4\,140	&	29\,212	\\
			9	&	4\,862	&	21\,147	&	417\,199	\\
			10	&	16\,796	&	115\,975	&	8\,283\,458	\\
			\hhline{|-||-|-|-|}
		\end{tabular}
	\end{center}
	\caption{The number of elements in $\ncpn$, $\pn$ and $\lam(\vs)$ for small $n$.}
	\label{tab:numbers}
\end{table}

Although the complex $\op$ has much fewer vertices than the building $\si$, $\op$ is dispersed in $\si$ in the following sense.

\begin{lem}
	Every rank $n-2$ vertex of the building $\si$ is contained in $\op$.
\end{lem} 

\begin{proof}
	The number of rank $n-2$ vertices in $\op$ is counted by the Stirling number of the second kind $S(n,2)$, which is $2^{n-1}-1$ \cite[p. 74f]{staece}. The number of rank $n-2$ vertices of $\si$ equals the Gaussian binomial coefficient $\binom{n-1}{n-2}_2$, which computes as
	\begin{align*}
	\binom{n-1}{n-2}_2 
	&= \prod_{i=0}^{n-3} \frac{2^{n-1}-2^i}{2^{n-2}-2^i}
	=  \prod_{i=0}^{n-3} \frac{2^i(2^{n-1-i}-1)}{2^i(2^{n-2-i}-1)}
	=  \prod_{i=0}^{n-3} \frac{2^{n-1-i}-1}{2^{n-2-i}-1}\\[6pt]
	&= \frac{2^{n-1-0}-1}{1} \cdot \left(\prod_{i=1}^{n-4} \frac{2^{n-1-i}-1}{2^{n-2-(i-1)}-1}\right) \cdot \frac{1}{2^1-1}\\[6pt]
	&= \left(2^{n-1}-1\right) \cdot\prod_{i=1}^{n-4} \frac{2^{n-1-i}-1}{2^{n-1-i}-1}= 2^{n-1}-1.
	\end{align*}
\end{proof}

\begin{figure}
	\begin{center}
		\begin{tikzpicture}
		\foreach \n/\hp/\hn [evaluate=\n as \j using int(\n+2)] in {1/5.000/5.000,
			2/4.688/4.375,
			3/3.881/3.134,
			4/2.714/1.765,
			5/1.552/0.759,
			6/0.709/0.245,
			7/0.253/0.058,
			8/0.070/0.010}
		{
			\draw[draw=none,fill=Superlightgray] (1.5*\n-1,0) rectangle (1.5*\n,5);
			\draw[draw=none,fill=Blue] (1.5*\n-1,0) rectangle (1.5*\n,\hp);
			\draw[draw=none,fill=Orange] (1.5*\n-1,0) rectangle (1.5*\n,\hn);
			\node at (1.5*\n-0.5,-0.5){$\j$}; 
		};
		\end{tikzpicture}
		\caption{The relative numbers of vertices of $\ncpn$ are shown in orange and those of $\pn$ in blue inside $\lam(\vs)$ for $n=3, \ldots, 10$.}
		\label{fig:numbers}
	\end{center}
\end{figure}

\subsubsection{Partition subspaces}
For the following considerations it is helpful to know which subspaces of $\vs$ describe partitions, that is, which subspaces lie in the image of $\on$ or $\op$, respectively. 

\begin{defi}
	A subspace $U \subseteq \vs$ is called \emph{partition subspace} if its corresponding vertex in $\si$ is contained in the subcomplex $\op \subseteq \si$. 
	Analogously, a subspace $U$ is called \emph{non-crossing partition subspace} if it corresponds to a vertex in $\on \subseteq \si$.
\end{defi}

A vector $v\in \vs$ is called \emph{partition vector} if it spans a one-dimensional partition subspace. This means it is of the form $e_i+e_j$ or $e_i$ for standard basis vectors of $\vs$ by the definition of $\emb$. Hence $v$ is a partition vector if and only if it has exactly one or two non-zero entries. Moreover, partition vectors are in bijection with edges of the $n$-gon via the map $\emb_\E$.

\begin{lem}
	Let $U \subseteq \vs$ be a vector space. Then $U$ is a partition vector space if and only if it has a basis consisting of partition vectors.
\end{lem}

\begin{proof}
	Let $U$ be a partition subspace and $\pi \in \pn$ a partition such that $\emb(\pi)=U$. If $A=\Set{a_1, \ldots, a_k}$ is a spanning forest of $\pi$, then $\emb(a_1), \ldots, \emb(a_k)$ is a basis of $\emb(\pi)$ by \cref{cor:basis_edges}, which consists of partition vectors. For the other direction, let $\Set{v_1, \ldots, v_k}$ be a basis of the subspace $U \subseteq \vs$ consisting of partition vectors. Their preimages are edges of $\pi$, which form a spanning tree, since $v_1, \ldots, v_k$ are linearly independent, and therefore do not form cycles. 
\end{proof}

\section{Apartments and chambers}

In this section we describe apartments and chambers of $\on$ and $\op$ by trees. Moreover, we study the properties of universal and base chambers, which are particular chambers of $\on$ and $\op$, respectively. 
\subsection{Descriptions by trees}\label{sec:aptms_cham_nc_trees}

The description of apartments of $\on$ by non-crossing trees first appeared in \cite{hks}. 

\subsubsection{Apartments}

Recall from \cref{sec:sph_build} that apartments of $\si$ are in bijection with frames of $\vs$, where a frame of $\vs$ is a set $\Set{L_1, \ldots, L_{n-1}}$ of one-dimensional subspaces such that $L_1+\ldots + L_{n-1}= \vs$. Because every one-dimensional subspace of $\vs$ contains exactly one element different from zero, the apartments of the building $\si$ are in bijection with the bases of $\vs$. 

\begin{lem}
	The number of apartments in $\si$ equals $\frac{1}{(n-1)!} \prod_{i=0}^{n-2} (2^{n-1}-2^i)$.
\end{lem}

\begin{proof}
	The number of apartments of $\si$ equals the number of the bases of $\vs$, which is well-known from linear algebra to be $\frac{1}{(n-1)!} \prod_{i=0}^{n-2} (2^{n-1}-2^i)$.
\end{proof}

An apartment $A\in \A(\si)$ is an apartment of $\on$ or $\op$ if all rank $1$ vertices of $A$ are contained in $\on$ or $\op$, respectively. In the sequel we describe the apartments of $\on$ and $\op$ pictorially by trees.

In \cref{lem:nc_spanning_tree_red_decomp} we already saw that every non-crossing spanning tree gives rise to a reduced decomposition of the Coxeter element and hence an apartment of $\on$ by \cref{cor:construction_boolean_sublatice}. The other way around, every apartment of the non-crossing partition complex $\on$ can be described by a non-crossing spanning tree as well \cite[Prop. 4.4]{hks}.

\begin{lem}\label{lem:aptms}
	Every apartment of $\on$ is described by a non-crossing spanning tree on the vertices of the $n$-gon and vice versa.
\end{lem} 

Using our embedding $\emb$, we get an explicit bijection from the set of non-crossing trees and apartments of $\on$ by sending the edge set $\Set{a_1, \ldots, a_{n-1}}$ of the non-crossing spanning tree to the frame $\Set{\emb(a_1), \ldots, \emb(a_{n-1})}$.

\begin{cor}
	The number of apartments in $\on$ equals $\frac{1}{2n-1}\binom{3n-3}{n-1}$.
\end{cor}

\begin{proof}
	The number of apartments of $\on$ equals the number of non-crossing spanning trees on $n$ vertices, which is $\frac{1}{2n-1}\binom{3n-3}{n-1}$ \cite[Cor. 1.2]{noy}.
\end{proof}

Recall that \cref{lem:nc_spanning_tree_red_decomp} relates reduced decompositions of the Coxeter element $\cox$ with non-crossing spanning trees: every reduced decomposition gives rise to a non-crossing spanning tree and vice versa, although the correspondence is not one-to-one. Different reduced decompositions of $\cox$ correspond to the same non-crossing spanning tree if they differ by commuting reflections. Hence we get the following group-theoretic description of apartments.

\begin{lem}
	Every apartment of $\on$ is described by a reduced decomposition of the Coxeter element $(1\ldots n)$ and vice versa.
\end{lem}

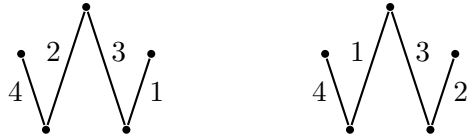
\begin{figure}
	\begin{center}
		\begin{tikzpicture}[every path/.style={thick}]
		\begin{scope}
		\fnfeck
		\draw(p1)to node[midway, right]{1}(p2)
		(p2)to node[pos=0.65, right]{3}(p5)
		(p5)to node[pos=0.35, left]{2}(p3)
		(p3)to node[midway, left]{4}(p4);
		\end{scope}
		
		\begin{scope}[xshift = 4cm]
		\fnfeck
		\draw(p1)to node[midway,  right]{2}(p2)
		(p2)to node[pos=0.65, right]{3}(p5)
		(p5)to node[pos=0.35, left]{1}(p3)
		(p3)to node[midway,  left]{4}(p4);
		\end{scope}
		
		\end{tikzpicture}
		\caption{The non-crossing tree correpsonding to the reduced decompositions $\tau_1=(1\,\;2)(3\,\;5)(2\,\;5)(3\,\;4)$ and $\tau_2=(3\,\;5)(1\,\;2)(2\,\;5)(3\,\;4)$. The induced labeling of $\tau_1$ is depicted on the left-hand side and that of $\tau_2$ on the right-hand side.}
		\label{fig:tree_of_aptm}
	\end{center}
\end{figure}

\begin{exa}
	Consider the reduced decomposition $\tau_1=(1\,\;2)(3\,\;5)(2\,\;5)(3\,\;4)$ of the Coxeter element of $S_5$. Sending each reflection of the reduced decomposition to a subspace of $\F_2^4$ via the embedding $\emb\colon \nc(S_5) \to \lam(\F_2^4)$ yields the basis 
	$\Set{e_1+e_2,e_3,e_2,e_3+e_4}$ of $\F_2^4$. The corresponding tree that arises from the reduced decomposition via the isomorphism $\iota\colon \ncp_5 \to \nc(S_5)$ is depicted in \cref{fig:tree_of_aptm}. The labeling of the tree on the left-hand side is induced from the reduced decomposition $\tau_1$ by assigning the edge corresponding to the reflection on position $i$ the label $i$.
	
	Now consider the reduced decomposition $\tau_2=(3\,\;5)(1\,\;2)(2\,\;5)(3\,\;4)$, which arises from $\tau_1$ by swapping the first two reflections. It gives rise to the same apartment as $\tau_1$ and hence the same tree, but the labeling is different and shown on the right-hand side of \cref{fig:tree_of_aptm}. 
	Note that there are labelings of trees which do \emph{not} correspond to a reduced decomposition of the Coxeter element, since swapping two non-commuting reflections alters the  decomposed element.
\end{exa}

The following is implied by \cite[Thm. 2.2, Le. 2.5]{gy}.

\begin{lem}\label{lem:good_labeling}
	Let $T$ be a non-crossing spanning tree labeled by $\lambda\colon\Set{1,\ldots, n-1} \to T$. Then $\iota(\lambda(1))\ldots \iota(\lambda(n-1))$ is a reduced decomposition of the Coxeter element of $S_n$ if and only if the edge labels are increasing counterclockwise in the interior of the $n$-gon around each vertex.
\end{lem}

\begin{figure}
	\begin{center}
		\begin{tikzpicture}
		\begin{scope}
		\foreach \w in {1,...,5} 
		\node (p\w) at (-\w * 360/5 +90  : 13mm) [kpunkt] {};
		\draw[thick](p1)to node[midway,right]{2}(p2)
		(p2)to node[pos=0.65, right]{3}(p5)
		(p5)to node[pos=0.35, left]{1}(p3)
		(p3)to node[midway, left]{4}(p4);
		\draw[->] (p2)+(30:3mm) arc (30:160:3mm);
		\draw[->] (p3)+(30:3mm) arc (30:160:3mm);
		\draw[->] (p4)+(-120:3mm) arc (-120:30:3mm);
		\draw[->] (p5)+(-150:3mm) arc (-150:-20:3mm);
		\draw[->] (p1)+(150:3mm) arc (150:320:3mm);
		\end{scope}
		
		\begin{scope}[xshift = 5 cm]
		\foreach \w in {1,...,5} 
		\node (p\w) at (-\w * 360/5 +90  : 13mm) [kpunkt] {};
		\draw[thick](p1)to node[midway,  right]{1}(p2)
		(p2)to node[pos=0.65, right]{4}(p5)
		(p5)to node[pos=0.35, left]{3}(p3)
		(p3)to node[midway,  left]{2}(p4);
		\draw[->] (p2)+(30:3mm) arc (30:160:3mm);
		\draw[->,thick, Orange] (p3)+(30:3mm) arc (30:160:3mm);
		\draw[->] (p4)+(-120:3mm) arc (-120:30:3mm);
		\draw[->] (p5)+(-150:3mm) arc (-150:-20:3mm);
		\draw[->] (p1)+(150:3mm) arc (150:320:3mm);
		\node[Orange, kpunkt] at (p3){};
		\end{scope}
		
		\end{tikzpicture}
	\end{center}
	\caption{The labeling on the left satisfies the condition of \cref{lem:good_labeling} and hence gives rise to the reduced decomposition $(3\,\;5)(1\,\;2)(2\,\;5)(3\,\;4)$ of the Coxeter element in $S_5$. The labeling on the right does not correspond to a reduced decomposition, as the condition is violated for the vertex $3$.}
	\label{fig:good_bad_labelings}
\end{figure}
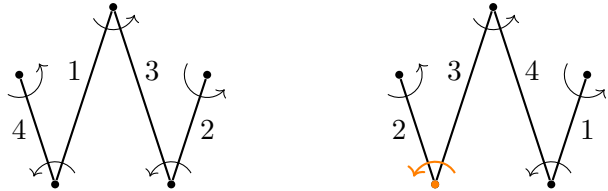

An example of a labeling satisfying the conditions of the lemma is shown on the left-hand side of \cref{fig:good_bad_labelings}, and a labeling that does not on the right-hand side. The apartments of the partition complex have a description in terms of trees as well.

\begin{lem}
	Every apartment of $\op$ is described by a spanning tree and vice versa.
\end{lem}

\begin{proof}
	Let $A$ be an apartment of $\op$ corresponding to the frame $\Set{L_1, \ldots, L_{n-1}}$ consisting of partition subspaces. Then the set of non-zero vectors $\Set{v_1, \ldots, v_{n-1}}$ corresponding to the frame is a basis of $\vs$, which consists of partition vectors. By \cref{lem:cycles_dependent} the set of edges corresponding to the basis under $\emb$ does not contain any cycle. Since it consists of $n-1$ edges, it is a spanning tree on $n$ vertices.
	
	For the other direction, let $A\subseteq \E$ be the edge set of a tree on $n$ vertices. Again by \cref{lem:cycles_dependent}, the $n-1$ non-zero vectors of  $\emb(A)$ are linearly independent and hence a basis.
\end{proof}

\begin{rem}
	The bijection of spanning trees and apartments of $\op$ is induced by the embedding $\emb \colon \op \to \si$. 
\end{rem}

\begin{cor}
	The number of apartments in $\op$ equals $n^{n-2}$.
\end{cor}

\begin{proof}
	The number of apartments on $\op$ equals the number of spanning trees on $n$ vertices, which is well-known to be $n^{n-2}$ by Cayley's formula \cite[Prop. 5.3.2]{staecz}.
\end{proof}

\cref{tab:numbers_aptm} shows the numbers of apartments of $\ncpn$, $\on$ and $\lam(\vs)$ for small $n$. In comparison to the number of vertices, the number of apartments of $\on$ and $\op$ inside $\si$ falls even faster.

\begin{table}
	\begin{center}
		\begin{tabular}{|r||r|r|r|}
			\hhline{|-||-|-|-|}
			n & $\#\An$ & $\#\A(\pn)$ &$\#\A(\si)$\\
			\hhline{:=::=:=:=:}
			1	&	1	&	1	&	1	\\
			2	&	1	&	1	&	1	\\
			3	&	3	&	3	&	3	\\
			4	&	12	&	16	&	28	\\
			5	&	55	&	125	&	840	\\
			6	&	273	&	1\,296	&	83\,328	\\
			7	&	1\,428	&	16\,807	&	27\,998\,208	\\
			8	&	7\,752	&	262\,144	&	32\,509\,919\,232	\\
			
			\hhline{|-||-|-|-|}
		\end{tabular}
	\end{center}
	\caption{The number of apartments of $\on$, $\op$ and $\lam(\vs)$ for small $n$ are shown here.}
	\label{tab:numbers_aptm}
\end{table}

\subsubsection{Chambers}

Recall from \cref{sec:sph_build} that the apartments of $\si$ are isomorphic to the barycentric subdivision of the boundary of the $(n-2)$-simplex, which in turn is isomorphic to the Coxeter complex of type $A_{n-2}$.
We denote this complex by $\Sigma$ in what follows. The chambers of $\Sigma$ are given by total orders of the rank $1$ vertices.

\begin{lem}\label{lem:labeled_trees_chambers}
	Let $T$ be a spanning tree corresponding to an apartment $A \in \A(\pn)$. Then the chambers of $A$ are given by labelings of $T$.
\end{lem}

\begin{proof}
	Every labeling of a spanning tree gives a total order of its edges. Because the edges of a tree are the rank $1$ vertices of the corresponding apartment, the labeling gives a total order of the rank $1$ vertices and hence a chamber. Since the chambers are uniquely characterized by the total orders of rank $1$ elements, every chamber of the apartment gives rise to a labeling of the associated tree. 
\end{proof}

\begin{cor}
	Let $C \in \C(\pn)$ be a chamber and $A\in\A(\pn)$ an apartment. Then $C$ is contained in $A$ if and only if there exists a labeling of the tree corresponding to $A$ that corresponds to the chamber $C$.
\end{cor}

\begin{figure}
	\begin{center}
		\begin{tikzpicture}
		\begin{scope}[yshift=1.1cm]
		\foreach \w in {1,...,5} 
		\node (p\w) at (-\w * 360/5 +90  : 13mm) [mpunkt] {};
		\draw[thick](p1)to node[midway,  right]{1}(p2)
		(p2)to node[pos=0.65, right]{4}(p5)
		(p5)to node[pos=0.35, left]{3}(p3)
		(p3)to node[midway,  left]{2}(p4);
		\end{scope}
		
		\begin{scope}[xshift = 8cm]
		\coordinate (q1) at (-2.4,0);
		\coordinate (q3) at (2.4,0);
		\coordinate (q2) at (0,2.8);
		\filldraw[fill=Superlightgray, thick] (q1)--(q2)--(q3)--(q1);
		\node[mpunkt] at (q1) {};
		\node[mpunkt] at (q2){};
		\node[mpunkt] at (q3){};
		\node[left=2mm] at (q1) {\begin{tikzpicture}\skfnfeck\draw (p1)--(p2);\end{tikzpicture}};
		\node[above=2mm] at (q2) {\begin{tikzpicture}\skfnfeck\draw(p1)--(p2)(p3)--(p4);\end{tikzpicture}};
		\node[right=2mm] at (q3) {\begin{tikzpicture}\skfnfeck\draw (p1)--(p2)(p3)--(p4)--(p5)--(p3);\end{tikzpicture}};
	\end{scope}
	
\end{tikzpicture}
\end{center}
\caption{A labeled tree and the corresponding chamber of $|\ncp_5|$.}
\label{fig:chamber_example}
\end{figure}
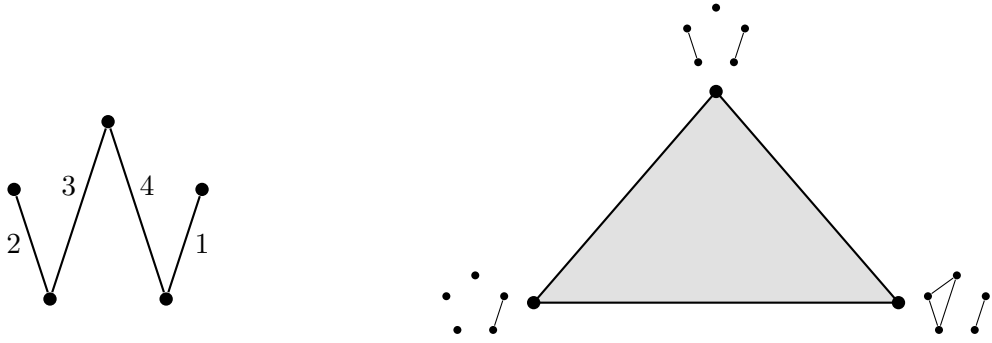

\begin{rem}
We already encountered labeled \emph{non-crossing} trees in \cref{lem:good_labeling}, where the labeling had to satisfy a certain condition in order to give rise to a reduced decomposition of the Coxeter element. The labelings we consider in \cref{lem:labeled_trees_chambers} are different! \emph{Every} labeling gives rise to a particular chamber, which is contained in the apartment corresponding to the tree. Note that if the tree is crossing, this labeling \emph{does not} yield the reduced decomposition that uniquely corresponds to a chamber in general. 

For a non-crossing tree, some of these labelings, at least one, satisfy the condition of \cref{lem:good_labeling}, but \enquote{most} of them do not. If a labeling satisfies the condition, then the reduced decomposition that uniquely corresponds to the described chamber gives rise to the tree we started with. If the labeling violates the condition, then the reduced decomposition associated to that chamber corresponds to another tree.
\end{rem}

\begin{exa}\label{exa:labeled_tree}
Consider the tree $T$ on the right-hand side of \cref{fig:good_bad_labelings} labeled by $\lambda$. Although the labeling does not give rise to a reduced decomposition of the Coxeter element, it corresponds to a chamber of $|\ncp_5|$. The edge set of $T$ is 
\[
A=\Set{a_1=(1,2),\ a_2=(3,4),\ a_3=(3,5),\ a_4=(2,5)},
\]
where we chose the names of the edges accordingly to the labeling of $T$, that is $\lambda(i)=a_i$.
The chamber of $|\ncp_5|$ corresponding to the labeled tree has the vertices $a_1$, $a_1\vee a_2$ and $a_1 \vee a_2 \vee a_3$ and is depicted in \cref{fig:chamber_example}.
\end{exa}

\begin{lem}\label{lem:dist_aptm}
Let $T$ be a tree corresponding to an apartment $A\in \A(\pn)$ and let $\lambda_1, \lambda_2\colon \Set{1, \ldots, n-1}\to T$ 
be two labelings of $T$ corresponding to chambers $C$ and $D$ of $A$, respectively. Consider the simplicial isomorphism $\jmath\colon A \to \Sigma$  that sends $C$ to the fundamental chamber $\Sigma$ of the Coxeter complex of $S_{n-1}$. Then $D$ gets mapped to the chamber with label $\lambda_1\inv\circ\lambda_2$.
\end{lem}

\begin{proof}
Let $T$ be a tree with edge set $\Set{a_1, \ldots, a_{n-1}}$ such that $a_i=\lambda_1(i)$ holds for $1\leq i < n$, that is $\lambda_1\inv(a_i)=i$. This means that on the rank $1$ vertices, the isomorphism  $\jmath\colon A \to \Sigma$ coincides with $\lambda_1\inv$. The chamber corresponding to the labeling $\lambda_2$ has vertices $\lambda_2(1)$, $\lambda_2(1)\vee\lambda_2(2)$, $\ldots$, $\lambda_2(1) \vee \ldots \vee \lambda_2(n-2)$. Hence its image in $\Sigma$ is given by $\jmath(\lambda_2(1))$, $\ldots$, $\jmath(\lambda_2(1)) \vee \ldots \vee \jmath(\lambda_2(n-2))$, which corresponds to the one line notation $\jmath(\lambda_2(1)) \ldots \jmath(\lambda_2(n-1))$ describing a chamber in $\Sigma$. But this is the one line notation of the permutation $\lambda_1\inv\circ\lambda_2$ of $\Set{1, \ldots, n-1}$ as well. This shows that the chamber corresponding to the labeling $\lambda_2$ of $T$ gets mapped to the chamber of $\Sigma$ corresponding to the permutation $\lambda_1\inv\circ\lambda_2$ by $\jmath$.
\end{proof}

\begin{cor}
Let $\lambda_1$ and $\lambda_2$ be two labelings of a spanning tree $T$. Then the corresponding chambers have distance $\ell_S(\lambda_1\inv\circ\lambda_2)$ in the apartment of $\op$ given by $T$. 
\end{cor}

\begin{exa}
Let $T$ be the tree from \cref{exa:labeled_tree} with edge set 
\[
A=\Set{a_1=(1,2),\ a_2=(3,4),\ a_3=(3,5),\ a_4=(2,5)},
\]
and $C$ the chamber given by the labeling $\lambda_1(i)=a_i$ for $1\leq i \leq 4$. Let $D$ be the chamber corresponding to the labeling $\lambda_2$ of $T$ depicted on the left-hand side of \cref{fig:good_bad_labelings}, which is given by
\[
\lambda_2(1)=a_3, \quad \lambda_2(2)=a_1, \quad \lambda_2(3)=a_4, \quad \lambda_2(4)=a_2.
\]
The total order $a_3a_1a_4a_2$ of the rank $1$ vertices of $A$ gives the chamber $D$. Identifying $A$ with the Coxeter complex of $S_{3}$ such that $C$ corresponds to $\id\in S_4$ implies that $D$ corresponds to the permutation with one line notation $3142$.
\end{exa}

\subsection{Universal and base chambers}

This section is devoted to the study of universal chambers in $\on$, which first appeared in \cite[Def. 4.5]{hks}, and their analog in $\op$, the base chambers. The properties of universal and base chambers resemble those of chambers of spherical buildings. Recall that non-trivial blocks of a partition are called base blocks.

\begin{defi}
A partition of $\ncpn$ is called \emph{universal} if it has exactly one base block that consists of circularly consecutive elements. Such a block is called \emph{universal block}. A chamber of $\on$ is a \emph{universal chamber} if and only if its vertices are universal partitions. The faces of a universal chamber are called \emph{universal faces}. In particular, we say that the vertices of a universal chamber are \emph{universal}. 
\end{defi}

\begin{lem}
There are $n\cdot 2^{n-3}$ universal chambers in $\on$.
\end{lem}

\begin{proof}
For every universal partition $\pi\in\ncpn$ there are exactly two universal partitions that cover $\pi$, since the universal block of a cover arises from the universal block of $\pi$ by joining a vertex circularly next to it. 

Now let us construct a universal chamber of $\on$. There are $n$ different universal rank $1$ vertices. Once we have chosen one of them, we have two choices at each step from rank $i$ to rank $i+1$. Since a chamber in $\on$ has $n-2$ vertices, we get a total of $n\cdot 2^{n-3}$ universal chambers.
\end{proof} 

The proof of the above lemma shows that universal chambers with fixed rank $1$ vertex are given by a binary sequence that encodes which vertices are added to obtain the universal block of the next rank. The other way around, each such binary sequence gives, together with an edge, a unique universal chamber. \cref{fig:univ_chamb_ncp5} shows the four universal chambers of $|\ncp_5|$ containing the edge $(2,3)$. If a $0$ in the binary sequence means going counterclockwise and a $1$ means going clockwise, then the binary sequences corresponding to the chambers are $00$, $01$, $11$ and $10$, read from left to right in each row.

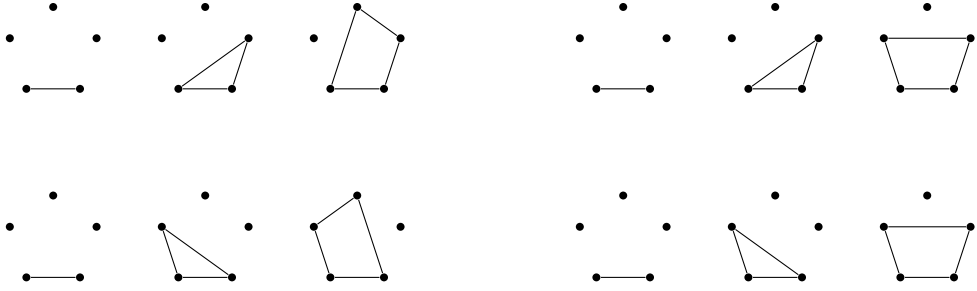
\begin{figure}
\begin{center}
\begin{tikzpicture}
\kfnfeck \draw (p2) -- (p3);
\begin{scope}[xshift = 2cm]
\kfnfeck \draw (p2) -- (p3) -- (p1) -- (p2);
\end{scope}
\begin{scope}[xshift = 4cm]
\kfnfeck \draw (p1) -- (p2) -- (p3) -- (p5) -- (p1);
\end{scope}

\begin{scope}[xshift = 7.5cm]
\kfnfeck \draw (p2) -- (p3);
\begin{scope}[xshift = 2cm]
\kfnfeck \draw (p2) -- (p3) -- (p1) -- (p2);
\end{scope}
\begin{scope}[xshift = 4cm]
\kfnfeck \draw (p1) -- (p2) -- (p3) -- (p4) -- (p1);
\end{scope}
\end{scope}

\begin{scope}[yshift = -2.5cm]
\kfnfeck \draw (p2) -- (p3);
\begin{scope}[xshift = 2cm]
\kfnfeck \draw (p2) -- (p3) -- (p4) -- (p2);
\end{scope}
\begin{scope}[xshift = 4cm]
\kfnfeck \draw (p2) -- (p3) -- (p4) -- (p5) -- (p2);
\end{scope}

\begin{scope}[xshift = 7.5cm]
\kfnfeck \draw (p2) -- (p3);
\begin{scope}[xshift = 2cm]
\kfnfeck \draw (p2) -- (p3) -- (p4) -- (p2);
\end{scope}
\begin{scope}[xshift = 4cm]
\kfnfeck \draw (p1) -- (p2) -- (p3) -- (p4) -- (p1);
\end{scope}
\end{scope}
\end{scope}
\end{tikzpicture}
\end{center}
\caption{The vertices of the four universal chambers on $|\ncp_5|$ containing the edge $(2,3)$ are depicted here.}
\label{fig:univ_chamb_ncp5}
\end{figure}

The next proposition shows that \cref{rem:homotopy_type_building}, which states that a building is the union of all apartments containing a fixed chamber, generalizes to $\on$ for universal chambers.

\begin{prop}\label{prop:univ_chamber_union}
Let $C\in \Cn$ be a chamber. Then the union of all apartments containing $C$ equals the whole complex $\on$ if and only if $C$ is a universal chamber.
\end{prop}

In \cite[Le. 4.7]{hks} it is shown that for a universal chamber and an arbitrary chamber there exists an apartment containing both, which implies one direction of the statement of the proposition. The existence of one such chamber is implied by the supersolvability of the non-crossing partition lattice, which was shown in \cite[Thm. 4.3.2]{her}, as well. The proof we include here uses different methods than \cite{hks}, which fit to the considerations above and allows us to translate the statement into the setting on $\op$ in \cref{prop:base_cham_union}.  \cref{prop:univ_chamber_union} and its proof are published as Proposition 4.10 of \cite{hs}.

\begin{defi}
Let $\pi = \Set{B_1, \ldots, B_k} \in \pn$ be a partition and $M \subseteq \Set{1, \ldots, n}$ be a subset. Then $\pi^M \coloneqq \Set{B_1\cap M, \ldots, B_k\cap M}$ is a partition of $M$ and is called the \emph{partition induced by $M$}. If $C=(C_1, \ldots, C_{n-2})$ is a chamber, then the \emph{chamber induced by $M$} is the chamber $C^M=(C_1^M, \ldots, C_{n-2}^M)$, whose vertices are the partitions induced by $M$.
\end{defi}

\begin{rem}\label{rem:induced_partition}
If $M\subseteq \Set{1,\ldots, n}$ is a subset with $m$ elements, then the natural order on $M$ gives a canonical identification of an induced partition $\pi^M$ for $\pi \in \pn$ with a partition in $\p_m$. Moreover, if $\pi$ is a non-crossing partition, then its induced partitions are non-crossing as well.
\end{rem}

\begin{proof}[Proof of \cref{prop:univ_chamber_union}]
First suppose that $C$ is not a universal chamber. Then there exists a vertex $C_k$ of $C$ that is not a universal partition. 
First suppose that $C_k$ has two non-trivial blocks $A$ and $B$ and let $\pi$ be a partition consisting of the two edges $\Set{a_1,b_1}$ and $\Set{a_2,b_2}$, where $a_1,a_2\in A$ are different elements as well as $b_1,b_2\in B$. Then there is no spanning tree that contains both $C_k$ and $\pi$, because the connection of the two blocks $A$ and $B$ with two edges causes cycles. Hence, there is no apartment containing both $C_k$ and $\pi$ and in particular, the union of all apartments containing $C$ is not equal to $\on$.
Now suppose that $C_k$ has exactly one non-trivial block $B$, which is not a universal block. Then there exists an edge that crosses $B$ and hence there is no apartment containing both $C_k$ and this edge. 

Now let $C$ be a universal chamber and $D$ an arbitrary chamber with rank $1$ vertex $D_1$. We show that there exists a non-crossing tree such that $C$ and $D$ are contained in the apartment corresponding to it.

We use induction on $n$. For $n=3$ the statement is trivial, since $|\ncp_3|$  consists of three vertices, which are the chambers, and any two vertices form an apartment. Hence suppose that $n\geq 4$. Let $a\coloneqq(i,j)$ be the unique edge in $D_1$.

First suppose that $i$ and $j$ are not circularly consecutive, hence the edge $a$ divides the vertices of the $n$-gon into two circularly consecutive parts $M \coloneqq \Set{m \str i\leq m \leq j}$ and $N\coloneqq (\Set{1, \ldots, n}\setminus M) \cup \Set{i,j}$, which both contain the vertices $i$ and $j$. Since $i$ and $j$ are not consecutive, we have that $\#M,\#N < n$. By induction there exists a non-crossing spanning tree $T_M$ on the subset $M$ that corresponds to an apartment containing the induced chambers $C^M$ and $D^M$. Analogously, we find a non-crossing tree $T_N$ for $C^N$ and $D^N$. Note that both $D_1^M$ and $D_1^N$ consist of the edge $a$, hence both $T_M$ and $T_N$ contain $a$. This allows us to merge the trees $T_N$ and $T_M$ along the edge $a$ to obtain a non-crossing spanning tree $T$, whose corresponding apartment contains both $C$ and $D$.

Now suppose that $i$ and $j$ are circularly consecutive. Let $M\coloneqq \Set{1, \ldots, n}\setminus \Set{j}$. By induction, there exists a tree $T_M$ that gives rise to an apartment containing both $C^M$ and $D^M$. Joining the edge $a$ to $T_M$ gives a tree $T$, whose corresponding apartment contains both $C$ and $D$.
\end{proof}

\begin{rem}
The non-crossing partition complexes of type $B_n$ and $D_n$ do not have the property for $n > 3$ that there exists a chamber such that the whole complex can be written as a union of apartments containing this particular chamber. For type $B$ we showed this in joint work with Schwer in \cite[Thm. 4.23]{hs}. The proof easily generalizes to type $D$. In particular, this implies that the lattices $\nc(B_n)$ and $\nc(D_n)$ are not supersolvable for $n>3$.
\end{rem}

The concept of universal chambers translates to the partition complex. 

\begin{defi}
A partition is called \emph{basic} if it has exactly one base block. A chamber of $\op$ is a \emph{base chamber} if and only if its vertices are basic partitions. 
\end{defi}

Note that base chambers are non-crossing and that every universal chamber in $\on$ is a base chamber of $\op$.

\begin{lem}
There are $\frac{1}{2}n!$ base chambers in $\op$.
\end{lem}

\begin{proof}
We count the possibilities to construct the $n-2$ vertices of a base chamber of $\op$. For the rank $1$ vertex, we can choose freely among the $n$ vertices of the $n$-gon to obtain an edge, which are $\binom{n}{2}=\frac{n(n-1)}{2}$. Since a base chamber has exactly one non-trivial block for each vertex, we have $n-i$ available vertices for the rank $i$ partition, which gives $\frac{n(n-1)}{2}\cdot (n-2)!=\frac{1}{2}n!$ possibilities in total.
\end{proof}

The analog of \cref{prop:univ_chamber_union} holds for base chambers in $\op$ as well.

\begin{prop}\label{prop:base_cham_union}
Let $C\in \C(\pn)$ be a chamber. Then the union of all apartments containing $C$ equals the whole complex $\op$ if and only if $C$ is a base chamber.
\end{prop}

\begin{proof}
The proof is the same as for \cref{prop:univ_chamber_union} with all \enquote{non-crossing}s deleted.
\end{proof}

\begin{cor}
The complex $\op$ is a union of apartments.
\end{cor}


Another property of chambers of the building $\si$, which translates to universal and base chambers, respectively, is the following.

\begin{lem}\label{lem:three_chambers_building}
Every codimension $1$ simplex of $\si$ is contained in exactly three chambers.
\end{lem}

\begin{proof}
Let $C$ be a chamber of $\si$ with vertices $C_1, \ldots, C_{n-2}$ such that $\rk(C_i)=i$ and let $F = C\setminus C_k$ be a codimension $1$ face. Set $C_0=\Set{0}$ and $C_{n-1}=\vs$. 
Let $B$ be a basis of $C_{k-1}$ and $a,b\in \vs$ be two vectors such that $B\cup \Set{a,b}$ is an arbitrary basis of $C_{k+1}$. 
Every vector of $C_{k+1}$ that is not contained in $C_{k-1}$ has to be contained in exactly one of the three subspaces $\Braket{B\cup \Set{a}}$, $\Braket{B \cup \Set{b}}$ or $\Braket{B\cup \Set{a+b}}$, respectively.
Hence $F$ is contained in exactly three chambers, namely the ones that have the above subspaces as rank $k$ vertices.
\end{proof}

\begin{lem}\label{lem:char_univ_chambers}
Let $C\in \Cn$ be a chamber. Then $C$ is a universal chamber if and only if all its codimension $1$ faces are contained in exactly three chambers of $\on$.
\end{lem}

\begin{proof}
Let $C$ be a chamber of $\on$ with vertices $C_1, \ldots, C_{n-2}$ such that $\rk(C_i)=i$. Let $F=C\setminus C_k$ be a codimension $1$ face and let $C_0$ be the rank $0$ partition in $\ncpn$ and $C_{n-1}$ the rank $n-1$ partition. 

First suppose that $C$ is a universal chamber.	
If $k=1$, then $C_2$ has three edges, which give rise to three chambers in $\on$ containing $F$. 
If $k>1$, then the universal block  of $C_{k+1}$ arises from the universal block $U$ of $C_{k-1}$ by adding two vertices $a$ and $b$ of the $n$-gon. The two non-crossing partitions with one base block $U\cup \Set{a}$ and $U \cup \Set{b}$ as well as the non-crossing partition with the two base blocks $U$ and $\Set{a,b}$ are all less than $C_{k+1}$ and greater than $C_{k-1}$. Note that $\Set{a,b}$ and $U$ are non-crossing blocks. Either, $a$ and $b$ are consecutive vertices, or $a$ is consecutive to one vertex of $U$ and $b$ to another. Hence $F$ is contained in exactly three chambers.

Now suppose that $C$ is not a universal chamber. We show that there exists a codimension $1$ face that is contained in exactly two chambers of $\on$.

We first consider the case that $C$ has a vertex whose partition has more than one base block. In particular, there has to be a vertex that has exactly two base blocks. Let $k$ be the minimal rank such that $C_{k+1}$ has two base blocks. Hence $C_{k-1}$ and $C_k$ are basic partitions and one of the base blocks of $C_{k+1}$ is a single edge $b$. \cref{fig:lem_univ_1} depicts the situation. Every basis of $C_{k+1}$ that contains a basis $B$ for $C_{k-1}$ has to contain the vector $v_b$ corresponding to the edge $b$ as well as a vector $v_a$ that corresponds to an edge $a$ contained in $C_k$ but not in $C_{k-1}$. Since $a$ and $b$ are in different blocks, their vertices are disjoint. Consequently, the sum $v_a+v_b$ is not a partition vector and $B \cup \Set{v_a+v_b}$ is not a basis for a non-crossing partition subspace. Hence the codimension $1$ face of $C$ that does not contain $C_k$ is contained in exactly two chambers of $\on$, namely the ones with respective rank $k$ vertex $C_k=\Braket{B \cup \Set{v_a}}$ or $\Braket{B \cup \Set{v_b}}$.

We may now assume that $C$ is a base chamber. Further, we suppose that $C$ is not a universal chamber. Let $k>1$ be such that $C_k$ is universal and $C_{k-1}$ is not universal. Such an integer exists, since $C_{n-2}$ is a universal partition. \cref{fig:lem_univ_2} depicts the situation. Let  $B_{k-1}$, $B_k$ and $B_{k+1}$ be the respective base blocks of $C_{k-1}$, $C_k$ and $C_{k+1}$. Suppose that $B_k\setminus B_{k-1}=\Set{i}$ and $B_{k+1}\setminus B_k=\Set{j}$ and without loss of generality that $i \neq 1,n$. 
By \cref{lem:three_chambers_building}, the three chambers of the building that contain $F$ have rank $k$ vertices $\Braket{C_{k-1} \cup \Set{e_i+e_m}}$, $\Braket{C_{k-1} \cup \Set{e_j+e_m}}$ and $\Braket{C_{k-1} \cup \Set{e_i+e_j}}$ for some $m\in B_{k-1}$, respectively.
But $\Braket{C_{k-1} \cup \Set{e_i+e_j}}$ is not a partition subspace, since the edge $\Set{i,j}$, which corresponds to the vector $e_i + e_j$, crosses the edge $\Set{i-1,i+1} \in B_{k-1}$. Hence $F$ is contained in exactly two chambers of $\on$.	 
\end{proof}

\begin{figure}
\begin{center}
\begin{tikzpicture}
\def\r{10mm}
\foreach \w in {1,...,10}
\node (p\w) at (-\w * 36 + 36 : \r) [kpunkt] {};
\draw (p5) -- (p9) -- (p8) -- (p7) -- (p6) -- (p5);
\node[below = 3mm] at ($(p3)!0.5!(p4)$) {$C_{k-1}$};

\begin{scope}[xshift = 4cm]
\foreach \w in {1,...,10}
\node (p\w) at (-\w * 36 + 36 : \r) [kpunkt] {};
\draw[Orange, thick] (p4) -- (p7) node[midway, above right] {$a$};
\draw (p4) -- (p9) -- (p8) -- (p7) -- (p6) -- (p5) -- (p4);
\node[below = 3mm] at ($(p3)!0.5!(p4)$) {$C_k$};
\end{scope}

\begin{scope}[xshift = 8cm]
\foreach \w in {1,...,10}
\node (p\w) at (-\w * 36 + 36 : \r) [kpunkt] {};
\draw[Orange,thick] (p1) -- (p3) node[midway, above left] {$b$};
\draw (p4) -- (p9) -- (p8) -- (p7) -- (p6) -- (p5) -- (p4);
\node[below = 3mm] at ($(p3)!0.5!(p4)$) {$C_{k+1}$};
\end{scope}
\end{tikzpicture}
\end{center}
\caption{The situation in the first case of the second implication of the proof of \cref{lem:char_univ_chambers}.}
\label{fig:lem_univ_1}
\end{figure}
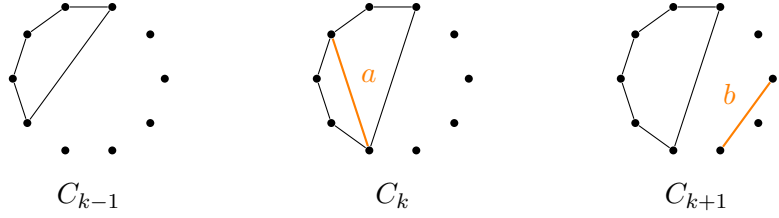

\begin{figure}
\begin{center}
\begin{tikzpicture}
\def\r{10mm}

\foreach \w in {1,...,10}
\node (p\w) at (-\w * 36 + 36 : \r) [kpunkt] {};
\node[above] at (p9) {$i$}; \node[below left] at (p5) {$j$};
\draw (p2) -- (p7) -- (p8) -- (p10) -- (p1) -- (p2);
\draw[Orange, thick] (p9) -- (p5);
\draw[Orange, thick] (p10) -- (p8);
\node[below = 3mm] at ($(p3)!0.5!(p4)$) {$C_{k-1}$};

\begin{scope}[xshift = 4cm]
\foreach \w in {1,...,10}
\node (p\w) at (-\w * 36 + 36 : \r) [kpunkt] {};
\node[above] at (p9) {$i$}; \node[below left] at (p5) {$j$};
\draw (p2) -- (p7) -- (p8) -- (p9)--(p10) -- (p1) -- (p2);
\node[below = 3mm] at ($(p3)!0.5!(p4)$) {$C_k$};
\end{scope}

\begin{scope}[xshift = 8cm]
\foreach \w in {1,...,10}
\node (p\w) at (-\w * 36 + 36 : \r) [kpunkt] {};
\node[above] at (p9) {$i$}; \node[below left] at (p5) {$j$}; \node[right] at (p1) {$m$};
\draw (p1) -- (p2) -- (p5) -- (p7) -- (p8) -- (p9) -- (p10) -- (p1);
\draw[Orange, thick] (p9) -- (p1) -- (p5);
\node[below = 3mm] at ($(p3)!0.5!(p4)$) {$C_{k+1}$};
\end{scope}
\end{tikzpicture}
\end{center}
\caption{The situation in the second case of the second implication of the proof of \cref{lem:char_univ_chambers}.}
\label{fig:lem_univ_2}
\end{figure}
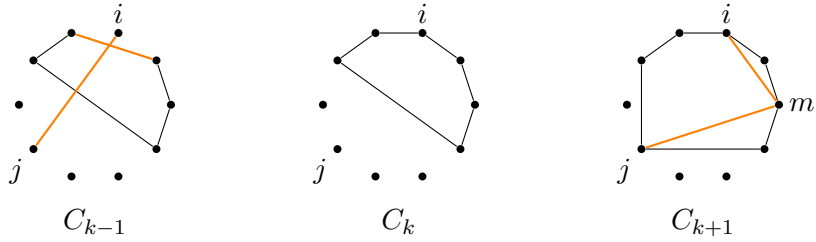

\begin{cor}
Let $C\in \C(\pn)$ be a chamber. Then $C$ is a base chamber if and only if all its codimension $1$ faces are contained in exactly three chambers of $\op$.	
\end{cor}

\begin{proof}
The proof is analogous to the one of \cref{lem:char_univ_chambers} with \enquote{universal chamber} being replaced by \enquote{base chamber} and \enquote{universal block} by \enquote{base block}. The argumentation to the first direction is exactly the same. For the second implication, we only have to consider the first case, where we assume that the chamber has a vertex with more than one non-trivial block.
\end{proof}

Intuitively, the two statements about universal chambers in $\on$ and base chambers in $\op$ fit together. If a chamber is a universal chamber in $\on$ or a base chamber in $\op$, then its neighborhood is not distinguishable from the one of a chamber in the building.

\section{More on the Coxeter complex of type $A$}\label{sec:more_Cox}

This section is devoted to a more detailed study of the apartments of $\si$, $\op$ and $\on$, seen as Coxeter complexes of type $A_{n-2}$ corresponding to the symmetric group $S_{n-1}$. In particular, we are interested in the structure of link complexes, and opposite vertices and chambers. We set $r\coloneqq n-2$.

The structure of the link complex in a Coxeter complex are described in terms of Coxeter complexes \cite[Prop. 3.16]{ab}.

\begin{lem}\label{lem:type_of_link}
Let $\Sigma$ be the Coxeter complex of type $A_r$ and $v\in \Sigma$ a vertex of type $1\leq i \leq r$. Then the link complex $\lk(p,A)$ is isomorphic to the Coxeter complex of type $A_{i-1}\times A_{r-i}$.
\end{lem}

For the study of link complexes, we need to know the structure of reducible Coxeter complexes \cite[Exc. 3.30]{ab}.

\begin{lem}
The Coxeter complex of the Coxeter group of type $A_i \times A_j$ is isomorphic to the simplicial join $\Sigma_i \ast \Sigma_j$, where $\Sigma_i$ is the Coxeter complex of type $A_i$ and $\Sigma_j$ is the Coxeter complex of type $A_j$.
\end{lem}

\begin{figure}
\begin{center}
\begin{tikzpicture}
\begin{scope}
\node[kpunkt](p1) at (0,0){};
\node[kpunkt](p2) at (1,0){};
\node[kpunkt](p3) at (2,0){};
\draw (p1)--(p2);
\draw (p2)--(p3);
\node[align=center] at (1,-1){$\Sigma$\\ of type $A_3$};
\end{scope}

\begin{scope}[xshift=3.5cm]
\node[lkpunkt](p1) at (0,0){};
\node[kpunkt](p2) at (1,0){};
\node[kpunkt](p3) at (2,0){};
\draw[loosely dotted] (p2)--(p1);
\draw (p2)--(p3);
\node[align=center] at (1,-1){$\lk(v_1, \Sigma)$\\ of type $A_2$};
\end{scope}

\begin{scope}[xshift=7cm]
\node[kpunkt](p1) at (0,0){};
\node[lkpunkt](p2) at (1,0){};
\node[kpunkt](p3) at (2,0){};
\draw[loosely dotted] (p1)--(p2);
\draw [loosely dotted] (p3)--(p2);
\node[align=center] at (1,-1) {$\lk(v_2, \Sigma)$\\ of type $A_1\times A_1$};
\end{scope}

\begin{scope}[xshift=10.5cm]
\node[kpunkt](p1) at (0,0){};
\node[kpunkt](p2) at (1,0){};
\node[lkpunkt](p3) at (2,0){};
\draw[loosely dotted] (p2)--(p3);
\draw (p2)--(p1);
\node[align=center] at (1,-1){$\lk(v_3,\Sigma)$\\ of type $A_2$};
\end{scope}

\end{tikzpicture}
\end{center}
\caption{The types of links in the Coxeter complex $\Sigma$ of type $A_3$.}
\label{fig:exa_cox_dia}
\end{figure}
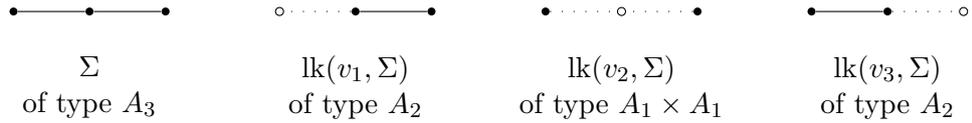

\begin{exa}
Let $\Sigma$ be the Coxeter complex of type $A_3$ and $v\in \Sigma$ be a vertex. If $\rk(v)=1$ or $\rk(v)=3$, then $\lk(v,\Sigma)$ is isomorphic to the Coxeter complex of type $A_2$, hence a hexagon. If $\rk(v)=2$, then $\lk(v,\Sigma)$ is isomorphic to the Coxeter complex of type $A_1 \times A_1$, which is the spherical join of two vertices with two vertices, that is a square. The type of the link complex can be deduced from the Coxeter diagram of $\Sigma$. The type of the link complex of a vertex of type $k$ is given by the Coxeter diagram of $\Sigma$ with the vertex corresponding to type $k$ deleted. The respective Coxeter diagrams for the links in the Coxeter complex of type $A_3$ are shown in \cref{fig:exa_cox_dia}
\end{exa}

\subsection{Opposition}

In this section we are concerned with opposite vertices and chambers in the Coxeter complex of type $A$.
Recall from \cref{sec:sph_build} that the Coxeter complex of type $A_r$ is isomorphic to the barycentric subdivision of the boundary of the $r$-simplex, whose vertices correspond to the proper subsets of $\Set{1, \ldots, r+1}$.

\subsubsection{Vertices}

\begin{defi}
The \emph{distance} of two vertices in a simplicial complex $\si$ is defined as their combinatorial distance in the $1$-skeleton $\si^{(1)}$. Two vertices are called \emph{opposite} if they have maximal distance.
\end{defi}

\begin{lem}
The maximal distance of vertices in a Coxeter complex of type $A_r$ is independent of $r$ and equals three.
\end{lem}

\begin{proof}
We identify the Coxeter complex of type $A_r$ with the barycentric subdivision of the boundary of the $r$-simplex and denote it by $\Sigma$. Let $A,B\subseteq \Set{1, \ldots, r+1}$ be two proper subsets that correspond to vertices of $\Sigma$. One of the following four cases holds for the pair $A,B$:
\begin{enumerate}
\item $A \subseteq B$ or $B \subseteq A$,
\item $A \cap B \neq \emptyset$ and $A \cup B \neq A,B$,
\item $A \cap B = \emptyset$ and $A \cup B \neq \Set{1, \ldots, r+1}$, or
\item $A \cap B = \emptyset$ and $A \cup B = \Set{1, \ldots, r+1}$.
\end{enumerate}
Recall that two vertices in $\Sigma$ are connected by an edge if one is a subset of the other. Hence in case a), $A$ and $B$ are connected by an edge and $\di(A,B)=1$. In case b), $A$ and $B$ are not connected by an edge, but both are connected to the vertex $A \cap B$, hence their distance is $\di(A,B)=2$. In the third case c), the two vertices $A$ and $B$ are both connected to the vertex $A \cup B$ and hence their distance is $2$ as well. For case d) note that $A\cup B$ is not a proper subset of $\Set{1, \ldots, r+1}$ and hence not a vertex of $\Sigma$. Consider a partition of, say, $B$ into $B_1 \cup B_2$. Then  
\[
A - A \cup B_1 - B_1 - B
\]
is a sequence of three edges that connects $A$ to $B$ and consequently, $\di(A,B)=3$.
\end{proof}

\begin{cor}\label{cor:opposite_vertices}
Two vertices in a Coxeter complex are opposite if and only if their distance equals three.
\end{cor}

\begin{cor}\label{lem:type_opposite_vertex}
Let $p$ be a vertex in the Coxeter complex of type $A_r$. Then there exists a unique vertex $q$ that is opposite to $p$. If $p$ has type $i$, then the type of $q$ equals $r+1-i$. Moreover, in the corresponding Boolean lattice $\B\subseteq \lam(\vs)$ it holds that 
\[
p \vee q = \ma \quad \text{ and } \quad p \wedge q = \mi.
\]	
\end{cor}

\begin{proof}
Let $\Sigma$ be the Coxeter complex of type $A_r$, interpreted as barycentric subdivision as above, and let $A\in \Sigma$ be a vertex of rank $i$. The unique vertex $B\in \Sigma$ that fulfills $A\cap B = \emptyset$ and $A \cup B = \Set{1, \ldots, r+1}$ is $B = \Set{1, \ldots, r+1}\setminus A$, hence $B$ is the unique vertex opposite to $A$. Since the rank in the Coxeter complex equals the cardinality of the corresponding subsets of $\Set{1, \ldots, r+1}$, $A$ has rank $i$ and $B$ has rank $r+1-i$.
The join in the Boolean lattice $\B \subseteq \lam(\vs)$ corresponding to $\Sigma$ is given by the union and the meet by the intersection of subsets. For vertices $A$ and $B$ of $\Sigma$ as above, this means that $A \vee B = \Set{1,\ldots,r,r+1}=\ma$ and $A\wedge B = \emptyset = \mi$.
\end{proof}

\begin{rem}
If we endow the Coxeter complex $\Sigma$ with the spherical metric, then two vertices of $\Sigma$ are opposite if and only if their distance equals $\pi$. 
\end{rem}

The following characterizes opposite vertices in link complexes \cite[Exc. 3.156]{ab}.

\begin{lem}\label{lem:opposition_join}
Let $\Sigma = \Sigma_1 \ast \ldots \ast \Sigma_k$ be a finite Coxeter complex and let $p \in \Sigma_i$ be a vertex for some $1\leq i \leq k$. Then a vertex $q \in \Sigma$ is opposite to $p$ in $\Sigma$ if and only if it is opposite to $p$ in $\Sigma_i$.
\end{lem}

\subsubsection{Chambers}
The notion of opposition exists for chambers as well \cite[Def. 4.68]{ab}.

\begin{defi}
Two chambers $C,D \in \Sigma$ in the Coxeter complex of type $A_r$ are called \emph{opposite} if and only if their distance equals $\di(C,D)=\binom{r+1}{2}$. Two chambers of the spherical building $\si$ are said to be \emph{opposite} if and only if they are opposite in an apartment.
\end{defi}

The concepts of opposite vertices and opposite chambers in the spherical building are connected \cite[Exc. 4.78]{ab}.

\begin{lem}
Let $C,D\in \si$ be two chambers with respective vertices $C_1, \ldots, C_r$ and $D_1, \ldots, D_r$ such that $\rk(C_i)=\rk(D_i)=i$. Then $C$ and $D$ are opposite if and only if $C_i$ and $D_{r+1-i}$ are opposite for all $1\leq i \leq r$.
\end{lem}

\subsubsection{Non-crossing spanning trees}

Let us translate the concept of opposite vertices and chambers in the context of the non-crossing partition complex $\on$. Recall that the apartments of $\on$ are in bijection with non-crossing spanning trees. 

In \cref{lem:type_opposite_vertex} we saw that two vertices in the Coxeter complex of type $A_r$, seen as barycentric subdivision of the boundary of the $r$-simplex, are opposite if and only if they are a partition of $\Set{1, \ldots, r+1}$. The vertices of an apartment $A\in \An$, given by a non-crossing spanning tree $T$, are in bijection with the proper subsets of the edge set of $T$. Hence two vertices in $A$ are opposite if and only if they correspond to a partition of edges of $T$. \cref{fig:exa_opposite_vertices} shows an example of a edge partition and resulting opposite vertices.

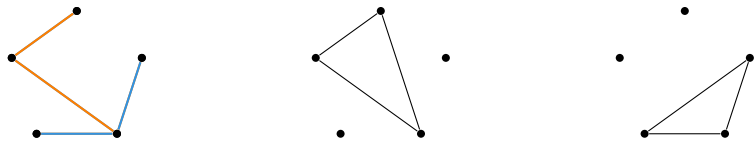
\begin{figure}
\begin{center}
\begin{tikzpicture}

\begin{scope}
\fnfeck
\draw[Orange, thick] (p1) to (p2)(p3) to (p2);
\draw[Blue, thick](p4) to (p2)(p4) to (p5);
\end{scope}

\begin{scope}
\fnfeck
\draw[Blue, thick] (p1) to (p2)(p3) to (p2);
\draw[Orange, thick](p4) to (p2)(p4) to (p5);
\end{scope}

\begin{scope}[xshift = 4cm]
\fnfeck
\draw (p4) to (p2)(p4) to (p5) (p5)to(p2);
\end{scope}

\begin{scope}[xshift = 8cm]
\fnfeck
\draw (p1) to (p2)(p3) to (p2) (p1)to(p3);
\end{scope}

\end{tikzpicture}
\end{center}
\caption{The partition of edges of the non-crossing tree on the left-hand side gives rise to the two opposite vertices in the corresponding apartment.}
\label{fig:exa_opposite_vertices}
\end{figure}

If $\lambda\colon \Set{1, \ldots, r+1} \to T$ is a labeling giving rise to a chamber $C$, then the labeling $\lambda'\colon \Set{1, \ldots, r+1} \to T$ that is defined by $\lambda'(i)\coloneqq \lambda(r+1-i)$ gives rise to the unique chamber opposite to $C$ in the apartment corresponding to $T$, as $\ell_S(\lambda\inv\circ\lambda')=\binom{r+1}{2}$, which is the maximal distance in the Coxeter complex of type $A_r$.
\subsection{Lengths of galleries}

Recall from \cref{sec:cox_cplx} that minimal galleries in the Coxeter complex are given by reduced words in the Coxeter generators $S$. Moreover, the length of a gallery equals the length of the corresponding element in the symmetric group with respect to the Coxeter generating set.

For the proofs of the Theorems \ref{thm:dist_pn_si} and \ref{thm:dist_pn_ncpn} we need a characterization of the length with respect to $S$ via inversions. If $1\leq i<j\leq n$, then the pair $(i,j)$ is an \emph{inversion} of $w\in S_n$ if $w(j)<w(i)$. The \emph{inversion number} of $w$ is the cardinality of its \emph{inversion set}, that is
it is defined as
\[
\invn(w)\coloneqq \#\Set{(i,j)\str 1\leq i<j\leq n, w(i)>w(j)}.
\]

\begin{rem}\label{rem:inv_number}
The inversion number $\invn(w)$ of an element $w\in S_n$ equals its length $\ell_S(w)$ with respect to the Coxeter generators $S$ \cite[Prop. 1.5.2]{bb}. Hence, the inversion number gives the length of a minimal gallery connecting $\id$ to $w$ in the Coxeter complex of $S_n$.
\end{rem}

\begin{exa}\label{exa:inv_number_gallery}
Let us illustrate the computation of the inversion number and a construction of a corresponding minimal gallery in $S_4$. We consider the one line notation $4132$ of $w=(1\,\;4\,\;2)$. To obtain the inversion number and hence the length of $w$, we have to count how many pairs $(i,j)$ are \enquote{out of order}. Starting with $4$ on position one, all numbers $1,2,3$ are less than $4$ and right of it, hence $(1,2)$, $(1,3)$ and $(1,4)$ are inversions. On position two is $1$, which does not contribute to the inversion set as \enquote{$j$}. On position $3$ we find the inversion $(3,4)$. Hence $\invn(w)=4=\ell_S(w)$.
Now we construct a gallery from the identity to $w$ in the Coxeter complex for $S_4$. For this it might me helpful to consider \cref{fig:bary_sd_3simpl}, which shows the Coxeter complex with chambers labeled with one line notations. First, we re-write the one line notation for $\id$ in terms of $w$, that is we have $w(2)w(4)w(3)w(1)$. The idea is to establish the order $w(1)w(2)w(3)w(4)$ by first moving $w(1)$ to position $1$, then $w(2)$ to position $2$, and so on. Note that the number of shifts needed to bring $w(i)$ to position $i$ in the $\ith$ step equals the number of inversions in which $i$ is the second entry. Hence the resulting gallery is a minimal gallery from $\id$ to $w$.
In our concrete example of $w$ being represented by $4132$ we get the following.
Moving $w(1)=4$ to the left in a sequence of one line notations yields
\begin{align*}
1\,\;2\,\;3\,\;\textbf{4} &= w(2)\,\;w(4)\,\;w(3)\,\;\textbf{w(1)},\\
1\,\;2\,\;\textbf{4}\,\;3 &= w(2)\,\;w(4)\,\;\textbf{w(1)}\,\;w(3),\\
1\,\;\textbf{4}\,\;2\,\;3 &= w(2)\,\;\textbf{w(1)}\,\;w(4)\,\;w(3),\\
\textbf{4}\,\;1\,\;2\,\;3 &= \textbf{w(1)}\,\;w(2)\,\;w(4)\,\;w(3),
\end{align*}
where we see that $w(2)$ is already on position $2$. Hence we have to shift $w(3)$ one position to the left and get
\begin{align*}
4\,\;1\,\;2\,\;\textbf{3} &= w(1)\,\;w(2)\,\;w(4)\,\;\textbf{w(3)},\\
4\,\;1\,\;\textbf{3}\,\;2 &= w(1)\,\;w(2)\,\;\textbf{w(3)}\,\;w(4),
\end{align*}
which all in all yields a gallery of type $(s_3,s_2,s_1,s_3)$.
\end{exa}

There is another canonical way to count inversions and construct a minimal gallery, which is demonstrated in the following example. 

\begin{exa}\label{exa:minimal_gallery2}
As in the previous example, we count the inversions of $w=(1\,\;4\,\;2)\in S_4$, but we use a different method here. For this we use the one line notation $4132$ of $w$. This time, we start from the last position and count how many entries on the left-hand side are greater. In our concrete case, there are two entries, namely $3$ and $4$, that are left of position four and greater than two. Left of position three and greater than $3$ is the entry $4$ on position one and finally, left of position two is $4$, which is greater than $1$. Hence, the number of inversions equals $4$. 
To construct a gallery from the identity to $w$ in the corresponding Coxeter complex, we re-write the one line notation of the identity in term of $w$ as before as $w(2)w(4)w(3)w(1)$. Dually to the procedure in \cref{exa:inv_number_gallery}, we establish the order of the one line notation of $w$ by moving the entries to the right, starting from $w(4)$ and proceeding in descended order. We start with moving $w(4)$ on position two to the right, that is
\begin{align*}
1\,\;\textbf{2}\,\;3\,\;4 &= w(2)\,\;\textbf{w(4)}\,\;w(3)\,\;w(1),\\
1\,\;3\,\;\textbf{2}\,\;4 &= w(2)\,\;w(3)\,\;\textbf{w(4)}\,\;w(1),\\
1\,\;3\,\;4\,\;\textbf{2} &= w(2)\,\;w(3)\,\;w(1)\,\;\textbf{w(4)}. 
\end{align*}
Note that we moved $2=w(4)$ two positions to the right, that is we passed the two greater numbers $3$ and $4$. These are exactly the pairs that cause the inversions we encountered above. Next, we move $w(3)$ to the right and finally $w(2)$ as well, which is 
\begin{align*}
1\,\;\textbf{3}\,\;4\,\;2 &= w(2)\,\;\textbf{w(3)}\,\;w(1)\,\;w(4),\\
\textbf{1}\,\;4\,\;\textbf{3}\,\;2 &= \textbf{w(2)}\,\;w(1)\,\;\textbf{w(3)}\,\;w(4),\\
4\,\;\textbf{1}\,\;3\,\;2 &= w(1)\,\;\textbf{w(2)}\,\;w(3)\,\;w(4).
\end{align*}
Note that $3=w(3)$ passed by $4$, which corresponds to the inversion we counted above, and the same is true for $1=w(2)$. Hence, the number of inversions coincides with the length of the gallery, which consequently is minimal. 
\end{exa}

\cleardoublepage

\chapter{Chamber distances and convex hulls}\label{chap:dist}

In this chapter we are filling in the details of the strategy that might prove the \cref{conj:ncpn_cat1}, that is that the non-crossing partition complex is $\cato$. The structure of this chapter follows the one of the outline in \cref{sec:strategy}.

As before, let $n$ be an arbitrary positive integer and $\si \coloneqq \lam(\vs)$ the spherical building associated to $\vs$. The Conventions \ref{conv:edges}, \ref{con:subcompl} and \ref{con:general} are still effective. Moreover, we endow the simplicial complexes $\on$, $\op$ and $\si$ with the spherical metrics induced from the natural round metric on the apartments as described in \cref{sec:metric}. Since we have $\on \subseteq \op \subseteq \si$ and the metrics on these complexes are length metrics, it holds that $\din(C,D) \geq \dip(C,D) \geq \dig(C,D)$ for all chambers $C,D \in \Cn$.

The following shows that we can make the assumptions required in \cref{bow}. We are not aware of any reference in the literature for it.

\begin{prop}\label{prop:locally_cato}
	If $|\ncp_k|$ is $\cato$ for all $k \leq n$, then $\on$ is locally $\cato$.
\end{prop}

\begin{proof}
	By \cref{gromov}, $\on$ is locally $\cato$ if all vertex links in $\on$ are $\cato$. 
	Let $v\in \on$ be a vertex. Then the link complex contains all elements of $\ncpn$ that are less than $v$ as well as all elements that are greater than $v$ in the absolute order, that is  $\lk(v,\on)= \left|[\mi,v]\right| \ast \left|[v,\ma]\right|$. By \cite[Prop. 2.6.11]{armstr}, every proper interval in $\nc(S_n)$ is isomorphic to $\nc(S_{n_1} \times \ldots \times S_{n_k})$ for some $k$ with $n_i<n$ for all $1\leq i \leq k$. Since $\nc(S_i \times S_j) \cong \nc(S_i) \times \nc(S_j)$, we get that
	\[
	\lk(v,\on) \cong \left|\ncp_{n_1} \times \ldots \times \ncp_{n_k}\right| \ast \left|\ncp_{n_{k+1}} \times \ldots \times \ncp_{n_m}\right|
	\]
	for some $m$ and $n_i < n$ for all $1\leq i \leq m$. The order complex of the direct product of two lattices can be written as spherical join. More precisely, if $K$ and $L$ are two finite lattices, then the order complex $|K \times L|$ is homeomorphic to $|K|\ast |L| \ast |\B_2|$, where $\B_2$ is the Boolean lattice of rank $2$ \cite[Eq. (7) of Chap. 7]{walker}. Note that $|\B_2|$ is a union of two vertices and as such $\cato$. Hence $\lk(v,\on)$ is the spherical join of spaces that can be equipped with a $\cato$ metric by assumption. Since the spherical join of $\cato$ spaces is again $\cato$ \cite[Cor. II.3.15]{bh}, the link complex $\lk(v,\on)$ is $\cato$ and consequently, the complex $\on$ is locally $\cato$.  
\end{proof}

\begin{rem} 
	Since $|\ncp_k|$ is $\cato$ for all $k \leq 6$, we know by the above proposition that $|\ncp_7|$ is locally $\cato$. In a uniform proof of the \cref{conj:ncpn_cat1} for all $n$ we hence may assume that $|\ncp_n|$ is locally $\cato$ for all $n$ by induction. In particular, stars are locally $\cato$ subcomplexes of $\on$ if we assume $\on$ to be locally $\cato$. 
\end{rem}
\section{The first cases}

In this section we consider convex hulls from chambers that are either contained in a common apartment of $\on$ or from chambers that have non-empty intersection.

\begin{lem}
	Let $C,D\in\Cn$ be two chambers such that there exists an apartment $A\in\An$ containing both chambers. Then $\convn(C,D)$ is $\pi$-uniquely geodesic.
\end{lem}

\begin{proof}
	The apartments of $\on$ are as Coxeter complexes of type $A_{n-2}$ isometric to the $(n-3)$-dimensional unit sphere. As such they are $\cato$ and hence $\pi$-uniquely geodesic by the \cref{unique}. Since all geodesics connecting two points in $A$ are contained in the respective convex hull of the chambers containing them, every convex hull is $\pi$-uniquely geodesic.
\end{proof}

Since every apartment of $\on$ is an apartment of $\si$ as well, the convex hulls in both complexes coincide, that is $\convn(C,D)=\convg(C,D)$. Convex hulls in the spherical building $\si$ have an explicit description \cite[Le. 2.21]{hks}.

\begin{lem}\label{lem:conv_hull_building}
	Let $C,D\in \si$ be two chambers with respective vertex sets $V_C$ and $V_D$. Then the convex hull of $C$ and $D$ is
	\[
	\convg(C,D)=|\Braket{V_C \cup V_D}|,
	\]
	where $\Braket{V_C \cup V_D}$ is the sublattice of $\lam(\vs)$ spanned by $V_C \cup V_D$.
\end{lem}

In other words, the convex hull of two chambers $C$ and $D$  in the building $\si$ is the smallest subcomplex of $\si$ containing $C$ and $D$ that is stable under taking meets and joins of all of its vertices. 

\begin{oq}
	Is there an explicit description of convex hulls in $\on$? For instance, if $C,D\in \Cn$ are chambers such that $\din(C,D)=\dig(C,D)$, does then $\convn(C,D) = \convg(C,D) \cap \on$ hold?
\end{oq}

The next statement in the setting of buildings follows from \cite[Le. 5.16, Def. 5.26]{ab}.

\begin{lem}\label{lem:star_building}
	Let $C,D\in \si$ be two chambers with $C\cap D \neq \emptyset$. Then their convex hull is contained in the star of their intersection, that is 
	\[
	\convg(C,D) \subseteq \sta(C\cap D, \si).
	\]
\end{lem}

The analogous statement holds for the non-crossing partition complex as well.

\begin{lem}\label{lem:star}
	Let $C,D\in\on$ be two chambers with $C\cap D \neq \emptyset$. Then their convex hull is contained in the star of their intersection, that is 
	\[
	\convn(C,D) \subseteq \sta(C\cap D, \on).
	\]
\end{lem}

\begin{proof}
	Let $C,D\in \Cn$ be two chambers such that $C\cap D \neq \emptyset$ and let $v\in C\cap D$ be a vertex of rank $k$. We show that every minimal gallery from $C$ to $D$ contains $v$. Since this holds for all vertices of the intersection, every chamber of each gallery connecting $C$ and $D$ contains the intersection $C\cap D$. Consequently, $\convn(C,D) \subseteq \sta(C\cap D, \on)$.
	
	Recall that every minimal gallery in $\on$ is described by a word in the simple generators $S$ of $s_{n-1}$. Let $r_1\ldots r_m$ be such a word for a minimal gallery $(C=C_0, \ldots, C_m=D)$. If $r_j=s_i$, then the chambers $C_j$ and $C_{j-1}$ differ by a rank $i$ vertex. We show that $r_j\neq s_k$ for all $1\leq j \leq m$. Suppose for the contrary that $r_j=s_k$ such that $r_i \neq s_k$ for all $i<j$. Then $C_{j}$ does not contain $v$ and by \cref{lem:star_building}, $C_j$ is not contained in a minimal gallery from $C$ to $D$ in the building. 
	This means that 
	$\dig(C_j,D)=\dig(C_{j-1},D)+1$ or $\dig(C_j,D)=\dig(C_{j-1},D)$, which implies that
	$\din(C_j,D)=\din(C_{j-1},D)+1$ or $\din(C_j,D)=\din(C_{j-1},D)$, as distances in $\si$ are less or equal to distances in $\on$. Hence if the distance in the building increases, then the distance in $\on$ has to increase as well. On the other hand, we have $\din(C_j,D)=\din(C_{j-1},D)-1$, since $C_{j-1}$ and $C_j$ are consecutive chambers of a gallery in $\on$ connecting $C$ and $D$. This implies that either $\din(C_{j-1},D)+1 = \din(C_{j-1},D)-1$ or $\din(C_{j-1},D) = \din(C_{j-1},D)-1$, which are both contradictions.
\end{proof}

\begin{cor}
	Let $C,D\in\ncpn$ be two chambers with $C\cap D \neq \emptyset$. Then their convex hull is shrinkable.
\end{cor}

\begin{proof}
	Let $C,D\in\on$ be two chambers with $C\cap D \neq \emptyset$ and $v\in C\cap D$ be a vertex. By the above lemma, every chamber $E\in \convn(C,D)$ contains $v$. Moreover, the maximal distance of vertices in $\convn(C,D)$ is two, as every vertex is connected to $v$. By \cref{cor:opposite_vertices}, $\convn(C,D)$ does not contain any vertices that are opposite in $\on$. Hence $\convn(C,D)$ is contained in a ball $B \subseteq \on$ with diameter less than $\pi$ and center $v$. The geodesic contraction of $B$ simultaneously shrinks every simplex in $\convn(C,D)$ onto $v$ and therefore every rectifiable loop in $\convn(C,D)$ as well. This means that $\convn(C,D)$ is shrinkable. 
\end{proof}

If we assume $\on$ to be locally $\cato$, \cref{bow} and the \cref{unique} imply the following.

\begin{cor}\label{cor:conv_intersec}
	Let $C,D\in\ncpn$ be two chambers with $C\cap D \neq \emptyset$. Then their convex hull is $\pi$-uniquely geodesic.
\end{cor}

\section{The equal-distance case}

In this section we investigate the case that the distance of two chambers $C,D\in\Cn$ equals their distance measured in the building $\si$.

\begin{lem}
	Let $C,D\in \Cn$ be two chambers such that $\din(C,D)=\dig(C,D)$. Then $\convn(C,D)$ is connected and contained in $\convg(C,D)$ and in particular contained in an apartment of the building $\si$.
\end{lem}

\begin{proof}
	Suppose that $C,D\in \Cn$ are such that $\din(C,D)=\dig(C,D)$. This means that minimal galleries connecting $C$ and $D$ in $\on$ have the same lengths as the ones in $\si$. Since $\on \subseteq \si$ is a subcomplex, every minimal gallery in $\on$ is also a minimal gallery in $\si$, hence $\convn(C,D) \subseteq \convg(C,D)$. Since $\convg(C,D)$ is contained in an apartment $A\in \A(\si)$ by \cref{prop:convex_aptms}, it holds that $\convn(C,D) \subseteq A$ as well. Finally, $\convn(C,D)$ is connected, since every chamber in $\convn(C,D)$ can be connected to both $C$ and $D$.
\end{proof}

Hence we have to investigate whether the chamber subcomplex $\convn(C,D)$ of a Coxeter complex of type $A$ is $\pi$-uniquely geodesic. Using \cref{bow}, we aim to show that $\convn(C,D)$ is shrinkable.

\subsection{Indications for shrinkability}

An obstruction for contractibility and hence shrinkability are holes. The following excludes the existence of small holes in $\on$.

\begin{nohole}\label{nohole}
	Suppose that $n \geq 5$. Let $A \in \A(\si)$ be an apartment in the building $\si$ and $p\in A$ a vertex. 
	If the link $\lk(p,A)$ is contained in $\on$, then the vertex  $p$ is contained in $\on$ as well. 
	Moreover, the statement is false for $n=4$.
\end{nohole}

For the proof we need the following lemma.

\begin{lem}\label{lem:geodesics_in_links}
	Let $A\in\A(\si)$ be an apartment and $p\in A$ a vertex. Moreover, let $a,b \in \lk(p,A)$ be such that $a$ is opposite $b$ in $\lk(p,A)$. Then the geodesic in $A$ connecting $a$ and $b$ is the concatenation of the edge $e_{a,p}$ between $a$ and $p$ with the edge $e_{p,b}$ between $p$ and $b$.
\end{lem}

\begin{proof}
	We show that the convex hull of $a$ and $b$ in $A$ is the simplicial complex spanned by $a,b$ and $p$, hence the concatenation of the edges $e_{a,p}$ and $e_{p,b}$. The claim follows since every geodesic is contained in the convex hull by \cref{lem:geod_contained_in_conv}, and the constructed convex hull is a concatenation of two edges.
	
	By \cref{lem:conv_hull_building} the convex hull $\convg(a,b)$ has as vertices the minimal set containing $a$ and $b$ that is stable under taking meets and joins in the corresponding lattice. Note that $a$ and $b$ are opposite in the same join factor $\Sigma$ of $\lk(p,A)$ by \cref{lem:opposition_join}. Hence we have that $a\vee b = \ma$ and $a\wedge b=\mi$ in the Boolean sublattice $\B\subseteq \lam$ corresponding to $\Sigma$ by \cref{lem:type_opposite_vertex}. 
	If $\rk_A(a), \rk_A(b) < \rk_A(p)$, then in the Boolean sublattice $\B_A$ associated to $A$ we have that $a \vee b = p$ and $a \wedge b = \mi$. Hence the set $\Set{a,b,p,\mi}$ is stable under taking meet and joins. 
	If $\rk_A(a), \rk_A(b) > \rk_A(p)$, we get $a\vee b =\ma$ and $a \wedge b = p$ hold in $\B_n$. Consequently, the set $\Set{a,b,p,\ma}$ is stable under taking meets and joins.
	In both cases, the span of the above vertex sets is the concatenation of the edges $e_{a,p}$ and $e_{p,b}$.
\end{proof}

\begin{proof}[Proof of the \cref{nohole}]
	Let $n\geq 5$, let $A\in\A(\si)$ be an apartment and $p\in A$ a vertex. Suppose that $\lk(p,A) \subseteq \on$. We show that $p\in \on$, where we consider the cases that $\rk(p)> 1$ and $\rk(p)=1$. Note that $r \coloneqq n-2 \geq 3.$
	
	First suppose that $\rk(p)=k> 1$. Then the link $\lk(p,A)$ has type $A_{k-1}\times A_{r-k}$ by \cref{lem:type_of_link} and therefore contains $k$ vertices $p_1, \ldots, p_k$ that have rank $1$ in the apartment $A$. Since $k>1$ and $r\geq 3$, there are at least two such vertices. 
	By assumption we have $\lk(p,A)\subseteq \on$, hence these vertices can be represented by edges in the $n$-gon. As rank $1$ vertices of the apartment $A$ they do not form cycles. Moreover, for all $1\leq i,j \leq k$ the join $p_i \vee p_j$ is an element of $\lk(p,A)$ and as such an element of $\on$. This means that the edges corresponding to the rank $1$ vertices $p_1, \ldots, p_k$ are pairwise non-crossing and consequently, it holds that  $p=p_1 \vee \ldots \vee p_k \in \on$.
	
	Now suppose that $\rk(p)=1$. Then $\lk(p,A)$ is of type $A_{r-1}$ and therefore contains $r$ vertices that represent $2$-dimensional subspaces $U_1,\ldots, U_r \subseteq \F_2^{n-1}$ in $\si$.  Suppose that $p$ is not contained in $\on$, which means that it is represented by a vector $v\in \F_2^{n-1}$ that has more than two non-zero entries. On the other hand, $v \in U_i$ for $1\leq i \leq r$, which implies that exactly three or four entries of $v$ are non-zero, since $v$ is the sum of two partition vectors. 
	Suppose that $v=e_a+e_b+e_c$ for $1\leq a < b < c <n$, that is $v$ has exactly three non-zero entries. There are exactly three $2$-dimensional subspaces that contain $v$ and have a basis with at most two non-zero entries. They are $\Braket{e_a+e_b, e_c}$, $\Braket{e_b+e_c,e_a}$ and $\Braket{e_a+e_c,e_b}$. This implies that $r = 3$ and $U_1$, $U_2$ and $U_3$ are these three subspaces. 
	The three $3$-dimensional subspaces that correspond to the remaining vertices in $\lk(p,A)$ are given by the pairwise join of the $2$-dimensional subspaces. But
	\[
	\Braket{e_a+e_b, e_c} \vee \Braket{e_b+e_c,e_a} =\Braket{e_a, e_b, e_c}= \Braket{e_b+e_c,e_a} \vee \Braket{e_a+e_c,e_b},
	\]
	which is a contradiction, because the $3$-dimensional subspaces in $\lk(p, A)$ are different. 
	The same argumentation works in the case that $v$ has four non-zero entries, that is $v=e_a+e_b+e_c+e_d$ for $1\leq a < b < c <d<n$. There are again exactly three subspaces containing $v$ with a basis with vectors corresponding to edges. They are given by $\Braket{e_a+e_b, e_c+e_d}$, $\Braket{e_b+e_c,e_a+e_d}$ and $\Braket{e_a+e_c,e_b+e_d}$, and their pairwise sum equals the subspace $\Braket{e_a+e_b,e_b+e_c,e_c+e_d}$, which is a contradiction.
	
	The statement is not true for $n=4$, since the link of the vertex corresponding to the subspace $\Braket{e_1+e_2+e_3}$ consists of the two non-crossing partitions $\Set{\set{1,2},\set{3,4}}$ and $\Set{\set{1,4},\set{2,3}}$.
\end{proof}

If we replace $\on$ by $\op$ in the above statement, we get a Link Property for the partition complex as well.

\begin{nohole}\label{pn-nohole}
	Suppose that $n \geq 5$. Let $A \in \A(\si)$ be an apartment in the building $\si$ and $p\in A$ a vertex. 
	If the link $\lk(p,A)$ is contained in $|\pn|$, then the vertex  $p$ is contained in $|\pn|$ as well. 
	Moreover, the statement is false for $n=4$.
\end{nohole}

\begin{proof}
	The proof is almost the same as for the \cref{nohole}. With the same assumptions with $\on$ being replaced by $|\pn|$, the proof for $\rk(p)=1$ is the same.
	In the case of $\rk(p)=k>1$, it is enough to know that $p$ is the join of $k$ edges in the $n$-gon that do not form cycles. This is the same argumentation as in the proof of the \cref{nohole}. The counterexample for $n=4$ is the same.
\end{proof}

The \cref{nohole} for $\on$ says that whenever we find a vertex of the building that is not contained in $\on$, then neither is its link complex.
The next theorem shows that in this situation, we even remove a geodesic of $\si$ through this vertex of length $\pi$. For the proof we need the following \cite[Prop. 4.8]{bra-mcc}.

\begin{prop}\label{lem:edge_lengths}
	Let $\Sigma$ be the Coxeter complex of type $A_{r}$ equipped with the standard metric of the $(r-1)$-sphere. The length $l_{ij}$ of an edge connecting a vertex of rank $i$ with a vertex of rank $j$ with $1 \leq i < j \leq r$ is for $s\coloneqq r+1$ given by
	\[
	\cos(l_{ij}) = \sqrt{ \frac{i(s-j)}{j(s-i)}}.
	\]
\end{prop}

\begin{thm}\label{thm:large_holes}
	Suppose that $n \geq 5$. Let $p\in\si \setminus \on$ be a vertex.
	Then there exists a geodesic $\gamma\subseteq \si$ of length $\pi$ with $p \in \gamma$ such that $\gamma \cap \on = \emptyset$.
\end{thm}

\begin{proof}
	Let $p\in \si\setminus \on$ be a vertex and $A\in \A(\si)$ be an apartment containing it. Note that $A$ is a Coxeter complex of type $A_r$ for $r\coloneqq n-2$.
	By the \cref{nohole}, there exists a vertex $q \in \lk(p,A)$ that is not contained in $\on$. Let $\rk(p)=x$ and $\rk(q)=y$, and assume without loss of generality that $x<y$.
	
	First we construct a path passing through $p$ and $q$ by considering opposite vertices in $\lk(p,A)$ and $\lk(q,A)$ by concatenating edges as depicted in \cref{fig:constr_geod}. 
	For two adjacent vertices $a,b$ we denote the edge connecting them by $e_{a,b}$. The concatenation of two edges $e$ and $e'$ is denoted by $e\circ e'$. 
	Then we show that the length of the constructed path indeed equals $\pi$ by using \cref{lem:edge_lengths}. 
	
	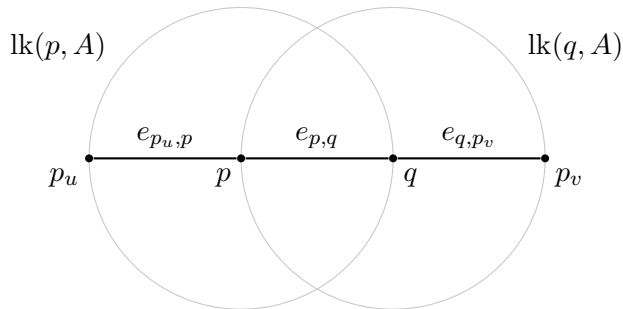
\begin{figure}%
		\begin{center}
			\begin{tikzpicture}
			\def\r{2} 
			
			\begin{scope}
			\draw[Lightgray] (0,0) circle (\r);
			\draw[Lightgray] (\r, 0) circle (\r);
			
			\node(p) at (0,0) [kpunkt] {};
			\node(q) at (\r,0) [kpunkt] {};
			\node(p_u) at (-\r,0) [kpunkt] {};
			\node (p_v) at (2*\r,0) [kpunkt] {};
			\node(e_pup) at (-\r/2,0) {};
			\node(e_pq) at (\r/2,0) {};
			\node(e_qpv) at (3*\r/2,0) {};
			
			
			\node[below left] at (p_u) {$p_u$};
			\node[below right] at (p_v) {$p_v$};
			\node[below left] at (p) {$p$};
			\node[below right] at (q) {$q$};
			
			\node[above] at (e_pup) {$e_{p_u,p}$};
			\node[above] at (e_pq) {$e_{p,q}$};
			\node[above] at (e_qpv) {$e_{q,p_v}$};

			\node[above left] at (-215:\r) {$\lk(p,A)$};
			
			\draw[thick] (p_u) -- (p)-- (q)--(p_v);
			\end{scope}
			
			\begin{scope}[xshift=\r cm]
			\node[above right] at (35:\r) {$\lk(q,A)$};
			\end{scope}
			
			\end{tikzpicture}
			\caption{The construction of a path passing through $p$ and $q$ by a concatenation of edges.}%
			\label{fig:constr_geod}%
		\end{center}
	\end{figure} 
	
	Let us consider the subcomplex $\lk(p,A)=\Sigma_1 \ast \Sigma_2\subseteq A$, which is isomorphic to the Coxeter complex of type $A_{x-1}\times A_{r-x}$. We compute the rank in $A$ of the unique vertex $p_u$ that is opposite to $q$ in $\lk(p,A)$. Since we assumed that $x<y$, the vertex $q$, and hence its opposite, is contained in $\Sigma_2$ by \cref{lem:opposition_join}. The type of $q$ in $\Sigma_2$ is $y-x$. Using \cref{lem:type_opposite_vertex}, we compute the type of the opposite vertex $p_u$ in $\Sigma_2$ as
	\[
	(r-x)+1-(y-x) = r+1-y = s-y
	\] 
	for $s\coloneqq r+1$.
	Consequently, the vertex $p_u$ has type $u \coloneqq s-y+x$ in the apartment $A$. \cref{fig:scheme_types} schematically shows the correspondences of the types of $p$ and $q$, and the corresponding opposite vertices, in the link subcomplexes $\lk(p,A)$ and $\lk(q,A)$. Note that the order of $x,y,u,v$ depends on the relative distance of $x$ and $y$ and may differ from the order depicted in \cref{fig:scheme_types}. 
	
	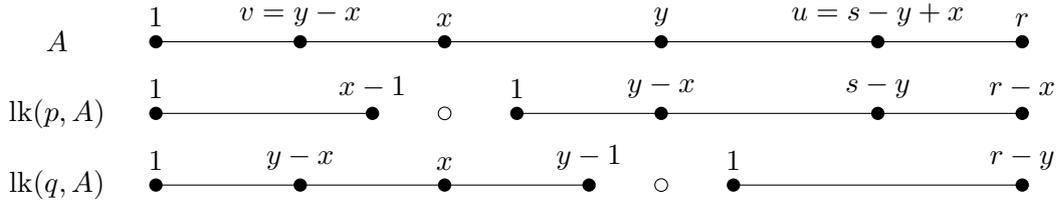
\begin{figure}%
		\begin{center}
			\begin{tikzpicture}[scale=0.95]
			\foreach \x in {1,...,13}
			\foreach \y/\l in {1/A,0/B,-1/C}
			\coordinate (\l\x) at (\x,\y);	
			
			\node[left of= A1, xshift = -0.3cm]{$A$};
			\draw (A1) -- (A13);
			\foreach \x in {1,3,5,8,11,13} \node[mpunkt] at (A\x) {};
			\foreach \x/\l in {1/1,3/v=y-x,5/x,8/y,11/u=s-y+x,13/r}
			\node [above=2pt] at (A\x) {$\l$};
			
			\node[left of= B1, xshift = -0.3cm]{$\lk(p,A)$};
			\draw (B1) -- (B4) (B6) -- (B13);
			\foreach \x in {1,4,6,8,11,13} \node[mpunkt] at (B\x) {};
			\draw (B5) circle (2.5pt);
			\foreach \x/\l in {1/1,4/x-1,6/1,8/y-x,11/s-y,13/r-x}
			\node [above=2pt] at (B\x) {$\l$};
			
			\node[left of= C1, xshift = -0.3cm]{$\lk(q,A)$};
			\draw (C1) -- (C7) (C9) -- (C13);
			\foreach \x in {1,3,5,7,9,13} \node[mpunkt] at (C\x) {};
			\draw (C8) circle (2.5pt);
			\foreach \x/\l in {1/1,3/y-x,5/x,7/y-1,9/1,13/r-y}
			\node [above=2pt] at (C\x) {$\l$};
			
			\end{tikzpicture}
			\caption{The schematic correspondences of vertex types in $A$, $\lk(p,A)$ and $\lk(q,A)$. }%
			\label{fig:scheme_types}%
		\end{center}
	\end{figure} 
	
	Similarly, we compute the type of the vertex $p_v$ that is opposite to $p$ in $\lk(q,A)$. Note that $\lk(q,A)=\Sigma_3\ast\Sigma_4 \subseteq A$ is isomorphic to the Coxeter complex of type $A_{y-1}\times A_{r-y}$. Since $x<y$, the vertex $p$ is contained in $\Sigma_3 \subseteq \lk(q,A)$. The type of $p$ in $\lk(q,A)$ equals its type in $A$, hence, by \cref{lem:type_opposite_vertex}, the type of the opposite vertex $p_v$ is 
	\[
	v \coloneqq (y-1)+1-x = y-x,
	\]
	which is the type of $p_v$ in $A$ as well.
	
	Now we show that  $e_{p_u,p} \circ e_{p,q}\circ e_{q,p_v}$ is a geodesic. By \cref{lem:geodesics_in_links} we know that both $e_{p_u,p} \circ e_{p,q}$ and $e_{p,q}\circ e_{q,p_v}$ are geodesics. Hence $e_{p_u,p} \circ e_{p,q}\circ e_{q,p_v}$ is locally distance minimizing and therefore a geodesic.

	We prove that lengths of the geodesic segments $e_{p_u,p}$, $e_{p,q}$ and $e_{q,p_v}$ add up to $\pi$ using \cref{lem:edge_lengths}. 
	Since $\rk(p)=x<u=s+x-y = \rk(p_u)$, the length of $e_{p_u,p}$ is
	\begin{align*}
	A \coloneqq \arccos\sqrt{a} \coloneqq l_{xu}=
	\arccos \sqrt{\frac{x(y-x)}{(s+x-y)(s-x)}},
	\end{align*}
	the length of $e_{p,q}$ is with $\rk(p)=x<y=\rk(q)$ equal to
	\begin{align*}
	B \coloneqq \arccos\sqrt{b}\coloneqq l_{xy}= 
	\arccos \sqrt{\frac{x(s-y)}{y(s-x)}}
	\end{align*}
	and finally the length of $e_{q,p_v}$ is
	\begin{align*}
	C \coloneqq\arccos \sqrt{c}\coloneqq l_{vy}=
	\arccos \sqrt{\frac{(y-x)(s-y)}{y(s+x-y)}},
	\end{align*}
	since $\rk(p_v)=v=y-x<y=\rk(q)$.
	We show that $\cos(A+B+C)=-1$, which implies the claim. Using the law of sines and cosines gives
	\begin{align*}
	\cos(A+B+C)&= \cos(A)\cos(B)\cos(C)-\sin(A)\sin(B)\cos(C)\\
	&-\sin(A)\cos(B)\sin(C)-\cos(A)\sin(B)\sin(C)\\
	&= \sqrt{abc} - \sqrt{c(1-a)(1-b)} - \sqrt{b(1-a)(1-c)} - \sqrt{a(1-b)(1-c)}.
	\end{align*}
	For the values $a=\frac{x(y-x)}{(s+x-y)(s-x)}$, $b=\frac{x(s-y)}{y(s-x)}$ and $c=\frac{(y-x)(s-y)}{y(s+x-y)}$ we have
	\begin{align*}
	\sqrt{abc} &= \frac{x(s-y)(y-x)}{y(s-x)(s+x-y)},\\[4pt]
	\sqrt{a(1-b)(1-c)} &= \frac{sx(y-x)}{y(s-x)(s+x-y)},\\[4pt]
	\sqrt{b(1-a)(1-c)}&= \frac{sx(s-y)}{y(s-x)(s+x-y)}\text{ and}\\[4pt]
	\sqrt{c(1-a)(1-b)}&= \frac{s(s-y)(y-x)}{y(s-x)(s+x-y)},\\
	\end{align*}
	where we simplified the terms on the left-hand side using the Mathematica-based online tool \cite{matheonline}. Note that the terms above are all positive, because each factor is positive. Comparing the terms on the right-hand side, we see that all fractions have the same denominator 
	\begin{align*}
	d&\coloneqq y(s-x)(s+x-y) = s^2y+sxy-sy^2-sxy-x^2y+xy^2\\
	&=s^2y-sy^2-x^2y+xy^2
	\end{align*}
	Hence we consider the nominators of the fractions of the right-hand sides of the equations above and get
	\begin{align*}
	&d\cdot\left(\sqrt{abc}-\sqrt{a(1-b)(1-c)}-\sqrt{b(1-a)(1-c)}-\sqrt{c(1-a)(1-b)}\right)\\
	&=x(s-y)(y-x)-sx(y-x)-sx(s-y)-s(s-y)(y-x)\\
	&=sxy -sx^2-xy^2+x^2y-sxy+sx^2-s^2x+sxy-sy^2+s^2x+sy^2-sxy\\
	&=x^2y-xy^2-s^2y+sy^2= -d,
	\end{align*}
	which is equivalent to
	\[
	\cos(A+B+C)=\sqrt{abc}-\sqrt{a(1-b)(1-c)}-\sqrt{b(1-a)(1-c)}-\sqrt{c(1-a)(1-b)}=-1.
	\]
\end{proof}

\subsection{Computing distances}

In this section we investigate which properties of pairs of chambers lead to the fact that their distances in $\on$ and $\si$ coincide. For this, the following simplification is helpful.

\begin{thm}\label{thm:dist_pn_si}
	The gallery distances in $\op$ and $\si$ coincide, that is 
	for all chambers $C,D\in \C(P_n)$ it holds that $\dip(C,D)=\dig(C,D)$.
\end{thm}

\begin{proof}
	Let $C,D\in \C(\pn)$ be chambers and set $r\coloneqq n-2$. 
	We construct a minimal gallery in $\si$ connecting $C$ and $D$ and show that each of its chambers is contained in $\op$.
	
	First we choose an apartment $A\in \A(\si)$ with $C,D\in A$ such that if $\Set{v_1, v_2, \ldots, v_{r+1}}$ is the basis of $\F_2^{n-1}$ corresponding to $A$, then the chamber $C$ has vertices
	\[
	C_1=\Braket{v_1},\ C_2=\Braket{v_1,v_2}, \ldots,\ C_r=\Braket{v_1, \ldots, v_{r}}
	\]
	and the chamber $D$ has vertices 
	\[
	D_1=\Braket{v_{w(1)}},\ D_2=\Braket{v_{w(1)},v_{w(2)}}, \ldots, \  D_{r}=\Braket{v_{w(1)}, \ldots, v_{w(r)}}
	\]
	for a $w\in S_{n-1}$. 
	This means that $C$ corresponds to the total ordering $v_1, v_2, \ldots,v_{r}, v_{r+1}$ of basis vectors and $D$ corresponds to $v_{w(1)}, \ldots, v_{w(r)}, v_{w(r+1)}$. The isomorphism from $A$ to the Coxeter complex $\Sigma$ of $S_{n-1}$ maps $C$ to $\id$ and $D$ to $w$. For clarity, we identify $A$ with $\Sigma$ and use the one line notation to represent chambers of $A$. In one line notation, the chamber $C$ is given by $1 2 \ldots r+1$, which is exactly the sequence of indices of the corresponding vectors $v_i$. Similarly, the chamber $D$ corresponds to $w(1)\ldots w(r+1)$. 
	
	To construct the minimal gallery from $C$ to $D$ in $A$, we use the same method as in \cref{exa:inv_number_gallery}. We represent the minimal gallery by a list of one line notations that correspond to adjacent chambers $(C=C^0, C^1, \ldots, C^{\ell_S(w)}=D)$. Recall that two one line notations represent adjacent chambers if and only if they differ by a swap of adjacent entries. The chambers are obtained as follows. First, move $w(1)$ in $12\ldots r+1$ step-by-step to the left until it is on position one. If $w(1)=k$, then $C^{k-1}$ is the first chamber with $w(1)$ at position one. Then, move $w(2)$ in the one line notation of $C^{k-1}$ to the left and proceed until $w(2)$ is on position two. We repeat this procedure until we end up with $w(1)w(2)\ldots w(r+1)$, which represents $D$. The argumentation that this gallery is indeed minimal is the same as in \cref{exa:inv_number_gallery}.
	
	What is left to show is that each chamber $C^i$ is indeed in $\op$ for all $0\leq i \leq \ell_S(w)$. We show that $C^i$ is in $\C(\pn)$ for all $i<w(1)$, which are all chambers in whose one line notation the entry $w(1)$ is moved to the left. In all subsequent one line notations, the first position is occupied with $w(1)$. Recall that $D_1=\Braket{v_{w(1)}}$, which means that $C^i$ and $D$ have the same rank $1$ vertex $D_1$ for all $i \geq w(1)$. Hence, all those chambers are contained in $\sta(D_1,A)$. 
	Since the link of a rank $1$ vertex in $A$ is isomorphic to an apartment of type $A_{r-1}$, it follows by induction that $C^i$ is in $\C(\pn)$ for all $i\leq \ell_S(w)$, provided that $C^i$ is in $\C(\pn)$ for all $i<w(1)$. 
	
	In order to prove that the chambers  $C=C^0, C^1, \ldots, C^{k-1}$ for $k\coloneqq w(1)$
	are in $\op$, we have to show that its vertices are in $\op$.
	We denote the rank $j$ vertex of the chamber $C^i$ by $C_j^i$ for $0\leq i < k$ and $1\leq j \leq r$ and set $C_0^i \coloneqq \Set{0}$. It holds that
	\[
	C_j^1=
	\begin{cases}
	C_j & \text{ if }j<k-1,\\[4pt]
	C_{j-1} \vee D_1& \text{ if }j=k-1,\\[4pt]
	C_j& \text{ if }j>k-1\\
	\end{cases}
	\]
	and the rank $j$ vertex of the chamber $C^i$ is given by
	\[
	C_j^i=
	\begin{cases}
	C_j & \text{ if }j<k-i,\\[4pt]
	C_{j-1} \vee D_1& \text{ if }j=k-i,\\[4pt]
	C_j^{i-1}& \text{ if }j>k-i,\\
	\end{cases}
	\]
	since in step $i$ the only vertex that gets changed, compared to the vertices of the previous chamber, is the one of rank $k-i-1$. The left-shift of $w(1)$ corresponds to taking the join with $D_1$. Considering the sequence in \cref{exa:inv_number_gallery} helps to unravel the indices.
	
	We show that $C_{k-i}^i=C_{j-1}\vee D_1$ is in $\op$ for all $1\leq j =k-i< k$. Since $C_{k-i}^i$ is the unique vertex that is not contained in $C^{i-1}$, this implies that $C^1, \ldots, C^{k-1}\in \op$, as $C\in \op$ by assumption.
	
	For this, let $1 \leq j <k=w(1)$. The vertex $C_{j-1}\vee D_1$ has rank $j$, hence it is described by a vector space of dimension $j$. Since $C\in \op$, the subspace for $C_{j-1}$ has a basis consisting of partition vectors. The union of this basis with $\Set{v_{w(1)}}$, which is a basis of the subspace for $D_1$, is a basis of the subspace for $C_{j-1}\vee D_1$. Since $D_1 \in \op$ by assumption, $v_{w(1)}$ is a partition vector, which means that $C_{j-1} \vee D$ has a basis consisting of partition vectors and hence $C_{j-1}\vee D_1 \in \op$.
\end{proof}

Hence instead of computing the distances of chambers in the building $\si$, it is enough to compute it in the partition complex $\op$. Using this, we get a class of chambers of $\on$ satisfying $\din(C,D)=\dig(C,D)$.

\begin{thm}\label{thm:dist_pn_ncpn}
	Let $C,D\in \Cn$ be chambers that are contained in a common apartment $A\in \A(\pn)$. Then $\din(C,D)=\dip(C,D)$.
\end{thm}

\begin{proof}
	The strategy of the proof is similar to the one of \cref{thm:dist_pn_si}. First we choose a suitable apartment $A$ of $\op$ and then construct a minimal gallery in $A$ that is contained in $\on$.
	
	Let $C,D\in \Cn$ be chambers and set $r\coloneqq n-2$. Let $A\in \A(\pn)$ be an apartment corresponding to the basis $\Set{v_1, \ldots, v_r}$ of $\F_2^{n-1}$ such that the chamber $C$ corresponds to the total ordering $v_1 \ v_2\ldots v_r\ v_{r+1}$ and the chamber $D$ to $v_{w(1)}\ v_{w(2)} \ldots  v_{w(r)}\  v_{w(r+1)}$ for an element $w\in S_{n-1}$. We use the method from \cref{exa:minimal_gallery2} to construct a gallery $(C=C^0, C^1, \ldots, C^{\ell_S(w)}=D)$ from $C$ to $D$, which is minimal by construction. 
	
	First we express the one line notation for $\id$, which represents the chamber $C$, in terms of $w$. Then we move the entry $w(r+1)$ step-by-step to the right until it is on position $r+1$. In the next step, we move $w(r)$ to position $r$ and proceed until we obtained the order $w(1)\ldots w(r+1)$, which represents the chamber $D$.
	Note that the rank $r$ vertex of each chamber $C^i$ for $i \geq r+1-w(r+1)$ equals $D_r=\Braket{v_{w(1)}\ v_{w(2)} \ldots  v_{w(r)}}$, since the last entry in each one line notation is $w(r+1)$. Hence, the chambers $C^i$ are all contained in $\sta(D_r,A)$ for $i \geq r+1-w(r+1)$. So it is enough to show that the chambers $C^i$ have vertices in the non-crossing partitions $\on$. The statement for the remaining chambers follows with induction, since the link of a rank $r$ vertex in $A$ is isomorphic to the Coxeter complex of type $A_{r-1}$. 
	
	We set $k\coloneqq w(r+1)$ and denote the rank $j$ vertex of the chamber $C^i$ by $C_j^i$ for all $1\leq j \leq r$ and $0 \leq i \leq r+1-k$.
	Then we have that the vertices of the chamber adjacent to $C$ in the minimal gallery are given by
	\[
	C^1_j=
	\begin{cases}
	C_j&\text{ if }j>k,\\
	C_{j+1}\wedge D_r &\text{ if }j=k,\\
	C_j &\text{ if }j>k
	\end{cases}
	\]
	and in general, the vertices of $C^i$ are
	\[
	C^i_j=
	\begin{cases}
	C_j^{i-1}&\text{ if }j>k,\\
	C_{j+1}\wedge D_r &\text{ if }j=k,\\
	C_j &\text{ if }j>k.
	\end{cases}
	\]
	
	We prove that $C_{j+1} \wedge D_r$ is in $\on$, which shows that all vertices of $C^1, \ldots, C^{r+1-k}$ are in $\on$.
	For the contrary suppose that the partition corresponding to $C_{j+1}\wedge D_r$ is crossing. Since $C_{j+1} \wedge D_r\in A$, there exist two vectors $u_1, u_2\in \Set{v_1, \ldots, v_{r+1}}$ that correspond to crossing edges in $C_{j+1}\wedge D_r$. Without loss of generality let $u_1=e_a+e_c$ and $u_2=e_b+e_d$ for some $1\leq a<b<c<d\leq n$. Since $u_1$ and $u_2$ are contained in the vector space corresponding to the non-crossing partition $D_r$, the vectors $e_a+e_b$ and $e_c+e_d$ are in $D_r$. The same is true for $C_{j+1}$, hence $e_a+e_b$ and $e_c+e_d$ are contained in the vector space corresponding to $C_{j+1}\wedge D_r$. This contradicts that $C_{j+1}\wedge D_r$ is crossing, since in a basis for the vector space of $C_{j+1}\wedge D_r$, the crossing pair $u_1,u_2$  can be replaced by the non-crossing pair $e_a+e_b,e_c+e_d$. Consequently, we get that $C_{j+1}\wedge D_r\in \on$ for all $1\leq j \leq r$.
\end{proof}

\begin{cor}
	Let $C,D\in \Cn$ be chambers that are contained in a common apartment $A\in \A(\pn)$. Then $\din(C,D)=\dig(C,D)$.
\end{cor}

Recall from \cref{lem:anti-auto} that the map $\ncaa\colon \nc(S_n) \to \nc(S_n)$ given by $\ncaa(w)=w\inv\cox$ is an anti-automorphism of the lattice $\nc(S_n)$. Hence $\ncaa$ induces a simplicial automorphism $\iota \circ \ncaa\circ \iota\inv$, which we denote by $\ncaa\colon \on \to \on$ as well. A detailed description of the pictorial anti-automorphism $\ncaa$ is given in \cite[Chap. 4.3]{armstr}. 

\begin{cor}\label{cor:dist_nc_si}
	Let $C,D\in \Cn$ be chambers such that their images $\ncaa(C)$ and $\ncaa(D)$ are contained in a common apartment $A\in \A(\pn)$. Then $\din(C,D)=\dig(C,D)$. 
\end{cor}

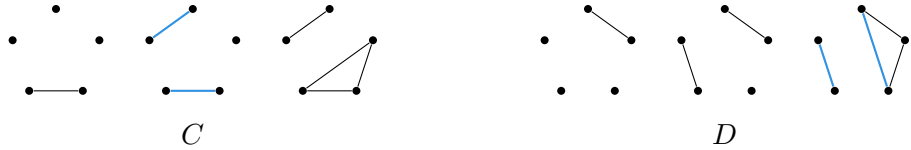
\begin{figure}
	\begin{center}
		\begin{tikzpicture}
		\kfnfeck \draw (p2) -- (p3);
		\begin{scope}[xshift = 1.8cm]
		\kfnfeck \draw[Blue, thick] (p2) -- (p3) (p4) -- (p5);
		\node[below = 3mm] at ($(p2)!0.5!(p3)$) {$C$};
		\end{scope}
		\begin{scope}[xshift = 3.6cm]
		\kfnfeck \draw (p4) -- (p5) (p1) -- (p2) -- (p3) -- (p1);
		\end{scope}
		
		\begin{scope}[xshift = 7cm]
		\kfnfeck \draw (p1) -- (p5);
		\begin{scope}[xshift = 1.8cm]
		\kfnfeck \draw (p1) -- (p5) (p4) -- (p3);
		\node[below = 3mm] at ($(p2)!0.5!(p3)$) {$D$};
		\end{scope}
		\begin{scope}[xshift = 3.6cm]
		\kfnfeck \draw[Blue, thick] (p4) -- (p3) (p2) -- (p5); \draw (p5) -- (p1) -- (p2);
		\end{scope}
		\end{scope}
		\end{tikzpicture}
	\end{center}
	\caption{The vertices of the chamber $C$ from \cref{exa:galleries} are depicted on the left-hand side and those of the chamber $D$ on the right-hand side. The blue edges prevent the chambers from being in a common $\on$ apartment.}
	\label{fig:exa_chambers_1}
\end{figure}

\begin{exa}\label{exa:galleries}
	Let us consider the chambers $C$ and $D$ in $|\ncp_5|$ corresponding to the reduced decompositions $(2\,\;3)(4\,\;5)(1\,\;5)(1\,\;3)$ and $(1\,\;5)(3\,\;4)(1\,\;2)(2\,\;4)$ of the Coxeter element, respectively. Their vertices are depicted in \cref{fig:exa_chambers_1}. Note that they are opposite chambers in the building and hence they are contained in a unique apartment, which is given by the basis
	\[
	\Set{
		\begin{pmatrix}
		0\\1\\1\\0
		\end{pmatrix},
		\begin{pmatrix}
		0\\1\\1\\1
		\end{pmatrix},
		\begin{pmatrix}
		1\\0\\1\\1
		\end{pmatrix},
		\begin{pmatrix}
		1\\0\\0\\0
		\end{pmatrix}
	}
	\]
	of $\F_2^4$. Note that the corresponding apartment is \emph{not} contained in $\op$, as not every element of the basis is a partition vector.
	The images of $C$ and $D$ under $\ncaa$ are opposite in $\si$ as well and given by the reduced decompositions $(3\,\;5)(1\,\;5)(3\,\;4)(1\,\;2)$ and $(2\,\;4)(1\,\;4)(2\,\;3)(4\,\;5)$, respectively. Their vertices are depicted in \cref{fig:exa_chambers_2}. Since $\ncaa(C)$ and $\ncaa(D)$ are both base chambers, they are contained in a common apartment $A$ of $\pn$. It is uniquely determined and shown in \cref{fig:exa_chambers_2}. 
	It holds that $\din(\ncaa(C),\ncaa(D))=\dig(\ncaa(C),\ncaa(D))$ by \cref{cor:dist_nc_si}. Since anti-automorphisms map as simplicial isomorphism galleries to galleries, it follows that $\din(C,D)=\dig(C,D)$. 
\end{exa}

\begin{figure}
	\begin{center}
		\begin{tikzpicture}
		\kfnfeck \draw[Orange, thick] (p5) -- (p3);
		\begin{scope}[xshift =1.8cm]
		\kfnfeck \draw (p1) -- (p3) -- (p5) -- (p1);
		\node[below = 3mm] at ($(p2)!0.5!(p3)$) {$\phi(C)$};
		\end{scope}
		\begin{scope}[xshift = 3.6cm]
		\kfnfeck \draw (p1) -- (p5) -- (p4) -- (p3)-- (p1);
		\end{scope}
		
		\begin{scope}[xshift = 6cm]
		\kfnfeck \draw[Orange, thick] (p2) -- (p4);
		\begin{scope}[xshift = 1.8cm]
		\kfnfeck \draw (p1) -- (p2) -- (p4) -- (p1);
		\node[below = 3mm] at ($(p2)!0.5!(p3)$) {$\phi(D)$};
		\end{scope}
		\begin{scope}[xshift = 3.6cm]
		\kfnfeck \draw (p4) -- (p3) -- (p2) -- (p2) -- (p1) -- (p4);
		\end{scope}
		\end{scope}
		
		\begin{scope}[xshift = 12cm]
		\kfnfeck \draw (p5) -- (p3) -- (p1) -- (p4) -- (p2);
		\node[below = 3mm] at ($(p2)!0.5!(p3)$) {$A$};
		\end{scope}
		\end{tikzpicture}
	\end{center}
	\caption{On the left we see the vertices of the chamber $\ncaa(C)$ from \cref{exa:galleries} and on the right those of the camber $\ncaa(D)$. The orange edges prevent the chambers from being in a common $\on$ apartment.}
	\label{fig:exa_chambers_2}
\end{figure}
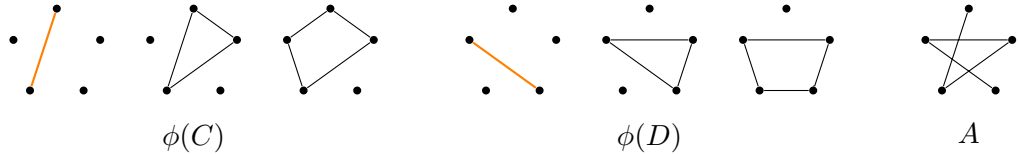

\begin{rem}\label{rem:distances}
	When we look at the chambers $C$ and $D$, and their images $\ncaa(C)$ and $\ncaa(D)$, we see that whenever vertices of $\ncaa(C)$ and $\ncaa(D)$ have crossing blocks, the corresponding vertices of $C$ and $D$ form a \enquote{box}. These are two different kinds of obstacles. Whenever two chambers of $\on$ only have \enquote{cross obstacles}, that is they are contained in a common $\op$ apartment, but not in a $\on$ apartment, then their distance measured in $\on$ coincides with the one measured in  the building $\si$. Dually, if two chambers only have \enquote{box obstacles}, then the two different ways of distance measure agree as well. The other way around, this means that whenever two chambers of $\on$ have a \emph{greater} distance in $\on$ than in $\si$, that they have to have both box and cross obstacles.  
\end{rem}

\begin{oq}
	How can we compute the distance of two chambers in $|\ncpn|$? For instance, does the number of cross and box obstacles give rise to the distance?
\end{oq}

\section{The different-distance case}\label{sec:dif_dis}

This section is devoted to the study of chambers that have greater distance in the non-crossing partition complex than in the building. 
There are two different cases to consider. Either, the chambers are already opposite in the building, or not. We illustrate our procedure, which we believe to work in every dimension, with examples for small $n$. These illustrations are no proofs, but rather ideas and sketches. 

\begin{figure}
	\begin{center}
		\begin{tikzpicture}
		\kfnfeck \draw[Orange] (p1) -- (p3);
		\begin{scope}[xshift = 2cm]
		\kfnfeck \draw[Blue, thick] (p1) -- (p3) (p4) -- (p5);
		\end{scope}
		\begin{scope}[xshift = 4cm]
		\kfnfeck \draw (p1) -- (p2) -- (p3) -- (p1) (p5) -- (p4);
		\end{scope}
		
		\begin{scope}[xshift = 7cm]
		\kfnfeck \draw[Orange, thick] (p2) -- (p4);
		\begin{scope}[xshift = 2cm]
		\kfnfeck \draw (p1) -- (p5) (p2) -- (p4);
		\end{scope}
		\begin{scope}[xshift = 4cm]
		\kfnfeck \draw[Blue, thick] (p1) -- (p5) (p3) -- (p4); \draw (p3) -- (p2) -- (p4);
		\end{scope}
		\end{scope}
		\end{tikzpicture}
	\end{center}
	\caption{The shown chambers are opposite in $\si$ and hence have distance equal to $6$ in $\si$, whereas their distance in $|\ncp_5|$ is $7$. The edges highlighted in orange from a cross obstacle and the blue edges a box obstacle.}
	\label{fig:chambers_far_away}
\end{figure}
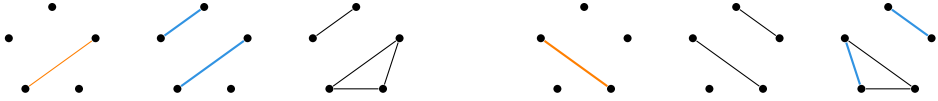

\subsubsection{Opposite in the building}

Let $C$ and $D$ be the chambers of $|\ncp_5|$ that are given by the reduced decompositions $(1\,\;3)(4\,\;5)(1\,\;2)(3\,\;5)$ and $(2\,\;4)(1\,\;5)(2\,\;3)(1\,\;4)$, respectively. Their pictorial representations are shown in \cref{fig:chambers_far_away}.
They are an example for two chambers of $|\ncp_5|$ that are opposite in $\si$, that is their distance equals $6$, but their distance in $|\ncp_5|$ equals $7$. 
The unique apartment $A\in \A(\si)$ containing $C$ and $D$ in $\si$ is given by the basis
\[
\Set{
	\begin{pmatrix}
	0\\1\\0\\1
	\end{pmatrix},
	\begin{pmatrix}
	1\\0\\1\\1
	\end{pmatrix},
	\begin{pmatrix}
	1\\1\\0\\1
	\end{pmatrix},
	\begin{pmatrix}
	0\\1\\0\\1
	\end{pmatrix}
}
\]
of $\F_2^4$. The only two chambers of $|\ncp_5|$ that are contained in $A$ are $C$ and $D$ themselves. This means that every gallery in $|\ncp_5|$ connecting $C$ and $D$ immediately leaves the apartment $A$.

Let us take a closer look at the convex hull $\convn(C,D)$. A schematic picture of it is shown in \cref{fig:conv_hull}. We see that $\convn(C,D)$ is not simply connected. It rather is a union of three strands connecting $C$ and $D$. Note that the three chambers adjacent to $D$ in $\convn(C,D)$, call them $D_1$, $D_2$ and $D_3$, all have distance $6$ from $C$ and they are opposite $C$ in the building. Hence we have

\[
\convn(C,D)=\convn(C,D_1) \cup \convn(C,D_2) \cup \convn(C,D_3) \cup D.
\]

In order to prove that $\convn(C,D)$ is $\pi$-uniquely geodesic, one has to show the following.

\begin{conj}
	If $x\in C$ and $y\in D$ are points and $\gamma_i \subseteq \convn(C,D_i)\cup D$ as well as  $\gamma_j \subseteq \convn(C,D_j)\cup D$ are geodesics joining $x$ and $y$ for $i \neq j$, then $\ell(\gamma_1)+\ell(\gamma_2)\geq \pi$.
\end{conj}

For instance, let $x\in C$ and $y\in D$ be the respective rank $1$ vertices. Let $D_i$ be the chamber whose common panel with $D$ has color $s_i$. Consider $\gamma_2$ and $\gamma_3$ connecting $x$ and $y$ as in the above conjecture. Then $\gamma_2$ is a concatenation of two edges that both connect a rank $1$ vertex with a rank $3$ vertex. This can be read off \cref{fig:conv_hull}. Similarly, $\gamma_3$ is a concatenation of four edges that all connect a rank $1$ and a rank $2$ vertex. Using \cref{lem:edge_lengths}, we compute that $\ell(\gamma_2)+\ell(\gamma_3)=2\pi$.  

Note that although $x$ and $y$ are contained in disjoint chambers that have distance $7$ in $|\ncp_5|$, their distance in $|\ncp_5|$ equals $\ell(\gamma_2) \approx 0.8 \pi$. Consequently, there exist points   $x'\in C$ and $y'\in D$ in the respective \emph{interiors} such that $\din(x',y')<\pi$.

\begin{figure}
	\begin{center}
		\begin{tikzpicture}
		\def\cola{Orange}
		\def\colb{Green}
		\def\colc{Blue}
		\def\k{1.5}
		\def\h{1}
		\foreach \x/\y in {2/1,8/1, 5/3, 4/5, 5/7} 
		\node(p\x\y) [mpunkt,\cola] at (\x*\k,\y*\h)  {};
		\foreach \x/\y in {3/2,6/2, 4/4} 
		\node(p\x\y) [mpunkt,\cola] at ($( \x *\k + \k/2,\y*\h)$)  {};
		\foreach \x/\y in {3/1,6/1, 3/3,6/3} 
		\node(p\x\y)[mpunkt,\colb] at (\x*\k,\y*\h)  {};
		\foreach \x/\y in {1/2,4/2,7/2, 4/4, 4/6, 4/8} 
		\node(p\x\y)[mpunkt,\colb] at ($( \x *\k + \k/2,\y*\h)$)  {};
		\foreach \x/\y in {1/1,7/1, 4/3, 5/5, 4/7} 
		\node(p\x\y)[mpunkt,\colc] at (\x*\k,\y*\h)  {};
		\foreach \x/\y in {2/2,5/2} 
		\node(p\x\y)[mpunkt,\colc] at ($( \x  *\k + \k/2,\y*\h)$)  {};
		\draw[\cola,thick] (p11) -- (p12) -- (p22)
		(p31) -- (p22) -- (p33) -- (p43) -- (p42) -- (p52) -- (p61) -- (p71) -- (p72)
		(p43) -- (p44) -- (p55) -- (p46) -- (p47) -- (p48)
		(p52) -- (p63);
		\draw[\colb,thick] (p11) -- (p21) -- (p22) --(p32) -- (p43) -- (p53) -- (p52) -- (p62) -- (p71) -- (p81)
		(p53) -- (p55) -- (p57) -- (p47) -- (p45) -- (p43)
		(p55) -- (p45);
		\draw[\colc,thick] (p12) -- (p21) -- (p31) -- (p32) -- (p33)
		(p32) -- (p42) -- (p53) -- (p63) -- (p62) -- (p72) -- (p81)
		(p53) -- (p44) -- (p45) -- (p46) -- (p57) -- (p48)
		(p61) -- (p62);
		\node at ($(p11)!0.5!(p21) + (0,\h/3)$) {$D$};
		\node at ($(p71)!0.5!(p81) + (0,\h/3)$) {$D$};
		\node at ($(p12)!0.5!(p22) - (0,\h/3)$) {$D_1$};
		\node at ($(p62)!0.5!(p72) - (0,\h/3)$) {$D_3$};
		\node at ($(p43)!0.5!(p53) - (0,\h/3)$) {$C$};
		\node at ($(p47)!0.5!(p57) + (0,\h/3)$) {$D$};
		\node at ($(p47)!0.5!(p57) - (0,\h/3)$) {$D_2$};
		\end{tikzpicture}
	\end{center}
	\caption{The schematic picture of $\convn(C,D)$. Blue vertices correspond to vertices of rank $1$, green vertices to rank $2$, and orange vertices to rank $3$.}
	\label{fig:conv_hull}
\end{figure}
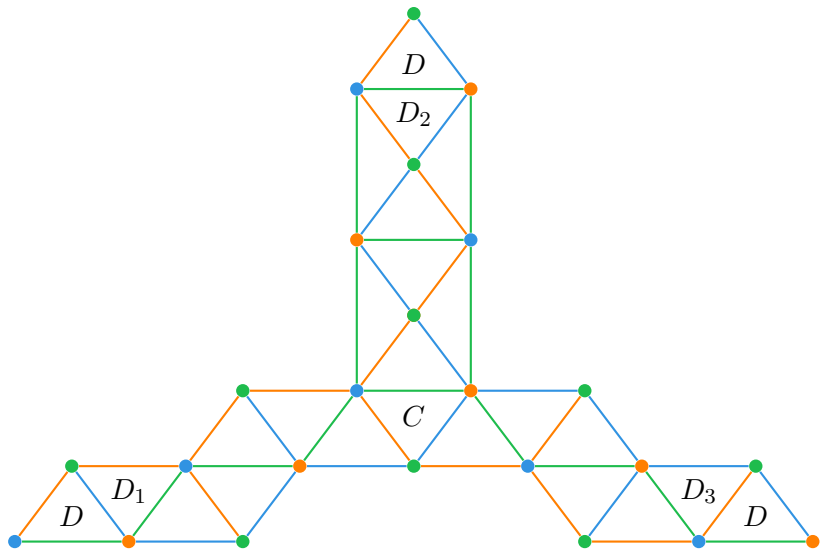

\subsubsection{Not opposite in the building}

The last case we consider is that the two chambers $C,D \in \Cn$ are disjoint, not opposite in the building and satisfy $\din(C,D)> \dig(C,D)$. In the following, we compute the convex hull in an example in $|\ncp_6|$.

Let us consider the chambers $C$ and $D$ of $|\ncp_6|$ corresponding to the reduced decompositions $(1\,\;2)(3\,\;6)(4\,\;5)(2\,\;6)(3\,\;5)$ and  $(2\,\;4)(1\,\;4)(5\,\;6)(2\,\;3)(4\,\;6)$, respectively. The pictorial representations of their vertices are shown in \cref{fig:chambers_C_D}. 
They are clearly disjoint and not opposite in the building, since the rank $1$ partition of $C$ is less than the rank $4$ partition of $D$ in $\ncp_6$ in the absolute order. Their distance in the building equals $\dig(C,D)=7$, which can be seen as follows. The chamber $D$ is adjacent to a chamber $A$ that is contained in the star of the partition $p$ corresponding to $(1\,\;2)\in S_6$. Since both $A$ and $C$ are contained in $\sta(p,\si)$, their distance in $\si$ equals their distance in $\sta(p,\si)$. Hence $\dig(C,A)=6$, because they are opposite in $\sta(p,\si)$. So we get that $\dig(C,D)=7$, since $D$ is not contained in $\sta(p,\si)$. 

\begin{rem}
	The chambers $C$ and $A$ correspond, after a suitable identification of $\sta(p,|\ncp_6|)$ with $|\ncp_5|$, to the chambers $C$ and $D$ from the previous section.
\end{rem}

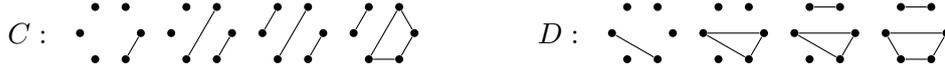
\begin{figure}
	\begin{center}
		\begin{tikzpicture}
		\sechseck \draw (p1) -- (p2);
		\node[left = 3mm] at (p4) {$C:$};
		\begin{scope}[xshift = 1.2cm]
		\sechseck \draw (p1) -- (p2) (p6) -- (p3);
		\end{scope}
		\begin{scope}[xshift = 2.4cm]
		\sechseck \draw (p1) -- (p2) (p6) -- (p3) (p5) -- (p4);
		\end{scope}
		\begin{scope}[xshift = 3.6cm]
		\sechseck \draw (p1) -- (p2) -- (p3) -- (p6) -- (p1) (p5) -- (p4);
		\end{scope}
		
		\begin{scope}[xshift = 7cm]
		\sechseck \draw (p2) -- (p4);
		\node[left = 3mm] at (p4) {$D:$};
		\begin{scope}[xshift = 1.2cm]
		\sechseck \draw (p1) -- (p4) -- (p2)--(p1);
		\end{scope}
		\begin{scope}[xshift = 2.4cm]
		\sechseck \draw (p6) -- (p5) (p1) -- (p4) -- (p2) -- (p1);
		\end{scope}
		\begin{scope}[xshift = 3.6cm]
		\sechseck \draw (p1) -- (p2) -- (p3) -- (p4) -- (p1) (p5) -- (p6);
		\end{scope}		
		\end{scope}
		\end{tikzpicture}
		\caption{The depicted chambers $C$ and $D$ of $|\ncp_6|$ are not opposite in the building $\si$, but their distance in $|\ncp_6|$ is greater than their distance in the building $\si$.}
		\label{fig:chambers_C_D}
	\end{center}
\end{figure}

In order to show that $\din(C,D)=8$, we show that the distance of all chambers in $|\ncp_6|$ that are adjacent to $D$ have distance $7$ or greater in $|\ncp_6|$. We denote the neighbors of $D$ by $A$, $B$, $E$, $F$ and $G$, where by \emph{neighbor}  we mean an adjacent chamber. 
The respective adjacency relations are shown in \cref{fig:nbs} and their vertices are depicted in \cref{fig:nbs_vert}.
We compute the respective distances of the neighbors of $D$ to $C$ in the building using the fact that each pair is contained in a common apartment. Since this is a Coxeter complex, we can adapt the procedure of \cref{lem:dist_aptm} in the setting of buildings. This yields that $\dig(C,B) = \dig(C,E) = \dig(C,G) = 7$ and $\dig(C,F)=8$. Hence we have that $\din(C,B),\din(C,E),\din(C,G) \geq 7$ and that $\din(C,F)\geq 8$. With the same argumentation as for the distance of $C$ and $D$ in $|\ncp_6|$, we get that $\din(C,G)=8$. Hence neither $F$ nor $G$ are contained in $\convn(C,D)$.
In order to compute $\din(C,B)$ and $\din(C,E)$, we consider the chambers $B_1$ and $E_1$, which are adjacent to $A$ and $B$, or $A$ and $E$, respectively. \cref{fig:nbs} shows their adjacency relations and their vertices are depicted in \cref{fig:nbs_BpEp}. Since $B_1,E_1 \in \sta(p,|\ncp_6|)$, we can easily compute their distances and obtain that $\din(C,B_1)=\din(C,E_1)=6$. Consequently the distance in $|\ncp_6|$ is $\din(C,B)=\din(C,E)=7$ and we get that 
\begin{align*}
\convn(C,D) &= \convn(C,A) \cup \convn(C,B) \cup \convn(C,E) \cup D.
\end{align*}
So we have to compute $\convn(C,B)$ and $\convn(C,E)$. Besides $B_1$, there is only one other neighbor, call it $B_2$, of $B \in |\ncp_6|$ that has lesser distance. The same is true for $E$ and its neighbor $E_2$. These chambers are depicted in \cref{fig:nbs_BpEp}. Note that both $B_2$ and $E_2$ have neighbors that are contained in $\sta(p,|\ncp_6|)$, which are called $B_3$ and $E_3$, respectively. They are depicted in \cref{fig:nbs_BpEp} as well. Note that $B_3$ is adjacent to $B_1$, and $E_3$ is adjacent to $E_1$. If we continue to investigate the convex hulls $\convn(C,B)$ and $\convn(C,E)$, we see that the following holds for all their respective chambers: Either they are contained in $\sta(p,|\ncp_6|)$, or they are adjacent to chambers of $|\ncp_6|$ that are contained in $\sta(p,|\ncp_6|)$. Since $\convn(C,D)\cap\sta(p,|\ncp_6|)$ is shrinkable by \cref{lem:star} and $\convn(C,D)$ retracts onto the former one, if follows that $\convn(C,D)$ is shrinkable as well.

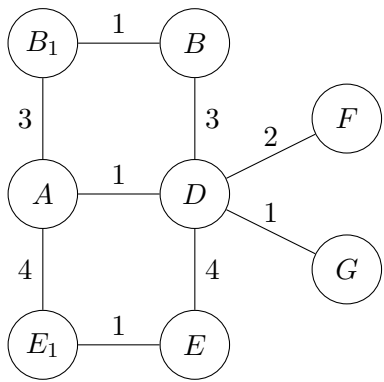
\begin{figure}
	\begin{center}
		\begin{tikzpicture}
		\def\r{2}
		
		\node[blub] (D) at (0,0) {$D$};
		\node[blub] (B) at (0,\r) {$B$}; 
		\node[blub] (E) at (0,-\r) {$E$}; 
		\node[blub] (A) at (-\r,0) {$A$}; 
		\node[blub] (F) at (\r,\r/2) {$F$};
		\node[blub] (G) at (\r,-\r/2) {$G$};
		\node[blub] (Bp) at (-\r,\r) {$B_1$}; 
		\node[blub] (Ep) at (-\r,-\r) {$E_1$};
		\draw (D) to node[midway,right]{3} (B);
		\draw (E) to node[midway,right]{4} (D);
		\draw (A) to node[midway,left]{3} (Bp);
		\draw (Ep) to node[midway,left]{4} (A);
		\draw (A) to node[midway,above]{1} (D);
		\draw (D) to node[midway,above]{2} (F);
		\draw (D) to node[midway,above]{1} (G);
		\draw (Bp) to node[midway,above]{1} (B);
		\draw (Ep) to node[midway,above]{1} (E);
		
		\end{tikzpicture}
	\end{center}
	\caption{A part of the Hurwitz graph corresponding to $|\ncp_6|$ around the chamber $D$ is shown here. The label of an edge corresponds to the rank of the partition the two respective adjacent chambers differ in.}
	\label{fig:nbs}
\end{figure}

\begin{figure}
	\begin{center}
		\begin{tikzpicture}
		\begin{scope}
		\sechseck \draw (p1) -- (p2);
		\node[left = 3mm] at (p4) {$A:$};
		
		\begin{scope}[xshift = 1.2cm]
		\sechseck \draw (p1) -- (p2) -- (p4)--(p1);
		\end{scope}
		
		\begin{scope}[xshift = 2.4cm]
		\sechseck \draw (p1) -- (p2) -- (p4)--(p1) (p5)--(p6);
		\end{scope}
		
		\begin{scope}[xshift = 3.6cm]
		\sechseck \draw (p1) -- (p2) -- (p3)--(p4)--(p1) (p5)--(p6);
		
		\end{scope}
		
		\begin{scope}[xshift = 7cm]
		\sechseck \draw (p2) -- (p4);
		\node[left = 3mm] at (p4) {$F:$};
		
		\begin{scope}[xshift = 1.2cm]
		\sechseck \draw (p2)--(p4)(p5)--(p6);
		\end{scope}
		
		\begin{scope}[xshift = 2.4cm]
		\sechseck \draw (p1) -- (p2) -- (p4)--(p1) (p5)--(p6);
		\end{scope}
		
		\begin{scope}[xshift = 3.6cm]
		\sechseck \draw (p1) -- (p2) -- (p3)--(p4)--(p1) (p5)--(p6);	
		\end{scope}
		
		\end{scope}
		
		\end{scope}
		
		\begin{scope}[yshift=-1.4cm]
		\sechseck \draw (p4) -- (p2);
		\node[left = 3mm] at (p4) {$B:$};
		
		\begin{scope}[xshift = 1.2cm]
		\sechseck \draw (p1) -- (p4) -- (p2)--(p1);
		\end{scope}
		
		\begin{scope}[xshift = 2.4cm]
		\sechseck \draw (p1) -- (p4) --(p3)-- (p2) -- (p1);
		\end{scope}
		
		\begin{scope}[xshift = 3.6cm]
		\sechseck \draw (p1) -- (p2) -- (p3) -- (p4) -- (p1) (p5) -- (p6);
		\end{scope}		
		
		\begin{scope}[xshift = 7cm]
		
		\sechseck \draw (p1) -- (p4);
		\node[left = 3mm] at (p4) {$G:$};
		
		\begin{scope}[xshift = 1.2cm]
		\sechseck \draw (p1) -- (p4) -- (p2)--(p1);
		\end{scope}
		
		\begin{scope}[xshift = 2.4cm]
		\sechseck \draw (p6) -- (p5) (p1) -- (p4) -- (p2) -- (p1);
		\end{scope}
		
		\begin{scope}[xshift = 3.6cm]
		\sechseck \draw (p1) -- (p2) -- (p3) -- (p4) -- (p1) (p5) -- (p6);
		\end{scope}		
		
		\end{scope}
		\end{scope}
		
		\begin{scope}[yshift=-2.8cm]
		\sechseck \draw (p4) -- (p2);
		\node[left = 3mm] at (p4) {$E:$};
		
		\begin{scope}[xshift = 1.2cm]
		\sechseck \draw (p1) -- (p4) -- (p2)--(p1);
		\end{scope}
		
		\begin{scope}[xshift = 2.4cm]
		\sechseck \draw (p6) -- (p5) (p1) -- (p4) -- (p2) -- (p1);
		\end{scope}
		
		\begin{scope}[xshift = 3.6cm]
		\sechseck \draw (p1) -- (p2) -- (p4) -- (p5) -- (p6) -- (p1);
		\end{scope}		
		
		\end{scope}
		
		\end{tikzpicture}
		\caption{The chambers $A$, $B$, $E$, $F$ and $G$ are shown here.}
		\label{fig:nbs_vert}
	\end{center}
\end{figure}
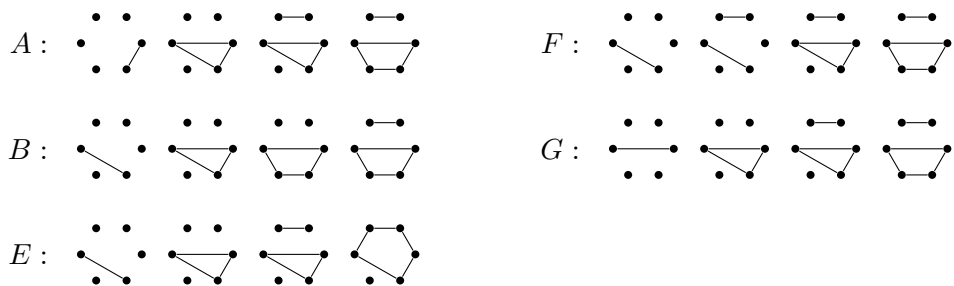

\begin{figure}
	\begin{center}
		\begin{tikzpicture}
		\sechseck \draw (p1) -- (p2);
		\node[left = 3mm] at (p4) {$B_1:$};
		
		\begin{scope}[xshift = 1.2cm]
		\sechseck \draw (p1) -- (p4) -- (p2)--(p1);
		\end{scope}
		
		\begin{scope}[xshift = 2.4cm]
		\sechseck \draw (p1) -- (p4) --(p3)-- (p2) -- (p1);
		\end{scope}
		
		\begin{scope}[xshift = 3.6cm]
		\sechseck \draw (p1) -- (p2) -- (p3) -- (p4) -- (p1) (p5) -- (p6);
		\end{scope}	
		
		\begin{scope}[xshift = 7cm]
		\sechseck \draw (p1) -- (p2);
		\node[left = 3mm] at (p4) {$E_1:$};
		
		\begin{scope}[xshift = 1.2cm]
		\sechseck \draw (p1) -- (p4) -- (p2)--(p1);
		\end{scope}
		
		\begin{scope}[xshift = 2.4cm]
		\sechseck \draw (p6) -- (p5) (p1) -- (p4) -- (p2) -- (p1);
		\end{scope}
		
		\begin{scope}[xshift = 3.6cm]
		\sechseck \draw (p1) -- (p2) -- (p4) -- (p5) -- (p6) -- (p1);
		\end{scope}		
		\end{scope}

		\begin{scope}[yshift=-1.4cm]
		\sechseck \draw (p4) -- (p2);
		\node[left = 3mm] at (p4) {$B_2:$};
		
		\begin{scope}[xshift = 1.2cm]
		\sechseck \draw (p1) -- (p4) -- (p2)--(p1);
		\end{scope}
		
		\begin{scope}[xshift = 2.4cm]
		\sechseck \draw (p1) -- (p4) --(p3)-- (p2) -- (p1);
		\end{scope}
		
		\begin{scope}[xshift = 3.6cm]
		\sechseck \draw (p1) -- (p2) -- (p3) -- (p4) -- (p6) -- (p1);
		\end{scope}	
		
		\begin{scope}[xshift = 7cm]
		\sechseck \draw (p4) -- (p2);
		\node[left = 3mm] at (p4) {$E_2:$};
		
		\begin{scope}[xshift = 1.2cm]
		\sechseck \draw (p1) -- (p4) -- (p2)--(p1);
		\end{scope}
		
		\begin{scope}[xshift = 2.4cm]
		\sechseck \draw  (p1) -- (p5)--(p4) -- (p2) -- (p1);
		\end{scope}
		
		\begin{scope}[xshift = 3.6cm]
		\sechseck \draw (p1) -- (p2) -- (p4) -- (p5) -- (p6) -- (p1);
		\end{scope}		
		\end{scope}
		\end{scope}
		
		\begin{scope}[yshift=-2.8cm]
		\sechseck \draw (p1) -- (p2);
		\node[left = 3mm] at (p4) {$B_3:$};
		
		\begin{scope}[xshift = 1.2cm]
		\sechseck \draw (p1) -- (p4) -- (p2)--(p1);
		\end{scope}
		
		\begin{scope}[xshift = 2.4cm]
		\sechseck \draw (p1) -- (p4) --(p3)-- (p2) -- (p1);
		\end{scope}
		
		\begin{scope}[xshift = 3.6cm]
		\sechseck \draw (p1) -- (p2) -- (p3) -- (p4) -- (p6) -- (p1);
		\end{scope}	
		
		\begin{scope}[xshift = 7cm]
		\sechseck \draw (p1) -- (p2);
		\node[left = 3mm] at (p4) {$E_3:$};
		
		\begin{scope}[xshift = 1.2cm]
		\sechseck \draw (p1) -- (p4) -- (p2)--(p1);
		\end{scope}
		
		\begin{scope}[xshift = 2.4cm]
		\sechseck \draw  (p1) -- (p5)--(p4) -- (p2) -- (p1);
		\end{scope}
		
		\begin{scope}[xshift = 3.6cm]
		\sechseck \draw (p1) -- (p2) -- (p4) -- (p5) -- (p6) -- (p1);
		\end{scope}		
		\end{scope}
		\end{scope}
		\end{tikzpicture}
		\caption{The chambers $B_1$, $B_2$, $B_2$, $E_1$, $E_2$ and $E_3$ are shown here.}
		\label{fig:nbs_BpEp}
	\end{center}
\end{figure}
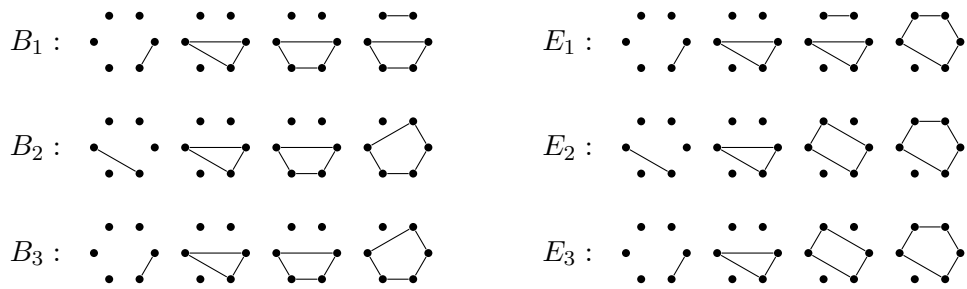


\cleardoublepage
\part*{Appendix}
\cleardoublepage
\appendix

\chapter{Proof of \cref{prop:typeB_anti-auto_extends}}\label{chap:typeB_details}	

In this appendix we are proving \cref{prop:typeB_anti-auto_extends}. For this, we need a listing of reduced decompositions of elements of type rank $2$ in $\nc(B_n)$.
There are 31 different types of reduced decompositions for the 13 different types of non-crossing partitions of rank 2. They are listed in \cref{tab:typeB_rk2_red_decomp}.

\renewcommand{\arraystretch}{1,2}
\begin{table}[h]
	\begin{center}
		\begin{tabular}{|l||l|}
			
			\hhline{|-||-|}	
			rank 2 element& reduced decompositions\\

			\hhline{|-||-|}	
			
			\multicolumn{2}{c}{ }\\[-10pt]
			\multicolumn{2}{c}{$1 \leq a < b <c \leq n$}\\
			
			\hhline{|-||-|}

			$[a\,\;b]$ 
			& $\dka a\,\;b\dkz [b]$, $\dka a\,\;-b\dkz [a]$, $[a]\dka a\,\;b\dkz$, $[b]\dka a\,\;-b\dkz$ \\
			
			$\dka a\,\;-c\dkz [b]$
			&$\dka a\,\;-c\dkz [b]$, $ [b]\dka a\,\;-c\dkz$	\\
			
			$\dka b\,\;c\dkz [a]$ 
			&$\dka b\,\;c\dkz [a]$, $ [a]\dka b\,\;c\dkz$	\\
			
			$\dka a\,\;b\dkz [c]$
			&$\dka a\,\;b\dkz [c]$, $[c]\dka a\,\;b\dkz $	\\
			
			$\dka a\,\;b\,\;c\dkz$
			&	$\dka a\,\;b\dkz\dka b\,\;c\dkz$, $\dka b\,\;c\dkz\dka a\,\;c\dkz$, $\dka a\,\;c\dkz\dka a\,\;b\dkz$	\\
			
			$\dka a\,\;b\,\;-c\dkz$
			&	$\dka a\,\;b\dkz\dka b\,\;-c\dkz$, $\dka b\,\;-c\dkz\dka a\,\;-c\dkz$, $\dka a\,\;-c\dkz\dka a\,\;-b\dkz$	\\
			
			$\dka a\,\;-b\,\;-c\dkz$
			&	$\dka a\,\;-b\dkz\dka b\,\;c\dkz$, $\dka b\,\;c\dkz\dka a\,\;-c\dkz$, $\dka a\,\;-c\dkz\dka a\,\;-b\dkz$	\\	
			
			\hhline{|-||-|}	
			
			\multicolumn{2}{c}{ }\\[-10pt]	
			\multicolumn{2}{c}{$1 \leq a < b <c <d \leq n$}\\
			
			\hhline{|-||-|}	
			
			$\dka a\,\;b\dkz\dka c\,\;d\dkz$
			& $\dka a\,\;b\dkz\dka c\,\;d\dkz$, $\dka c\,\;d\dkz\dka a\,\;b\dkz$	\\ 
			
			$\dka a\,\;b\dkz \dka c\,\;-d\dkz$
			&$\dka a\,\;b\dkz \dka c\,\;-d\dkz$, $\dka c\,\;-d\dkz\dka a\,\;b\dkz$	\\
			
			$\dka a\,\;-b\dkz \dka c\,\;d\dkz$
			&$\dka a\,\;-b\dkz \dka c\,\;d\dkz$, $ \dka c\,\;d\dkz\dka a\,\;-b\dkz$	\\
			
			$\dka a\,\;d\dkz \dka b\,\;c\dkz$
			&$\dka a\,\;d\dkz \dka b\,\;c\dkz$, $\dka b\,\;c\dkz\dka a\,\;d\dkz $	\\
			
			$\dka a\,\;-d\dkz \dka b\,\;c\dkz$
			&$\dka a\,\;-d\dkz \dka b\,\;c\dkz$, $ \dka b\,\;c\dkz\dka a\,\;-d\dkz$	\\
			
			$\dka a\,\;-d\dkz \dka b\,\;-c\dkz$
			&$\dka a\,\;-d\dkz \dka b\,\;-c\dkz$, $ \dka b\,\;-c\dkz\dka a\,\;-d\dkz$	\\	
			\hhline{|-||-|}	
		\end{tabular}
		\caption{The reduced decompositions of rank 2 elements in $\nc(B_n)$.}\label{tab:typeB_rk2_red_decomp}
	\end{center}
\end{table}
\renewcommand{\arraystretch}{1}

\begin{lem}
	Every element in $\nc(B_n)$ of rank $2$ is of one of the following forms for $1\leq a <b<c<d\leq n$:
	\begin{align*}
	&[a\,\;b], \
	\dka a\,\;-c\dkz [b],\ \dka b\,\;c\dkz [a],\ \dka a\,\;b\dkz [c],\ 
	\dka a\,\;b\,\;c\dkz ,\ \dka a\,\;b\,\;-c\dkz ,\ \dka a\,\;-b\,\;-c\dkz,\\
	&\dka a\,\;b\dkz \dka c\,\;d\dkz ,\ 
	\dka a\,\;b\dkz \dka c\,\;-d\dkz ,\ 
	\dka a\,\;-b\dkz \dka c\,\;d\dkz ,\\
	&\dka a\,\;d\dkz \dka b\,\;c\dkz ,\ \dka a\,\;-d\dkz \dka b\,\;c\dkz ,\ \dka a\,\;-d\dkz \dka b\,\;-c\dkz .
	\end{align*}
\end{lem}

\begin{proof}
	All listed elements are of rank 2, oriented consistently and non-crossing. By \cref{prop:typeB_consistent_orientation-NC} they are all contained in $\nc(B_n)$. Moreover, every element of rank 2 in $W(B_n)$ has a cycle structure that appears in the list. Note that the element $[a][b]=\dka a\,\;b\dkz \dka a\,\;-b\dkz$ is \emph{not} contained in $\nc(B_n)$. Again by \cref{prop:typeB_consistent_orientation-NC} it is enough to check that all other elements in $W$ of the above cycle structures are crossing or not oriented consistently. The crossing elements are $\dka a\,\;c\dkz [b]$, $\dka b\,\;-c\dkz [a]$, $\dka a\,\;-b\dkz [c]$, $\dka a\,\;-b\dkz \dka c\,\;-d\dkz $ and $\dka a\,\;d\dkz \dka b\,\;-c\dkz $. The non-consistently oriented elements are $[a\,\;-b]$, $\dka a\,\;c\,\;b\dkz $, $\dka a\,\;c\,\;-b\dkz $ and $\dka a\,\;-c\,\;-b\dkz$. This completes the list of elements of rank 2.
\end{proof}

We use the convention for the notation of $\dka x\,\;y\dkz$ that either $1\leq x < y$ or $1\leq x <-y$.

\begin{lem}\label{lem:typeB_classes_red_decomp}
	Let $w\in \nc(B_n)$ be of rank two with reduced decomposition $st$ where $t=\dka i\,\;j\dkz$ or $t=[i]$ for $1\leq i < j \leq n$. Then every reduced decomposition $st$ of $w$ is in one of the following classes for $1\leq k<l\leq n$.
	
	\begin{itemize}
		\item 
		In the case of $s=\dka k\,\;l\dkz $, exactly one of the following cases holds:
		\begin{enumerate}
			\item $i,|j|<k$,
			\item $i,|j|\geq l$,
			\item $i<k$ and $|j|\geq l$,
			\item if $t=[i]$: $i \geq l$,
			\item $k\leq i < j < l$. 
		\end{enumerate}
		\item
		If $s=\dka k\,\;-l\dkz$, then exactly one of the following cases holds:
		\begin{enumerate}\setcounter{enumi}{5}
			\item $k \leq i,|j| < l$,
			\item $i,j < k$, 
			\item $l \leq i,j $, 
			\item $i < k < l \leq -j $. 
		\end{enumerate}
		\item 
		If $s=[k]$, then exactly one of the following cases holds:
		\begin{enumerate}\setcounter{enumi}{9}
			\item $ i,j < k$,
			\item $k \leq i,j$,
			\item $i<k \leq -j$.
		\end{enumerate}
	\end{itemize}
\end{lem}
\begin{proof}
	The proof is by inspection. All possibilities for $s$ and $t$ are listed in \cref{tab:typeB_rk2_red_decomp}.
\end{proof}

\begin{proof}[Proof of \cref{prop:typeB_anti-auto_extends}]
	We show that $\bilb(s,t)=0$ for all reflections $s,t$ such that $st\in \nc(B_n)$. The 31 different reduced decompositions for rank 2 elements are listed in \cref{tab:typeB_rk2_red_decomp}. They can be divided into three classes, depending on the structure of the first reflection $s$.
	
	Let $1\leq a < b \leq n$. Recall that the root in $V$ corresponding to $\dka a\,\; b\dkz$ is $e_a-e_b$, the root corresponding to $\dka a\,\; -b\dkz$ is $e_a + e_b$, and the root corresponding to $[a]$ is $e_a$.
	
	Let $s= \dka k\,\;l \dkz$ with $1\leq k < l \leq n$. This is the first class of \cref{lem:typeB_classes_red_decomp}. For $v\in V$ it holds that
	\[
	\bilb(s,v)= -\left(\sum_{m=k}^{l-1} e_m\tr\right)\cdot v.
	\]
	If $st\in \nc(B_n)$, then $t=\dka i\,\; j\dkz$ or $t=[i]$ fulfills one of the cases \emph{a)}-\emph{e)} in \cref{lem:typeB_classes_red_decomp}. We show that in all of these cases it holds that $\bilb(s,t)=0$.
	
	If $t$ is in case \emph{a)}-\emph{e)}, all non-zero entries of $\alpha_t$ are less than $k$ or greater than $l-1$, hence $\bilb(s,t)=0$.
	In case $\emph{e)}$ the bilinear form evaluates to $\bilb(s,t)=-e_i\tr\cdot e_i+e_j\tr\cdot e_j=0$.
	
	Let $s= \dka k\,\;-l \dkz$ with $1\leq k < l \leq n$. This is the second class of \cref{lem:typeB_classes_red_decomp}. For $v\in V$ it holds that
	\[
	\bilb(s,v)= -2\left(\sum_{m=1}^{k-1} e_m\tr\right)\cdot v + 2\left(\sum_{m=l}^{n} e_m\tr\right)\cdot v.
	\]
	
	As before, let $st\in\nc(B_n)$ and let $t=\dka i\,\; j\dkz$ or $t=[i]$ fulfill case \emph{f)} in \cref{lem:typeB_classes_red_decomp}. Then all non-zero entries of $\alpha_s$ are greater than or equal to $k$, and less than $l$, hence $\bilb(s,t)=0$.
	
	In case \emph{g)} or \emph{h)} of \cref{lem:typeB_classes_red_decomp}, the two non-zero entries of $\alpha_t$ have different signs and get multiplied with the same sign in the evaluation of $\bilb(s,t)$. Hence $\bilb(s,t)=0$. The last case \emph{i)} in this class is similar. Now the two non-zero entries have the same sign, but they are multiplied with different signs, hence they again add up to $\bilb(s,t)=0$.
	
	Let $s=[k]$ with $1\leq k \leq n$ and let $st\in\nc(B_n)$. This is the third class of \cref{lem:typeB_classes_red_decomp}. For $v\in V$ it holds that
	\[
	\bilb(s,v)= -\left(\sum_{m=1}^{k-1} e_m\tr\right)\cdot v + \left(\sum_{m=k}^{n} e_m\tr\right)\cdot v.
	\]
	Comparing this to the previous class, we see that for $k=l$ we are in the same situation except for the factor 2. The analog of case \emph{f)} does not appear, and the argumentation of cases \emph{j}-\emph{l)} is the same as for the cases \emph{g)}-\emph{h)}, in this order.
\end{proof}

\cleardoublepage

\chapter{Proof of \cref{prop:typeD_anti-auto_extends}}\label{chap:typeD_details}

In this appendix we are proving \cref{prop:typeD_anti-auto_extends}. For this, we need a listing of reduced decompositions of elements of type rank $2$ in $\nc(D_n)$, which is given in \cref{tab:typeD_rk2_red_decomp}.

\renewcommand{\arraystretch}{1,15}
\begin{table}[h]
	\begin{center}
		\begin{tabular}{|l||l|}
			\hhline{|-||-|}
			
			rank 2 element& reduced decompositions\\
			\hhline{|-||-|}
			
			\multicolumn{2}{c}{ }\\[-10pt]
			\multicolumn{2}{c}{$1 \leq a < b <c < n$}\\
			
			\hhline{|-||-|}	
			
			$\dka a\,\;b\,\;c\dkz$
			&	$\dka a\,\;b\dkz\dka b\,\;c\dkz$, $\dka b\,\;c\dkz\dka a\,\;c\dkz$, $\dka a\,\;c\dkz\dka a\,\;b\dkz$	\\
			
			$\dka a\,\;b\,\;-c\dkz$
			&	$\dka a\,\;b\dkz\dka b\,\;-c\dkz$, $\dka b\,\;-c\dkz\dka a\,\;-c\dkz$, $\dka a\,\;-c\dkz\dka a\,\;-b\dkz$	\\
			
			$\dka a\,\;-b\,\;-c\dkz$
			&	$\dka a\,\;-b\dkz\dka b\,\;c\dkz$, $\dka b\,\;c\dkz\dka a\,\;-c\dkz$, $\dka a\,\;-c\dkz\dka a\,\;-b\dkz$	\\

			$\dka a\,\;b\dkz \dka c\,\;n\dkz $
			&$\dka a\,\;b\dkz \dka c\,\;n\dkz $, $\dka c\,\;n\dkz \dka a\,\;b\dkz $		\\
			
			$\dka a\,\;b\dkz \dka -c\,\;n\dkz $
			&$\dka a\,\;b\dkz \dka -c\,\;n\dkz $, $\dka -c\,\;n\dkz \dka a\,\;b\dkz $		\\
			
			$\dka b\,\;c\dkz \dka a\,\;n\dkz $
			&$\dka b\,\;c\dkz \dka a\,\;n\dkz $, $\dka a\,\;n\dkz \dka b\,\;c\dkz $		\\
			
			$\dka b\,\;c\dkz \dka -a\,\;n\dkz $
			&$\dka b\,\;c\dkz \dka -a\,\;n\dkz $, $\dka -a\,\;n\dkz \dka b\,\;c\dkz $		\\
			
			$\dka a\,\;-c\dkz \dka b\,\;n\dkz $
			&$\dka a\,\;-c\dkz \dka b\,\;n\dkz $, $\dka b\,\;n\dkz \dka a\,\;-c\dkz $		\\
			
			$\dka a\,\;-c\dkz \dka -b\,\;n\dkz $
			&$\dka a\,\;-c\dkz \dka -b\,\;n\dkz $, $\dka -b\,\;n\dkz \dka a\,\;-c\dkz $		\\

			\hhline{|-||-|}
			
			\multicolumn{2}{c}{ }\\[-10pt]
			\multicolumn{2}{c}{$1 \leq a < b < n$ and $c < n$}\\
			
			\hhline{|-||-|}

			$\dka a\,\;b\,\;n\dkz $
			&$\dka a\,\;b\dkz \dka b\,\;n\dkz $, $\dka b\,\;n\dkz \dka a\,\;n\dkz $, $\dka a\,\;n\dkz \dka a\,\;b\dkz $		\\
			
			$\dka -a\,\;-b\,\;n\dkz $
			&$\dka a\,\;b\dkz \dka -b\,\;n\dkz $, $\dka -b\,\;n\dkz \dka -a\,\;n\dkz $, $\dka -a\,\;n\dkz \dka a\,\;b\dkz $			\\
			
			$\dka b\,\;-a\,\;n\dkz $
			&$\dka a\,\;-b\dkz \dka -a\,\;n\dkz $, $\dka -a\,\;n\dkz \dka b\,\;n\dkz $, $\dka b\,\;n\dkz \dka a\,\;-b\dkz $			\\
			
			$\dka -b\,\;a\,\;n\dkz $
			&$\dka a\,\;-b\dkz \dka a\,\;n\dkz $, $\dka a\,\;n\dkz \dka -b\,\;n\dkz $, $\dka -b\,\;n\dkz \dka a\,\;-b\dkz $			\\
			
			$[c][n]$
			& $\dka a\,\;n\dkz \dka -a\,\;n\dkz $, $\dka -a\,\;n\dkz \dka a\,\;n\dkz $\\

			\hhline{|-||-|}	
			
			\multicolumn{2}{c}{ }\\[-10pt]
			\multicolumn{2}{c}{$1 \leq a < b <c <d <n$}\\
			
			\hhline{|-||-|}

			$\dka a\,\;b\dkz\dka c\,\;d\dkz$
			& $\dka a\,\;b\dkz\dka c\,\;d\dkz$, $\dka c\,\;d\dkz\dka a\,\;b\dkz$	\\ 
			
			$\dka a\,\;b\dkz \dka c\,\;-d\dkz$
			&$\dka a\,\;b\dkz \dka c\,\;-d\dkz$, $\dka c\,\;-d\dkz\dka a\,\;b\dkz$	\\
			
			$\dka a\,\;-b\dkz \dka c\,\;d\dkz$
			&$\dka a\,\;-b\dkz \dka c\,\;d\dkz$, $ \dka c\,\;d\dkz\dka a\,\;-b\dkz$	\\
			
			$\dka a\,\;d\dkz \dka b\,\;c\dkz$
			&$\dka a\,\;d\dkz \dka b\,\;c\dkz$, $\dka b\,\;c\dkz\dka a\,\;d\dkz $	\\
			
			$\dka a\,\;-d\dkz \dka b\,\;c\dkz$
			&$\dka a\,\;-d\dkz \dka b\,\;c\dkz$, $ \dka b\,\;c\dkz\dka a\,\;-d\dkz$	\\
			
			$\dka a\,\;-d\dkz \dka b\,\;-c\dkz$
			&$\dka a\,\;-d\dkz \dka b\,\;-c\dkz$, $ \dka b\,\;-c\dkz\dka a\,\;-d\dkz$	\\	
			\hhline{|-||-|}
		\end{tabular}
		\caption{The reduced decompositions of rank 2 elements in $\nc(D_n)$.}\label{tab:typeD_rk2_red_decomp}
	\end{center}
\end{table}

\begin{lem}\label{lem:typeD_classes_red_decomp}
	Let $w\in \nc(D_n)$ be of rank two. Then every reduced decomposition $st$ of $w$ is in one of the following classes.
	
	\begin{itemize}
		\item If neither $s$ nor $t$ contains the letter $n$, then $s,t\leq [1\ldots n-1]$. Hence they are also elements of $\nc(B_{n-1})$ and the classification of \cref{lem:typeB_classes_red_decomp} holds. 
		
		\item 
		In the case of $s=\dka k\,\;l\dkz $ for $1\leq k < l < n$, exactly one of the following cases holds for $t=\dka i\,\; n\dkz$ with $1\leq |i|< n$:
		\begin{enumerate}
			\item $|i|<k$,
			\item $l\leq |i|$.
		\end{enumerate}
		
		\item 
		In the case of $s=\dka k\,\;-l\dkz $ for $1\leq k < l < n$, it holds  for $t=\dka i\,\; n\dkz$ with $1\leq |i|$ that
		\begin{enumerate}\setcounter{enumi}{2}
			\item $k \leq |i|< l$.
		\end{enumerate}
		
		\item 
		In the case of $s=\dka k\,\;n\dkz $ for $1\leq k < n$, exactly one of the following cases holds for $t=\dka i\,\; j\dkz$ with $1\leq |i| < |j| \leq n $:
		\begin{enumerate}\setcounter{enumi}{3}
			\item $k \leq i<j <n$,
			\item $k \leq -i <j =n$,
			\item $i<k\leq -j <n$,
			\item $i<k<j =n$,
			\item $i<j<k$.
		\end{enumerate}
		
		\item
		In the case of $s=\dka -k\,\;n\dkz $ for $1\leq k < n$, exactly one of the following cases holds for $t=\dka i\,\; j\dkz$ for $1\leq |i| < |j| \leq n$:
		\begin{enumerate}\setcounter{enumi}{8}
			\item $k\leq i<j$,
			\item $i<j<k$,
			\item $i<k\leq -j<n$,
			\item $-i < k < j =n$.
		\end{enumerate}
	\end{itemize}
\end{lem}

\begin{proof}
	The proof is by inspection of \cref{tab:typeD_rk2_red_decomp}.
\end{proof}

\begin{proof}[Proof of \cref{prop:typeD_anti-auto_extends}]
	As in the proof of \cref{prop:typeB_anti-auto_extends}, we show that $\bild(s,t)=0$ for all reflections $s,t$ such that $st\in \nc(D_n)$. 
	The 47 different reduced decompositions for rank 2 elements are listed in \cref{tab:typeD_rk2_red_decomp}. They can be partitioned into five classes, which are listed in \cref{lem:typeD_classes_red_decomp}. 
	
	Recall that the root in $V$ corresponding to $\dka a\,\; b\dkz$ is $e_a-e_b$ and the root corresponding to $\dka a\,\; -b\dkz$ is $e_a + e_b$ for $1 \leq a < b \leq n$.
	
	If $st$ is in the first class of \cref{lem:typeD_classes_red_decomp}, then the letter $n$ neither appears in $s$ nor in $t$, which means that $st\in \nc(B_{n-1})$ by \cref{cor:typeBD_elements_ncd_as_ncb}.
	The bilinear form $\bild$ restricted to the first $n-1$ coordinates equals the bilinar form $\bilb$ for a vector space of dimension $n-1$. By \cref{prop:typeB_anti-auto_extends} we therefore have $\bild(s,t)=0$ for $st$ in the first class.
	
	Let $v\in V$ be an arbitrary element.
	
	Let $st$ be in the second class of \cref{lem:typeD_classes_red_decomp}, then $s=\dka k\,\;l \dkz$ for $1 \leq k < l <n$. The bilinear form evaluated on $s$ and $v$ is
	\[
	\bild(s,v)=   2\left(\sum_{m=k}^{l-1} e_m\tr\right)\cdot v.
	\]
	Let $t=\dka i\,\; n\dkz$ for $1 \leq |i| <n$.
	Both for $|i|<k$ and $|i|\geq l$ we have that $\bild(s,t)=0$. These are all cases for $i$ if $st$ is in the second class.
	
	If $st$ is in the third class of \cref{lem:typeD_classes_red_decomp}, then $s=\dka k\,\;-l \dkz$ for $1 \leq k < l <n$ and
	\[
	\bilb(s,v)= -2\left(\sum_{m=1}^{k-1} e_m\tr\right)\cdot v + 2\left(\sum_{m=l}^{n-1} e_m\tr\right)\cdot v.
	\]
	The second reflection is $t=\dka i\,\; n\dkz$ with $1\leq k \leq |i| < l$ and the evaluation of the bilinear form is $\bild(s,t)=0$.
	
	Now let $st$ be in the fourth class of \cref{lem:typeD_classes_red_decomp}, in which $s=\dka k\,\;n\dkz$ for $1 \leq k <n$. We get that
	\[
	\bilb(s,v)= -\left(\sum_{m=1}^{k-1} e_m\tr +  e_n\tr\right)\cdot v + \left(\sum_{m=k}^{n-1} e_m\tr\right)\cdot v .
	\]
	Let $t=\dka i\,\; j\dkz$ with $1 \leq |i| < |j| \leq n$. If $t$ satisfies case \emph{e)} or \emph{f)}, then $i$ and $j$ have different signs and $\alpha_t=e_i+e_j$. In both cases, they are multiplied with different signs and $\bild(s,t)=0$.
	In every other case of the fourth class, $i$ and $j$ have the same sign and the summands of $\alpha_t$ get multiplied with the same sign, hence the sum adds up to $\bild(s,t)=0$.
	
	For the last class of \cref{lem:typeD_classes_red_decomp}, let $s=\dka -k\,\; n\dkz$ for $1\leq k <n$. Then 
	\[
	\bilb(s,v)= -\left(\sum_{m=1}^{k-1} e_m\tr \right)\cdot v + \left(\sum_{m=k}^{n} e_m\tr\right)\cdot v.
	\]
	Let  $t=\dka i\,\; j\dkz$ with $1 \leq |i| < |j| \leq n$. In the cases \emph{i)} and \emph{j)}, the root $\alpha_t$ is a sum of a positive and a negative basis vector. But since they get multiplied with the same sign, they sum up to $\bild(s,t)=0$. In the remaining two cases of the last class, $i$ and $j$ have different signs and hence the root $\alpha_t$ has two positive summands. Since $|i|<k<|j|$, they get multiplied with different signs in the evaluation of $\bild(s,t)$, which means $\bild(s,t)=0$.
\end{proof}

\cleardoublepage



\begin{thebibliography}{ABW07}
	
	\bibitem[AB08]{ab}
	P.~Abramenko and K.~S. Brown.
	\newblock {\em Buildings, {T}heory and applications}, volume 248 of {\em
		Graduate Texts in Mathematics}.
	\newblock Springer, New York, 2008.
	
	\bibitem[ABW07]{abw}
	C.~Athanasiadis, T.~Brady, and C.~Watt.
	\newblock Shellability of noncrossing partition lattices.
	\newblock {\em Proc. Amer. Math. Soc.}, 135(4):939--949, 2007.
	
	\bibitem[AR04]{ath_rei}
	C.~Athanasiadis and V.~Reiner.
	\newblock Noncrossing partitions for the group ${D}_n$.
	\newblock {\em SIAM J. Discrete Math.}, 18(2):397--417, 2004.
	
	\bibitem[AR14]{ar}
	R.~M. Adin and Y.~Roichman.
	\newblock On maximal chains in the non-crossing partition lattice.
	\newblock {\em J. Combin. Theory Ser. A}, 125:18--46, 2014.
	
	\bibitem[Arm09]{armstr}
	D.~Armstrong.
	\newblock Generalized noncrossing partitions and combinatorics of {C}oxeter
	groups.
	\newblock {\em Mem. Amer. Math. Soc.}, 202(949):x+159, 2009.
	
	\bibitem[BB05]{bb}
	A.~Bj{\"o}rner and F.~Brenti.
	\newblock {\em Combinatorics of Coxeter groups}, volume 231 of {\em Graduate
		Texts in Mathematics}.
	\newblock Springer, New York, 2005.
	
	\bibitem[BC06]{bes_cor}
	D.~Bessis and R.~Corran.
	\newblock Non-crossing partitions of type {$(e,e,r)$}.
	\newblock {\em Adv. Math.}, 202(1):1--49, 2006.
	
	\bibitem[Bes03]{bes}
	D.~Bessis.
	\newblock The dual braid monoid.
	\newblock {\em Ann. Sci. École Norm. Sup.}, 36(4):647--683, 2003.
	
	\bibitem[BG17]{bg}
	B.~Baumeister and T.~Gobet.
	\newblock Simple dual braids, noncrossing partitions and {M}ikado braids of
	type {$D_n$}.
	\newblock {\em Bull. Lond. Math. Soc.}, 49(6):1048--1065, 2017.
	
	\bibitem[BH99]{bh}
	M.~Bridson and A.~Haefliger.
	\newblock {\em Metric spaces of non-positive curvature}, volume 319 of {\em
		Grundlehren der Mathematischen Wissenschaften [Fundamental Principles of
		Mathematical Sciences]}.
	\newblock Springer-Verlag, Berlin, 1999.
	
	\bibitem[Bia97]{biane}
	P.~Biane.
	\newblock Some properties of crossings and partitions.
	\newblock {\em Discrete Math.}, 175(1-3):41--53, 1997.
	
	\bibitem[BM10]{bra-mcc}
	T.~Brady and J.~McCammond.
	\newblock Braids, posets and orthoschemes.
	\newblock {\em Algebr. Geom. Topol.}, 10(4):2277--2314, 2010.
	
	\bibitem[Bou02]{bour}
	N.~Bourbaki.
	\newblock {\em Lie groups and Lie algebras. Chapters 4–6.}
	\newblock Elements of Mathematics (Berlin). Springer, Berlin, 2002.
	
	\bibitem[Bow95]{bow}
	B.~H. Bowditch.
	\newblock Notes on locally {${\textrm CAT}(1)$} spaces.
	\newblock In {\em Geometric group theory ({C}olumbus, {OH}, 1992)}, volume~3 of
	{\em Ohio State Univ. Math. Res. Inst. Publ.}, pages 1--48. de Gruyter,
	Berlin, 1995.
	
	\bibitem[Bra01]{bra_kpi}
	T.~Brady.
	\newblock {A partial order on the symmetric group and new $K(\pi, 1)$'s for the
		braid groups}.
	\newblock {\em Adv. Math.}, 161(1):20--40, 2001.
	
	\bibitem[Bro89]{bro}
	K.~S. Brown.
	\newblock {\em Buildings}.
	\newblock Springer, New York, 1989.
	
	\bibitem[BW02a]{bra_watt_kpi}
	T.~Brady and C.~Watt.
	\newblock {$K(\pi,1)$}'s for {A}rtin groups of finite type.
	\newblock In {\em Proceedings of the {C}onference on {G}eometric and
		{C}ombinatorial {G}roup {T}heory, {P}art {I} ({H}aifa, 2000)}, volume~94,
	pages 225--250, 2002.
	
	\bibitem[BW02b]{bra_watt_par_ord}
	T.~Brady and C.~Watt.
	\newblock A partial order on the orthogonal group.
	\newblock {\em Comm. Algebra}, 30(8):3749--3754, 2002.
	
	\bibitem[BW08]{bra_watt_lattice}
	T.~Brady and C.~Watt.
	\newblock Non-crossing partition lattices in finite real reflection groups.
	\newblock {\em Trans. Amer. Math. Soc.}, 360(4):1983--2005, 2008.
	
	\bibitem[Car72]{car}
	R.~W. Carter.
	\newblock Conjugacy classes in the {W}eyl groups.
	\newblock {\em Compositio Math.}, 25:1--59, 1972.
	
	\bibitem[Dav08]{davis}
	M.~Davis.
	\newblock {\em The geometry and topology of {C}oxeter groups}, volume~32 of
	{\em London Mathematical Society Monographs Series}.
	\newblock Princeton University Press, Princeton, NJ, 2008.
	
	\bibitem[Ede80]{edel}
	P.~H. Edelman.
	\newblock Chain enumeration and noncrossing partitions.
	\newblock {\em Discrete Math.}, 31(2):171--180, 1980.
	
	\bibitem[EO]{matheonline}
	F.~Embacher and P.~Oberhuemer.
	\newblock mathe-online.at.
	\newblock \url{https://www.mathe-online.at/Mathematica/eingabe_frame_new.html}.
	\newblock accessed 28-09-2018.
	
	\bibitem[FZ02]{fom_zel}
	S.~Fomin and A.~Zelevinsky.
	\newblock Cluster algebras. {I}. {F}oundations.
	\newblock {\em J. Amer. Math. Soc.}, 15(2):497--529, 2002.
	
	\bibitem[Gar97]{gar}
	P.~Garrett.
	\newblock {\em Buildings and classical groups}.
	\newblock Chapman \& Hall, London, 1997.
	
	\bibitem[Gr{\"a}11]{graetzer}
	G.~Gr{\"a}tzer.
	\newblock {\em Lattice theory: foundation}.
	\newblock Birkh{\"a}user/Springer Basel AG, Basel, 2011.
	
	\bibitem[GY02]{gy}
	I.~Goulden and A.~Yong.
	\newblock Tree-like properties of cycle factorizations.
	\newblock {\em J. Combin. Theory Ser. A}, 98(1):106--117, 2002.
	
	\bibitem[Her99]{her}
	P.~Hersh.
	\newblock {\em Decomposition and enumeration in partially ordered sets}.
	\newblock ProQuest LLC, Ann Arbor, MI, 1999.
	\newblock Thesis (Ph.D.)--Massachusetts Institute of Technology.
	
	\bibitem[HKS16]{hks}
	T.~Haettel, D.~Kielak, and P.~Schwer.
	\newblock The 6-strand braid group is {CAT}(0).
	\newblock {\em Geom. Dedicata}, 182(1):263--286, 2016.
	
	\bibitem[HS18]{hs}
	J.~Heller and P.~Schwer.
	\newblock Generalized non-crossing partitions and buildings.
	\newblock {\em Electron. J. Combin.}, 25(1):Paper 1.24, 29, 2018.
	
	\bibitem[Hum90]{hum}
	J.~E. Humphreys.
	\newblock {\em Reflection groups and {C}oxeter groups}, volume~29 of {\em
		Cambridge Studies in Advanced Mathematics}.
	\newblock Cambridge University Press, Cambridge, 1990.
	
	\bibitem[Kre72]{kre}
	G.~Kreweras.
	\newblock Sur les partitions non crois\'ees d'un cycle.
	\newblock {\em Discrete Math.}, 1(4):333--350, 1972.
	
	\bibitem[Lor92]{lorenz}
	F.~Lorenz.
	\newblock {\em Lineare {A}lgebra. {II}}.
	\newblock Bibliographisches Institut, Mannheim, 3rd ed. edition, 1992.
	
	\bibitem[Mar14]{marq}
	T.~Marquis.
	\newblock Conjugacy classes and straight elements in {C}oxeter groups.
	\newblock {\em J. Algebra}, 407:68--80, 2014.
	
	\bibitem[McC06]{mcc_surp}
	J.~McCammond.
	\newblock Noncrossing partitions in surprising locations.
	\newblock {\em Amer. Math. Monthly}, 113(7):598--610, 2006.
	
	\bibitem[Noy98]{noy}
	M.~Noy.
	\newblock Enumeration of noncrossing trees on a circle.
	\newblock In {\em Proceedings of the 7th {C}onference on {F}ormal {P}ower
		{S}eries and {A}lgebraic {C}ombinatorics ({N}oisy-le-{G}rand, 1995)}, volume
	180, pages 301--313, 1998.
	
	\bibitem[PW93]{pen_water}
	R.~C. Penner and M.~S. Waterman.
	\newblock Spaces of {RNA} secondary structures.
	\newblock {\em Adv. Math.}, 101(1):31--49, 1993.
	
	\bibitem[Rei97]{rei}
	V.~Reiner.
	\newblock Non-crossing partitions for classical reflection groups.
	\newblock {\em Discrete Math.}, 177(1-3):195--222, 1997.
	
	\bibitem[Sim00]{simion_ncp}
	R.~Simion.
	\newblock Noncrossing partitions.
	\newblock {\em Discrete Math.}, 217(1-3):367--409, 2000.
	\newblock Formal power series and algebraic combinatorics (Vienna, 1997).
	
	\bibitem[Sim03]{simion_ass}
	R.~Simion.
	\newblock A type-{B} associahedron.
	\newblock {\em Adv. in Appl. Math.}, 30(1-2):2--25, 2003.
	\newblock Formal power series and algebraic combinatorics (Scottsdale, AZ,
	2001).
	
	\bibitem[Spe97]{speicher}
	R.~Speicher.
	\newblock Free probability theory and non-crossing partitions.
	\newblock {\em S\'{e}m. Lothar. Combin.}, 39:Art. B39c, 38 pp., 1997.
	
	\bibitem[Sta99]{staecz}
	R.~P. Stanley.
	\newblock {\em Enumerative combinatorics. {V}ol. 2}, volume~62 of {\em
		Cambridge Studies in Advanced Mathematics}.
	\newblock Cambridge University Press, Cambridge, 1999.
	
	\bibitem[Sta12]{staece}
	R.~P. Stanley.
	\newblock {\em Enumerative combinatorics, {V}ol. 1}, volume~49 of {\em
		Cambridge Studies in Advanced Mathematics}.
	\newblock Cambridge University Press, Cambridge, 2nd ed. edition, 2012.
	
	\bibitem[Tho18]{thom}
	H.~Thomas.
	\newblock Coxeter groups and quiver representations.
	\newblock In {\em Surveys in representation theory of algebras}, volume 716 of
	{\em Contemp. Math.}, pages 173--186. Amer. Math. Soc., Providence, RI, 2018.
	
	\bibitem[Wal81]{walker}
	J.~W. Walker.
	\newblock {\em Topology and Combinatorics of Ordered Sets}.
	\newblock ProQuest LLC, Ann Arbor, MI, 1981.
	\newblock Thesis (Ph.D.)--Massachusetts Institute of Technology.
	
\end{thebibliography}

\end{document}